\documentclass[12pts]{amsart}

\usepackage{amscd,latexsym,amsthm,amsfonts,amssymb,amsmath,amsxtra}
\usepackage[mathscr]{eucal}
\newcommand{\BA}{{\mathbb {A}}}

\newcommand{\BC}{{\mathbb {C}}}

\newcommand{\diag}{{\mathrm{diag}}}

\newcommand{\GL}{{\mathrm{GL}}}

\newcommand{\Hom}{{\mathrm{Hom}}}

\newcommand{\Ind}{{\mathrm{Ind}}}

\newcommand{\I}{{\mathrm{I}}}

\newcommand{\ind}{{\mathrm{ind}}}

\renewcommand{\Re}{{\mathrm{Re}}}

\newcommand{\Res}{{\mathrm{Res}}}

\newcommand{\SL}{{\mathrm{SL}}}

\newcommand{\SO}{{\mathrm{SO}}}

\newcommand{\Sp}{{\mathrm{Sp}}}

\newtheorem{thm}{Theorem}[section]
\newtheorem{cor}[thm]{Corollary}
\newtheorem{lem}[thm]{Lemma}
\newtheorem{prop}[thm]{Proposition}

\newtheorem{rmk}[thm]{Remark}

\begin{document}
\renewcommand{\theequation}{\arabic{equation}}
\numberwithin{equation}{section}

\title{Double Descent in Classical Groups}

\author{David Ginzburg}

\address{School of Mathematical Sciences, Sackler Faculty of Exact Sciences, Tel-Aviv University, Israel
69978} \email{ginzburg@post.tau.ac.il}

\thanks{This research was supported by the ISRAEL SCIENCE FOUNDATION
	(grant No. 461/18).}

\author{David Soudry}
\address{School of Mathematical Sciences, Sackler Faculty of Exact Sciences, Tel-Aviv University, Israel
69978} \email{soudry@post.tau.ac.il}

\subjclass{Primary 11F70 ; Secondary 22E55}



\keywords{Eisenstein series, Speh representations, Cuspidal
automorphic representations, Fourier coefficients}

\begin{abstract}
Let $\BA$ be the ring of adeles of a number field $F$. Given a self-dual irreducible, automorphic, cuspidal representation $\tau$ of $\GL_n(\BA)$, with trivial central characters, we construct its full inverse image under the weak Langlands functorial lift from the appropriate split classical group $G$. We do this by a new automorphic descent method, namely the double descent. This method is derived from the recent generalized doubling integrals of Cai, Friedberg, Ginzburg and Kaplan \cite{CFGK17}, which represent the standard $L$-functions for $G\times \GL_n$. Our results are valid also for double covers of symplectic groups.
\end{abstract}
\maketitle
\setcounter{section}{-1}

\section{Introduction}

Let $F$ be a number field, and let $\BA$ be its ring of adeles. Let $\tau$ be an irreducible, automorphic, cuspidal representation of $\GL_n(\BA)$. Assume that $\tau$ is self-dual with a trivial central character. Then we know that $\tau$ is the functorial lift of an irreducible, automorphic, cuspidal representation $\sigma$ of $G(\BA)$, where $G$ is an appropriate symplectic group $\Sp_m$, or a split special orthogonal group $\SO_m$, viewed as algebraic groups defined over $F$. Moreover, in \cite{GRS11}, we found an explicit construction of such a generic $\sigma$. Our main goal in this paper is to construct explicitly all irreducible, automorphic, cuspidal representations $\sigma$ of $G(\BA)$, which lift (weakly) to $\tau$. More precisely, for each isomorphism class of such a $\sigma$, we will construct a certain unique representative. Thanks to Arthur (Theorem 1.5.2 in \cite{A13}), we know that when $G$ is symplectic, or odd orthogonal these multiplicities are 1, and for $G$ even orthogonal, these multiplicities are 1, or 2.

We recall that in \cite{GRS11}, we constructed the generic representation $\sigma$ above by our automorphic descent method, derived from the Rankin-Selberg integrals which represent the $L$-functions $L(\sigma\times\tau,s)$, and are nontrivial exactly when $\sigma$ is generic. See \cite{GRS11}, Sec. 3.4, 3.6, for the definition of these integrals and of the descent construction. For example, we recall that if $L(\tau,\wedge^2,s)$ has a pole at $s=1$, then necessarily $n=2n'$ is even and $G=\SO_m$, $m=2n'+1=n+1$. We will realize the split classical groups as matrix groups, as in the beginning of Sec. 1.1 below. In this case, we took an Eisenstein series $E(f_{\tau,s})$, corresponding to a smooth, holomorphic section $f_{\tau,s}$ of the parabolic induction
$$
\rho_{\tau,s}=\Ind_{Q_n(\BA)}^{\SO_{2n}(\BA)}\tau|\det\cdot|^s.
$$
Here, $Q_n$ is the standard parabolic subgroup of $\SO_{2n}$, with Levi part isomorphic to $\GL_n$. As the section varies, $E(f_{\tau,s})$ has a simple pole at $s=\frac{1}{2}$. (In \cite{GRS11}, we used a shift $s-\frac{1}{2}$ and considered the pole at $s=1$.) Next, we applied to the residues $Res_{s=\frac{1}{2}}E(f_{\tau,s})$ a Fourier coefficient with respect to a character $\psi_{U_{1^{n'-1}}}$ of $U_{1^{n'-1}}(\BA)$, trivial on $U_{1^{n'-1}}(F)$, where $U_{1^{n'-1}}$ is the unipotent radical of the standard parabolic subgroup $Q_{1^{n'-1}}$ of  $\SO_{2n}$, with Levi part isomorphic to $\GL_1^{n'-1}\times \SO_{2n'+2}$. The character $\psi_{U_{1^{n'-1}}}$ is such that it is stabilized by $G(\BA)=\SO_{2n'+1}(\BA)$, where $\SO_{2n'+1}(\BA)$ is embedded inside $\SO_{2n'+2}(\BA)$ as the stabilizer of an anisotropic vector. Denote this embedding by $t$. We defined the descent of $\tau$, $\mathcal{D}_\psi(\tau)$, as the space of automorphic functions on $G(\BA)$, spanned by the above $\psi_{U_{1^{n'-1}}}$ - Fourier coefficients applied to the residues $Res_{s=\frac{1}{2}}E(f_{\tau,s})$, as the section varies, viewing these Fourier coefficients as automorphic functions on $G(\BA)$ through the embedding $t$. Then one of the main results of \cite{GRS11}, specialized to this case, is
\begin{thm}\label{thm 0.1}
	\begin{enumerate}
		\item $\mathcal{D}_\psi(\tau)$ is nontrivial.
		\item All elements of $\mathcal{D}_\psi(\tau)$ are cuspidal.
		\item Decompose $\mathcal{D}_\psi(\tau)$ into a direct sum of irreducible, automorphic, cuspidal representations $\pi$ of $G(\BA)$,
		 
		$$
		\mathcal{D}_\psi(\tau)=\oplus_{\sigma\in J} \sigma.
		$$
		Then this decomposition is multiplicity free. Every $\sigma\in J$ is globally generic and lifts weakly to $\tau$.
		\item Let $\sigma$ be an irreducible, automorphic, cuspidal representation  $G(\BA)$, which is globally generic and lifts weakly to $\tau$. Then there is a (unique) representation $\sigma'\in J$, such that $\sigma$ has a nontrivial $L^2$-pairing with $\bar{\sigma}'$.
\end{enumerate}
\end{thm} 
The main idea behind the descent in this theorem is the following. Let $\sigma$ be an irreducible, automorphic, cuspidal representation of $G(\BA)$. Denote the $\psi_{U_{1^{n'-1}}}$ - Fourier coefficient above of $E(f_{\tau,s})$ by $E^{\psi_{U_{1^{n'-1}}}}(f_{\tau,s})$. Consider, for $\varphi_\sigma$ in the space of $\sigma$, the integral
\begin{equation}\label{0.1}
\mathcal{L}(\varphi_\sigma, f_{\tau,s})=\int_{G(F)\backslash G(\BA)}\varphi_\sigma(g)E^{\psi_{U_{1^{n'-1}}}}(f_{\tau,s})(t(g))dg.
\end{equation}
This integral is identically zero, unless $\sigma$ is globally generic. Assume that $\sigma$ is globally generic. Then the global Rankin-Selberg integral \eqref{0.1} represents the $L$-function $L(\sigma\times \tau,s+\frac{1}{2})$. See \cite{S93} for details. For decomposable data, unramified and properly normalized outside a finite set of places $S$, containing the Archimedean places, we have
\begin{equation}\label{0.2}
\mathcal{L}(\varphi_\sigma, f_{\tau,s})=\prod_{v\in S}
\mathcal{L}_v(\varphi_{\sigma_v}, f_{\tau_v,s})\frac{L^S(\sigma\times \tau,s+\frac{1}{2})}{L^S(\tau,\wedge^2,2s+1)}.
\end{equation}
The product of "local integrals" over $v\in S$ can be nicely controlled, so that, in particular, if $L^S(\sigma\times \tau,s+\frac{1}{2})$ has a pole at $s=\frac{1}{2}$, then $\mathcal{L}(\varphi_\sigma, f_{\tau,s})$ has a pole at $s=\frac{1}{2}$, for some choice of data. This means that the following pairing is nontrivial,
\begin{equation}\label{0.3}
\int_{G(F)\backslash G(\BA)}\varphi_\sigma(g)(Res_{s=\frac{1}{2}}E(f_{\tau,s}))^{\psi_{U_{1^{n'-1}}}}(t(g))dg.
\end{equation}
Thus, if $\sigma$ weakly lifts to $\tau$, then it pairs nontrivially (in the sense of \eqref{0.3} being nontrivial) with the descent $\mathcal{D}_\psi(\tau)$.
The Rankin-Selberg integral \eqref{0.1} and \eqref{0.2} work for any irreducible, automorphic, cuspidal representation $\tau$ of $\GL_n(\BA)$, and similar variants exist for $\SO_{2n'+1}\times \GL_k$, for any $n', k$.   
We proved similar theorems for special even orthogonal groups, split, or quasi-split, as well as for quasi-split unitary groups, symplectic groups and their double covers.

Our main goal in this paper is to construct, in a certain sense, to be made precise later, all irreducible, automorphic, cuspidal representations $\sigma$ of $G(\BA)$, globally generic, or otherwise, which lift weakly to $\tau$. We will do this here for $G$ a split orthogonal group, a symplectic group, or its double cover. For this, we will use the recent generalized doubling integrals in \cite{CFGK17}, which represent the standard $L$-functions $L(\sigma\times\tau,s)$, for any pair of irreducible, automorphic, cuspidal representations $\sigma$ and $\tau$ of $G(\BA)$ and $\GL_n(\BA)$, respectively. The case where $n=1$ was done by Piatetski-Shapiro and Rallis, who discovered the douling integrals for $G\times \GL_1$. See \cite{PSR87}. There is no requirement on $\sigma$ to be globally generic, or to have some other global model, such as a global Bessel model, or a global Fourier-Jacobi model. In fact, the doubling integral unfolds to an Eulerian integral, where the dependence on $\sigma$ is through matrix coefficients of $\sigma$. We review this in Sec. 1.2. For this introduction, we explain this in brief for a split special odd orthogonal group  $G=\SO_m$, $m=2m'-1$. Consider the parabolic induction
$$
\rho_{\Delta(\tau,m),s}=\Ind_{Q_{nm}(\BA)}^{\SO_{2nm}(\BA)}\Delta(\tau,
m)|\det\cdot|^s,
$$		
where $Q_{nm}$ is the standard parabolic subgroup of $\SO_{2nm}$, with Levi part isomorphic to $\GL_{nm}$; $\Delta(\tau,m)$ is the Speh representation of $\GL_{nm}(\BA)$, corresponding to $\tau$. For a smooth, holomorphic section $f_{\Delta(\tau,m),s}$ of $\rho_{\Delta(\tau,m),s}$, let $E(f_{\Delta(\tau,m),s})$ be the corresponding Eisenstein series on $\SO_{2nm}(\BA)$. We apply to $E(f_{\Delta(\tau,m),s})$ a certain Fourier coefficient with respect to a character $\psi_{U_{m^{n-1}}}$ of $U_{m^{n-1}}(\BA)$, where $U_{m^{n-1}}$ is the unipotent radical of the standard parabolic subgroup $Q_{m^{n-1}}$ of 
 $\SO_{2nm}$, whose Levi part $M_{m^{n-1}}$ is isomorphic to $\GL_m^{n-1}\times \SO_{2m}$. The character $\psi_{U_{m^{n-1}}}$ is stabilized by the adele points of a subgroup of $M_{m^{n-1}}$, which is isomorphic to $\SO_m\times \SO_m$. Denote by $\mathcal{F}_\psi(E(f_{\Delta(\tau,m),s}))$ the resulting Fourier coefficient, viewed, through the last isomorphism, as an automorphic function on the "doubled group"  $\SO_m(\BA)\times \SO_m(\BA)$. As in \eqref{0.1}, we pair $\mathcal{F}_\psi(E(f_{\Delta(\tau,m),s}))$ with an irreducible, automorphic, cuspidal representation of $\SO_m(\BA)\times \SO_m(\BA)$. Thus, let
$\sigma,  \pi$ be two irreducible, automorphic, cuspidal representations
of $\SO_m(\BA)$. The integrals of the generalized doubling method of \cite{CFGK17} have the following form,\\
\\
$\mathcal{L}(f_{\Delta(\tau,m),s},\varphi_\sigma, \varphi_\pi)=$
\begin{equation}\label{0.4}
\int_{\SO_m(F)\times \SO_m(F)\backslash
	\SO_m(\BA)\times \SO_m(\BA)}\mathcal{F}_\psi(E(f_{\Delta(\tau,
	m),s}))(g,h)\varphi_\sigma(g)\varphi_\pi(h)dgdh,
\end{equation}
where $\varphi_\sigma, \varphi_\pi $ are in the spaces of $\sigma, \pi$, respectively. 
The unfolding of this global integral, carried out in \cite{CFGK17}, shows that it is identically zero, unless the following pairing is nontrivial
$$
c(\varphi_\sigma,\varphi_\pi^\iota)=\int_{\SO_m(F)\backslash
	\SO_m(\BA)}\varphi_\sigma(g)\varphi_\pi(g^\iota)dg,
$$
where $\iota$ is a certain outer conjugation of order 2 of $\SO_m$ (see \eqref{1.10.1}). Thus, let us take $\pi=\bar{\sigma}^\iota$. Then it is shown in \cite{CFGK17} that for decomosable data, appropriately normalized, there is an expression for $\mathcal{L}(f_{\Delta(\tau,m),s},\varphi_\sigma, \bar{\xi}^\iota_\sigma)$  similar to \eqref{0.2}, with the product outside $S$ being
$$
\frac{L^S(\sigma\times \tau, s+\frac{1}{2})}{D^{\SO_{2nm},S}_\tau(s)},
$$
where $D^{\SO_{2nm},S}_\tau(s)$ is the normalizing factor of the Eisenstein series above, on $\SO_{2nm}(\BA)$. See right after \eqref{1.10.8}. This is valid for any pair of irreducible, automorphic, cuspidal representations $\sigma, \tau$ of $\SO_m(\BA)$, $\GL_n(\BA)$, respectively. 

Assume that $n=m-1=2m'-2$, and $\tau$ is such that $L(\tau,\wedge^2,s)$ has a pole at $s=1$. Assume that $\sigma$ lifts weakly to $\tau$. Then $L^S(\sigma\times \tau, s+\frac{1}{2})$ has a pole at $s=\frac{1}{2}$. Again, it follows that $\mathcal{F}_\psi(E(f_{\Delta(\tau,
	m),s}))$ has a pole at $s=\frac{1}{2}$, and the following pairing along 
$\SO_m\times \SO_m$ is nontrivial (as data vary),
\begin{equation}\label{0.5}
\int_{\SO_m(F)\times \SO_m(F)\backslash
	\SO_m(\BA)\times \SO_m(\BA)}\mathcal{F}_\psi(Res_{s=\frac{1}{2}}(E(f_{\Delta(\tau,
	m),s})))(g,h)\varphi_\sigma(g)\bar{\xi}_\sigma(h^\iota)dgdh.
\end{equation}
This suggests that we consider the following space of smooth automorphic functions on $\SO_m(\BA)\times \SO_m(\BA)$,
\begin{equation}\label{0.6}
\mathcal{D}\mathcal{D}_\psi(\tau)=\{\mathcal{F}_\psi(Res_{s=\frac{1}{2}}(E(f_{\Delta(\tau,
	m),s})))\Big|_{\SO_m(\BA)\times \SO_m(\BA)}\}.
\end{equation}
We call this space, and the resulting automorphic representation of $\SO_m(\BA)\times \SO_m(\BA)$, the double descent of $\tau$. This term was suggested by Erez Lapid. The double descent of $\tau$ is the main object of study of our paper. We will now state our main theorems. We have a similar construction of the double descent
$\mathcal{D}\mathcal{D}_\psi(\tau)$, on $G(\BA)\times G(\BA)$, and similar theorems for symplectic groups, their double covers and for special split even orthogonal groups, for a given irreducible, automorphic, cuspidal representation $\tau$ of $\GL_n(\BA)$, which is self-dual and has a trivial central character.

Theorem \ref{thm 0.1} implies that  
$$
\mathcal{D}\mathcal{D}_\psi(\tau)\neq 0.
$$
 This is a highly non-trivial fact. We don't know how to prove this without Theorem \ref{thm 0.1}. See Theorem \ref{thm 2.2} for the proof. The main work of this paper (Sec. 4 - Sec. 11) is to prove
\begin{thm}\label{thm 0.2}
	Let $\tau$ be an irreducible, automorphic, cuspidal representation of $\GL_n(\BA)$, such that $L(\tau,\wedge^2,s)$ has a pole at $s=1$ ($n=m-1=2m'-2$). Then the double descent of $\tau$ is a (nontrivial) cuspidal representation of $G(\BA)\times G(\BA)$.
\end{thm}
This implies the following theorem in a relatively simple way.
\begin{thm}\label{thm 0.3}
	We have a decomposition of the double descent of $\tau$ into a direct sum of irreducible, automorphic, cuspidal representations of $G(\BA)\times G(\BA)$ 
	\begin{equation}\label{0.7}
	\mathcal{D}\mathcal{D}_\psi(\tau)=\oplus_{\sigma\in J}\sigma\otimes\bar{\sigma}^\iota.
	\end{equation}
	The decomposition is multiplicity free. Let $\sigma$ be an irreducible, automorphic, cuspidal representation of $G(\BA)$, which lifts weakly to $\tau$. Then $\bar{\sigma}\otimes\sigma^\iota$ has a nontrivial $L^2$ pairing with $\mathcal{D}\mathcal{D}_\psi(\tau)$, and hence there is a unique $\sigma'\in J$, such that $\sigma$ is isomorphic to $\sigma'$.	
\end{thm}
In Sec. 12, 13, we prove the unramified correspondence, that is
\begin{thm}\label{thm 0.4}
Let $\sigma$ be an irreducible, automorphic, cuspidal representation of $G(\BA)$, such that $\sigma\in J$, that is $\sigma\otimes\bar{\sigma}^\iota$ is a direct summand in \eqref{0.7}. Then $\sigma$ lifts to $\tau$ at all finite places where $\tau$ is unramified.
\end{thm}
In this paper, we work in a more general set-up, which we will need for our next project, where we will try to achieve by double descent the irreducible, automorphic representations of $G(\BA)$ which lift to Speh representations. Returning to the example of $G=\SO_m$, $m=2m'-1$, we start with an irreducible, automorphic, cuspidal representation $\tau$ of $\GL_n(\BA)$, which is self-dual and has a trivial central character. Let $m$ be a positive integer. Consider an Eisenstein series $E(f_{\Delta(\tau,m),s})$ on $\SO_{2nm}(\BA)$, corresponding to $\rho_{\Delta(\tau,m),s}$. The set of all possible poles of such Eisenstein series, with $Re(s)\geq 0$, is determined in \cite{JLZ13}. The set depends on whether $L(\tau,\wedge^2,s)$, or $L(\tau,\vee^2,s)$ has a pole at $s=1$. In the first case, we get the set of points 
$$
e^{\SO_{2nm}}_k(\wedge^2)=k-\frac{1}{2},\ k=1,2,...,\frac{m+1}{2}.
$$
In the second case,
$$
e^{\SO_{2nm}}_k(\vee^2)=k,\ k=1,2,...,\frac{m-1}{2}.
$$	
For $\eta=\wedge^2, \vee^2$, we consider the leading term $a(f_{\Delta(\tau,m),e^{\SO_{2nm}}_k(\eta)})$  of the Laurent expansion around $e_k^{\SO_{2nm}}(\eta)$ of the Eisenstein series $E(f_{\Delta(\tau,m),s})$. We view it as an automorphic function on $\SO_{2nm}(\BA)$. We do not assume that the points $e_k^{\SO_{2nm}}(\eta)$ are poles of the Eisenstein series $E(f_{\Delta(\tau,m),s})$. We simply study their leading terms at these points. Let $A(\Delta(\tau,m),\eta, k)$ denote the space generated by these leading terms. It is an automorphic module of $\SO_{2nm}(\BA)$. Denote, for each point $e^{\SO_{2nm}}_k(\eta)$,
$$
\mathcal{E}_k(f_{\Delta(\tau,m),e^{\SO_{2nm}}_k(\eta)})=
\mathcal{F}_\psi(a(f_{\Delta(\tau,m),e^{\SO_{2nm}}_k(\eta)})).
$$
We view these functions as automorphic functions on $\SO_m(\BA)\times \SO_m(\BA)$. Denote the module generated by them by $\mathcal{E}_k(\Delta(\tau,m),\eta)$.
Note that the case of functoriality corresponds to $e_1^{\SO_{2(m-1)m}}(\wedge^2)=\frac{1}{2}$. Thus, $\mathcal{E}_k(\Delta(\tau,m),\eta)$ is a generalization of the double descent of $\tau$. We prove in Prop. \ref{prop 3.3}, that $\mathcal{E}_k(\Delta(\tau,m),\eta)=0$, unless $k$ is relatively small. For example, if $\mathcal{E}_k(\Delta(\tau,m),\wedge^2)$ is nonzero, then we must have that $k\leq \frac{m}{2n}+\frac{1}{2}$. We fix such a $k_0$. Let $1\leq r<\frac{m}{2}$, and let $U_r^{\SO_m}$ be the unipotent radical of the standard parabolic subgroup $Q_r^{\SO_m}$ of $\SO_m$, whose Levi part is isomorphic to $\GL_r\times \SO_{m-2r}$. In Sec. 4 - Sec. 11, we compute the constant term of $\mathcal{E}_{k_0}(\Delta(\tau,m),\eta)$ along $U_r^{\SO_m}\times I_m$, that is along $U_r^{\SO_m}$ inside the first copy of $\SO_m$. In Cor. \ref{cor 10.3}, we will further assume that $m$ and $n$ are related as follows.
\begin{equation}\label{0.8}
\end{equation}
\begin{enumerate}
\item $m-1=(2k_0-1)n,\ \  \eta=\wedge^2$,\\
\item $m-1=2k_0n,\ \eta=\vee^2, n\ \textit{even}$,\\
\item $m-1=2k_0(n-1),\ \eta=\vee^2, n\ \textit{odd}$.
\end{enumerate}
 
 These assumptions, in the first two cases, are compatible with our future goal of achieving by double descent irreducible, automorphic representations of $\SO_m(\BA)$, which lift to Speh representations. We continue with this assumption from Cor. \ref{cor 10.3} until the end of the paper. Note that the first case in \eqref{0.8}, with $k_0=1$, corresponds to the case of functoriality (from $\SO_m$ to $\GL_{m-1}$). In Theorem \ref{thm 11.2}, we get a nice expression of the constant term above, along $U_r^{\SO_m}\times I_m$. For an element $\xi\in A(\Delta(\tau,m),\eta, k)$, this constant term is expressed in terms of $\xi^{U^{\SO_{2nm}}}_{(2n-1)r}$, the constant term of $\xi$, along the unipotent radical $U^{\SO_{2nm}}_{(2n-1)r}$ of $\SO_{2nm}$ of the standard parabolic subgroup of $\SO_{2nm}$, with Levi part isomorphic to $\GL_{(2n-1)r}\times \SO_{2(nm-(2n-1)r}$. This exhibits a kind of a tower property. We believe that Theorem \ref{thm 11.2} will be important for achieving by double descent irreducible, automorphic representations of $\SO_m(\BA)$, which lift to Speh representations of $\GL_{m-1}(\BA)$. The most encouraging sign that this would hopefully work is Theorem \ref{thm 12.3}.
 Denote in \eqref{0.8}, $\mu_0=2k_0-1$, when $\eta=\wedge^2$ and $\mu_0=2k_0$, when $\eta=\vee^2$. Assume, for example, that we are in the first case of \eqref{0.8}, and write $n=2n'$. Let $v$ be a finite place of $F$, where $\tau_v$ is unramified, and write $\tau_v$ as a parabolic induction from the standard Borel subgroup
 $$
 \tau_v=\chi_1\times\cdots\times \chi_{n'}\times \chi^{-1}_{n'}\times\cdots\times\chi^{-1}_1,
 $$
 where $\chi_i$ are unramified characters of $F_v^*$. The local version at $v$ of
 $A(\Delta(\tau,m),\wedge^2,k_0)$ is the representation 
 \begin{multline}\nonumber
 \rho^{\SO_{2nm}}_{\chi,,\wedge^2,k_0}=
 \Ind^{\SO_{2nm}(F_v)}_{Q_{(m+\mu_0)^{n'},(m-\mu_0)^{n'}}(F_v)}\chi\\
 \chi=[(\otimes_{i=1}^{n'}\chi_i\circ det_{\GL_{m+\mu_0}}) \otimes  (\otimes_{i=1}^{n'}\chi_i\circ det_{\GL_{m-\mu_0}})],
 \end{multline}
where $Q_{(m+\mu_0)^{n'},(m-\mu_0)^{n'}}$ is the standard parabolic subgroup of $\SO_{2nm}$ with Levi part isomorphic to $\GL_{m+\mu_0}^{n'}\times \GL_{m-\mu_0}^{n'}$. The local version at $v$ of $\mathcal{E}_{k_0}(\Delta(\tau,m),\wedge^2)$ is the Jacquet module with respect to $U_{m^{n-1}}(F_v)$ and the local character at $v$, $(\psi_v)_{U_{m^{n-1}}}$, of $\psi_{U_{m^{n-1}}}$. Then we prove in this case
	\begin{thm}\label{thm 0.5}
		We have
\begin{multline}\nonumber
J_{(\psi_v)_{U_{m^{n-1}}(F_v)}}(\rho^{\SO_{2nm}}_{\chi,\wedge^2,k_0})\cong\\ (\Ind_{Q_{\mu_0^{n'}}(F_v)}^{\SO_m(F_v)}\chi_1\circ det_{\GL_{\mu_0}}\otimes\cdots\otimes\chi_{n'}\circ det_{\GL_{\mu_0}})\otimes\\ \otimes(\Ind_{Q_{\mu_0^{n'}}(F_v)}^{\SO_m(F_v)}\chi_1\circ det_{\GL_{\mu_0}}\otimes\cdots\otimes\chi_{n'}\circ det_{\GL_{\mu_0}}),
\end{multline}
In particular, if $\pi_1$, $\pi_2$ are two irreducible, unramified representations of $\SO_m(F_v)$, such that
$$
\Hom_{\SO_m(F_v)\times \SO_m(F_v)}(\pi_1\otimes \pi_2,
J_{(\psi_v)_{U_{m^{n-1}}(F_v)}}(\rho^{\SO_{2nm}}_{\chi,\wedge^2,k_0})\neq 0,
$$
then $\pi_1\cong \pi_2$ is the unramified constituent of 
$$
\Ind_{Q_{\mu_0^{n'}}(F_v)}^{\SO_m(F_v)}(\chi_1\circ det_{\GL_{\mu_0}}\otimes\cdots\otimes\chi_{n'}\circ det_{\GL_{\mu_0}}),
$$
and thus it lifts to the unramified constituent of
$$
\Ind_{P_{\mu_0^n}(F_v)}^{\GL_{\mu_0n}(F_v)}(\chi_1\circ det_{\GL_{\mu_0}}\otimes\cdots\otimes\chi_{n'}\circ det_{\GL_{\mu_0}}\otimes
\chi^{-1}_{n'}\circ det_{\GL_{\mu_0}}\otimes\cdots\otimes \chi^{-1}_1\circ det_{\GL_{\mu_0}}). 
$$
\end{thm}

If we denote the local factor at $v$ of the Speh representation $\Delta(\tau,\mu_0)$ by $\Delta(\tau_v,\mu_0)$, the theorem says that the local analog at $v$ of the generalization of the double descent applied to the first case of \eqref{0.8} is the local Speh representation $\Delta(\tau_v,\mu_0)$.

\section{Preliminaries and notation}

In this section, we review the global integrals of the generalized doubling method of \cite{CFGK17}. We borrow the same notation and conventions from \cite{GS18}. We then consider the Fourier coefficient of the Eisenstein series on the appropriate group $H$, which appears in the global integral. We take these integrals, such that they represent the standard $L$-function for $G\times \GL_n$. Then the Eisenstein series depends on an irreducible, automorphic, cuspidal representation $\tau$ of $\GL_n(\BA)$. When we choose ($H$ and) $G$ and $\tau$, such that $\tau$ is lifted from $G$, the Eisenstein series has a simple pole at $s=\frac{1}{2}$, and the above Fourier coefficient of the corresponding residue at $s=\frac{1}{2}$ becomes the main object of this paper. \\
\\
{\bf 1. The groups}\\

Let $F$ be a number field and $\BA$ its ring of adeles. We will
consider symplectic groups $\Sp_{2k}$ and split orthogonal groups
$\SO_k$ over $F$. We will realize these groups as matrix groups in
the following standard way. Let $w_k$ denote the $k\times k$
permutation matrix which has $1$ along the main anti-diagonal. Then
the corresponding matrix algebraic groups are
$$      
\Sp_{2k}=\{g\in \GL_{2k} \ |\
{}^tg\begin{pmatrix}&w_k\\-w_k\end{pmatrix}g=\begin{pmatrix}&w_k\\-w_k\end{pmatrix}
\},
$$
$$              
 \SO_k=\{g\in \SL_k \ |\ {}^tgw_kg=w_k\}.
$$
We will also view these groups as algebraic groups over the completion of $F$ at $v$, $F_v$, for each place $v$ of $F$. Put
\begin{equation}\label{1'.1}
J_{2k}=\begin{pmatrix}&w_k\\-w_k\end{pmatrix}.
\end{equation}
It will sometimes be convenient to denote $J_{\SO_k}=w_k$, and  $J_{2k}=J_{\Sp_{2k}}$, and
we will oftentimes denote such a group by $H_\ell$, where $\ell$ is the
number of variables of the corresponding anti-symmetric, or
symmetric form. We let $H_\ell$, considered as an algebraic group over $F$ (resp. over $F_v$) act on the $\ell$-dimensional column space over $F$ (resp. over $F_v$).  
For each place $v$, we fix a maximal compact subgroup $K_{H_\ell(F_v)}$ of $H_\ell(F_v)$. When $v$ is finite, we will take it to be $H_\ell(\mathcal{O}_v)$, where $\mathcal{O}_v$ is the ring of integers inside $F_v$. We denote
\begin{equation}\label{1'.2}
K_{H_\ell(\BA)}=\prod_v K_{H_\ell(F_v)}\subset H_\ell(\BA).
\end{equation}
We will consider also double covers of symplectic groups over the local fields $F_v$. We will denote these groups by $\Sp^{(2)}_{2k}(F_v)$. In case $F_v=\BC$, $\Sp_{2k}^{(2)}(\BC)=\Sp_{2k}(\BC)\times 1$. For the other places $v$, we realize $\Sp^{(2)}_{2k}(F_v)$ as $\Sp_{2k}(F_v)\times \{\pm 1\}$, using the normalized Ranga Rao cocycle, corresponding to the standard Siegel parabolic subgroup (\cite{Rao93}). Although $\Sp^{(2)}_{2k}(F_v)$ is not the group of $F_v$ - points of an algebraic group defined over $F_v$, we will apply the language of algebraic groups by saying that the $F_v$ - points of $\Sp_{2k}^{(2)}$ is the group $\Sp_{2k}^{(2)}(F_v)$. 
We know that unipotent subgroups $U_v$ of $\Sp_{2k}(F_v)$ split in $\Sp^{(2)}_{2k}(F_v)$, and if $U_v$ consists of upper triangular unipotent matrices, then the Ranga Rao cocycle is trivial on $U_v\times U_v$. Thus, $U_v\times 1$ is a subgroup of $\Sp^{(2)}_{2k}(F_v)$.  We will identify $U_v$ and $U_v\times 1$.

We will consider the corresponding double cover, $\Sp^{(2)}_{2k}(\BA)$, of $\Sp_{2k}(\BA)$. See \cite{GS18}, Sec. 1.1. In particular, we recall that for each place $v$, outside a finite set of places $S_0=S_{0,F}$, containing the Archimedean places, there is a unique embedding 
$\zeta_v: K_{\Sp_{2k}(F_v)}\rightarrow \Sp^{(2)}_{2k}(F_v)$ of the form 
\begin{equation}\label{1'.3}
\zeta_v(r)=(r,\lambda_v(r)). 
\end{equation}
We will denote  $K_{\Sp^{(2)}_{2k}(F_v)}=\zeta_v(K_{\Sp_{2k}(F_v)})$. In \cite{Sw90}, Prop. 1.6.6, Sweet extends the function $\lambda_v$ to $\Sp_{2k}(F_v)$. In particular, he shows that $\lambda_v$ is trivial on $Q_k(F_v)$ ($Q_k$ is the Siegel parabolic subgroup of $\Sp_{2k}$). In the places $v\in S_0$, we will take $K_{\Sp^{(2)}_{2k}(F_v)}$ to be the inverse image of $K_{\Sp_{2k}(F_v)}$ inside $\Sp^{(2)}_{2k}(F_v)$. Also, for $v\in S_0$, we define (as Sweet does in \cite{Sw90}) $\lambda_v=1$, as a function on $\Sp_{2k}(F_v)$.  
Consider the restricted direct product 
\begin{equation}\label{1'.4}
\widetilde{\Sp}_{2k}(\BA)=\prod_v{}'\ Sp^{(2)}_{2k}(F_v),
\end{equation}
with respect to the groups $\{K_{\Sp^{(2)}_{2k}(F_v)}\}_{v\notin S_0}$. Then the double cover $\Sp^{(2)}_{2k}(\BA)$ is the quotient of $\widetilde{\Sp}_{2k}(\BA)$ by the subgroup 
\begin{equation}\label{1'.5}
C'=\{\Pi'_v(I_{2k},\mu_v)\in \widetilde{\Sp}_{2k}(\BA) \ |\ \Pi_v\mu_v=1\}.
\end{equation}
Let 
\begin{equation}\label{1'.6}
p=p_{2k}:\widetilde{\Sp}_{2k}(\BA)\rightarrow \Sp^{(2)}_{2k}(\BA),
\end{equation}
denote the quotient map. 
\begin{equation}\label{1'.7}
p(\Pi'_v(g_v,\mu_v))=C'\Pi'_v(g_v,\mu_v).
\end{equation}
The kernel of $p$ is
\begin{equation}\label{1'.8}
C_2=\{C'\Pi_v'(I_{2k},\mu_v)\in \Sp^{(2)}_{2k}(\BA)\}\cong \{\pm 1\}.
\end{equation}
Denote 
\begin{equation}\label{1'.9}
K_{\Sp^{(2)}_{2k}(\BA)}=p(\prod_v K_{\Sp^{(2)}_{2k}(F_v)}).
\end{equation}
The group $\Sp_{2k}(F)$ embeds "diagonally" in $\Sp^{(2)}_{2k}(\BA)$ by
\begin{equation}\label{1'.10}
\gamma\mapsto C'\Pi'_v(\gamma,1).
\end{equation}
We will identify $\Sp_{2k}(F)$ as a subgroup of $\Sp^{(2)}_{2k}(\BA)$, and oftentimes, we will denote the r.h.s. of \eqref{1'.10} simply by $(\gamma,1)$ or even by $\gamma$. It will be convenient to denote by $C^{(1)}_2$ the trivial subgroup of $H_m(\BA)=H_m^{(1)}(\BA)$, and
\begin{equation}\label{1'.11}
C_2^{(2)}=C_2\subset \Sp^{(2)}_{2k}(\BA).
\end{equation}
We don't mention $m$ or $k$ in our notation, and there won't be any confusion. Thus, in both cases we may consider the quotient $C^{(\epsilon)}H_m(F)\backslash H_m^{(\epsilon)}(\BA)$.

When we consider an irreducible, automorphic, cuspidal representation $\sigma$ of $\Sp^{(2)}_{2k}(\BA)$, we always assume that it is genuine. This means that an element $C'\cdot \Pi'_v(I_{2k},\mu_v)$ acts by multiplication by $\Pi_v\mu_v$. The represntation $\sigma$ decomposes as a restricted tensor product of local irreducible genuine representations $\sigma_v$ of $\Sp^{(2)}(F_v)$, $\otimes'_v\sigma_v$, in the sense that 
\begin{equation}\label{1'.12}
\sigma\circ p\cong \otimes'_v\sigma_v.
\end{equation}
We will use sometimes the following  notation. Let $\Pi'_v(g_v,\mu_v)\in  \widetilde{\Sp}_{2k}(\BA)$. Let $g=\Pi'_vg_v\in \Sp_{2k}(\BA)$, and let $\bar{\mu}$ denote the sequence $(\mu_v)_v$. Note, that for almost all $v\notin S_0$, $g_v\in K_{\Sp_{2k}(F_v)}$ and $\mu_v=\lambda_v(g_v)$. We will denote
\begin{equation}\label{1'.13}
(g,\bar{\mu})=\Pi'_v(g_v,\mu_v).
\end{equation}
Sometimes we will denote, for short,
\begin{equation}\label{1'.14}
\sigma(p((g,\bar{\mu}))):=\sigma((g,\bar{\mu})).
\end{equation}
For an $F$- subgroup $U$ of $\Sp_{2k}$, consisting of upper unipotent matrices, we will identify $U(\BA)$ with the subgroup of elements $p((u,\bar{1}))$, where $u\in U(\BA)$, and $\bar{1}$ is the sequence with coordinate $1$, at all places $v$. We will usually denote $p((u,\bar{1}))$ by $(u,1)$, or by $u$.

In order to unify notation, let $\epsilon=1,2$. We will consider the groups $H_\ell^{(1)}=H_\ell$ in the linear case, and the groups $H_\ell^{(2)}$ only when $H_\ell$ is symplectic, and then $\ell$ is even. In both cases we will use the notation $H^{(\epsilon)}_\ell$ with the agreement that when $H_\ell$ is orthogonal then $\epsilon=1$. 

For two positive integers $n, m$, we will consider soon certain Eisenstein series on $H^{(\epsilon)}_{2nm}$. When convenient,
we will shorten our notation by putting $H=H^{(\epsilon)}_{2nm}$ and
$r_H=2nm$. We let $\delta_H=1$ when $H$ is orthogonal, and
$\delta_H=-1$, when $H$ is symplectic, or metaplectic.

Throughout the paper, if $m$ is odd, then we take $H_{2nm}$ to be orthogonal.  

We denote the standard basis of $F^{2nm}$ by $\{e_1,...,e_{nm}, e_{-nm},...,e_{-1}\}$.
Assume that $H$ is linear. For positive integers, $k_1,...,k_\ell$, such that
$k_1+\cdots+k_\ell\leq \frac{r_H}{2}$, let
$Q_{k_1,...,k_\ell}=Q^{(1)}_{k_1,...,k_\ell}=Q_{k_1,...,k_\ell}^H$ denote the standard
parabolic subgroup of $H$, whose Levi part is isomorphic to
$\GL_{k_1}\times\cdots\times \GL_{k_\ell}\times H'$, where $H'$ is a split
classical group of the same type as $H$. Denote the corresponding
Levi part by $M_{k_1,...,k_\ell}=M_{k_1,...,k_\ell}^H$, and
unipotent radical by $U_{k_1,...,k_\ell}=U_{k_1,...,k_\ell}^H$. When
$k_1=\cdots=k_\ell=k$, we will simply denote $Q_{k^\ell}$,
$M_{k^\ell}$, $U_{k^\ell}$. When $H$ is metaplectic, we consider the analogous subgroups $Q^{(2)}_{k_1,...,k_\ell}=Q^H_{k_1,...,k_\ell}$, obtained as the inverse image in $H$ of the similar parabolic subgroup of the corresponding symplectic group. 
When $H$ is linear and $j\leq \frac{r_H}{2}$, denote by $\hat{a}$, $a\in
\GL_j$, the following element of $H$,
\begin{equation}\label{1'.15}
\hat{a}=diag(a,I_{r_H-2j},a^*),\quad\quad a^*=w_j{}^ta^{-1}w_j.
\end{equation}
Given positive integers $k_1,...,k_\ell$, such that
$k_1+\cdots+k_\ell=k$, we denote by $P_{k_1,...,k_\ell}$ the
standard parabolic subgroup of $\GL_k$, consisting of upper
triangular block matrices, with diagonal of the form
$diag(g_1,...,g_\ell)$, where $g_i\in\GL_{k_i}$, for $1\leq i\leq
\ell$. We denote the Levi part by $L_{k_1,...,k_\ell}$, and the unipotent radical by
$V_{k_1,...,k_\ell}$.

Consider the parabolic subgroup $Q_{m^{n-1}}$ of $H$. The elements of $U_{m^{n-1}}$ have the form
\begin{equation}\label{1.2}
u=\begin{pmatrix}I_m&x_1&&&&\star&&&&\star\\&\ddots&&&&\cdots&&&&\cdots\\&&I_m&x_{n-2}&&\star&&&&\star\\
&&&I_m&y_1&y_2&y_3&\star&&\star\\&&&&I_{[\frac{m}{2}]}&&&y'_3&&\star\\&&&&&I_{2[\frac{m+1}{2}]}&&y'_2&&\star\\
&&&&&&I_{[\frac{m}{2}]}&y'_1&&\star\\&&&&&&&I_m&x'_{n-2}\
\cdots&\star\\&&&&&&&&\ddots&\star\\&&&&&&&&&x'_1\\&&&&&&&&&I_m\end{pmatrix}\in H,
\end{equation}
We fix a nontrivial character
$\psi$ of $F\backslash \BA$. It defines the following character
$\psi_{U_{m^{n-1}}}$ of $U_{m^{n-1}}(\BA)$, trivial on
$U_{m^{n-1}}(F)$. Its value on the element $u$ of the form
\eqref{1.2}, with adele coordinates, is
\begin{equation}\label{1.3}
\psi_{U_{m^{n-1}}}(u)=\psi(tr(x_1+\cdots+x_{n-2}))\psi(tr((y_1,y_2,y_3)A_H),
\end{equation}
where $A_H$ is the following matrix. When $m$ is even,
\begin{equation}\label{1.4}
A_H=\begin{pmatrix}I_{\frac{m}{2}}&0\\0_{m\times \frac{m}{2}}&0_{m\times \frac{m}{2}}\\0&I_{\frac{m}{2}}\end{pmatrix}.
\end{equation}
When $m$ is odd (and hence $H$ is orthogonal),
\begin{equation}\label{1.5}
A_H=\begin{pmatrix}I_{[\frac{m}{2}]}&0&0\\0&0&0\\0&1&0\\0&\frac{1}{2}&0\\0&0&0\\0&0&I_{[\frac{m}{2}]}\end{pmatrix},
\end{equation}
where the second and fifth block rows of zeroes contain each $[\frac{m}{2}]$
rows. We note that $\psi_{U_{m^{n-1}}}$ corresponds to the nilpotent orbit in
the Lie algebra of $H$ corresponding to the partition
$((2n-1)^m, 1^m)$. See \cite{MW87}, \cite{GRS03}.\\
For later use, we will need the following notation. We will denote for the element $u\in U_{m^{n-1}}$ in \eqref{1.2},
$$
x_i=x_i(u),\ 1\leq i\leq n-2,
$$
\begin{equation}\label{1.5.1}
x_{n-1}(u)=(y_1,y_2,y_3)A_H.
\end{equation}
Thus, when $m=2m'$ is even,
$$
x_{n-1}(u)=(y_1,y_3),
$$
and when $m=2m'-1$ is odd
$$
x_{n-1}(u)=(y_1, y_2\begin{pmatrix}0_{m'-1}\\1\\ \frac{1}{2}\\0_{m'-1}\end{pmatrix},y_3).
$$
Assume that $H$ is linear. The stabilizer of $\psi_{U_{m^{n-1}}}$ in
$M_{m^{n-1}}(\BA)$ is the adele points of an algebraic group over $F$, which we denote by $D=D_{\psi_{U_{m^{n-1}}}}$. It is isomorphic to $H_m\times
H_m$. See \cite{CM93}, Theorem 6.1.3. The elements of $D$ are realized
as
\begin{equation}\label{1.6}
t(g,h)=diag(g^{\Delta_{n-1}},j(g,h),(g^*)^{\Delta_{n-1}}),\ g, h\in H_m,
\end{equation}
where $g^{\Delta_{n-1}}=diag(g,...,g)$ ($n-1$ times), and $j(g,h)$
is as follows.\\
Assume that $m=2m'$ is even. Then $g, h\in H_{2m'}$. Write
$g=\begin{pmatrix}a&b\\c&d\end{pmatrix}$, where $a,...,d$ are
$m'\times m'$ matrices. Then
\begin{equation}\label{1.7}
j(g,h)=\begin{pmatrix}a&&b\\&h\\c&&d\end{pmatrix}.
\end{equation}
Assume that $m=2m'-1$ is odd. Then $H=\SO_{2n(2m'-1)}$, $g\in \SO_{2m'-1}$, $h\in
H_{2m'-1}\cong\SO_{2m'-1}$. In this case, we will write the elements
of $H_{2m'-1}$ with respect to the symmetric matrix
$w'_{2m'-1}$, where
\begin{equation}\label{1.7.1}
w'_{2m'-1}=\begin{pmatrix}&&w_{m'-1}\\&-1\\w_{m'-1}\end{pmatrix},
\end{equation}
so that 
$$
H_{2m'-1}=\{g\in \SL_{2m'-1}\ | \
{}^tgw'_{2m'-1}g=w'_{2m'-1}\}.
$$
Write
$$
g=\begin{pmatrix}a_1&b_1&c_1\\a_2&b_2&c_2\\a_3&b_3&c_3\end{pmatrix},\
h=\begin{pmatrix}A_1&B_1&C_1\\A_2&B_2&C_2\\A_3&B_3&C_3\end{pmatrix},
$$
where the first and third block rows (resp. columns) of $g$ contain
each $m'-1$ rows (resp. columns), and similarly for $h$. Then
\begin{equation}\label{1.8}
j(g,h)=\begin{pmatrix}a_1&0&\frac{1}{2}b_1&b_1&0&c_1\\0&A_1&\frac{1}{2}B_1&-B_1&C_1&0\\
a_2&A_2&\frac{1}{2}(b_2+B_2)&b_2-B_2&C_2&c_2\\\frac{1}{2}a_2&-\frac{1}{2}A_2&\frac{1}{4}(b_2-B_2)&\frac{1}{2}(b_2+B_2)&-\frac{1}{2}C_2&\frac{1}{2}c_2\\
0&A_3&\frac{1}{2}B_3&-B_3&C_3&0\\a_3&0&\frac{1}{2}b_3&b_3&0&c_3\end{pmatrix}.
\end{equation}
Assume that $H$ is metaplectic. Then the analogue of \eqref{1.7} is given by the homomorphism (over $F_v$)
$$
t_v^{(2)}:\Sp_{2m'}^{(2)}(F_v)\times \Sp_{2m'}^{(2)}(F_v)\mapsto \Sp_{4nm'}^{(2)}(F_v),
$$
defined by
\begin{equation}\label{1.9}
t_v^{(2)}((g,\alpha),(h,\beta))=(t(g,h),\alpha\beta (x_1(g),x_2(h))),
\end{equation}
where $t(g,h)$ is given by \eqref{1.6}, $\alpha,\beta=\pm 1$, and $x_1,x_2$ are the Ranga Rao $x$-functions on $\Sp_{2m'}(F_v)$, $\Sp_{2ni+2m'}(F_v)$, respectively. They take values in $F_v^*/(F_v^*)^2$; $(x_1(g),x_2(h))$ is the Hilbert symbol of $F_v$. See \cite{GS18}, Sec.1 for some more details. 

Let us write the adelic version of \eqref{1.9}. The collection $\{t_v^{(2)}\}$ defines the homomorphism
$$
\tilde{t}:\widetilde{\Sp}_{2m'}(\BA)\times \widetilde{\Sp}_{2m'}(\BA)\rightarrow \widetilde{\Sp}_{4nm'}(\BA),
$$
such that, for each place $v$, $\tilde{t}$ on $\Sp_{2m'}^{(2)}(F_v)\times \Sp_{2m'}^{(2)}(F_v)$ is $t^{(2)}_v$. Thus, we have the homomorphism
$$
t^{(2)}:\Sp^{(2)}_{2m'}(\BA)\times \Sp^{(2)}_{2m'}(\BA)\rightarrow \Sp^{(2)}_{4nm'}(\BA),
$$
\begin{equation}\label{1.9.1}
t^{(2)}(p((g,\bar{\alpha})),p((h,\bar{\beta})))=p((t(g,h),\bar{\alpha}\bar{\beta} (x_1(g),x_2(h)))).
\end{equation}
We used the notation \eqref{1'.13}. Also, $t(g,h)$ is the element whose coordinate at $v$ is $t(g_v,h_v)$, and $(x_1(g),x_2(h))$ is the element in $\BA^*$, whose coordinate at $v$ is $(x_{1,v}(g_v),x_{2,v}(h_v))_v$.

In order to unify and ease our notations, we will re-denote, when convenient, $t(g,h)$, $j(g,h)$, in the linear case, by $t^{(1)}(g,h)$, $j^{(1)}(g,h)$. Similarly, we will re-denote, in the linear case, the stabilizer of $\psi_{U_{m^{n-1}}}$ by $D^{(1)}$, and in the metaplectic case, we will denote by $D^{(2)}(F_v)$ the image of the homomorphism \eqref{1.9}. When there is no risk of confusion, we will simply denote $t(g,h)=t^{(\epsilon)}(g,h)$, $D=D^{(\epsilon)}$, $\epsilon=1,2$.\\  
\\
{\bf 2. Eisenstein series}\\

Let $\tau$ be an irreducible, automorphic, cuspidal representation
of $\GL_n(\BA)$. Assume that $\tau$ has a unitary central character.
Let $\Delta(\tau, m)$ be the Speh representation of
$\GL_{nm}(\BA)$, corresponding to $\tau$ and of length $m$. This representation is spanned by the (multi-)
residues of Eisenstein series corresponding to the parabolic induction
from
$$
\tau|\det\cdot|^{s_1}\times
\tau|\det\cdot|^{s_2}\times\cdots\times
\tau|\det\cdot|^{s_m},
$$
at the point
$$
(\frac{m-1}{2},\frac{m-3}{2},\frac{m-5}{2},...,-\frac{m-1}{2}).
$$
In general, for a positive integer $a$, divisible by $n$, $a=n\ell$, with $\ell>1$, $\Delta(\tau,\ell)$ is an irreducible, unitary, automorphic representation of $\GL_a(\BA)$, which appears with multiplicity one in the discrete, non-cuspidal part of $L^2(\GL_a(F)\backslash \GL_a(\BA))_\chi$, where $\chi=\omega_\tau^\ell$. As we run over the positive integers $n$ and $\ell>1$, such that $a=n\ell$, and $\tau$, as above, with $\chi=\omega_\tau^\ell$, the sum of all $\Delta(\tau,\ell)$ exhausts the non-cuspidal part of $L^2(\GL_a(F)\backslash \GL_a(\BA))_\chi$. This is proved in \cite{MW89}.

We consider the Eisenstein series on $H(\BA)$, $E(f_{\Delta(\tau,
	m)\gamma^{(\epsilon)}_\psi,s})$, corresponding to a smooth, holomorphic section
$f_{\Delta(\tau,m)\gamma^{(\epsilon)}_\psi,s}$ of the parabolic induction
$$
\rho_{\Delta(\tau,m)\gamma^{(\epsilon)}_\psi,s}=\Ind_{Q^{(\epsilon)}_{nm}(\BA)}^{H(\BA)}\Delta(\tau,
m)\gamma^{(\epsilon)}_\psi|\det\cdot|^s.
$$
Here, $\epsilon=1,2$, according to whether $H$ is linear, or metaplectic. When $\epsilon=1$, $H$ is linear and $\gamma^{(1)}_\psi=1$. When $\epsilon=2$, $m=2m'$ is even, $H(\BA)=\Sp^{(2)}_{4nm'}(\BA)$, and $\gamma^{(2)}_\psi=\gamma_\psi\circ \det$ is the Weil factor attached to $\psi$, composed with the determinant. This is a character of the double cover of $\GL_{2nm'}(\BA)$.
We will denote the value at $h$ of our Eisenstein series by $E(f_{\Delta(\tau, m)\gamma^{(\epsilon)}_\psi,s},h)$.
Consider the Fourier coefficient of $E(f_{\Delta(\tau, m)\gamma^{(\epsilon)}_\psi,s})$
along $U_{m^{n-1}}$ with respect to the character $\psi_{U_{m^{n-1}}}$, and
view it as a function on $D(\BA)=D_{\psi_{U_{m^{n-1}}}}(\BA)$.\\
\\
$\mathcal{F}_\psi(E(f_{\Delta(\tau, m)\gamma^{(\epsilon)}_\psi,s}))(g,h)=$
\begin{equation}\label{1.9.1}
\int_{U_{m^{n-1}}(F)\backslash
	U_{m^{n-1}}(\BA)}E(f_{\Delta(\tau,
	m)\gamma^{(\epsilon)}_\psi,s},ut(g,h))\psi^{-1}_{U_{m^{n-1}}}(u)du,
\end{equation}
where $g, h\in H^{(\epsilon)}_m(\BA)$. Since $D(\BA)$
stabilizes $\psi_{U_{m^{n-1}}}$, the function \eqref{1.9.1} is left
$D(F)$ - invariant.

Let $\sigma,  \pi$ be two irreducible, automorphic, cuspidal representations
of $H^{(\epsilon)}_m(\BA)$. The integrals of the generalized doubling method of \cite{CFGK17} have the following form,\\
$\mathcal{L}(f_{\Delta(\tau,m)\gamma^{(\epsilon)}_\psi,s},\varphi_\sigma, \varphi_\pi)=$
\begin{equation}\label{1.10}
\int_{C_2^{(\epsilon)}H_m(F)\times C_2^{(\epsilon)}H_m(F)\backslash
	H^{(\epsilon)}_m(\BA)\times H^{(\epsilon)}_m(\BA)}\mathcal{F}_\psi(E(f_{\Delta(\tau,
	m)\gamma^{(\epsilon)}_\psi,s}))(g,h)\varphi_\sigma(g)\varphi_\pi(h)dgdh,
\end{equation}
where $\varphi_\sigma, \varphi_\pi $ are in the spaces of $\sigma, \pi$, respectively. 
The unfolding of this global integral, carried out in \cite{CFGK17}, shows that it is identically zero, unless the following pairing is nontrivial (as a bilinear form)
\begin{equation}\label{1.10.0}
c(\varphi_\sigma,\varphi_\pi^\iota)=\int_{C_2^{(\epsilon)}H_m(F)\backslash
	H^{(\epsilon)}_m(\BA)}\varphi_\sigma(g)\varphi_\pi(g^\iota)dg,
\end{equation}
where, for $g\in H_m(\BA)$,
\begin{equation}\label{1.10.1}
g^\iota=J_0^{-1}gJ_0,
\end{equation} 
with
$$
J_0=\begin{pmatrix}&&I_{[\frac{m}{2}]}\\&I_{m-2[\frac{m}{2}]}\\-\delta_HI_{[\frac{m}{2}]}\end{pmatrix},
$$
and 
$$
\varphi_\pi^\iota(g)=\varphi_\pi(g^\iota).
$$
See \cite{GS18}, (1.39), for the lift of the automorphism \eqref{1.10.1} in the metaplectic case. Denote $\pi^\iota (g)=\pi(g^\iota)$ (acting in the space of $\pi$). In particular, $\pi^\iota\cong \hat{\sigma}$. Note that
\begin{equation}\label{1.10.2}
\rho(h)(\varphi_\pi^\iota)=(\pi^\iota(h)\varphi_\pi)^\iota,
\end{equation}
where $\rho(h)$ denotes the right translation by $h$. Thus, the space of the cusp forms $\varphi_\pi^\iota$ is an automorphic realization of $\pi^\iota$. The unfolding of \eqref{1.10}, for $Re(s)$ sufficiently large, has the following form\\
\\
$\mathcal{L}(f_{\Delta(\tau,m)\gamma^{(\epsilon)}_\psi,s},\varphi_\sigma, \varphi_\pi)=$
\begin{equation}\label{1.10.3}
\int_{C_2^{(\epsilon)}\backslash H^{(\epsilon)}_m(\BA)}c(\varphi_\sigma,\rho(h^\iota)(\varphi_\pi^\iota))\int_{U'_{m(n-1)}(\BA)}
f^\psi_{\Delta(\tau, m)\gamma^{(\epsilon)}_\psi,s}(\delta_0ut(1,h))\psi^{-1}_{U_{m^{n-1}}}(u)dudh.
\end{equation}
Here, $U'_{m(n-1)}$ is a certain subgroup of $U_{m(n-1)}$,
$\delta_0$ is a certain element in $H(F)$, and the upper $\psi$ on the section denotes a composition of the section with a Fourier
coefficient on $\Delta(\tau,m)$. This Fourier coefficient is along $V_{m^n}$, with respect to the character
\begin{equation}\label{1.10.3.1}
\psi_{V_{m^n}}:\begin{pmatrix}I_m&x_1&\star&\cdots&\star\\&I_m&x_2&\cdots&\star\\&&\ddots\\&&&I_m&x_{n-1}\\
&&&&I_m\end{pmatrix}\mapsto\psi(tr(x_1+x_2+\cdots+x_{n-1})).
\end{equation}
We can say more. Consider in \eqref{1.10} the $dg$ integration only, namely
\begin{equation}\label{1.10.4}
\Lambda(f_{\Delta(\tau,	m)\gamma^{(\epsilon)}_\psi,s},\varphi_\sigma)(h)=
\int_{C_2^{(\epsilon)}H_m(F)\backslash
	H^{(\epsilon)}_m(\BA)}\mathcal{F}_\psi(E(f_{\Delta(\tau,
	m)\gamma^{(\epsilon)}_\psi,s}))(g,h)\varphi_\sigma(g)dg.
\end{equation} 
Then, for $Re(s)$ sufficiently large,\\
\\
$\Lambda(f_{\Delta(\tau,	m)\gamma^{(\epsilon)}_\psi,s},\varphi_\sigma)(h)=$ 
\begin{equation}\label{1.10.5}
\int_{C_2^{(\epsilon)}\backslash H^{(\epsilon)}_m(\BA)}\varphi_\sigma(h^\iota g)\int_{U'_{m(n-1)}(\BA)}
f^\psi_{\Delta(\tau, m)\gamma^{(\epsilon)}_\psi,s}(\delta_0ut(g,1))\psi^{-1}_{U_{m^{n-1}}}(u)dudg.
\end{equation}
This is a special case of the first main result of \cite{GS18}. See (3.40) and Prop. 3.4 in \cite{GS18}. Assuming that $f^\psi_{\Delta(\tau, m)\gamma^{(\epsilon)}_\psi,s}$ is $t(1\times K_{H_m^{(\epsilon)}(\BA)})$-finite,  the integral \eqref{1.10.5}  exhibits $\Lambda(f_{\Delta(\tau,
	m)\gamma^{(\epsilon)}_\psi,s},\varphi_\sigma)$ as a family of cusp forms in the automorphic realization, as above, of $\sigma^\iota$.  
See \cite{GS18}, Prop. 3.6, Remark 3.7. Note, also, that since the integral on the r.h.s. of \eqref{1.10.4} is absolutely convergent, except when $s$ is a pole of the Eisenstein series, we get that the integral \eqref{1.10.4} defines a meromorphic function in the whole plane.	
	
	We remark that, for $b\in H^{(\epsilon)}_m(\BA)$,
$$
\int_{U'_{m(n-1)}(\BA)}
f^\psi_{\Delta(\tau, m)\gamma^{(\epsilon)}_\psi,s}(\delta_0ut(b^\iota g,bh))\psi^{-1}_{U_{m^{n-1}}}(u)du=
$$
\begin{equation}\label{1.10.6}
\int_{U'_{m(n-1)}(\BA)}
f^\psi_{\Delta(\tau, m)\gamma^{(\epsilon)}_\psi,s}(\delta_0ut(g,h))\psi^{-1}_{U_{m^{n-1}}}(u)du.
\end{equation}
For decomposable data, the integral \eqref{1.10} is Eulerian. It is equal to the product of the corresponding local integrals
$$
\mathcal{L}(f_{\Delta(\tau,m)\gamma^{(\epsilon)}_\psi,s},\varphi_\sigma, \varphi_\pi)=
\prod_v \mathcal{L}_v(f_{\Delta(\tau_v,m)\gamma^{(\epsilon)}_{\psi_v},s},\varphi_{\sigma_v}, \tilde{\varphi}_{\sigma_v}),
$$
where $\varphi_\sigma$ (resp. $\varphi^\iota_\pi$) corresponds under an isomorphism $\sigma\cong \otimes_v \sigma_v$ (resp. $\pi^\iota\cong \otimes_v\hat{\sigma}_v$), to a decomposable vector $\otimes_v \varphi_{\sigma_v}$ (resp. $\otimes_v \tilde{\varphi}_{\sigma_v}$), with pre-chosen unramified local vectors $\varphi^0_{\sigma_v}$ (resp. $\tilde{\varphi}^0_{\sigma_v}$), outside a finite set of places $S$, containing the Archimedean places, such that $\sigma$ is unramified outside $S$. Assume that also $\tau$ is unramified outside $S$. Similarly, $f_{\Delta(\tau,m)\gamma^{(\epsilon)}_\psi,s}$ corresponds to a decomposable section $\otimes_v f_{\Delta(\tau_v,m)\gamma^{(\epsilon)}_{\psi_v},s}$, where we denote by $\Delta(\tau_v,m)$ the local factor at $v$ of $\Delta(\tau,m)$, and we also fixed an isomorphism $\Delta(\tau,m)\cong \otimes'_v \Delta(\tau_v,m)$. We have, for $Re(s)$ large,\\
\\
$\mathcal{L}_v(f_{\Delta(\tau_v,m)\gamma^{(\epsilon)}_{\psi_v},s},\varphi_{\sigma_v}, \tilde{\varphi}_{\sigma_v})=$
\begin{equation}\label{1.10.7}
\int_{C_2^{(\epsilon)}\backslash H^{(\epsilon)}_m(F_v)}\xi_{\varphi_{\sigma_v},\tilde{\varphi}_{\sigma_v}}(h)\int_{U'_{m(n-1)}(F_v)}
f^\psi_{\Delta(\tau_v, m)\gamma^{(\epsilon)}_{\psi_v},s}(\delta_0ut(1,h))\psi^{-1}_{U_{m^{n-1}},v}(u)dudh,
\end{equation}
where $\xi_{\varphi_{\sigma_v},\tilde{\varphi}_{\sigma_v}}$ is the matrix coefficient of $\sigma_v$ resulting from a local pairing $c_v$ in the one dimensional space $\Hom_{H_m^{(\epsilon)}(F_v)\times H_m^{(\epsilon)}(F_v)}(\sigma_v\otimes \pi_v^\iota,1)$, with pre-chosen pairings $c_v^0$, for $v\notin S$, such that $\xi_{\varphi^0_{\sigma_v},\tilde{\varphi}^0_{\sigma_v}}(1)=1$. See \cite{GS18}, Theorem 4.7, for more details. For $v\notin S$, $\psi_v$ normalized and $f_{\Delta(\tau_v,m)\gamma^{(\epsilon)}_{\psi_v},s}=f^0_{\Delta(\tau_v,m)\gamma^{(\epsilon)}_{\psi_v},s}$ unramified and suitably normalized, the unramified computation gives
\begin{equation}\label{1.10.8}
\mathcal{L}_v(f^0_{\Delta(\tau_v,m)\gamma^{(\epsilon)}_{\psi_v},s},\varphi^0_{\sigma_v}, \tilde{\varphi}^0_{\sigma_v})=\frac{L_{\epsilon,\psi_v}(\sigma_v\times \tau_v,s+\frac{1}{2})}{D_{\tau_v}^H(s)}.
\end{equation}
Here, if $\epsilon=1$, $L_{\epsilon,\psi_v}(\sigma_v\times \tau_v,s+\frac{1}{2})=L(\sigma_v\times \tau_v,s+\frac{1}{2})$, and if $\epsilon=2$, 
$L_{\epsilon,\psi_v}(\sigma_v\times \tau_v,s+\frac{1}{2})=L_{\psi_v}(\sigma_v\times \tau_v,s+\frac{1}{2})$. The denominator $D_{\tau_v}^H(s)$ comes from the normalizing factor of the Eisenstein series and is as follows.\\
When $H=\Sp_{4nm'}$ ,
$$
D^H_{\tau_v}(s)=L(\tau_v,s+m'+\frac{1}{2})\prod_{k=1}^{m'}L(\tau_v,\wedge^2,2s+2k)L(\tau_v,sym^2,2s+2k-1).
$$
When $H=\Sp^{(2)}_{4nm'}$,
$$
D^H_{\tau_v}(s)=\prod_{k=1}^{m'}L(\tau_v,\wedge^2,2s+2k-1)L(\tau_v,sym^2,2s+2k)
$$
When $H=\SO_{4nm'}$,
$$
D^H_{\tau_v}(s)=\prod_{k=1}^{m'}L(\tau_v,\wedge^2,2s+2k)L(\tau_v,sym^2,2s+2k-1).
$$
When $H=\SO_{2n(2m'-1)}$,
$$
D^H_{\tau_v}(s)=\prod_{k=1}^{m'}L(\tau_v,\wedge^2,2s+2k-1)\prod_{k=1}^{m'-1}L(\tau_v,sym^2,2s+2k).
$$
For decomposable data in the global integral, all unramified and normalized outside $S$, and assuming that the pairing \eqref{1.10.0} is nontrivial, we get that
\begin{equation}\label{1.11}
\mathcal{L}(f_{\Delta(\tau,m)\gamma^{(\epsilon)}_\psi,s},\varphi_\sigma, \varphi_\pi)=\prod_{v\in S} \mathcal{L}_v(f_{\Delta(\tau_v,m)\gamma^{(\epsilon)}_{\psi_v},s},\varphi_{\sigma_v}, \tilde{\varphi}_{\sigma_v})\frac{L^S_{\epsilon,\psi}(\sigma\times \tau,s+\frac{1}{2})}{D_\tau^{H,S}(s)},
\end{equation}
where 
$$
D^{H,S}_\tau(s)=\prod_{v\notin S} D^H_{\tau_v}(s). 
$$
For the local theory of the  integrals $\mathcal{L}_v(f_{\Delta(\tau_v,m)\gamma^{(\epsilon)}_{\psi_v},s},\varphi_{\sigma_v}, \tilde{\varphi}_{\sigma_v})$, see \cite{CFK18}. Let $D^H_\tau(s)=\prod_v D^H_{\tau_v}(s)$. Denote
\begin{equation}\label{1.12}
E^*(f_{\Delta(\tau,
	m)\gamma^{(\epsilon)}_\psi,s})=D^H_\tau(s)E(f_{\Delta(\tau,
	m)\gamma^{(\epsilon)}_\psi,s}).
\end{equation}
This is the normalized Eisenstein series.

\section{Statement of the main theorem}

\begin{lem}\label{lem 2.1}
In the notation of the last section, if $L_{\epsilon,\psi}^S(\sigma\times\tau,s)$ has a pole at a point $s_0$, such that $\Re(s_0)\geq \frac{1}{2}$, then the normalized Eisenstein series $E^*(f_{\Delta(\tau,m)\gamma^{(\epsilon)}_\psi,s})$ has a pole at $s_0$, as the section $f_{\Delta(\tau,m)\gamma^{(\epsilon)}_\psi,s}$ varies among the $K_{H(\BA)}$-finite, holomorphic sections.
\end{lem}	
\begin{proof}
This follows from \eqref{1.11} and Cor. 44 in \cite{CFK18}. Note that $D_{\tau_v}^H$ is never zero, for any place $v$. 
\end{proof}

Assume that $\tau$ is self-dual and its central character is trivial. Then we know that $\tau$ is a functorial lift from an irreducible, automorphic, cuspidal representation $\sigma$ of an appropriate group $H^{(\epsilon)}_m(\BA)$. Moreover, by \cite{GRS11}, we can find such a generic $\sigma=\sigma_\tau$, the descent of $\tau$,  and then $L^S_{\epsilon,\psi}(\sigma\times \tau, s+\frac{1}{2})$ has a pole at $s=\frac{1}{2}$. In this case, it is easy to see that $D_\tau^H$ is holomorphic and non-zero at $s=\frac{1}{2}$, and hence $E(f_{\Delta(\tau,m)\gamma^{(\epsilon)}_\psi,s})$ has a pole at $\frac{1}{2}$, as the section varies. Moreover, the corresponding Fourier coefficient appearing in \eqref{1.10}, is not identically zero on the residue. Thus, there is a smooth holomorphic section (and even a $K_{H(\BA)}$-finite one) $f_{\Delta(\tau,m)\gamma^{(\epsilon)}_\psi,s}$, such that
\begin{equation}\label{2.1}
\mathcal{E}(f_{\Delta(\tau,m)\gamma^{(\epsilon)}_\psi,\frac{1}{2}})=\mathcal{F}_\psi(Res_{s=\frac{1}{2}}E(f_{\Delta(\tau,m)\gamma^{(\epsilon)}_\psi,s}))\neq 0.
\end{equation}
Let us write in detail the groups $H^{(\epsilon)}_m$ according to $\tau$.
\begin{enumerate}
	\item Assume that $L(\tau,\wedge^2,s)$ has a pole at $s=1$. Then $n=2(m'-1)$ is even, $m=2m'-1=n+1$, $H_m^{(\epsilon)}=H_m=\SO_{2m'-1}$, and $H=\SO_{4(m'-1)(2m'-1)}$. \\
	\item Assume that $L(\tau,\wedge^2,s)$ has a pole at $s=1$, and $L(\tau,\frac{1}{2})\neq 0$. Then $n=2m'$ is even, $m=2m'=n$, $H_m^{(\epsilon)}=H^{(2)}_m=\Sp^{(2)}_{2m'}$, and $H=\Sp^{(2)}_{8(m')^2}$.\\
	\item Assume that $L(\tau, \vee^2,s)$ has a pole at $s=1$, and $n=2m'$ is even. Then $m=2m'=n$, $H_m^{(\epsilon)}=H_m=\SO_{2m'}$, and $H=\SO_{8(m')^2}$.\\
	\item Assume that $L(\tau, \vee^2,s)$ has a pole at $s=1$, and $n=2m'-1$ is odd. Then $m=2m'-2=n-1$, $H_m^{(\epsilon)}=H_m=\Sp_{2m'-2}$, and $H=\Sp_{4(m'-1)(2m'-1)}$. 
\end{enumerate}
Here we denote $Sym^2=\vee^2$. We proved 
\begin{thm}\label{thm 2.2}
Let $\tau$ be an irreducible, automorphic, cuspidal representation of $\GL_n(\BA)$. Assume that $\tau$ is self-dual and its central character is trivial. Then in each of the four cases above, with the indicated group $H$ and the indicated value of $m$, the Eisenstein series  $E(f_{\Delta(\tau,m)\gamma^{(\epsilon)}_\psi,s})$ has a pole at $s=\frac{1}{2}$, for some smooth, holomorphic section. Moreover, the Fourier coefficient \eqref{2.1} is not identically zero on the corresponding residual representation.
\end{thm}	
In the set-up of the last theorem, consider the space of automorphic forms on $H^{(\epsilon)}_m(\BA)\times H^{(\epsilon)}_m(\BA)$, generated by the functions $(g,h)\mapsto \mathcal{E}(f_{\Delta(\tau,m)\gamma^{(\epsilon)}_\psi,\frac{1}{2}})(g,h)$. By Theorem \ref{thm 2.2}, this space is nontrivial. Denote this space (and the resulting representation of $H^{(\epsilon)}_m(\BA)\times H^{(\epsilon)}_m(\BA)$) by 
$\mathcal{D}\mathcal{D}_\psi(\tau)$. We will call this representation the double descent of $\tau$. 

\begin{prop}\label{prop 2.3}
Let $\tau$ and $H^{(\epsilon)}_m$ be as in Theorem \ref{thm 2.2}. Let $\sigma$ be an irreducible, automorphic, cuspidal representation of $H^{(\epsilon)}_m(\BA)$. \\
1. Assume that $\sigma$ (weakly) lifts functorially to $\tau$ on $\GL_n(\BA)$. Then $\bar{\sigma} \otimes \sigma^\iota$ has a non-trivial $L^2$-pairing with $\mathcal{D}\mathcal{D}_\psi(\tau)$.\\ 
2. Assume that $L^S_{\epsilon,\psi}(\sigma \times \tau,s)$ has a pole at $s=1$, then $\bar{\sigma} \otimes \sigma^\iota$ has a non-trivial $L^2$-pairing with $\mathcal{D}\mathcal{D}_\psi(\tau)$.\\
\end{prop}
\begin{proof}
Assume that $\sigma$ lifts to $\tau$. Since $L^S_{\epsilon,\psi}(\sigma \times \tau,s+\frac{1}{2})=L^S(\tau \times \tau,s+\frac{1}{2})$, the partial $L$-function $L^S_{\epsilon,\psi}(\sigma \times \tau,s+\frac{1}{2})$ has a pole at $s=\frac{1}{2}$. By \eqref{1.10}, \eqref{1.11}, and Lemma \ref{lem 2.1}, we conclude that the following integral is not identically zero, 
\begin{equation}\label{2.2}
\int_{C_2^{(\epsilon)}H_m(F)\times C_2^{(\epsilon)}H_m(F)\backslash
	H^{(\epsilon)}_m(\BA)\times H^{(\epsilon)}_m(\BA)}\mathcal{E}(f_{\Delta(\tau,m)\gamma^{(\epsilon)}_\psi,\frac{1}{2}})(g,h)\varphi_\sigma(g)\bar{\xi}_\sigma(h^\iota)dgdh.
\end{equation}
Here $\varphi_\sigma, \xi_\sigma$ vary in the space of $\sigma$. This is the $L^2$-pairing of $\mathcal{E}(f_{\Delta(\tau,m)\gamma^{(\epsilon)}_\psi,\frac{1}{2}})$ with $\bar{\varphi}_\sigma\otimes \xi_\sigma^\iota$. Note that we only used the fact that $L^S_{\epsilon,\psi}(\sigma \times \tau,s)$ has a pole at $s=1$, and hence the second assertion of the proposition follows as well.
\end{proof}

We remark that thanks to the work of Arthur \cite{A13}, and independently, the work \cite{CFK18}, we now know that the cuspidal representation $\sigma$ lifts weakly to an irreducible automorphic representation of $\GL_n(\BA)$, and then it follows that $\sigma$ lifts to $\tau$ if and only if $L^S_{\epsilon,\psi}(\sigma \times \tau,s)$ has a pole at $s=1$. 

The main theorem of this paper is
\begin{thm}\label{thm 2.4} 
Let $\tau$ and $H^{(\epsilon)}_m$ be as in Theorem \ref{thm 2.2}. The (nontrivial) representation $\mathcal{D}\mathcal{D}_\psi(\tau)$ of $H^{(\epsilon)}_m(\BA)\times H^{(\epsilon)}_m(\BA)$ is cuspidal. It decomposes into a multiplicity free direct sum of the form
\begin{equation}\label{2.2.1}
\mathcal{D}\mathcal{D}_\psi(\tau)=\oplus_\sigma (\sigma\otimes \bar{\sigma}^\iota),
\end{equation}
where $\sigma$ varies over a set $J$ of irreducible, automorphic, cuspidal representations of $H^{(\epsilon)}_m(\BA)$, which lift weakly to $\tau$. Each irreducible, automorphic, cuspidal representation of $H^{(\epsilon)}_m(\BA)$, which lifts weakly to $\tau$ is isomorphic to a (unique) representation in $J$.
\end{thm}
We will call the set-up of Theorem \ref{thm 2.2}, (as well as Theorem \ref{thm 2.4}), "the case of functoriality". This is detailed in cases (1) - (4) before Theorem \ref{thm 2.2}.
We remark that with the knowledge of Arthur's multiplicity formulas, we know that except the case $H^{(\epsilon)}_m=\SO_{2m'}$, $J$ in the theorem consists exactly of all irreducible, automorphic, cuspidal representations of $H^{(\epsilon)}_m(\BA)$, which lift weakly to $\tau$. 

A large part of the work of this paper will be to show that the space $D^1_\psi(\tau)$ of automorphic functions on $H^{(\epsilon)}_m(\BA)$, generated by the functions $g\mapsto \mathcal{E}(f_{\Delta(\tau,m)\gamma^{(\epsilon)}_\psi,\frac{1}{2}})(g,1)$, is cuspidal.  This will be proved in Sec. 4 - 11. Assuming this, we may decompse
\begin{equation}\label{2.3}
D^1_\psi(\tau)=\oplus _\sigma\sigma
\end{equation}
as an orthogonal direct sum of irreducible, automorphic, cuspidal representations of $H^{(\epsilon)}_m(\BA)$.
Fix, for each $\sigma$ appearing in \eqref{2.3}, an orthonormal basis $\{\varphi^i_\sigma\}_i$ of $V_\sigma$ - the given space of $\sigma$. Write 
\begin{equation}\label{2.4}
\mathcal{E}(f_{\Delta(\tau,m)\gamma^{(\epsilon)}_\psi,\frac{1}{2}})(g,h)=\sum_\sigma\sum_i\alpha_{\sigma,i}(h)\varphi^i_\sigma(g).
\end{equation}
Then
$$
\alpha_{\sigma,i}(h)=\int_{C_2^{(\epsilon)}H_m(F)\backslash H^{(\epsilon)}_m(\BA)}\mathcal{E}(f_{\Delta(\tau,m)\gamma^{(\epsilon)}_\psi,\frac{1}{2}})(g,h)\bar{\varphi^i}_\sigma(g)dg=
$$
$$
=Res_{s=\frac{1}{2}}\Lambda(f_{\Delta(\tau,m)\gamma^{(\epsilon)}_\psi,s},\bar{\varphi^i}_\sigma)(h):=\xi^i_{\bar{\sigma}^\iota}(h).
$$
Assume that $f^\psi_{\Delta(\tau, m)\gamma^{(\epsilon)}_\psi,s}$ is $t(1\times K_{H_m^{(\epsilon)}(\BA)})$-finite. Then we explained, right after \eqref{1.10.5}, that this is an element in the space of the cuspidal representation $\bar{\sigma}^\iota$. Thus, we can rewrite \eqref{2.4} as
\begin{equation}\label{2.5}
\mathcal{E}(f_{\Delta(\tau,m)\gamma^{(\epsilon)}_\psi,\frac{1}{2}})(g,h)=\sum_\sigma\sum_i
\varphi^i_\sigma(g)\xi^i_{\bar{\sigma}^\iota}(h).
\end{equation}
This gives a decomposition of the form \eqref{2.2.1}. Let us show that the decomposition \eqref{2.2.1} is multiplicity free. Indeed, assume that there are two isomorphic summands, with an isomorphism $T$
\begin{equation}\label{2.6}
T:\sigma_1\otimes \bar{\sigma}_1^\iota \mapsto \sigma_2\otimes \bar{\sigma}_2^\iota.
\end{equation}
Due to the irreducibility of $\sigma_1$, we may assume that $T$ preserves the $L^2$-pairing (on $\sigma_1$),
$$
\int_{C_2^{(\epsilon)}H_m(F)\backslash
	H^{(\epsilon)}_m(\BA)}f(g,g^\iota)dg=\int_{C_2^{(\epsilon)}H_m(F)\backslash
	H_m(\BA)}T(f)(g,g^\iota)dg.
$$
Consider the space generated by the functions 
$$
(g,h)\mapsto T(\varphi_{\sigma_1}\otimes \bar{\xi_{\sigma_1}}^\iota)(g,h)- \varphi_{\sigma_1}(g) \bar{\xi_{\sigma_1}}(h^\iota), 
$$
as $\varphi_{\sigma_1}$ and $\xi_{\sigma_1}$ vary in the space of $\sigma_1$. If $T$ is different than the identity, then this space defines an irreducible, automorphic, cuspidal representation $A$ of $H^{(\epsilon)}_m(\BA)\times H^{(\epsilon)}_m(\BA)$, which is isomorphic to $\sigma_1\otimes \bar{\sigma}_1^\iota$. Of course, the $L^2$ -  pairing between the elements of $\mathcal{D}\mathcal{D}_\psi(\tau)$ and $A$ is nontrivial. By \eqref{1.10}, \eqref{1.10.0}, \eqref{1.10.3},  the following functional on $A$ is nontrivial
$$
\int_{C_2^{(\epsilon)}H_m(F)\backslash H^{(\epsilon)}_m(\BA)}a(g,g^\iota)dg.
$$
This contradicts \eqref{2.6}. Now, it remains to prove that for each summand  $\sigma\otimes \bar{\sigma}^\iota$ in \eqref{2.2.1}, $\sigma$ lifts weakly to $\tau$. We will prove this in Sec. 12, 13, and this will conclude the proof of Theorem \ref{thm 2.4}. 

Note that as a corollary of Theorem \ref{thm 2.4}, we get 
\begin{thm}\label{thm 2.5}
Let $\tau$ and $H^{(\epsilon)}_m$ be as in Theorem \ref{thm 2.2}. Let $\sigma$ be an irreducible, automorphic, cuspidal representation of $H^{(\epsilon)}_m(\BA)$. Then $\sigma$ lifts to $\tau$ if and only if $L^S_{\epsilon,\psi}(\sigma\times\tau,s)$ has a pole at $s=1$.
\end{thm}
\begin{proof}
We explained the first direction in proof of Prop. \ref{prop 2.3}. Conversely, assume that $L^S_{\epsilon,\psi}(\sigma\times\tau,s)$ has a pole at $s=1$. Then, by Prop. \ref{prop 2.3}, $\bar{\sigma}\otimes\sigma^\iota$ has a nontrivial $L^2$-pairing with $\mathcal{D}\mathcal{D}_\psi(\tau)$. By Theorem \ref{thm 2.4}, $\sigma$ lifts to $\tau$.
\end{proof}
We remark again that we proved the last theorem without using the fact that $\sigma$ has a functorial lift on $\GL_n(\BA)$.

\section{On the residues of the Eisenstein series $E(f_{\Delta(\tau,
		m)\gamma^{(\epsilon)}_\psi,s})$}
	
As we explained in the introduction, we will place ourselves in a more general framework. We fix the self-dual, cuspidal representation $\tau$ of $\GL_n(\BA)$, as before. Again, we assume that the central character of $\tau$ is trivial. We fix, for the moment, a positive integer $m$, and consider again the Eisenstein series $E(f_{\Delta(\tau,m)\gamma^{(\epsilon)}_\psi,s})$ on $H(\BA)=H^{(\epsilon)}_{2nm}(\BA)$. In \cite{JLZ13}, the possible poles of the normalized Eisenstein series $E^*(f_{\Delta(\tau,m)\gamma^{(\epsilon)}_\psi,s})$ in $Re(s)\geq 0$ \eqref{1.12} are determined. As remarked in \cite{CFK18}, the proof in \cite{JLZ13} uses Arthur's results. We now recall the list of possible poles from \cite{JLZ13}. We include the case of metaplectic groups which does not appear in \cite{JLZ13}, but can be obtained similarly. We will address the existence of these poles in a forthcoming work of ours. For the moment, we consider the above list merely as a list of points.

{\bf Case $\wedge^2$:} Assume that $L(\tau,\wedge^2,s)$ has a pole at $s=1$.
\begin{enumerate}
	\item  Assume that $H=\Sp_{2nm}$ ($m$ is even by our assumptions). Then 
	$$
	e^H_k(\wedge^2)=k, k=1,2,...,\frac{m}{2}. 
	$$
	If $L(\tau,\frac{1}{2})=0$, omit $k=\frac{m}{2}$.\\
	\item  Assume that $H=\Sp^{(2)}_{2nm}$ ($m$ is even by our assumptions). Then 
	$$
	e^H_k(\wedge^2)=k-\frac{1}{2}, k=1,2,...,\frac{m}{2}.
	$$
	\\
	\item Assume that $H=\SO_{2nm}$ and $m$ even. Then 
	$$
	e^H_k(\wedge^2)=k, k=1,2,...,\frac{m}{2}.
	$$
	\\
	\item Assume that $H=\SO_{2nm}$ and $m$ is odd. Then  
	$$
	e^H_k(\wedge^2)=k-\frac{1}{2}, k=1,2,...,\frac{m+1}{2}.
	$$	
\end{enumerate}

{\bf Case $\vee^2$:} Assume that $L(\tau,\vee^2,s)$ has a pole at $s=1$.
\begin{enumerate}
	\item  Assume that $H=\Sp_{2nm}$ ($m$ is even by our assumptions). Then 
	$$
	e^H_k(\vee^2)=k-\frac{1}{2}, k=1,2,...,\frac{m}{2}. 
	$$
	If $L(\tau,\frac{1}{2})=0$, omit $k=\frac{m}{2}$.\\
	\item  Assume that $H=\Sp^{(2)}_{2nm}$ ($m$ is even by our assumptions). Then 
	$$
	e^H_k(\vee^2)=k, k=1,2,...,\frac{m}{2}.
	$$
	\\
	\item Assume that $H=\SO_{2nm}$ and $m$ even. Then  
	$$
	e^H_k(\vee^2)=k-\frac{1}{2}, k=1,2,...,\frac{m}{2}.
	$$
	\\
	\item Assume that $H=\SO_{2nm}$ and $m$ is odd. Then  
	$$
	e^H_k(\vee^2)=k, k=1,2,...,\frac{m-1}{2}.
	$$	
\end{enumerate}

Let $\eta$ be either $\wedge^2$ or $\vee^2$. Consider the leading term of the Laurent expansion around $e_k^H(\eta)$ of the Eisenstein series $E(f_{\Delta(\tau,m)\gamma^{(\epsilon)}_\psi,s})$. Denote this leading term by $a(f_{\Delta(\tau,m)\gamma^{(\epsilon)}_\psi,e^H_k(\eta)})$, and view it as an automorphic function on $H(\BA)$. Let $A(\Delta(\tau,m)\gamma_\psi^{(\epsilon)},\eta, k)$ denote the space generated by the leading terms\\ 
$a(f_{\Delta(\tau,m)\gamma^{(\epsilon)}_\psi,e^H_k(\eta)})$. As usual, we also denote by $A(\Delta(\tau,m)\gamma_\psi^{(\epsilon)},\eta, k)$ the representation of $H(\BA)$ by right translations in this space. We do not assume that all the points $e^H_k(\eta)$ are actually poles of our Eisenstein series. Note that the normalizing factor $D_\tau^H(s)$ of this Eisenstein series is holomorphic and nonzero at $e^H_k(\eta)$, and so we may replace $E^*(f_{\Delta(\tau,m)\gamma^{(\epsilon)}_\psi,s})$ by $E(f_{\Delta(\tau,m)\gamma^{(\epsilon)}_\psi,s})$. 

Let $\mathcal{O}$ be a nilpotent orbit of the Lie algebra of $H$ over $F$, corresponding to a partition $\underline{P}$ of $2nm$. Assume that $A(\Delta(\tau,m)\gamma_\psi^{(\epsilon)}, \eta,k)$ admits a nontrivial Fourier coefficient corresponding to $\mathcal{O}$. We will bound $\mathcal{O}$ (or $\underline{P}$) by examining at a finite place $v\notin S$ (i.e. $\tau_v$ is unramified) 
the unramified component $\rho_{\tau_v,m,\gamma_\psi^{(\epsilon)};\eta,k}$ of $\rho_{\Delta(\tau,m)\gamma_\psi^{(\epsilon)},s}$ at the place $v$, and at the various points $s=e^H_k(\eta)$. We will determine in each case a nilpotent orbit in the Lie algebra of $H$ over the algebraic closure of $F_v$, which bounds all nilpotent orbits corresponding to degenerate Whittaker models of $\rho_{\tau_v,m,\gamma_\psi^{(\epsilon)};\eta,k}$.

Consider Case $\wedge^2$. This forces $n=2n'$ to be even. Fix $v\notin S$. Since $\tau$ is self dual with trivial central character, we can write $\tau_v$ as a parabolic induction from an unramified character of the standard Borel subgroup,
\begin{equation}\label{3.1}
\tau_v=\chi_1\times\cdots\times \chi_{n'}\times\chi^{-1}_{n'}\times\cdots\times\chi^{-1}_1,
\end{equation}
where $\chi_i$ are unramified characters of $F_v^*$. In cases (1), (3) (of Case $\wedge^2$), $H=\Sp_{2nm},\  \SO_{2nm}$, $m$ is even and $e_k^H(\wedge^2)=k$, $1\leq k\leq \frac{m}{2}$. Then $\rho_{\tau_v,m;\wedge^2,k}$ is the unramified constituent of the following parabolic induction, 
\begin{equation}\label{3.2}
\rho_{\chi,k}=\Ind^{H(F_v)}_{Q_{((m+2k),(m-2k))^{n'}}(F_v)}\otimes_{i=1}^{n'}[(\chi_i\circ det_{\GL_{m+2k}}) \otimes (\chi_i\circ det_{\GL_{m-2k}})].
\end{equation}
Recall that $Q_{(m+2k)^{n'},(m-2k)^{n'}}$ denotes the standard parabolic subgroup, whose Levi part $M_{(m+2k)^{n'},(m-2k)^{n'}}$ is isomorphic to $\GL_{m+2k}^{n'}\times \GL_{m-2k}^{n'}$.
Consider the germ expansion of the character of $\rho_{\chi,k}$.
It has the following form (see \cite{MW87})
\begin{equation}\label{3.3}
tr (\rho_{\chi,k}(\varphi))=\sum_{\mathcal{O}\in
	ind_{M_{(m+2k)^{n'},(m-2k)^{n'}}(F_v)}^{H(F_v)}0}c_{\mathcal{O}}\int_{\mathcal{O}}\widehat{\varphi\circ
	exp}(X)d\beta_{\mathcal{O}}(X),
\end{equation}
for a smooth, compactly supported function $\varphi$ on
$H(F_v)$; the coefficients $c_{\mathcal{O}}$ are certain
positive integers. The induced nilpotent orbit
$ind_{M_{((m+2k),(m-2k))^{n'}}(F_v)}^{H(F_v)}0$ corresponds, over the algebraic
closure of $F_v$, to the partition $((2n)^{m-2k},n^{4k})$ of $2nm$. See
\cite{CM93}, Chapter 7. By Theorem I.16 in \cite{MW87}, this orbit
tells us that all maximal degenerate Whittaker models (terminology
of \cite{MW87}) of $\rho_{\chi,k}$ correspond to a unique
nilpotent orbit, over the algebraic closure of $F_v$, namely the one
corresponding to the partition $((2n)^{m-2k},n^{4k})$. In particular,
$\rho_{\tau_v,m;\wedge^2,k}$ does not have a nontrivial degenerate Whittaker model,
corresponding to any nilpotent orbit which is unrelated to, or
strictly larger than $((2n)^{m-2k},n^{4k})$. We conclude the first part of the following proposition. The second part is proved in the same way. The notation is as above. 
\begin{prop}\label{prop 3.1}
Assume that $L(\tau,\wedge^2,s)$ has a pole at $s=1$. 	
\begin{enumerate}
\item Let $H=\Sp_{2nm}$, $\SO_{2nm}$, with $m$ even. Let $1\leq k\leq \frac{m}{2}$. Then $e_k^H(\wedge^2)=k$ and
$$
\underline{P}\leq ((2n)^{m-2k},n^{4k}).
$$
\item Let $H=\Sp^{(2)}_{2nm}$, with $m=2m'$ even, or $H=\SO_{2nm}$, with $m=2m'-1$ odd. Let $1\leq k\leq m'$. Then $e_k^H(\wedge^2)=k-\frac{1}{2}$ and
$$
\underline{P}\leq ((2n)^{m-2k+1},n^{4k-2}).
$$
\end{enumerate}
\end{prop}
For the second part of the proposition, we remark that, for $v\notin S$, $\rho_{\tau_v,m;\wedge^2,k}$ is the unramified constituent of the parabolic induction, 
\begin{equation}\label{3.4}
\rho_{\chi,k}=\Ind^{H(F_v)}_{Q_{((m+2k-1),(m-2k+1))^{n'}}(F_v)}((\chi_i\circ det_{\GL_{m+2k-1}}) \otimes (\chi_i\circ det_{\GL_{m-2k+1}}))\gamma_\psi^{(\epsilon)}.
\end{equation}
We have an analogous proposition, with similar proof, for the case where $L(\tau,\vee^2,s)$ has a pole at $s=1$.
\begin{prop}\label{prop 3.2}
Assume that $L(\tau,\vee^2,s)$ has a pole at $s=1$. 	
\begin{enumerate}
\item Let $H=\Sp_{2nm}$ with $m$ even. Let $1\leq k\leq \frac{m}{2}$. Then $e_k^H(\vee^2)=k-\frac{1}{2}$ and
$$
\underline{P}\leq ((2n)^{m-2k+1},n^{4k-2}),\  n \ \text{even},
$$
$$
\underline{P}\leq ((2n)^{m-2k+1}, (n+1)^{2k-1}, (n-1)^{2k-1}),\  n \ \text{odd}.
$$
\item Let $H=\Sp^{(2)}_{2nm}$, with $m$ even. Let $1\leq k\leq \frac{m}{2}$. Then $e_k^H(\vee^2)=k$ and
$$
\underline{P}\leq ((2n)^{m-2k},n^{4k}), \ n \ \text{even}, 
$$
$$
\underline{P}\leq ((2n)^{m-2k}, (n+1)^{2k}, (n-1)^{2k}), \ n \ \text{odd}.
$$
\item Let $H=\SO_{2nm}$ with $m$ even. Let $1\leq k\leq \frac{m}{2}$. Then $e_k^H(\vee^2)=k-\frac{1}{2}$ and
$$
\underline{P}\leq ((2n)^{m-2k},2n-1, n+1, n^{4k-4}, n-1,1), \ n \ \text{even},
$$
$$
\underline{P}\leq ((2n)^{m-2k}, 2n-1, n+2, (n+1)^{2k-2}, (n-1)^{2k-2}, n-2,1), \ n\ \text{odd}.
$$
\item Let $H=\SO_{2nm}$ with $m$ odd. Let $1\leq k\leq \frac{m-1}{2}$. Then $e_k^H(\vee^2)=k$ and
$$
\underline{P}\leq ((2n)^{m-2k-1},2n-1, n+1, n^{4k-2}, n-1,1), \ n \ \text{even},
$$
$$
\underline{P}\leq ((2n)^{m-2k-1}, 2n-1, n+2, (n+1)^{2k-2}, n^2, (n-1)^{2k-2}, n-2,1), \ n\ \text{odd}.
$$
\end{enumerate}
\end{prop}
\begin{proof}
The proof is the same as for Prop. \ref{prop 3.1}. Let us sketch, for example, the third case where $H=\SO_{2nm}$, $m$ is even, $n=2n'+1$ is odd, and $e_k^H(\vee^2)=k-\frac{1}{2}$. Let $v\notin S$. Write as in \eqref{3.1}
\begin{equation}\label{3.5}
\tau_v=\chi_1\times\cdots\times \chi_{n'}\times 1\times\chi^{-1}_{n'}\times\cdots\times\chi^{-1}_1.
\end{equation}
Then $\rho_{\tau_v,m;\vee^2,k}$ is the unramified constituent of the following parabolic induction, 
\begin{equation}\label{3.6}
\Ind^{H(F_v)}_{Q_{((m+2k-1),(m-2k+1))^{n'},m}}\otimes_{i=1}^{n'}[(\chi_i\circ det_{\GL_{m+2k-1}}) \otimes (\chi_i\circ det_{\GL_{m-2k+1}})]\otimes |det_{GL_m}|^{k-\frac{1}{2}}.
\end{equation}
Recall that $Q_{((m+2k-1),(m-2k+1))^{n'},m}$ is the standard parabolic subgroup, whose Levi part is isomorphic to $[\prod_{i=1}^{n'} (\GL_{m+2k-1}\times \GL_{m-2k+1})]\times \GL_m$. Now, in order to find the induced nilpotent orbit, as in \eqref{3.3}, we need to take the $D_{mn}$-collapse of the partition of $2nm$, which is dual to $((m+2k-1)^{2n'},m^2,(m-2k+1)^{2n'})$, and this is the partition 
$(((2n)^{m-2k}, 2n-1, n+2, (n+1)^{2k-2}, (n-1)^{2k-2}, n-2,1))$. See Lemma 6.3.3 in \cite{CM93}.
\end{proof}
Denote, for each point $e^H_k(\eta)$,
\begin{equation}\label{3.7}
\mathcal{E}_k(f_{\Delta(\tau,m)\gamma^{(\epsilon)}_\psi,e^H_k(\eta)})=
\mathcal{F}_\psi(a(f_{\Delta(\tau,m)\gamma^{(\epsilon)}_\psi,e^H_k(\eta)})).
\end{equation}
We view the functions \eqref{3.7} as automorphic functions on $H_m^{(\epsilon)}(\BA)\times H_m^{(\epsilon)}(\BA)$. Denote the space generated by them by $\mathcal{E}_k(\Delta(\tau,m)\gamma_\psi^{(\epsilon)},\eta)$.
Note that in the set-up of \eqref{2.1},
\begin{equation}\label{3.8}
\mathcal{E}(f_{\Delta(\tau,m)\gamma^{(\epsilon)}_\psi,\frac{1}{2}})=\mathcal{E}_1(f_{\Delta(\tau,m)\gamma^{(\epsilon)}_\psi,e^H_1(\eta)}).
\end{equation}
Thus, $\mathcal{E}_k(\Delta(\tau,m)\gamma_\psi^{(\epsilon)},\eta)$ is a generalization of our double descent of $\tau$. As we explained in the introduction, we will use this generalization in a further work, where we will try to construct, by double descent, the irreducible, automorphic representations of $H_m^{(\epsilon)}(\BA)$, which lift to Speh representations.
\begin{prop}\label{prop 3.3}
Assume that $n>1$. Then $\mathcal{E}_k(\Delta(\tau,m)\gamma_\psi^{(\epsilon)},\eta)=0$ in the following cases.
\begin{enumerate}
	\item  In Case (1) of Prop. \ref{prop 3.1}, and Cases (2), (4) for $n$ even, of Prop. \ref{prop 3.2},
	$$
	k>\frac{m}{2n};
	$$ \\
	\item In Cases (2), (4) for $n$ odd of Prop. \ref{prop 3.2},
	$$
	k>\frac{m}{2(n-1)};
	$$\\
	\item In Case (2) of Prop. \ref{prop 3.1}, and Cases (1), (3) for $n$ even of Prop. \ref{prop 3.2}, 
	$$
	k>\frac{m}{2n}+\frac{1}{2};
	$$\\
	\item In Cases (1), (3) for $n$ odd of Prop. \ref{prop 3.2},
	$$
	k>\frac{m}{2(n-1)}+\frac{1}{2}.
	$$  
\end{enumerate}
In particular, if $k>\frac{m}{4}+\frac{1}{2}$, then
$$
\mathcal{E}_k(\Delta(\tau,m)\gamma_\psi^{(\epsilon)},\eta)=0.
$$
\end{prop}
\begin{proof}
Assume that $\mathcal{E}_k(\Delta(\tau,m)\gamma_\psi^{(\epsilon)},\eta)$ is nontrivial. Then the Fourier coefficient $\mathcal{F}_\psi$ is nontrivial on $A(\Delta(\tau,m)\gamma_\psi^{(\epsilon)},\eta, k)$. Recall that $\mathcal{F}_\psi$ corresponds to the partition $((2n-1)^m,1^m)$ of $2nm$. The last two propositions give a majorization of this partition, from which we get the condition on $k$. We show two examples. Consider Case (2) in Prop. \ref{prop 3.1}. Then we must have
\begin{equation}\label{3.9}
((2n-1)^m,1^m)\leq ((2n)^{m-2k+1},n^{4k-2}).
\end{equation}
The condition \eqref{3.9} is violated if $(2n-1)m>2n(m-2k+1)+n(2k-1)$,
which is equivalent to $k>\frac{m}{2n}+\frac{1}{2}$. Consider now Case (3) of Prop. \ref{prop 3.2}, with $n>1$ odd. Here, as in \eqref{3.9}, we must have 
\begin{equation}\label{3.10}
((2n-1)^m,1^m)\leq ((2n)^{m-2k},2n-1,n+2,(n+1)^{2k-2}, (n-1)^{2k-2},n-2,1).
\end{equation}
The condition \eqref{3.10} is violated if $(2n-1)m>2n(m-2k)+(2n-1)+(n+2)+(n+1)(2k-2)$,
which is equivalent to $k>\frac{m}{2(n-1)}+\frac{1}{2}$.
\end{proof}
Denote by $\alpha_{m,n}$ the rational number indicated in the last proposition, such that 
\begin{equation}\label{3.11}
\mathcal{E}_k(\Delta(\tau,m)\gamma_\psi^{(\epsilon)},\eta)=0,\ for\ all \ k>\alpha_{m,n}.
\end{equation}
For example, in Case 1 of the last proposition, $\alpha_{m,n}=\frac{m}{2n}$.
As a corollary, we obtain
\begin{cor}\label{cor 3.4}
Let $\sigma$ be an irreducible, automorphic, cuspidal representation of $H_m^{(\epsilon)}(\BA)$. Assume that $\sigma$ is unramified outside $S$ (as well as $\tau$).
Assume that $L^S_{\epsilon,\psi}(\sigma\times \tau, s)$ has a pole at the point $s=e_k(\eta)+\frac{1}{2}$, for some integer $1\leq k\leq \frac{m}{2}$. Then $1\leq k\leq \alpha_{m,n}$.
\end{cor}
Thus, for example, if $L^S(\tau,\wedge^2,s)$ has a pole at $s=1$, and $\sigma$ is on $\Sp_m(\BA)$, or $\SO_m(\BA)$, with $m$ even, such that $L^S(\sigma\times \tau, s)$ has a pole at the point $s=k_0+\frac{1}{2}$, where $1\leq k_0\leq \frac{m}{2}$ is an integer, then $m\geq 2nk_0$.

Due to Prop. \ref{prop 3.3}, we will be interested from now on in the spaces\\
 $\mathcal{E}_k(\Delta(\tau,m)\gamma_\psi^{(\epsilon)},\eta)$, only for $1\leq k\leq\alpha_{m,n}$.
The main work of this paper is the analysis of constant terms of these spaces along $U^{H_m}_r\times 1$, where $r\leq [\frac{m}{2}]$;  $U^{H_m}_r$ is the unipotent radical of the parabolic subgroup $Q_r^{H_m}$ of $H_m$.

\section{ A first reduction}

Our goal from this point until the end of Sec. 11 is to prove the cuspidality part of Theorem \ref{thm 2.4}. Recall that, in the notation of Theorem \ref{thm 2.4}, we defined the space $D^1_\psi(\tau)$ of automorphic functions on $H^{(\epsilon)}_m(\BA)$, generated by the functions $g\mapsto \mathcal{E}(f_{\Delta(\tau,m)\gamma^{(\epsilon)}_\psi,\frac{1}{2}})(g,1)$. We will prove
\begin{thm}\label{thm 4.1}
Let $\tau$ and $H^{(\epsilon)}_m$ be as in Theorem \ref{thm 2.2}. Then $D^1_\psi(\tau)$ is a cuspidal module over $H^{(\epsilon)}_m(\BA)$.
\end{thm}
We explained right after the statement of Theorem \ref{thm 2.4} that Theorem \ref{thm 4.1} implies that $\mathcal{D}\mathcal{D}_\psi(\tau)$ is cuspidal, and that it breaks into a multiplicity free direct sum of irreducible, automorphic, cuspidal representations of 
$H^{(\epsilon)}_m(\BA)\times H^{(\epsilon)}_m(\BA)$ of the form \eqref{2.2.1}. 

We start in this section with a certain reduction, and, as we explained at the beginning of the previous section, we carry out large parts of the work towards cuspidality, more generally, for $A(\Delta(\tau,m)\gamma_\psi^{(\epsilon)}, \eta,k)$. Only in Cor. \ref{cor 10.3} will we restrict to cases where $m,n,k$ satisfy certain relations, such that, according to the various cases of Prop. \ref{prop 3.1}, \ref{prop 3.2}, $m$, or $m-1$, is equal to an even multiple, $2k$, or an odd multiple, $2k\pm 1$, of $n$, or $n-1$. See Cor. \ref{cor 10.3}. These cases contain the cases of functoriality.

Assume that  $1\leq k\leq\alpha_{m,n}$. Let $r\leq [\frac{m}{2}]$. We consider the constant terms along $U^{H_m}_r\times 1$ of automorphic functions in $\mathcal{E}_k(\Delta(\tau,m)\gamma_\psi^{(\epsilon)},\eta)$. Thus, let $\xi$ be an automorphic function in $A(\Delta(\tau,m)\gamma_\psi^{(\epsilon)},\eta,k)$. Then this constant term is given by
\begin{equation}\label{4.1}
\mathcal{F}_\psi^r(\xi)(h)=\int_{V^{m,n,r}(F)\backslash V^{m,n,r}(\BA)}\xi(uh)\psi^{-1}_{V^{m,n,r}}(u)du,
\end{equation}
where $V^{m,n,r}=U^H_{m^{n-1}}\cdot t(U^{H_m}_r\times 1)$, and $\psi_{V^{m,n,r}}$ is the character of $V^{m,n,r}(\BA)$ obtained from $\psi_{U_{m^{n-1}}}$ by the trivial extension. Note the form of elements of $t(U^{H_m}_r\times 1)$. For this, write a typical element $e$ of $U^{H_m}_r$ as follows. When $m=2m'$ is even,
$$
e=e(v,z)=\begin{pmatrix}I_r&v_1&v_2&z\\&I_{m'-r}&0&v'_2\\&&I_{m'-r}&v'_1\\&&&I_r\end{pmatrix},\ \  v=(v_1,v_2)
$$
and when $m=2m'-1$ is odd,
$$
e=e(v,z)=\begin{pmatrix}I_r&v_1&v_2&v_3&z\\&I_{m'-1-r}&0&0&v'_3\\&&1&0&v'_2\\&&&I_{m'-1-r}&v'_1\\&&&&I_r\end{pmatrix},\ v=(v_1,v_2,v_3).
$$
Then, when $m=2m'$ is even, we get the elements
\begin{equation}\label{4.2}
t(e,1)=diag(e(v,z)^{\Delta_{n-1}},\begin{pmatrix}I_r&v_1&0&v_2&z\\&I_{m'-r}&0&0&v'_2\\&&I_m&0&0\\&&&I_{m'-r}&v'_1\\&&&&I_r\end{pmatrix}, (e(v,z)^*)^{\Delta_{n-1}}).
\end{equation}
When $m=2m'-1$, we get the elements\\
\\
$t(e,1)=$
\\
\begin{equation}\label{4.3}
diag(e^{\Delta_{n-1}},\begin{pmatrix}I_r&v_1&0&\frac{1}{2}v_2&v_2&0&v_3&z\\&I_{m'-1-r}&0&0&0&0&0&v'_3\\&&I_{m'-1}&0&0&0&0&0\\&&&1&0&0&0&v'_2\\&&&&1&0&0&\frac{1}{2}v'_2\\&&&&&I_{m'-1}&0&0\\&&&&&&I_{m'-1-r}&v'_1\\&&&&&&&I_r\end{pmatrix}, (e^*)^{\Delta_{n-1}}).
\end{equation}
Recall the $m\times m$ block matrices $x_i=x_i(u)$ in \eqref{1.2} and \eqref{1.5.1},
$1\leq i\leq n-1$. Write
\begin{equation}\label{4.4}
x_i(u)=\begin{pmatrix}x_i^{1,1}(u)&x_i^{1,2}(u)&x_i^{1,3}(u)\\x_i^{2,1}(u)&x_i^{2,2}(u)&x_i^{2,3}(u)
\\x_i^{3,1}(u)&x_i^{3,2}(u)&x_i^{3,3}(u)\end{pmatrix},\quad 1\leq i\leq n-1,
\end{equation}
where $x_i^{1,1}(u), x_i^{3,3}(u)\in M_{r\times r}$ (and then $x_i^{1,2}(u)\in M_{r\times
	(m-2r)}$ etc.) We will also denote $x_i^{\alpha,\beta}=x_i^{\alpha,\beta}(u)$.\\
Let $g\in H_m$. Define, for $1\leq i\leq n-1$,
\begin{equation}\label{4.5}
p_i(g)=\diag (I_{m(i-1)},g,I_{2m(n-i)},g^*,I_{m(i-1)}),
\end{equation}
and for $i=n$,
\begin{equation}\label{4.6}
p_n(g)=\diag(I_{m(n-1)},j(g,I_m),I_{m(n-1)}).
\end{equation}
See \eqref{1.7}. 
The first reduction for the computation of the constant term \eqref{4.1} 
is obtained by applying the process of
exchanging roots, in the sense of Lemma 7.1 in \cite{GRS11}, so
that, for $i=1,2...,n-1$, the blocks below the diagonal in $x_i$
(the blocks $x_i^{2,1}, x_i^{3,1}, x_i^{3,2}$) are exchanged with the blocks
above the diagonal in the $i$-th block of the image of the diagonal embedding of $U_r^{H_m}$, given by the map $e(v,z)\mapsto t(e(v,z),1)$.
The precise formulation is as follows. Let, for $0\leq \ell\leq
n-1$, $\mathbf{D}_\ell$ be the unipotent subgroup of all elements of
the form
\begin{equation}\label{4.7}
v=p_1(e_1)p_2(e_2)\cdot...\cdot
p_\ell(e_\ell)p_{\ell+1}(e)p_{\ell+2}(e)\cdot...\cdot p_n(e)u,
\end{equation}
where $e\in U_r^{H_m}$; for $j\leq \ell$, $e_j$ is an element of the unipotent radical $V_{r,m-2r,r}$ of the parabolic subgroup $P_{r,m-2r,r}$ of $\GL_m$,
and $u\in U_{m^{n-1}}$ is of the form \eqref{1.2}, with $x_j^{2,1}(u)=0,
x_j^{3,1}(u)=0, x_j^{3,2}(u)=0$, for $j\leq \ell$. For $v\in
\mathbf{D}_\ell(\BA)$, define
\begin{equation}\label{4.8}
\psi_{\mathbf{D}_\ell}(v)=\psi_{U_{m^{n-1}}}(u).
\end{equation}
This is a character of $\mathbf{D}_\ell(\BA)$, trivial on
$\mathbf{D}_\ell(F)$. For $\ell=0$, we define
$$
\mathbf{D}_0=V^{m,n,r}.
$$
Let, for $\ell\leq n-2$, $\mathbf{Y}_\ell$ be the subgroup of $U_{m^{n-1}}$ consisting of the
elements $u'$ of the form \eqref{1.2}, where except the identity diagonal blocks, all other coordinates are zero, except in the blocks $x_\ell, x^*_\ell$,
and
\begin{equation}\label{4.9}
x_\ell(u')=\begin{pmatrix}0&0&0\\a&0&0\\b&c&0\end{pmatrix},\quad a\in
M_{(m-2r)\times r},\ b\in M_{r\times r},\ c\in M_{r\times (m-2r)}.
\end{equation}
Similarly, we define $\mathbf{Y}_{n-1}$ to be the subset of $U_{m^{n-1}}$ consisting of similar elements $u'\in U_{m^{n-1}}$, with $x_{n-1}(u')$ given by \eqref{4.9}, except that we need to make sure that these elements lie in $H_{2nm}$, and for this we need to allow appropriate coordinates in the block in position $(n-1)\times (n+3)$ in \eqref{1.2}. Then $\mathbf{Y}_{n-1}$ is a subgroup only modulo $\mathbf{D}_{n-1}$. Note that $\psi_{\mathbf{D}_{n-1}}$ is trivial on commutators of any two elements of $\mathbf{Y}_{n-1}(\BA)$.
\begin{prop}\label{prop 4.1}
	Let $1\leq i\leq n-1$. Define, for $\xi\in A(\Delta(\tau,m)\gamma_\psi^{(\epsilon)},\eta,k)$, $h\in
	H(\BA)$,
	\begin{equation}\label{4.10}
	c_i(\xi)(h)=\int_{\mathbf{D}_i(F)\backslash \mathbf{D}_i(\BA)}\xi(vh)\psi_{\mathbf{D}_i}^{-1}(v)dv,
	\end{equation}
	Then
	\begin{equation}\label{4.11}
	c_{i-1}(\xi)(h)=\int_{\mathbf{Y}_i(\BA)}c_i(\xi)(yh)dy,
	\end{equation}
	We have that $c_{i-1}=0$ on $A(\Delta(\tau,m)\gamma_\psi^{(\epsilon)},\eta,k)$, if and only if $c_i=0$ on \\
	$A(\Delta(\tau,m)\gamma_\psi^{(\epsilon)},\eta,k)$. Hence $\mathcal{F}_\psi^r=0$ if and
	only if $c_{n-1}=0$. This proposition is valid for any automorphic
	representation of $H(\BA)$.
\end{prop}
\begin{proof}
	It will be clear that we are not going to use in this proof any
	particular property of $A(\Delta(\tau,m)\gamma_\psi^{(\epsilon)},\eta,k)$, except that it is a representation by right translations in a space of automorphic
	functions. However, in our notations, we
	will keep referring to $A(\Delta(\tau,m)\gamma_\psi^{(\epsilon)},\eta,k)$. Define
	$\mathbf{B}_i'=\mathbf{D}_{i-1}$. Let $\mathbf{C}_i'$ be the
	subgroup of $\mathbf{B}_i'$ consisting of all elements of the form
	\eqref{4.7}, with $\ell=i-1$, such that 
	$x_i^{3,1}(u)=0$. Let $\mathbf{Y}_i'$ be the subgroup of all elements
	$u'\in \mathbf{Y}_i$, such that in \eqref{4.9}, $a=0$ and $c=0$, that is
	\begin{equation}\label{4.12}
	x_i(u')=\begin{pmatrix}0&0&0\\0&0&0\\b&0&0\end{pmatrix},\quad 
	 b\in M_{r\times r}.
	\end{equation}
	Define $\mathbf{X}_i'$ to be the subgroup of all elements of the
	form
	\begin{equation}\label{4.13}
	p_i(\begin{pmatrix}I_r&
	&\beta\\&I_{m-2r}&0\\&&I_r\end{pmatrix}),\quad 
	\beta\in M_{r\times r}.
	\end{equation}
	Denote
	$$
	\mathbf{D}_i'=\mathbf{C}_i'\mathbf{X}_i'.
	$$
	Note that
	$$
	\mathbf{B}_i'=\mathbf{C}_i'\mathbf{Y}_i'=\mathbf{D}_{i-1}.
	$$
	These are unipotent groups, and they satisfy the set-up of Lemma 7.1
	of \cite{GRS11}. For example, the commutator of the element \eqref{4.13}
	of $\mathbf{X}_i'$ with the element of $\mathbf{Y}_i'$, satisfying
	\eqref{4.12}, is the element $c$ in $\mathbf{C}_i'$, all of whose
	blocks above the diagonal are zero, except the block $x_i(c)$, which
	has the form
	\begin{equation}\label{4.14}
	x_i(c)=\begin{pmatrix}b \beta
	&0&0\\0&0&0\\0&0&0\end{pmatrix}.
	\end{equation}
	Note that
	$\psi_{\mathbf{B}_i'}(c)=\psi(tr(b\beta))$ (when we take adele coordinates). Denote by
	$\psi_{\mathbf{C}_i'}$ the restriction of $\psi_{\mathbf{B}_i'}$ to
	$\mathbf{C}_i'(\BA)$. Then $\psi_{\mathbf{C}_i'}$ is preserved
	under conjugation by $\mathbf{Y}_i'(\BA)$ and by
	$\mathbf{X}_i'(\BA)$. Extend $\psi_{\mathbf{C}_i'}$ to
	$\mathbf{D}_i'(\BA)$ by making it trivial on $\mathbf{X}_i'(\BA)$. Denote this extension by $\psi_{\mathbf{D}_i'}$. Then by Lemma
	7.1 in \cite{GRS11}, we have the identity
	\begin{equation}\label{4.15}
	c_{i-1}(\xi)(h)=\int_{\mathbf{Y}'_i(\BA)}\int_{\mathbf{D}_i'(F)\backslash \mathbf{D}_i'(\BA)}\xi(vyh)\psi_{\mathbf{D}_i'}^{-1}(v)dvdy.
	\end{equation}
	By Corollary 7.1 in \cite{GRS11}, $c_{i-1}$ is trivial if and only
	if the following Fourier coefficient is identically zero (in $\xi,\
	h$),
	$$
	c'_i(\xi)(h)=\int_{\mathbf{D}_i'(F)\backslash \mathbf{D}_i'(\BA)}\xi(vh)\psi_{\mathbf{D}_i'}^{-1}(v)dv.
	$$
		
	Let $ \mathbf{B}_i''=\mathbf{D}_i'$, and let $ \mathbf{C}_i''$ be the
	subgroup of all elements $cx\in \mathbf{B}_i''$, with $x\in
	\mathbf{X}_i'$, $c\in \mathbf{C}_i'$, such that $x_i^{2,1}(c)=0$. Let
	$\mathbf{Y}''_i$ the subgroup of all elements $u'\in \mathbf{Y}_i$,
	such that in \eqref{4.12}, $b=0,\ c=0$, that is
	\begin{equation}\label{4.16}
	x_i(u')=\begin{pmatrix}0&0&0\\a&0&0\\0&0&0\end{pmatrix},\quad a\in
	M_{(m-2r)\times r}.
	\end{equation}
	Define $\mathbf{X}_i''$ to be the subgroup of all elements of the form
	\begin{equation}\label{4.17}
	p_i(\begin{pmatrix}I_r&\alpha&0\\&I_{m-2r}&0\\&&I_r\end{pmatrix}),\quad \alpha\in
	M_{r\times(m-2r)}.
	\end{equation}
	Denote 
	$$
	\mathbf{D}_i''=\mathbf{C}_i''\mathbf{X}_i''.
	$$
	We have
	$$
	\mathbf{B}_i''=\mathbf{C}_i''\mathbf{Y}''_i=\mathbf{D}_i'.
	$$
	The set-up of Lemma 7.1 of \cite{GRS11} is satisfied. For example,
	the commutator the element \eqref{4.17} of $\mathbf{X}_i''$ with the
	element of  $\mathbf{Y}_i''$, satisfying \eqref{4.16}, is the element
	$e$ in $\mathbf{C}_i''$, all of whose blocks above the diagonal are
	zero, except the block $x_i(e)$, which has the form
	$$
	x_i(e)=\begin{pmatrix}\alpha a&0&0\\0&0&0\\0&0&0\end{pmatrix}.
	$$
	For such $e\in \mathbf{C}_i''(\BA)$,
	$\psi_{\mathbf{B}_i}(e)=\psi(tr(\alpha a))$. As before, denote by
	$\psi_{\mathbf{C}_i''}$ the restriction of $\psi_{\mathbf{B}_i''}$ to
	$\mathbf{C}_i''(\BA)$. Note that $\psi_{\mathbf{C}_i''}$ is also the
	restriction of $\psi_{\mathbf{D}_i''}$ to $\mathbf{C}_i''(\BA)$. Then
	by Lemma 7.1 in \cite{GRS11}, we have the identity
	\begin{equation}\label{4.18}
	c'_i(\xi)(h)=\int_{\mathbf{Y}''_i(\BA)}\int_{\mathbf{D}_i''(F)\backslash \mathbf{D}_i''(\BA)}\xi(vyh)\psi_{\mathbf{D}_i''}^{-1}(v)dvdy,
	\end{equation}
	and hence, by \eqref{4.14},
	\begin{equation}\label{4.19}
	c_{i-1}(\xi)(g)=\int_{\mathbf{Y}''_i\mathbf{Y}'_i(\BA)}\int_{\mathbf{D}''_i(F)\backslash \mathbf{D}''_i(\BA)}\xi(vyg)\psi_{\mathbf{D}''_i}^{-1}(v)dvdy.
\end{equation}
	By Corollary 7.1 in \cite{GRS11},
	$c'_i=0$ if and only if the following integral is identically zero,
	$$
	c_i''(\xi)(h)=\int_{\mathbf{D}_i''(F)\backslash \mathbf{D}_i''(\BA)}\xi(vh)\psi_{\mathbf{D}_i''}^{-1}(v)dv.
	$$ 
	Thus, $c_{i-1}=0$ if and only if $c_i''=0$.
	 
	Let $\mathbf{B}_i=\mathbf{D}_i''$, and let $\mathbf{C}_i$ be the subgroup of elements $cx\in \mathbf{B}_i$, with $x\in \mathbf{X}_i''$, such that $x^{3,2}_i(c)=0$. Let $\mathbf{Y}'''_i$ be the subgroup of all elements $u'\in \mathbf{Y}_i$, such that in \eqref{4.12}, $a=0, b=0$, that is 
	\begin{equation}\label{4.20}
	x_i(u')=\begin{pmatrix}0&0&0\\0&0&0\\0&c&0\end{pmatrix}, \ c\in M_{r\times (m-2r)}.
	\end{equation}
	Define $\mathbf{X}'''_i$ to be the subgroup of elements of the form 
	\begin{equation}\label{4.21}
	p_i(\begin{pmatrix}I_r&0&0\\&I_{m-2r}&x\\&&I_r\end{pmatrix}),\ x\in M_{(m-2r)\times r}.
	\end{equation}
	Note that 
	$$
	\mathbf{D}_i=\mathbf{C}_i\mathbf{X}'''_i.
	$$
	We have 
	$$
	\mathbf{B}_i=\mathbf{C}_i\mathbf{Y}_i'''=\mathbf{D}_i''.
	$$
	The set-up of Lemma 7.1 of \cite{GRS11} is satisfied. For example, the commutator of the element \eqref{4.21} of $\mathbf{X}'''_i$ with the element of $\mathbf{Y}_i$, satisfying \eqref{4.20}, is the element $e$ of $\mathbf{C}_i$, all of whose blocks above the diagonal are zero, except the block $x_i(e)$, which has the form
	$$
	x_i(e)=\begin{pmatrix}0&0&0\\0&xc&0\\0&0&0\end{pmatrix}.
	$$
	For such $e\in \mathbf{C}_i(\BA)$, $\psi_{\mathbf{B}_i}(e)=\psi(tr(xc))$. Denote by $\psi_{\mathbf{C}_i}$ the restriction of $\psi_{\mathbf{B}_i}$ to $\mathbf{C}_i(\BA)$. Then by Lemma 7.1 in \cite{GRS11}, we have the identity
	\begin{equation}\label{4.22}
	c''_i(\xi)(h)=\int_{\mathbf{Y}'''_i(\BA)}\int_{\mathbf{D}_i(F)\backslash \mathbf{D}_i(\BA)}\xi(vyh)\psi_{\mathbf{D}_i}^{-1}(v)dvdy,
	\end{equation}
	and hence, by \eqref{4.14},
	\begin{equation}\label{4.23}
	c_{i-1}(\xi)(h)=\int_{\mathbf{Y}_i(\BA)}\int_{\mathbf{D}_i(F)\backslash \mathbf{D}_i(\BA)}\xi(vyh)\psi_{\mathbf{D}_i}^{-1}(v)dvdy.
	\end{equation}
	This is \eqref{4.11}. By Corollary 7.1 in \cite{GRS11},
	$c''_i=0$ if and only if $c_i=0$. Thus, $c_{i-1}=0$ if and only if $c_i=0$.
\end{proof}
The elements of $\mathbf{D}_{n-1}$ have the form (see \eqref{4.7})
\begin{equation}\label{4.24}
v=p_1(e_1)p_2(e_2)\cdot...\cdot
p_{n-1}(e_{n-1})p_n(e)u,
\end{equation}
where, for $j\leq n-1$, $e_j$ is an element of the unipotent radical $V_{r,m-2r,r}$ of the parabolic subgroup $P_{r,m-2r,r}$ of $\GL_m$; $e=e(v,z)\in U_r^{H_m}$, $p_n(e)$ is as in \eqref{4.6}, and $u\in U_{m^{n-1}}$ is of the form \eqref{1.2} with $x_j^{2,1}(u)=0,
x_j^{3,1}(u)=0, x_j^{3,2}(u)=0$, for $j\leq n-1$. (Recall \eqref{1.5.1}.) Note that for $v\in
\mathbf{D}_{n-1}(\BA)$, as in \eqref{4.24},
\begin{equation}\label{4.25}
\psi_{\mathbf{D}_{n-1}}(v)=\psi(\sum_{i=1}^{n-1}(tr(x_i^{1,1}(u))+tr(x_i^{2,2}(u))+tr(x_i^{3,3}(u)))).
\end{equation}
From Prop. \ref{prop 4.1}, we conclude
\begin{cor}\label{cor 4.2}
	Let $A$ be an automorphic representation of $H(\BA)$. For an automorphic form $\xi$ in the space of $A$, we have, for any $1\leq r\leq [\frac{m}{2}]$,\\
	\\
	$\mathcal{F}_\psi^r(\xi)(h)=$
    $$
	\int_{\mathbf{Y}_1(\BA)}\cdots \int_{\mathbf{Y}_{n-1}(\BA)}\int_{\mathbf{D}_{n-1}(F)\backslash \mathbf{D}_{n-1}(\BA)}\xi(vy_{n-1}\cdots y_1h)\psi_{\mathbf{D}_{n-1}}^{-1}(v)dvdy_{n-1}\cdots dy_1.
	$$
Moreover, $\mathcal{F}_\psi^r=0$ if and
	only if, for all $\xi$ and all $h$, 
	\begin{equation}\label{4.26}
	c_{n-1}(\xi)(h)=\int_{\mathbf{D}_{n-1}(F)\backslash \mathbf{D}_{n-1}(\BA)}\xi(vh)\psi_{\mathbf{D}_{n-1}}^{-1}(v)dv=0.
	\end{equation}
\end{cor}

\section{ A conjugation by a Weyl element}

We will conjugate the integration inside \eqref{4.26} by the
following product of Weyl elements $w_0=w'_0w''_0$. The element $w''_0$ is easy to describe,
\begin{equation}\label{6.0}
w''_0=diag(I_{m(n-1)+r},\begin{pmatrix}&I_{[\frac{m}{2}]}\\I_{[\frac{m}{2}]-r}\end{pmatrix},I_{2m-4[\frac{m}{2}]},\begin{pmatrix}&I_{[\frac{m}{2}]-r}\\I_{[\frac{m}{2}]}\end{pmatrix}, I_{m(n-1)+r}).
\end{equation}
Let $\omega_0=I_{2mn}$, when $H$ is symplectic, and when $H$ is orthogonal,
\begin{equation}\label{6.0'}
\omega_0=diag(I_{mn-1},\begin{pmatrix}&1\\1\end{pmatrix},I_{mn-1}).
\end{equation}
The element $w'_0$ has the following form
\begin{equation}\label{6.1}
w'_0=\begin{pmatrix}A_1&0&A_2\\A_3&0&0\\0&I_{2(m-r)}&0\\0&0&A_4\\ A_5&0&A_6\end{pmatrix}\omega_0^{r(n-1)}.
\end{equation}
The block $A_1$ has the following form. It has $2n-1$ block rows,
each one of size $r$; the last $n-1$ block rows are all zero. It has
$n$ block columns, the first $n-1$ of which are each of size $m$,
and the last one is of size $r$.
\begin{equation}\label{6.2}
A_1=\begin{pmatrix}a\\&a\\&&a\\&&&\cdots\\&&&&a\\&&&&&I_r\\0&&&\cdots
&&0\\ \vdots &&&&&\vdots\\ 0&&&\cdots &&0\end{pmatrix},\quad
a=\begin{pmatrix}I_r&0_{r\times (m-2r)}&0_{r\times r}\end{pmatrix}.
\end{equation}
The block $A_2$ has the same block row division as $A_1$; its first
$n$ block rows are zero. It has $n$ block columns, the first one is
of size $r$ and the next $n-1$ ones are of size $m$ each.
\begin{equation}\label{6.3'}
A_2=\begin{pmatrix}0&&&&0\\ \vdots
&&&&\vdots\\0&a&&&0\\&&a\\&&&\cdots
\\&&&&a\end{pmatrix}.
\end{equation}
The block $A_3$ has $n-1$ block rows, each one of size $m-2r$. It has the same block column division as $A_1$.
\begin{equation}\label{6.3''}
A_3=\begin{pmatrix} b\\&b\\&&b\\&&&\cdots\\&&&&b&0_{(m-2r)\times r}\end{pmatrix},
\quad b=\begin{pmatrix}0_{(m-2r)\times r}&I_{m-2r}&0_{(m-2r)\times r}\end{pmatrix}.
\end{equation}
The blocks $A_4, A_5, A_6$ are already determined by $A_1, A_2, A_3$. For example, the block $A_4$ has the same block row division as $A_3$ and the same block column division as $A_2$.
\begin{equation}\label{6.3'''}
A_4=\begin{pmatrix}0_{(m-2r)\times r}& b\\&&b\\&&&b\\&&&&\cdots\\&&&&&b\end{pmatrix}.
\end{equation}
Let us describe the subgroup
$\mathcal{D}_{n,m,r}=w_0\mathbf{D}_{n-1}w_0^{-1}$. It consists
of elements of the form
\begin{equation}\label{6.3}
v=\begin{pmatrix}U&X_{1,2}&X_{1,3}\\Y_{2,1}&V&X'_{1,2}\\Y_{3,1}&Y'_{2,1}&U'\end{pmatrix}\in
H,
\end{equation}
where $U$ is a $(2n-1)r\times (2n-1)r$ unipotent matrix lying in the
unipotent radical $V_{r^{2n-1}}$ of $P_{r^{2n-1}}\subset \GL_{(2n-1)r}$, 
\begin{equation}\label{6.4}
U=\begin{pmatrix}I_r&U_1&*&\cdots&*&*\\
&I_r&U_2&\cdots&*&*\\
& &I_r&\cdots &* & * \\
& & & \cdots& \cdots&\cdots &\\
& & &       &I_r&U_{2n-2}\\
& & &       & &I_r\end{pmatrix}.
\end{equation}
The block $U'$ is of the same size as $U$ and has the form
\eqref{6.4}. The block $V$ is of size $(2(m-2r)(n-1)+2(m-r))\times
(2(m-2r)(n-1)+2(m-r))$ and has the form 
\begin{equation}\label{6.5}
V=\begin{pmatrix}
S&E&C\\
&I_{2(m-r)}&E'\\
& &S'\end{pmatrix},
\end{equation}
where $S$ is $(m-2r)(n-1)\times (m-2r)(n-1)$ upper unipotent matrix of the
form
\begin{equation}\label{6.5.1}
S=\begin{pmatrix}I_{m-2r}&S_1&*&\cdots&*&*\\
&I_{m-2r}&S_2&\cdots&*&*\\
& &I_{m-2r}&\cdots &* & * \\
& & & \cdots& \cdots&\cdots &\\
& & &       &I_{m-2r}&S_{n-2}\\
& & &       & &I_{m-2r}\end{pmatrix}.
\end{equation}
$S'$ has a similar form to that of $S$. In order to describe later the
character obtained as the result of conjugating $\psi_{\mathbf{D}_{n-1}}$
by $w_0$, we will need the following notation. Write the
$(m-2r)(n-1)\times
2(m-r)$ matrix $E$ in \eqref{6.5} in the form $E=\begin{pmatrix}*\\
Y\end{pmatrix}$, with $Y$ having size $(m-2r)\times 2(m-r)$. Next,
write 
\begin{equation}\label{6.5.2}
Y=(Y_1,S_{n-1},Y_3), 
\end{equation}
with $Y_1$ and $Y_3$ of size $(m-2r)\times[\frac{m}{2}]$ and $S_{n-1}$ of size $(m-2r)\times2([\frac{m+1}{2}]-r)$. We continue with the description of $v$
in \eqref{6.3}.

Write the block $X_{1,3}$ (which is a matrix of size $(2n-1)r\times
(2n-1)r$) as a $(2n-1)\times (2n-1)$ matrix of blocks
$X_{1,3}^{(i,j)}$, each one of size $r\times r$ ($1\leq i,j\leq
2n-1$). Then $X_{1,3}^{(i,j)}=0$, for all $i>j$. Thus, $X_{1,3}$ has
an upper triangular shape as a $(2n-1)\times (2n-1)$ matrix of
$r\times r$ blocks.

The form of the block $X_{1,2}$ in \eqref{6.3} can be described as
follows. We  will write it as a block matrix. It has $2n-1$ block
rows, each one of size $r$ (i.e. each one contains $r$ rows). It has $2n+1$ block columns. The first and last
$n-1$ block columns are each of size $m-2r$ (i.e. each such block contains $m-2r$ columns).
The first and third of the middle three block columns are of size $[\frac{m}{2}]$, and
the second has $2([\frac{m+1}{2}]-r)$ columns. Denote by $\widetilde{X}_{1,2}$ the matrix obtained from
$X_{1,2}$ by deleting block column number $n$ and block column
number $n+2$. Each one of them contains $[\frac{m}{2}]$ columns. Then
$\widetilde{X}_{1,2}$ is a matrix of size $(2n-1)r\times (2n-1)(m-2r)$.
Write it as a $(2n-1)\times (2n-1)$ block matrix, where each block
is of size $r\times (m-2r)$, except the blocks of the middle column, which are each of size $r\times 2([\frac{m+1}{2}]-r)$. Then $\widetilde{X}_{1,2}$ has an upper
triangular shape. That is, if we denote the block in position
$(i,j)$, $1\leq i,j\leq 2n-1$, by $\widetilde{X}_{1,2}^{(i,j)}$, then $\widetilde{X}_{1,2}^{(i,j)}=0$, whenever
$i>j$. The two above columns deleted from $X_{1,2}$ have the same shape. If we write each
of these two as a column of $2n-1$ matrices of size $r\times
[\frac{m}{2}]$, then the last $n$ matrices are all zero. We also have to note the following when $m=2m'-1$ is odd. In this case, the block of $X_{1,2}$ in position $(n,n+1)$, which is a matrix of size $r\times 2(m'-r)$ has the form, 
\begin{equation}\label{6.5.2.1}
\begin{pmatrix}v_1,(\frac{1}{2}v_2,v_2)w_2^{r(n-1)},v_3\end{pmatrix}, 
\end{equation}
where $v_1,v_3$ are matrices of size $r\times (m'-1-r)$, and $v_2$ is a  column of $r$ coordinates. This originates from $v$ in $e=e(v,z)$ in \eqref{4.24}. For example, when $n=4$, $X_{1,2}$ has the following form
\begin{equation}\label{6.5.3}
X_{1,2}=\begin{pmatrix}\ast&\ast&\ast&\ast&\ast&\ast&\ast&\ast&\ast\\
0&\ast&\ast&\ast&\ast&\ast&\ast&\ast&\ast\\
0&0&\ast&\ast&\ast&\ast&\ast&\ast&\ast\\
0&0&0&0&\ast&0&\ast&\ast&\ast\\0&0&0&0&0&0&\ast&\ast&\ast\\
0&0&0&0&0&0&0&\ast&\ast\\0&0&0&0&0&0&0&0&\ast\end{pmatrix}.
\end{equation}
In this example, all blocks are of size $r\times (m-2r)$ except the
blocks in columns 4, 5, 6. The blocks in columns 4, 6 are each of size $r\times [\frac{m}{2}]$. The blocks in column 5 are each of size $r\times 2([\frac{m+1}{2}]-r)$.
 
The block $X_{1,2}'$ has a form "dual" to that of $X_{1,2}$. It has $2n-1$
block columns, each one of size $r$, and it has $2n+1$ block rows,
the first and last $n-1$ block rows, each of which of size $m-2r$, and
the three middle block rows which are of sizes $[\frac{m}{2}],2([\frac{m+1}{2}]-r),[\frac{m}{2}]$. When we delete
rows number $n, n+2$, we obtain a $(2n-1)\times (2n-1)$ block matrix $\tilde{X}'_{1,2}$,
where each of its blocks is of size $(m-2r)\times r$, except the blocks of the middle row, which are each of size $r\times 2([\frac{m+1}{2}]-r)$.  Then $\tilde{X}'_{1,2}$ has an upper triangular shape. The rows number $n, n+2$ of $X_{1,2}'$ are rows of blocks of size
$[\frac{m}{2}]\times r$, and their first $n$ blocks are all zero. Again, when $m=2m'-1$ is odd, the block of $X'_{1,2}$ in position $(n+1,n)$ has the form, 
\begin{equation}\label{6.5.3.1}
-\begin{pmatrix}w_{m'-1-r}{}^tv_3w_r\\w_2^{r(n-1)}\begin{pmatrix}{}^tv_2w_r\\ \frac{1}{2}{}^tv_2w_r\end{pmatrix}\\w_{m'-1-r}{}^tv_1w_r\end{pmatrix},
\end{equation}
where $v_1,v_2,v_3$ are as in \eqref{6.5.2.1}. For example, when $n=4$, $X'_{1,2}$ has the form
\begin{equation}\label{6.5.4}
X'_{1,2}=\begin{pmatrix}\ast&\ast&\ast&\ast&\ast&\ast&\ast\\
0&\ast&\ast&\ast&\ast&\ast&\ast\\
0&0&\ast&\ast&\ast&\ast&\ast\\
0&0&0&0&\ast&\ast&\ast\\
0&0&0&\ast&\ast&\ast&\ast\\
0&0&0&0&\ast&\ast&\ast\\
0&0&0&0&\ast&\ast&\ast\\
0&0&0&0&0&\ast&\ast\\
0&0&0&0&0&0&\ast\end{pmatrix}. 
\end{equation}
In this example, all blocks are of size $(m-2r)\times r$ except the
blocks in rows 4, 5, 6. The blocks in rows 4, 6 are each of size $[\frac{m}{2}]\times r$. The blocks in row 5 are each of size $2([\frac{m+1}{2}]-r)\times r$.

Write the block $Y_{3,1}$ (which is a matrix of size $(2n-1)r\times
(2n-1)r$) as a $(2n-1)\times (2n-1)$ matrix of blocks
$Y_{3,1}^{(i,j)}$, each one of size $r\times r$ ($1\leq i,j\leq
2n-1$). Then $Y_{3,1}^{(i,j)}=0$, for all $i\geq j-1$. Thus, as a
$(2n-1)\times (2n-1)$ matrix of $r\times r$ blocks, $Y_{3,1}$ has an
upper triangular shape, with zero blocks on the diagonal and also on
the second upper diagonal.

The block $Y_{2,1}$ has the same block division as that of $X'_{1,2}$, that is $2n-1$ block columns, each one of size $r$. It has $2n+1$ block rows, the size of the first and last
$n-1$ block rows is $m-2r$ and the sizes of
the three middle block rows are $[\frac{m}{2}],2([\frac{m+1}{2}]-r),[\frac{m}{2}]$. Denote by
$\widetilde{Y}_{2,1}$ the matrix obtained from $Y_{2,1}$ by deleting
block rows number $n, n+2$ (each one is of size $[\frac{m}{2}]$). Then
$\widetilde{Y}_{2,1}$ has the same block division as that of $\tilde{X}_{1,2}'$, and it has a shape of an upper triangular matrix of blocks, where all the diagonal blocks are zero, and when $m$ is even, also the blocks in the second upper diagonal are zero. When $m=2m'-1$ is odd, the blocks in the second upper diagonal of $\tilde{Y}_{2,1}$ are zero, except the middle block in row $n$ and column $n+1$, which has the following form,
 \begin{equation}\label{6.6}
\begin{pmatrix}0_{(m'-1-r)\times r}\\w_2^{r(n-1)}\begin{pmatrix}x\\-\frac{1}{2}x\end{pmatrix}\\0_{(m'-1-r)\times r}\end{pmatrix},
\end{equation}
where $x$ is a row of $r$ coordinates. The rows number $n, n+2$ of $Y_{2,1}$ are rows of blocks of size $[\frac{m}{2}]\times r$, and their first $n$ blocks are all zero. This originates from the condition that $x_{n-1}(u)$ in \eqref{4.24} has an upper triangular shape. (See \eqref{1.5.1}, \eqref{4.4}.) When $n=4$, and $m$ is even, $Y_{2,1}$ has the following form
\begin{equation}\label{6.6.0}
Y_{2,1}=\begin{pmatrix}0&0&\ast&\ast&\ast&\ast&\ast\\
0&0&0&\ast&\ast&\ast&\ast\\
0&0&0&0&\ast&\ast&\ast\\
0&0&0&0&\ast&\ast&\ast\\
0&0&0&0&0&\ast&\ast\\
0&0&0&0&\ast&\ast&\ast\\
0&0&0&0&0&0&\ast\\
0&0&0&0&0&0&0\\
0&0&0&0&0&0&0\end{pmatrix}. 
\end{equation}
In this example, all blocks are of size $(m-2r)\times r$ except the
blocks in rows 4, 6. The blocks in rows 4, 6 are each of size $[\frac{m}{2}]\times r$. The blocks in row 5 are each of size $2([\frac{m+1}{2}]-r)\times r$. For $n=4$ and $m$ odd, we have to replace the zero block in position $(5,5)$ by a block of the form \eqref{6.6}.

Finally, the block $Y_{2,1}'$ has a form "dual" to that of $Y_{2,1}$. It has the same block division as that of $X_{1,2}$, that is
$2n-1$ block rows, each one of size $r$. It has $2n+1$ block
columns, the first and last $n-1$ columns are of size
$m-2r$ each and the three middle columns are of sizes $[\frac{m}{2}],2([\frac{m+1}{2}]-r),[\frac{m}{2}]$.
When we delete columns number $n, n+2$, we obtain
$\widetilde{Y'}_{2,1}$, a $(2n-1)\times (2n-1)$ matrix of blocks, where each block
is of size $r\times (m-2r)$, except the blocks of the middle column, which are each of size $r\times 2([\frac{m+1}{2}]-r)$. Then $\widetilde{Y}'_{2,1}$ has a shape of an upper triangular matrix of blocks, where all the diagonal blocks are zero, and when $m$ is even, also the blocks in the second upper diagonal are zero. When $m=2m'-1$ is odd, the blocks in the second upper diagonal of $\tilde{Y}_{2,1}'$ are zero, except the middle block in row $n$ and column $n+1$, which has the form,
\begin{equation}\label{6.6.1}
(0_{r\times (m'-1-r)},(-\frac{1}{2}w_r{}^tx,w_r{}^tx)w_2^{r(n-1)},0_{r\times (m'-1-r)}),
\end{equation}
where $x$ is as in \eqref{6.6}. The columns number $n, n+2$ of $Y'_{2,1}$ are columns of blocks of size $r\times [\frac{m}{2}]$, and their last $n$ blocks are all zero. For example, for $n=4$, and $m$ even,
\begin{equation}\label{6.6.2}
Y'_{2,1}=\begin{pmatrix}0&0&\ast&\ast&\ast&\ast&\ast&\ast&\ast\\
0&0&0&\ast&\ast&\ast&\ast&\ast&\ast\\
0&0&0&\ast&0&\ast&\ast&\ast&\ast\\
0&0&0&0&0&0&0&\ast&\ast\\0&0&0&0&0&0&0&0&\ast\\
0&0&0&0&0&0&0&0&0\\0&0&0&0&0&0&0&0&0\end{pmatrix}.
\end{equation}
In this example, all blocks are of size $r\times (m-2r)$ except the
blocks in columns 4, 5, 6. The blocks in columns 4, 5, 6 are each of sizes $r\times [\frac{m}{2}], r\times 2([\frac{m+1}{2}]-r), r\times [\frac{m}{2}]$. When $n=4$ and $m$ odd, we have to replace the zero block in position $(3,5)$ by a block of the form \eqref{6.6.1}.
Let $\psi_{\mathcal{D}_{n,m,r}}$ be the character of
$\mathcal{D}_{n,m,r}(\BA)$ defined by
$$
\psi_{\mathcal{D}_{n,m,r}}(v)=\psi_{\mathbf{D}_{n-1}}(w_0^{-1}vw_0).
$$
Then in the notation of \eqref{6.3}-\eqref{6.5.2}, we have, for $v\in
\mathcal{D}_{n,m,r}(\BA)$,
\begin{equation}\label{6.7.1}
\psi_{\mathcal{D}_{n,m,r}}(v)=\prod_{i=1}^{n-1}\psi(tr(U_i)\psi^{-1}(tr(U_{n-1+i}))\prod_{i=1}^{n-2}\psi(tr(S_i))\psi(tr(S_{n-1}A'_H)),
\end{equation}
where $A'_H=I_{m-2r}$, when $H_{2nm}$ is symplectic. When $H_{2nm}$ is orthogonal, and $m=2m'$ is even,
\begin{equation}\label{6.7.1.1}
A'_H=\begin{pmatrix} I_{m'-r-1}\\&w_2^{r(n-1)}\\&&I_{m'-r-1}\end{pmatrix}.
\end{equation}
(If $2r=m$, then we ignore $\psi(tr(S_{n-1}A'_H))$ in \eqref{6.7.1}.) 
When $m=2m'-1$ is odd, 
\begin{equation}\label{6.7.2}
A'_H=\begin{pmatrix}I_{m'-1-r}\\&w_2^{r(n-1)}\begin{pmatrix}1\\\frac{1}{2}\end{pmatrix}\\&&I_{m'-1-r}\end{pmatrix}.
\end{equation}
Carrying out in \eqref{4.26} the conjugation by $w_0$, we get
\begin{equation}\label{6.7.3}
c_{n-1}(\xi)(h)=\int_{\mathcal{D}_{n,m,r}(F)\backslash \mathcal{D}_{n,m,r}(\BA)}\xi(vw_0h)\psi_{\mathcal{D}_{n,m,r}}^{-1}(v)dv.
\end{equation}
Let us denote, for $\xi$ in the space of an automorphic representation $A$ of $H(\BA)$, and $1\leq r\leq [\frac{m}{2}]$,
\begin{equation}\label{6.7}
\epsilon_{n,m,r}(\xi)(h)=\int_{\mathcal{D}_{n,m,r}(F)\backslash
	\mathcal{D}_{n,m,r}(\BA)}\xi(vh)\psi_{\mathcal{D}_{n,m,r}}^{-1}(v)dv=c_{n-1}(\xi)(w_0^{-1}h).
\end{equation}
Then Cor. \ref{cor 4.2} says that $\mathcal{F}_\psi^r$ is zero on $A$ if and only if $\epsilon_{n,m,r}$ is zero on $A$. Note that
\begin{equation}\label{6.7.4}
\mathcal{F}_\psi^r(\xi)(h)=\int_{\mathbf{Y}_1(\BA)}\cdots \int_{\mathbf{Y}_{n-1}(\BA)}\epsilon_{n,m,r}(\xi)(w_0y_{n-1}\cdots y_1h)dy_{n-1}\cdots dy_1.
\end{equation}
Recall that one of our goals is to analyze $\mathcal{F}_\psi^r(\xi)$, for $\xi$ varying in $A(\Delta(\tau,m)\gamma_\psi^{(\epsilon)},\eta,k)$, and, in particular, find conditions that it is zero. Recall that $\mathcal{F}_\psi^r(\xi)$ computes the constant term along $U_r^{H_m}\times 1$ of automorphic functions in $\mathcal{E}_k(\Delta(\tau,m)\gamma_\psi^{(\epsilon)},\eta)$. 
Our first major step towards this goal will be to prove the following theorem. Denote
\begin{equation}\label{6.8}
\mathrm{N}=\left
\{\begin{pmatrix}U&0&0\\0&V&0\\0&0&U^*\end{pmatrix}\in
\mathcal{D}_{n,m,r},\quad \textrm{of the form}
\eqref{6.3}-\eqref{6.5.1}\right\},
\end{equation}
and let $\psi_N$ be the restriction of $\psi_{\mathcal{D}_{n,m,r}}$
to $\mathrm{N}(\BA)$. Denote
\begin{equation}\label{6.8'}
\mathcal{D'}_{n,m,r}=\mathrm{N}\widetilde{\mathrm{X}},
\end{equation}
where $\widetilde{\mathrm{X}}$ is the unipotent subgroup consisting
of the matrices
\begin{equation}\label{6.8''}
\begin{pmatrix}I_{(2n-1)r}&X&C\\&I_{2n(m-2r)+2r}&X'\\&&I_{(2n-1)r}\end{pmatrix}\in
H_{2nm},
\end{equation}
where if we write $X$ as a matrix with $2n-1$ block rows, of size
$r$ each, then its last block row, when regarded as a matrix of size
$r\times (2(n-1)(m-2r)+2(m-r))$, is such that all its first $(2n-3)(m-2r)+2(m-r)$
columns are zero. Now, regard $X$ as a matrix with $2n-1$ block columns, according
to the block column division of $S, S^*$ in \eqref{6.5} and of $Y$ is \eqref{6.5.2}. Thus, the first and last $n-1$ block columns
of $X$ are of size $m-2r$ each, and the middle three columns are of sizes $[\frac{m}{2}], 2([\frac{m+1}{2}]-r), [\frac{m}{2}]$, respectively.
Then, for the definition of $\widetilde{\mathrm{X}}$, we require that the matrices $X$ in \eqref{6.8''} are such that in the last
block row of $X$, all block columns are zero, except the last one.
$\mathcal{D'}_{n,m,r}$ is an $F$-unipotent
subgroup, and the extension of $\psi_N$ to
$\mathcal{D'}_{n,m,r}(\BA)$ by the trivial character of
$\widetilde{\mathrm{X}}(\BA)$ is a character
$\psi_{\mathcal{D'}_{n,m,r}}$ of $\mathcal{D'}_{n,m,r}(\BA)$,
trivial on $\mathcal{D'}_{n,m,r}(F)$. 
Then our first major step
mentioned before is
\begin{thm}\label{thm 6.1}
The Fourier coefficient $\epsilon_{n,m,r}$ is identically zero on
$A(\Delta(\tau,m)\gamma_\psi^{(\epsilon)},\eta,k)$ if and only if
$\epsilon'_{n,m,r}$ is identically zero on
$A(\Delta(\tau,m)\gamma_\psi^{(\epsilon)},\eta,k)$, where
$$
\epsilon'_{n,m,r}(\xi)(h)=\int_{\mathcal{D'}_{n,m,r}(F)\backslash
\mathcal{D'}_{n,m,r}(\BA)}\xi(vh)\psi_{\mathcal{D'}_{n,m,r}}^{-1}(v)dv.
$$
Moreover, there is a unipotent subgroup $\mathrm{Y}$ of $H$ (to be specified in the proof), such that 
$$
\epsilon_{n,m,r}(\xi)(h)=\int_{\mathrm{Y}(\BA)}\epsilon'_{n,m,r}(\xi)(yh)dy.
$$ 
\end{thm}
The proof of this theorem will be concluded in the end of Sec. 7.
\bigskip

As a preparation for the following sections, we need to set
up more notation. Denote
\begin{equation}\label{6.9}
\mathcal{X}=\left\{\begin{pmatrix}I_{(2n-1)r}&X_{1,2}&X_{1,3}\\&I_{2n(m-2r)+2r}&X'_{1,2}\\
&&I_{(2n-1)r}\end{pmatrix}\in H_{2nm}\right\}=U_{(2n-1)r}.
\end{equation}
$$
\mathrm{X}=\mathcal{X}\cap \mathcal{D}_{n,m,r}.
$$
We will write the matrices $X_{1,3}, X_{1,2}$ as matrices of blocks
with the same block division as above when we described the elements
of $\mathcal{D}_{n,m,r}$. Thus, we regard $X_{1,3}$ as a
$(2n-1)\times (2n-1)$ matrix of $r\times r$ blocks, and, as we did in \eqref{6.8''}, we regard $X_{1,2}$ as
a matrix with $2n-1$ block rows, each one of size $r$, and $2n+1$
block columns, so that the first $n-1$ as well as the last $n-1$
columns are of size $m-2r$ each, and the middle three columns are of
sizes $[\frac{m}{2}], 2([\frac{m+1}{2}]-r), [\frac{m}{2}]$. We will denote the block of $X_{1,3}$ in position $(i,j)$ by $X_{1,3}^{(i,j)}$, and similarly for $X_{1,2}$ and
$X'_{1,2}$. For $1\leq i,j\leq 2n-1$, denote by
$\mathrm{X}_{1,3}^{i,j}$ the subgroup of $\mathcal{X}$ given by the
elements of the form \eqref{6.9}, with $X_{1,2}=0$ and with
$X_{1,3}$, such that all blocks $X_{1,3}^{(i',j')}$ are zero, except
blocks $X_{1,3}^{(i,j)}, X_{1,3}^{(2n-j,2n-i)}$. Note that $\mathrm{X}_{1,3}^{i,j}=\mathrm{X}_{1,3}^{2n-j,2n-i}$. Similarly, for $j\neq n+1$, let
$\mathrm{X}_{1,2}^{i,j}$ denote the subgroup of $\mathcal{X}$ given by
the elements of the form \eqref{6.9}, with $X_{1,3}=0$, and such
that all the blocks of $X_{1,2}$ (in the block division just
described) are zero except block $(i,j)$. When $j=n+1$, we let  $\mathrm{X}_{1,2}^{i,n+1}$ 
denote any set of representatives, modulo $\mathrm{X}_{1,3}^{i,2n-i}$, of the subgroup of $\mathcal{X}$,
consisting of the matrices \eqref{6.9}, such that all blocks $X_{1,2}^{(i',j')}$ are zero,
except when $(i',j')=(i,n+1)$, and all blocks $X_{1,3}^{(i',j')}$ are zero, except when $(i',j')=(i,2n-i)$.
 \\
Let
\begin{equation}\label{6.10}
\mathrm{Y}=\left\{\begin{pmatrix}I_{(2n-1)r}\\Y_{2,1}&I_{2n(m-2r)+2r}\\
Y_{3,1}&Y'_{2,1}&I_{(2n-1)r}\end{pmatrix}\in
\mathcal{D}_{n,m,r}\right\}.
\end{equation}
Recall that we have already described the matrices $Y_{3,1}$, $Y_{2,1}$
(and hence also $Y'_{2,1}$) as block matrices. Denote, for $3\leq
j\leq 2n-1$ and $1\leq i\leq j-2$, by $\mathrm{Y}_{1,3}^{i,j}$ the
subgroup of $\mathrm{Y}$ consisting of the elements \eqref{6.10},
with $Y_{2,1}=0$ and all blocks of $Y_{3,1}$ are zero, except the
blocks $Y_{3,1}^{(i,j)}$ and $Y_{3,1}^{(2n-j,2n-i)}$. Similarly, let, for $i\neq n+1$,
$\mathrm{Y}_{2,1}^{i,j}$ denote the subgroup of $\mathrm{Y}$ given by
the elements of the form \eqref{6.10}, with $Y_{3,1}=0$, and such
that all the blocks of $Y_{2,1}$ are zero except block $(i,j)$. (We take the pair of indices $(i,j)$, such that
$\mathrm{Y}_{2,1}^{i,j}\subset \mathcal{D}_{n,m,r}$.) As before, we let
$\mathrm{Y}_{2,1}^{n+1,j}$ denote any set of representatives, modulo $\mathrm{Y}^{2n-j,n+1}$,
of the subgroup of $\mathrm{Y}$, consisting of the matrices \eqref{6.10}, such that $Y_{2,1}^{(i',j')}=0$, except when $(i',j')=(n+1,j)$,
and $Y_{3,1}^{(i',j')}=0$, except when $(i',j')=(2n-j,j)$.\\
We have
\begin{equation}\label{6.11}
\mathcal{D}_{n,m,r}=\mathrm{Y}\mathrm{N}\mathrm{X}.
\end{equation}
Let $3\leq j\leq n$ and $1\leq i\leq j-2$. Denote by
$\mathrm{Y}_{3,1}(i,j)$ the subgroup of $\mathrm{Y}$ generated by
the subgroups
$\mathrm{Y}_{3,1}^{i',j'},\mathrm{Y}_{2,1}^{i',j'}\subset
\mathrm{Y}$, with $j'\geq j$, such that if $j'=j$, then $1\leq
i'\leq i$. Similarly, denote by $\mathrm{Y}_{2,1}(i,j)$ the subgroup
of $\mathrm{Y}$ generated by the subgroups
$\mathrm{Y}_{3,1}^{i',j'}, \mathrm{Y}_{2,1}^{i'',j''}\subset
\mathrm{Y}$, with $j', j''\geq j$, such that if $j'=j$, then $1\leq
i'<i$, and
if $j''=j$, then $1\leq i''\leq i$.\\
Let $2\leq i< n$ and $1\leq j\leq i$. Denote by
$\mathrm{X}_{1,3}(i,j)$ the unipotent group generated by
$\mathrm{X}$ and the subgroups $\mathrm{X}_{1,3}^{i',j'},
\mathrm{X}_{1,2}^{i'',j''}\subset \mathcal{X}$, with $2\leq
i',i''\leq i$, such that if $i'=i$, then $j\leq j'$, and if $i''=i$,
then $j<j''$. Similarly, denote by $\mathrm{X}_{1,2}(i,j)$ the
unipotent group generated by $\mathrm{X}$ and the subgroups
$\mathrm{X}_{1,3}^{i',j'}, \mathrm{X}_{1,2}^{i',j'}\subset
\mathcal{X}$, with $2\leq i'\leq i$, such that if $i'=i$, then
$j\leq j'$.\\
Let $3\leq j\leq n$ and $1\leq i\leq j-2$. Define
\begin{equation}\label{6.12}
\begin{array}{rcl}
\mathrm{D}_{3,1}^{i,j}&=&\mathrm{Y}_{2,1}(i,j)\mathrm{N}\mathrm{X}_{1,3}(j-1,i);\\
\mathrm{D}_{2,1}^{i,j}&=&\mathrm{Y}_{3,1}(i-1,j)\mathrm{N}\mathrm{X}_{1,2}(j-1,i)\qquad
(i>1);\\
\mathrm{D}_{2,1}^{1,j}&=&\mathrm{Y}_{3,1}(j-1,j+1)\mathrm{N}\mathrm{X}_{1,2}(j-1,1).
\end{array}
\end{equation}
Here, in case $j=n$, $\mathrm{Y}_{3,1}(j-1,j+1)$ is defined as
the subgroup of elements of the form
\eqref{6.10} in $\mathrm{Y}_{2,1}(1,n)$, such that the $(m-2r)\times r$
block of $Y_{2,1}$ in position $(1,n)$ is zero. Note that
\begin{equation}\label{6.13}
\mathcal{D}_{n,m,r}=\mathrm{Y}_{3,1}(1,3)\mathrm{N}\mathrm{X}_{1,2}(1,1)=\mathrm{D}_{2,1}^{1,2}.
\end{equation}
The proof of the next lemma is straightforward.
\begin{lem}\label{lem 6.1}
	Let $3\leq j\leq n$ and $1\leq i\leq j-2$. Then
	$\mathrm{D}_{3,1}^{i,j}$ and $\mathrm{D}_{2,1}^{i,j}$ are subgroups.
	These are $F$-unipotent subgroups of $H_{2nm}$. Consider the
	restriction of the character $\psi_{\mathcal{D}_{n,m,r}}$ to
	$(\mathrm{Y}_{2,1}(i,j)\mathrm{N})({\BA})$
	(resp. $(\mathrm{Y}_{3,1}(i-1,j)\mathrm{N})({\BA})$,
	$(\mathrm{Y}_{3,1}(j-1,j+1)\mathrm{N})({\BA)})$). Then it extends
	to a character of $\mathrm{D}_{3,1}^{i,j}(\BA)$ (resp.
	$\mathrm{D}_{2,1}^{i,j}(\BA)$) by the trivial character of
	$(\mathrm{X}_{1,3}(j-1,i))_{\BA}$ (resp.
	$(\mathrm{X}_{1,2}(j-1,i))_{\BA}$). Denote this character by
	$\psi_{\mathrm{D}_{3,1}^{i,j}}$ (resp.
	$\psi_{\mathrm{D}_{2,1}^{i,j}}$). It is trivial on
	$\mathrm{D}_{3,1}^{i,j}(F)$ (resp. $\mathrm{D}_{2,1}^{i,j}(F)$).
\end{lem}

\section{A second reduction}

Let $3\leq j\leq n$ and $1\leq i\leq j-2$. Let $A$ be an automorohic representation of $H(\BA)$. Define, for $\xi \in A$, $h\in H(\BA)$,
\begin{equation}\label{7.1}
\epsilon_{2,1}^{i,j}(\xi)(h)=\int_{\mathrm{D}_{2,1}^{i,j}(F)\backslash
	\mathrm{D}_{2,1}^{i,j}(\BA)}\xi(vh)\psi_{\mathrm{D}_{2,1}^{i,j}}^{-1}(v)dv,
\end{equation}
$$
\epsilon_{3,1}^{i,j}(\xi)(h)=\int_{\mathrm{D}_{3,1}^{i,j}(F)\backslash
	\mathrm{D}_{3,1}^{i,j}(\BA)}\xi(vh)\psi_{\mathrm{D}_{3,1}^{i,j}}^{-1}(v)dv.
$$
Note that
\begin{equation}\label{7.1.1}
\epsilon_{2,1}^{1,2}(\xi)=\epsilon_{n,m,r}(\xi).
\end{equation}
\begin{prop}\label{prop 7.1}
	Let $3\leq j\leq n$ and $1\leq i\leq j-2$. We have the following identities.\\
	{\bf 1.}
	$$
	\epsilon_{3,1}^{i,j}(\xi)(h)=\int_{\mathrm{Y}_{2,1}^{i,j}(\BA)}
	\epsilon_{2,1}^{i,j}(\xi)(yh)dy.
	$$
	Moreover, $\epsilon_{3,1}^{i,j}$ is zero on $A$ if and
	only if $\epsilon_{2,1}^{i,j}$ is zero on $A$.\\
	{\bf 2.} Assume that $i>1$. Then
	$$
	\epsilon_{2,1}^{i,j}(\xi)(h)=\int_{\mathrm{Y}_{3,1}^{i-1,j}(\BA)}
	\epsilon_{3,1}^{i-1,j}(\xi)(yh)dy.
	$$
	Moreover, $\epsilon_{2,1}^{i,j}$ is zero on $A$ if and
	only if $\epsilon_{3,1}^{i-1,j}$ is zero on $A$. \\
	{\bf 3.}
	$$
	\epsilon_{2,1}^{1,j}(\xi)(h)=\int_{\mathrm{Y}_{3,1}^{j-1,j+1}(\BA)}
	\epsilon_{3,1}^{j-1,j+1}(\xi)(yh)dy.
	$$
	Moreover, $\epsilon_{2,1}^{1,j}$ is zero on $A$ if and
	only if $\epsilon_{3,1}^{j-1,j+1}$ is zero on $A$.\\
\end{prop}
\begin{proof}
	Assume that $i>1$. Denote
	$$
	\begin{array}{rcl}
	\mathrm{C}_{2,1}^{i,j}&=&\mathrm{Y}_{3,1}(i-1,j)\mathrm{N}\mathrm{X}_{1,3}(j-1,i),\\
	\mathrm{B}_{2,1}^{i,j}&=&\mathrm{D}_{3,1}^{i,j}.
	\end{array}
	$$
	It is straightforward to check that $\mathrm{C}_{2,1}^{i,j}$ is an
	$F$-unipotent subgroup of $H_{2nm}$. We have
	$$
	\begin{array}{rcl}
	\mathrm{B}_{2,1}^{i,j}&=&\mathrm{C}_{2,1}^{i,j}\mathrm{Y}_{2,1}^{i,j},\\
	\mathrm{D}_{2,1}^{i,j}&=&\mathrm{C}_{2,1}^{i,j}\mathrm{X}_{1,2}^{j-1,i}.
	\end{array}
	$$
	Denote by $\psi_{\mathrm{C}_{2,1}^{i,j}}$ the
	restriction of the character $\psi_{\mathrm{B}_{2,1}^{i,j}}$ (which
	is $\psi_{\mathrm{D}_{3,1}^{i,j}}$) to $\mathrm{C}_{2,1}^{i,j}(\BA)$. The
	unipotent groups $\mathrm{C}_{2,1}^{i,j}$, $\mathrm{Y}_{2,1}^{i,j}$,
	$\mathrm{X}_{1,2}^{j-1,i}$, $\mathrm{B}_{2,1}^{i,j}$,
	$\mathrm{D}_{2,1}^{i,j}$ satisfy the set-up of Lemma 7.1 of
	\cite{GRS11}. Let us verify the condition on the commutator
	$[\mathrm{X}_{1,2}^{j-1,i},\mathrm{Y}_{2,1}^{i,j}]$. For this,
	consider $x\in \mathrm{X}_{1,2}^{j-1,i}$, $y\in
	\mathrm{Y}_{2,1}^{i,j}$ and write them as in \eqref{6.9},
	\eqref{6.10}. Thus, $X_{1,2}$ has an $r\times (m-2r)$ block $X$ in the
	$(j-1,i)$ position and zero blocks elsewhere (also $X_{1,3}=0$), and
	$Y_{2,1}$ has an $(m-2r)\times r$ block $Y$ in the $(i,j)$ position and zero
	blocks elsewhere (also $Y_{3,1}=0$). Then $[x,y]\in \mathrm{N}$, and
	writing this commutator using the notation of \eqref{6.8},
	\eqref{6.4}, we have $V=I$ and $U$ has the block $XY$ in the
	position $(j-1,j)$ and zero blocks elsewhere above the diagonal.
	Thus, $[x,y]\in \mathrm{C}_{2,1}^{i,j}$, and over $\BA$,
	$\psi_{\mathrm{C}_{2,1}^{i,j}}([x,y])=\psi(tr(XY))$. Now we can exchange
	roots, replacing $\mathrm{Y}_{2,1}^{i,j}$ by
	$\mathrm{X}_{1,2}^{j-1,i}$ in $\epsilon_{3,1}^{i,j}$. That is, we
	apply Lemma 7.1, Corollary 7.1 and the proof of Corollary 7.2 in
	\cite{GRS11}, and obtain the first part of the proposition in case
	$i>1$. When $i=1$, we do exactly the same with
	$$
	\mathrm{C}_{2,1}^{1,j}=\mathrm{Y}_{3,1}(j-1,j+1)\mathrm{N}\mathrm{X}_{1,3}(j-1,1),
	$$
	and the rest is as before.\\
	Assume that $i>1$. Denote
	$$
	\begin{array}{rcl}
	\mathrm{C}_{3,1}^{i-1,j}&=&\mathrm{Y}_{2,1}(i-1,j)\mathrm{N}\mathrm{X}_{1,2}(j-1,i),\\
	\mathrm{B}_{3,1}^{i-1,j}&=&\mathrm{D}_{2,1}^{i,j}.
	\end{array}
	$$
	Again, it is straightforward to check that
	$\mathrm{C}_{3,1}^{i-1,j}$ is a unipotent subgroup of
	$\mathrm{B}_{3,1}^{i-1,j}$. We have
	$$
	\begin{array}{rcl}
	\mathrm{B}_{3,1}^{i-1,j}&=&\mathrm{C}_{3,1}^{i-1,j}\mathrm{Y}_{3,1}^{i-1,j},\\
	\mathrm{D}_{3,1}^{i-1,j}&=&\mathrm{C}_{3,1}^{i-1,j}\mathrm{X}_{1,3}^{j-1,i-1}.
	\end{array}
	$$
    Denote by
	$\psi_{\mathrm{C}_{3,1}^{i-1,j}}$ the restriction of the character
	$\psi_{\mathrm{B}_{3,1}^{i-1,j}}$ (which is
	$\psi_{\mathrm{D}_{2,1}^{i,j}}$) to $\mathrm{C}_{3,1}^{i-1,j}(\BA)$. The
	unipotent groups $\mathrm{C}_{3,1}^{i-1,j}$,
	$\mathrm{Y}_{3,1}^{i-1,j}$, $\mathrm{X}_{1,3}^{j-1,i-1}$,
	$\mathrm{B}_{3,1}^{i-1,j}$, $\mathrm{D}_{3,1}^{i-1,j}$ satisfy the
	set-up of Lemma 7.1 of \cite{GRS11}. Again, we conclude, as before,
	that we can replace $\mathrm{Y}_{3,1}^{i-1,j}$ in
	$\epsilon_{2,1}^{i,j}$ by $\mathrm{X}_{1,3}^{j-1,i-1}$, and thus
	obtain the second part of the proposition. When $i=1$, we repeat the
	argument with
	$$
	\begin{array}{rcl}
	\mathrm{C}_{3,1}^{j-1,j+1}&=&\mathrm{Y}_{2,1}(j-1,j+1)\mathrm{N}\mathrm{X}_{1,2}(j-1,1),\\
	\mathrm{B}_{3,1}^{j-1,j+1}&=&\mathrm{D}_{2,1}^{1,j}.
	\end{array}
	$$
	We have
	$$
	\begin{array}{rcl}
	\mathrm{B}_{3,1}^{j-1,j+1}&=&\mathrm{C}_{3,1}^{j-1,j+1}\mathrm{Y}_{3,1}^{j-1,j+1},\\
	\mathrm{D}_{3,1}^{j-1,j+1}&=&\mathrm{C}_{3,1}^{j-1,j+1}\mathrm{X}_{1,3}^{j,j-1}.
	\end{array}
	$$
	Note that for $j<n$,
	$$
	\mathrm{X}_{1,2}(j-1,1)=\mathrm{X}_{1,2}(j,j).
	$$
	Now, we can apply Lemma 7.1, Corollary 7.1 and the proof of
	Corollary 7.2 in \cite{GRS11}, and perform the root exchange in
	$\epsilon_{2,1}^{1,j}$, and replace $\mathrm{Y}_{3,1}^{j-1,j+1}$ with
	$\mathrm{X}_{1,3}^{j,j-1}$, thus proving the third part of the
	proposition. 
\end{proof}

Let us apply Proposition \ref{prop 7.1} repeatedly as follows. Start
with $\epsilon_{n,m,r}=\epsilon_{2,1}^{1,2}$. This is (identically)
zero if and only if $\epsilon_{3,1}^{1,3}$ is zero. Then this is
equivalent to $\epsilon_{2,1}^{1,3}$ being zero, and this is
equivalent to $\epsilon_{3,1}^{2,4}$ being zero, which, in turn, is
equivalent to $\epsilon_{2,1}^{2,4}$, $\epsilon_{3,1}^{1,4}$,
$\epsilon_{2,1}^{1,4}$,
$\epsilon_{3,1}^{3,5}$, $\epsilon_{2,1}^{3,5}$,...,$\epsilon_{2,1}^{1,n}$ being zero. Similarly, repeating the identities of Proposition \ref{prop 7.1} in the above order, we can express $\epsilon_{n,m,r}(\xi)(h)$ as an integral of $\epsilon_{2,1}^{1,n}(\xi)$ along $\mathrm{Y}^n_\BA$, where that $\mathrm{Y}^n$ is the subgroup of $\mathrm{Y}$ generated by the subgroups $\mathrm{Y}_{2,1}^{i,j}, \mathrm{Y}_{3,1}^{i,j}\subset \mathrm{Y}$, with $j\leq n$. 
Therefore, we conclude

\begin{prop}\label{prop 7.2}
For all $\xi\in A$, we have
$$
\epsilon_{n,m,r}(\xi)(h)=\int_{\mathrm{Y}^n(\BA)}\epsilon_{2,1}^{1,n}(\xi)(yh)dy
=\int_{\mathrm{Y}^n(\BA)}\int_{\mathrm{D}_{2,1}^{1,n}(F)\backslash \mathrm{D}_{2,1}^{1,n}(\BA)}\xi(vyh)\psi^{-1}_{\mathrm{D}_{2,1}^{1,n}}(v)dvdy.
$$
Moreover, $\epsilon_{n,m,r}$ is trivial on 	$A$ if and only if 
$\epsilon_{2,1}^{1,n}$ is trivial on $A$.
\end{prop}

There are two more steps similar to those in Proposition \ref{prop
7.1} that we can make in general. In the first step, we exchange, in
$\epsilon_{2,1}^{1,n}$,
$\mathrm{Y}_{2,1}^{n,n+1}\mathrm{Y}_{2,1}^{n+1,n+1}\mathrm{Y}_{2,1}^{n+2,n+1}$ with
$\mathrm{X}_{1,2}^{n,n}\mathrm{X}_{1,2}^{n,n+1;0}\mathrm{X}_{1,2}^{n,n+2}$, where in case $m$ is even, $\mathrm{X}_{1,2}^{n,n+1;0}$ is trivial. Note that in this case $\mathrm{Y}_{2,1}^{n+1,n+1}$ is trivial. When $m=2m'-1$ is odd, $\mathrm{X}_{1,2}^{n,n+1;0}$ is (a set of representatives, modulo $\mathrm{X}_{1,3}^{n,n}$,) such that the block of $X_{1,2}$ in position $(n,n+1)$, which is a matrix of size $r\times 2(m'-r)$ has the form
$$
(0_{r\times (m'-1-r)},(-\frac{1}{2}x,x)w_2^{r(n-1)},0_{r\times (m'-1-r)}),
$$
where $x$ is a column vector. See \eqref{6.6}, \eqref{6.5.2.1}. Note that in this case $\mathrm{Y}_{2,1}^{n+1,n+1}$ is as in \eqref{6.6},
modulo $\mathrm{Y}_{3,1}^{n-1,n+1}$. Recall that, in this case, the set of representatives $\mathrm{Y}_{2,1}^{n+1,n+1}$ is a group, modulo $\mathrm{Y}_{3,1}^{n-1,n+1}$, and, similarly, $\mathrm{X}_{1,2}^{n,n+1;0}$ and $\mathrm{X}_{1,2}^{n,n+1}$, are groups, modulo $\mathrm{X}_{1,3}^{n,n}$. Now, checking the conditions of Lemma
7.1 in \cite{GRS11} is straightforward. We conclude, as before, that
$\epsilon_{2,1}^{1,n}$ is trivial if and only if the following
Fourier coefficient is identically zero
\begin{equation}\label{7.2}
\widetilde{\epsilon}_{2,1}^{1,n}(\xi)(g)=\int_{\widetilde{\mathrm{D}}_{2,1}^{1,n}(F)\backslash
	\widetilde{\mathrm{D}}_{2,1}^{1,n}(\BA)}\xi(vg)\psi_{\widetilde{\mathrm{D}}_{2,1}^{1,n}}^{-1}(v)dv,
\end{equation}
where
\begin{equation}\label{7.3}
\widetilde{\mathrm{D}}_{2,1}^{1,n}=\widetilde{\mathrm{Y}}_{3,1}(n-1,n+1)\mathrm{N}\mathrm{X}_{1,2}(n,n),
\end{equation}
and $\widetilde{\mathrm{Y}}_{3,1}(n-1,n+1)$ is the subgroup of
elements of the form \eqref{6.10} in $\mathrm{Y}_{3,1}(n-1,n+1)$,
such that the $2(m-r)\times r$ block of $Y_{2,1}$ corresponding to 
$\mathrm{Y}_{2,1}^{n,n+1}$, $\mathrm{Y}_{2,1}^{n+1,n+1}$, $\mathrm{Y}_{2,1}^{n+2,n+1}$ is zero. The subgroup $\mathrm{X}_{1,2}(n,n)$ is
generated by $\mathrm{X}_{1,2}(n-1,1)$ and
$\mathrm{X}_{1,2}^{n,n}$, $\mathrm{X}_{1,2}^{n,n+1;0}$, $\mathrm{X}_{1,2}^{n,n+2}$. Finally,
$\psi_{\widetilde{\mathrm{D}}_{2,1}^{1,n}}$ is the character of
$\widetilde{\mathrm{D}}_{2,1}^{1,n}(\BA)$ obtained by extending
$\psi_N$ by the trivial character of $\mathrm{X}_{1,2}(n,n)(\BA)$
and of $\widetilde{\mathrm{Y}}_{3,1}(n-1,n+1)({\BA})$.\\
Similarly, we conclude that for all $\xi\in A$, we have
\begin{equation}\label{7.2'}
\epsilon_{n,m,r}(\xi)(h)=\int_{\widetilde{\mathrm{Y}}^n(\BA)}\tilde{\epsilon}_{2,1}^{1,n}(\xi)(yh)dy
=\int_{\widetilde{\mathrm{Y}}^n(\BA)}\int_{\widetilde{\mathrm{D}}_{2,1}^{1,n}(F)\backslash \widetilde{\mathrm{D}}_{2,1}^{1,n}(\BA)}\xi(vyh)\psi^{-1}_{\widetilde{\mathrm{D}}_{2,1}^{1,n}}(v)dvdy,
\end{equation}
where $\widetilde{\mathrm{Y}}^n$ denotes the subgroup of $\mathrm{Y}$, generated by $\mathrm{Y}^n$ and $\mathrm{Y}_{2,1}^{n,n+1}$, $\mathrm{Y}_{2,1}^{n+1,n+1}$, $\mathrm{Y}_{2,1}^{n+2,n+1}$.\\
The elements $y$ of $\widetilde{\mathrm{Y}}_{3,1}(n-1,n+1)$ have the
following form. Write $y$ as in \eqref{6.10}. Then $Y_{3,1}$ has the
form
\begin{equation}\label{7.4}
Y_{3,1}=\begin{pmatrix}0_{(n-1)r\times nr}&e\\0_{nr\times
	nr}&0_{nr\times (n-1)r}\end{pmatrix},
\end{equation}
where, if $Y_{2,1}$ is zero, then the $(n-1)r\times (n-1)r$ matrix
$e$ is such that $ w_{(n-1)r}e$ is symmetric (resp. anti-symmetric) when $H_{2nm}$ is symplectic (resp. orthogonal). The block $Y_{2,1}$ of $y$ above has the form
\begin{equation}\label{7.5}
Y_{2,1}=\begin{pmatrix}0&h_{1,1}&h_{1,2}&h_{1,3}&\cdots&h_{1,n-2}&h_{1,n-1}\\
0&0&h_{2,2}&h_{2,3}&&h_{2,n-2}&h_{2,n-1}\\0&0&0&h_{3,3}&&h_{3,n-2}&h_{3,n-1}\\
\vdots&&&&\cdots&&\vdots\\ 0&0&0&0&\dots& 0&h_{n-1,n-1}\\0&0&0&0&& 0&0\\0&0&0&0&& 0&0\end{pmatrix}.
\end{equation}
Here, the first zero block column has $nr$ columns. All other block columns are each of size $r$. The first block row is of size $(n-1)(m-2r)$. The second block row is of size $2(m-r)$. All the remaining block rows are of size $m-2r$.

In the second step, we exchange in
$\widetilde{\epsilon}_{2,1}^{1,n}$, in \eqref{7.3},
$\mathrm{Y}_{3,1}^{n-1,n+1}$ "into" $\mathrm{X}_{1,3}^{n,n-1}$. Note
that $\mathrm{Y}_{3,1}^{n-1,n+1}$ is isomorphic to the space of  $r\times r$
symmetric matrices, or the antisymmetric matrices, according to whether $H_{2nm}$ is symplectic or orthogonal, while $\mathrm{X}_{1,3}^{n,n-1}$ is
isomorphic to $M_{r\times r}$. The verification of the conditions of
Lemma 7.1 in \cite{GRS11} is straightforward. More precisely, let
$\mathrm{Y}_{2,1}(n-1,n+1)$ be the subgroup of elements in
$\widetilde{\mathrm{Y}}_{3,1}(n-1,n+1)$, written in the form
\eqref{6.10}, such that the block $Y_{3,1}^{n-1,n+1}$ is zero. Let
$\mathfrak{X}_{1,3}^{n,n-1}$ be the subgroup of elements in
$\mathrm{X}_{1,3}^{n,n-1}$, written as in \eqref{6.9}, such that the
block $X_{1,3}^{n,n-1}$ has the property that $ w_r X_{1,3}^{n,n-1}$
is symmetric (resp. anti-symmetric). Let $\mathfrak{X}_{1,3}(n,n-1)$ be the subgroup
generated by $\mathrm{X}_{1,2}(n,n)$ and
$\mathfrak{X}_{1,3}^{n,n-1}$. Define
\begin{equation}\label{7.6}
\mathfrak{D}_{3,1}^{n-1,n+1}=\mathrm{Y}_{2,1}(n-1,n+1)\mathrm{N}\mathfrak{X}_{1,3}(n,n-1).
\end{equation}
This is a unipotent subgroup of $H_{2nm}$. Denote, as usual,
by $\psi_{\mathfrak{D}_{3,1}^{n-1,n+1}}$ the character of
$\mathfrak{D}_{3,1}^{n-1,n+1}(\BA)$ obtained by extending
$\psi_N$ trivially. Then $\widetilde{\epsilon}_{2,1}^{1,n}$ is
trivial if and only if the following Fourier coefficient is
identically zero,
\begin{equation}\label{7.7}
\varepsilon_{3,1}^{n-1,n+1}(\xi)=\int_{\mathfrak{D}_{3,1}^{n-1,n+1}(F)\backslash
	\mathfrak{D}_{3,1}^{n-1,n+1}(\BA)}\xi(v)\psi_{\mathfrak{D}_{3,1}^{n-1,n+1}}^{-1}(v)dv.
\end{equation}
Let $\hat{\mathrm{Y}}^n$ be the subgroup generated by $\widetilde{\mathrm{Y}}^n$ and $\mathrm{Y}_{3,1}^{n-1,n+1}$. Using Proposition \ref{prop 7.2}, \eqref{7.2}, \eqref{7.2'} and \eqref{7.7}, we obtain our second reduction, valid for any automorphic representation of
$H(\BA)$.

\begin{prop}\label{prop 7.3}
	For all $\xi\in A(\Delta(\tau,m)\gamma_\psi^{(\epsilon)},\eta,k)$, we have
	\begin{multline*}
			\epsilon_{n,m,r}(\xi)(h)=\int_{\hat{\mathrm{Y}}^n_\BA}\varepsilon_{3,1}^{n-1,n+1}(\xi)(yh)dy
	=\\
	=\int_{\hat{\mathrm{Y}}^n_\BA}\int_{\mathfrak{D}_{3,1}^{n-1,n+1}(F)\backslash \mathfrak{D}_{3,1}^{n-1,n+1}(\BA)}\xi(vyh)\psi^{-1}_{\mathfrak{D}_{3,1}^{n-1,n+1}}(v)dvdy.
	\end{multline*}
	Moreover, $\epsilon_{n,m,r}$ is trivial on 	$A(\Delta(\tau,m)\gamma_\psi^{(\epsilon)},\eta,k)$ if and only if 
	$\varepsilon_{3,1}^{n-1,n+1}$ is trivial on $A(\Delta(\tau,m)\gamma_\psi^{(\epsilon)},\eta,k)$. This proposition is valid for any automorphic representation $A$ of $H(\BA)$.
\end{prop}

\section{Fourier expansions I}

We will show that the Fourier coefficient
$\varepsilon_{3,1}^{n-1,n+1}$ remains unchanged when we replace
$\mathfrak{X}_{1,3}^{n,n-1}$ by $\mathrm{X}_{1,3}^{n,n-1}$. Here
will be the first time that we use special properties of
$A(\Delta(\tau,m)\gamma_\psi^{(\epsilon)},\eta,k)$. Denote by $\mathrm{X}_{1,3}(n,n-1)$ the group
generated by $\mathrm{X}_{1,2}(n,n)$ and $\mathrm{X}_{1,3}^{n,n-1}$,
and let
\begin{equation}\label{8.1}
\mathrm{D}_{3,1}^{n-1,n+1}=\mathrm{Y}_{2,1}(n-1,n+1)\mathrm{N}\mathrm{X}_{1,3}(n,n-1).
\end{equation}
This is an $F$- unipotent subgroup of $H_{2nm}$. Denote by
$\psi_{\mathrm{D}_{3,1}^{n-1,n+1}}$ the character of
$\mathrm{D}_{3,1}^{n-1,n+1}(\BA)$ obtained by extending $\psi_N$
trivially.
\begin{prop}\label{prop 8.1}
	Let $\xi\in A(\Delta(\tau,m)\gamma_\psi^{(\epsilon)},\eta,k)$. Then
	\begin{equation}\label{8.1'}
	\varepsilon_{3,1}^{n-1,n+1}(\xi)(g)=\int_{\mathrm{D}_{3,1}^{n-1,n+1}(F)\backslash
		\mathrm{D}_{3,1}^{n-1,n+1}(\BA)}\xi(vg)\psi_{\mathrm{D}_{3,1}^{n-1,n+1}}^{-1}(v)dv.
	\end{equation}
		\end{prop}
\begin{proof}
Recall that we denote $\delta_H=1$ when $H_{2nm}$ is symplectic, and $\delta_H=-1$ when $H_{2nm}$ is orthogonal. Let $S_r$ denote the subspace of all $z\in M_{r\times r}$, such
that ${}^t(w_rz)=-\delta_H(w_rz)$. Denote by $x(z)$ the element in
	$\mathrm{X}_{1,3}^{n,n-1}$, such that its corresponding
	$X_{1,3}^{n,n-1}$ block is $z$. Consider the following function on
	$S_r(\BA)$,
	\begin{equation}\label{8.2}
	\Phi_\xi(z)=\varepsilon_{3,1}^{n-1,n+1}(\xi)(x(z))=\int_{\mathfrak{D}_{3,1}^{n-1,n+1}(F)\backslash
		\mathfrak{D}_{3,1}^{n-1,n+1}(\BA)}\xi(vx(z))\psi_{\mathfrak{D}_{3,1}^{n-1,n+1}}^{-1}(v)dv.
	\end{equation}
A direct verification shows that $\Phi_\xi$ is $S_r(F)$ invariant.
	Thus it defines a smooth function on the compact abelian group
	$S_r(F)\backslash S_r(\BA)$. Consider its Fourier expansion. A
	typical character of $S_r(F)\backslash S_r(\BA)$ has the form
	$\psi(tr(zL))$, where $L\in S_r(F)$. The corresponding Fourier coefficient of $\Phi_\xi$ is
	\begin{equation}\label{8.3}
	\int_{S_r(F)\backslash S_r(\BA)}\Phi_\xi(z)\psi^{-1}(tr(zL))dz.
	\end{equation}
	We will show that the Fourier coefficient \eqref{8.3} is zero on
	$A(\Delta(\tau,m)\gamma_\psi^{(\epsilon)},\eta,k)$, for all nonzero $L$. In fact, we will show that an
	inner integration inside \eqref{8.3}, after substituting \eqref{8.2},
	is zero, for all nonzero $L$. To describe this inner integration,
	consider the unipotent radical $U_{r^n}$ of
	$Q_{r^n}=Q_{r^n}^H$. Denote $E=U_{r^n}\cap
	\mathrm{X}_{1,3}(n,n-1)$. The elements of $E$ have the form
	\begin{equation}\label{8.4}
	v=\begin{pmatrix}U&B&C\\&I_{2n(m-r)}&B'\\&&U^*\end{pmatrix}\in H_{2nm},
	\end{equation}
	where $U$ has the form
	\begin{equation}\label{8.5}
	U=\begin{pmatrix}I_r&x_1&\ast&\cdots&\ast\\&I_r&x_2&&\ast\\&&\cdots\\
	&&&&x_{n-1}\\&&&&I_r\end{pmatrix}.
	\end{equation}
	Write $B$ as a matrix of blocks, which has $n$ row blocks, each one
	of size $r$, and it has $4n-1$ block columns, the first and last
	$n-1$ of which have size $r$; the next $n-1$ block columns from both ends have size $m-2r$ each, and the three middle block
	columns have sizes $[\frac{m}{2}], 2([\frac{m+1}{2}]-r), [\frac{m}{2}]$, respectively. Then the last
	block row of $B$ has the following form (of size $r$)
	\begin{equation}\label{8.6}
	\begin{pmatrix}x_n,&\alpha,&0_{r\times (n-1)(m-2r)},&\beta,&0_{r\times (n-2)r},&z\end{pmatrix},
	\end{equation}
	where $x_n,z\in M_{r\times r}$, $\alpha\in M_{r\times (n-2)r}$, $\beta\in M_{r\times (2(m-r)+(n-1)(m-2r))}$. The
	following is an inner integration of \eqref{8.3} (when we let $\xi$ vary in the space of\\
	 $A(\Delta(\tau,m)\gamma_\psi^{(\epsilon)},\eta,k)$),
	\begin{equation}\label{8.7}
	\phi_\xi=\int_{E(F)\backslash E(\BA)}\xi(v)\psi_E^{-1}(v)dv,
	\end{equation}
	where $\psi_E$ is the character of $E(\BA)$, whose value on $v$ in
	\eqref{8.4}, with coordinates in $\BA$, is
	$\psi(tr(x_1+x_2+\cdots+x_n+Lz))$, using the notation in
	\eqref{8.5}, \eqref{8.6}. We will show that the Fourier coefficient
	\eqref{8.7} is zero on $A(\Delta(\tau,m)\gamma_\psi^{(\epsilon)},\eta,k)$, for all nonzero $L$. For this,
	we will further consider the Fourier expansion of \eqref{8.7} along
	the coordinates in \eqref{8.6} in position of the zero matrices.
	That is, let $a\in M_{r\times (n-1)(m-2r)}(\BA)$ and $b\in M_{r\times
		(n-2)r}(\BA)$. Let $x(a,b)$ be any element in $U_{r^n}(\BA)$
	of the form \eqref{8.4}, with $U=I_{nr}$, $B$ with its first $n-1$
	block rows (of size $r$ each) being zero, and the last block row,
	such that in \eqref{8.6}, we replace the two zero blocks by $a$,
	$b$, and replace $x_n$, $\alpha$, $\beta$, $z$ by zero. Define
	\begin{equation}\label{8.8}
	\phi_\xi(a,b)=\int_{E(F)\backslash E(\BA)}\xi(vx(a,b))\psi_E^{-1}(v)dv.
	\end{equation}
	Note that this is independent of the choice of $x(a,b)$ (we did not
	specify $C$). Clearly, $\phi_\xi$ in \eqref{8.8} is a smooth
	function on the compact abelian group
	$$
	(M_{r\times (n-1)(m-2r)}(F)\backslash M_{r\times (n-1)(m-2r)}(\BA))\times (
	M_{r\times (n-2)r}(F)\backslash M_{r\times (n-2)r}(\BA).
	$$
	Consider its Fourier expansion. A typical Fourier coefficient of
	$\phi_\xi$ has the form
	\begin{equation}\label{8.9}
	\int\int\phi_\xi(a,b)\psi^{-1}(tr(aH)+tr(bJ))dbda,
	\end{equation}
	where the $db$-integration is over $M_{r\times (n-2)r}(F)\backslash M_{r\times (n-2)r}(\BA)$, and the $da$-integration is over $M_{r\times (n-1)(m-2r)}(F)\backslash M_{r\times (n-1)(m-2r)}(\BA)$. Also, $H\in M_{(n-1)(m-2r)\times r}(F)$, $J\in M_{(n-2)r\times r}(F)$.
	When we substitute \eqref{8.8}, we obtain a Fourier coefficient of
	$\xi$ along $U_{r^n}(F)\backslash U_{r^n}(\BA)$ of the form
	\begin{equation}\label{8.10}
	\phi_{\xi,L,A}=\int_{U_{r^n}(F)\backslash U_{r^n}(\BA)}\xi(v)\psi_{U_{r^n},L,A}^{-1}(v)dv.
	\end{equation}
	Here, $A\in M_{2(nm-r(n+1))\times r}(F)$ is the following
	matrix
	$$
	A=\begin{pmatrix}0_{(n-2)r\times r}\\H\\0_{((n+1)m-2nr)\times
		r}\\J\end{pmatrix}.
	$$
	The character $\psi_{U_{r^n},L,A}$ assigns to an element $v\in
	U_{r^n}(\BA)$, written in the form \eqref{8.4}, with $U$ as in
	\eqref{8.5} and $B$ with last row block of size $r$ of the form
	\begin{equation}\label{8.11}
	(x_n,f,z),\quad f\in M_{r\times 2(nm-r(n+1))}(\BA),
	\end{equation}
	the value
	\begin{equation}\label{8.12}
	\psi_{U_{r^n},L,A}(v)=\psi(tr(x_1+\cdots x_n)+tr(fA)+tr(zL)).
	\end{equation}
	Of course $\psi_{U_{r^n},L,A}$ is an extension of $\psi_E$ by the
	character $\psi(tr(aH)+tr(bJ))$ used in \eqref{8.9}. We want to show
	that $\phi_{\xi,L,A}=0$, identically on $A(\Delta(\tau,m)\gamma_\psi^{(\epsilon)},\eta,k)$, for all matrices $A$
	as above, and all $L\neq 0$. Note that the character
	$\psi_{U_{r^n},L,A}$ is conjugate to $\psi_{U_{r^n},L,0}$. Indeed,
	let
	$$
	\zeta_A=\diag (I_{nr},\begin{pmatrix}
	I_r\\A&I_{2nm-2(n+1)r}\\0&A'&I_r\end{pmatrix},I_{nr})\in
	H_{2nm}(F).
	$$
	Then
	$$
	\phi_{\xi,L,A}=\int_{U_{r^n}(F)\backslash U_{r^n}(\BA)}\xi(\zeta_A
	v)\psi_{U_{r^n},L,A}^{-1}(v)dv=\int_{U_{r^n}(F)\backslash
		U_{r^n}(\BA)}\xi(v\zeta_A)\psi_{U_{r^n},L,0}^{-1}(v)dv.
	$$
	We changed variable $\zeta_A v \zeta_A^{-1}\mapsto v$ and we used
	the fact that $A'A=0$. Thus, we may assume that $A=0$, i.e. $H=0$,
	$J=0$. Denote, for short,
	$$
	\phi_{\xi,L,0}=\phi_{\xi,L}.
	$$
	The character $\psi_{U_{r^n},L,0}$ corresponds, in the sense of
	\cite{MW87}, to the nilpotent element $Y$ in $Lie(H_{2nm})(F)$
	and to the one parameter subgroup $\varphi(s)$, where
	\begin{equation}\label{8.13}
	Y=\begin{pmatrix}\mathcal{U}\\
	\mathcal{V}&0_{2n(m-r)\times 2n(m-r)}\\0_{nr\times
		nr}&\mathcal{V}'&-\mathcal{U}\end{pmatrix},
	\end{equation}
	and
	$$
	\mathcal{U}=\begin{pmatrix}0\\I_r&0\\&I_r\\
	&&\cdots\\&&&I_r&0\end{pmatrix}\in M_{nr\times nr}(F);
	$$
	$$
	\mathcal{V}=\begin{pmatrix}0_{r\times
		(n-1)r}&I_r\\0_{(2nm-2(n+1)r)\times(n-1)r}&0_{(2nm-2(n+1)r)\times
		r}\\0_{r\times (n-1)r}&L\end{pmatrix};
	$$
	$$
	\varphi(s)=\diag(s^{2n}I_r,s^{2n-2}I_r,...,s^2I_r, I_r,
	I_{2n(m-r)},I_r,s^{-2}I_r,...,s^{-2n}I_r).
	$$
	A simple calculation shows that the nilpotent orbit of $Y$
	corresponds to the partition $((2n+1)^\ell,
	(n+1)^{2(r-\ell)},1^{2nm-2(n+1)r+\ell})$, where $\ell=rank(L)$.
	By Propositions \ref{prop 3.1}, \ref{prop 3.2}, we obtain that if $\ell>0$, then
	$\phi_{\xi,L}=0$, for all $\xi\in A(\Delta(\tau,m)\gamma_\psi^{(\epsilon)},\eta,k)$. This proves the proposition.
\end{proof}
Now, we can repeat the steps of Proposition \ref{prop 7.1} and
exchange, in the integral \eqref{8.1'}, $\mathrm{Y}_{2,1}^{n-1,n+1}$
with $\mathrm{X}_{1,2}^{n,n-1}$, $\mathrm{Y}_{3,1}^{n-2,n+1}$ with
$\mathrm{X}_{1,3}^{n,n-2}$, $\mathrm{Y}_{2,1}^{n-2,n+1}$ with
$\mathrm{X}_{1,2}^{n,n-2}$ and so on, until we fill block row $n$ in
$\mathrm{X}_{1,3}$ and in $\mathrm{X}_{1,2}$, using block column
$n+1$ in $\mathrm{Y}_{3,1}$ and in $\mathrm{Y}_{2,1}$. More
precisely, let $1\leq i\leq n-2$. Denote by
$\mathrm{Y}_{3,1}(i,n+1)$, the subgroup of elements in
$\mathrm{Y}_{2,1}(n-1,n+1)$, such that their blocks
$Y_{3,1}^{i',n+1}$ and $Y_{2,1}^{i',n+1}$ are zero, for $i<i'<n-1$.
Similarly, denote by $\mathrm{Y}_{2,1}(i,n+1)$, the subgroup of
elements in $\mathrm{Y}_{3,1}(i,n+1)$, such that their
$Y_{3,1}^{i,n+1}$ block is zero. Next, denote by
$\mathrm{X}_{1,3}(n,i)$ the subgroup generated by
$\mathrm{X}_{1,3}(n,n-1)$ and the subgroups
$\mathrm{X}_{1,3}^{n,j}$, $\mathrm{X}_{1,2}^{n,j'}$, for $i\leq
j\leq n-2$ and $i<j'\leq n-1$. Similarly, denote by
$\mathrm{X}_{1,2}(n,i)$ the subgroup generated by
$\mathrm{X}_{1,3}(n,i)$ and $\mathrm{X}_{1,2}^{n,i}$.\\
Define, for $1\leq i\leq n-1$,
\begin{equation}\label{8.14}
\begin{array}{rcl}
\mathrm{D}_{3,1}^{i,n+1}&=&\mathrm{Y}_{2,1}(i,n+1)\mathrm{N}\mathrm{X}_{1,3}(n,i);\\
\mathrm{D}_{2,1}^{i,n+1}&=&\mathrm{Y}_{3,1}(i-1,n+1)\mathrm{N}\mathrm{X}_{1,2}(n,i)\qquad
(i>1);\\
\mathrm{D}_{2,1}^{1,n+1}&=&\mathrm{Y}_{3,1}(n-2,n+2)\mathrm{N}\mathrm{X}_{1,2}(n,1).
\end{array}
\end{equation}
Here, $\mathrm{Y}_{3,1}(n-2,n+2)$ is defined as before, namely, this
is the subgroup of elements inside $\mathrm{Y}_{1,2}(1,n+1)$, such
their block $Y_{1,2}^{1,n+1}$ is zero. As before, these are
$F$-unipotent subgroups of $H_{2nm}$, and we also have the
characters $\psi_{\mathrm{D}_{3,1}^{i,n+1}}$ and
$\psi_{\mathrm{D}_{2,1}^{i,n+1}}$ of the corresponding adele groups,
obtained by the trivial extension of $\psi_N$. Now, applying the notation
\eqref{7.1}, let, for $1\leq i\leq n-1$, $\xi \in A(\Delta(\tau,m)\gamma_\psi^{(\epsilon)},\eta,k)$,
$g\in H_{2nm}(\BA)$,
\begin{equation}\label{8.15}
\epsilon_{2,1}^{i,n+1}(\xi)(g)=\int_{\mathrm{D}_{2,1}^{i,n+1}(F)\backslash
	\mathrm{D}_{2,1}^{i,n+1}(\BA)}\xi(vg)\psi_{\mathrm{D}_{2,1}^{i,n+1}}^{-1}(v)dv,
\end{equation}

$$
\epsilon_{3,1}^{i,n+1}(\xi)(g)=\int_{\mathrm{D}_{3,1}^{i,n+1}(F)\backslash
	\mathrm{D}_{3,1}^{i,n+1}(\BA)}\xi(vg)\psi_{\mathrm{D}_{3,1}^{i,n+1}}^{-1}(v)dv,
$$
Proposition \ref{prop 8.1} says that, for $\xi\in A(\Delta(\tau,m)\gamma_\pi^{(\epsilon)},\eta,k)$,
$$
\epsilon_{3,1}^{n-1,n+1}(\xi)=\varepsilon_{3,1}^{n-1,n+1}(\xi),
$$
and now the proof of Proposition \ref{prop 7.1} works, word for word,
and gives

\begin{prop}\label{prop 8.2}
	Let $1\leq i\leq n-1$.\\
	{\bf 1.} 
	$$
	\epsilon_{3,1}^{i,n+1}(\xi)(h)=\int_{\mathrm{Y}_{2,1}^{i,n+1}(\BA)}
	\epsilon_{2,1}^{i,n+1}(\xi)(yh)dy.
	$$ 	
	Moreover, $\epsilon_{3,1}^{i,n+1}$ is zero on $A(\Delta(\tau,m)\gamma_\psi^{(\epsilon)},\eta,k)$ if and
	only if $\epsilon_{2,1}^{i,n+1}$ is zero on $A(\Delta(\tau,m)\gamma_\psi^{(\epsilon)},\eta,k)$.\\
	{\bf 2.} Assume that $i>1$. Then 
	$$
	\epsilon_{2,1}^{i,n+1}(\xi)(h)=\int_{\mathrm{Y}_{3,1}^{i-1,n+1}(\BA)}
	\epsilon_{3,1}^{i-1,n+1}(\xi)(yh)dy.
	$$	
	Moreover, $\epsilon_{2,1}^{i,n+1}$ is zero on
	$A(\Delta(\tau,m)\gamma_\psi^{(\epsilon)},\eta,k)$ if and only if $\epsilon_{3,1}^{i-1,n+1}$ is zero on
	$A(\Delta(\tau,m)\gamma_\psi^{(\epsilon)},\eta,k)$.\\
	\end{prop}
Let $\mathrm{Y}^{n+1}$ be the subgroup of $\mathrm{Y}$ generated by $\hat{\mathrm{Y}}^n$ and the subgroups $\mathrm{Y}_{3,1}^{i,n+1}$, $1\leq i\leq n-2$, $\mathrm{Y}_{2,1}^{i',n+1}$, $1\leq i'\leq n-1$. From Proposition \ref{prop 7.3} and the last two propositions, we
conclude
\begin{prop}\label{prop 8.3}
	For all $\xi\in A(\Delta(\tau,m)\gamma_\pi^{(\epsilon)},\eta,k)$, we have
	\begin{multline*}
	\epsilon_{n,m,r}(\xi)(h)=\int_{\mathrm{Y}^{n+1}_\BA}\epsilon_{2,1}^{1,n+1}(\xi)(yh)dy
	=\\
	=\int_{\mathrm{Y}^{n+1}_\BA}\int_{\mathrm{D}_{2,1}^{1,n+1}(F)\backslash \mathrm{D}_{2,1}^{1,n+1}(\BA)}\xi(vyh)\psi^{-1}_{\mathrm{D}_{2,1}^{1,n+1}}(v)dvdy.
	\end{multline*}
	Moreover, $\epsilon_{n,m,r}$ is trivial on
	$A(\Delta(\tau,m)\gamma_\pi^{(\epsilon)},\eta,k)$ if and only if $\epsilon_{2,1}^{1,n+1}$ is trivial on 	$A(\Delta(\tau,m)\gamma_\pi^{(\epsilon)},\eta,k)$.
\end{prop}

Note the shape of the subgroups
$\mathrm{Y}_{3,1}(n-2,n+2)$, $\mathrm{X}_{1,2}(n,1)$ in the
definition of $\mathrm{D}_{2,1}^{1,n+1}$. The elements $y$ of
$\mathrm{Y}_{3,1}(n-2,n+2)$ have the following form. Write $y$ as in
\eqref{6.10}. Then $Y_{3,1}$ has the form
$$
Y_{3,1}=\begin{pmatrix}0_{(n-2)r\times (n+1)r}&e\\0_{(n+1)r\times
	(n+1)r}&0_{(n+1)r\times (n-2)r}\end{pmatrix},
$$
where, if $Y_{2,1}$ is zero, then the $(n-2)r\times (n-2)r$ matrix
$e$ is such that $ {}^t(w_{(n-2)r}e)=\delta_H(w_{(n-2)r}e)$. The block
$Y_{2,1}$ of $y$ above has the form \eqref{7.5} with $h_{1,1}=0$.
The elements $x$ of $\mathrm{X}_{1,2}(n,1)$ have the following form.
Write $x$ as in \eqref{6.9}. Then $X_{1,3}$ is of size
$(2n-1)r\times (2n-1)r$ and has the form
\begin{equation}\label{8.16}
X_{1,3}=\begin{pmatrix}J&K\\0_{(n-1)r\times (n-1)r}&J'\end{pmatrix},
\end{equation}
where if $X_{1,2}$ is zero, then
${}^t(w_{(2n-1)r}X_{1,3})=\delta_H(w_{(2n-1)r}X_{1,3})$. The block $X_{1,2}$ is of
size $(2n-1)r\times (2n(m-2r)+2r)$ and has the form
\begin{equation}\label{8.17}
X_{1,2}=\begin{pmatrix}Q&R\\0_{(n-1)r\times
	(n(m-2r)+m)}&S\end{pmatrix},
\end{equation}
where $S$ is upper triangular as a $(n-1)\times (n-1)$ matrix of
$r\times (m-2r)$ blocks. We will now prove that $\epsilon_{2,1}^{1,n+1}$
is invariant to $\mathrm{X}_{1,3}^{n+1,n-1}(\BA)$. Equivalently,
let $\mathrm{X}_{1,3}(n+1,n-1)$ be the subgroup generated by
$\mathrm{X}_{1,2}(n,1)$ and $\mathrm{X}_{1,3}^{n+1,n-1}$. Let
$$
\mathbf{D}_{2,1}^{1,n+1}=\mathrm{D}_{2,1}^{1,n+1}\mathrm{X}_{1,3}^{n+1,n-1}
=\mathrm{Y}_{3,1}(n-2,n+2)\mathrm{N}\mathrm{X}_{1,3}(n+1,n-1),
$$
and extend the character $\psi_{\mathrm{D}_{2,1}^{1,n+1}}$ to
$\mathbf{D}_{2,1}^{1,n+1}(\BA)$ by the trivial character of
$\mathrm{X}_{1,3}^{n+1,n-1}(\BA)$. Denote the extended character
by $\psi_{\mathbf{D}_{2,1}^{1,n+1}}$. Then we prove
\begin{prop}\label{prop 8.4}
	For all $\xi\in A(\Delta(\tau,m)\gamma_\pi^{(\epsilon)},\eta,k)$, $g\in H(\BA)$,
	\begin{equation}\label{8.16.1}
	\epsilon_{2,1}^{1,n+1}(\xi)(g)=\int_{\mathbf{D}_{2,1}^{1,n+1}(F)\backslash
		\mathbf{D}_{2,1}^{1,n+1}(\BA)}\xi(vg)\psi_{\mathbf{D}_{2,1}^{1,n+1}}^{-1}(v)dv.
	\end{equation}
	\end{prop}
\begin{proof}
	The proof is similar to that of Proposition \ref{prop 8.1}. Let $A_r$ be the space of all matrices $z\in M_{r\times r}$,
	such that ${}^t(w_rz)=\delta_H(w_rz)$. Denote by $x(z)$ the element in
	$\mathrm{X}_{1,3}^{n+1,n-1}$, such that its corresponding block
	$X_{1,3}^{n+1,n-1}$ is $z$. Consider the function on $A_r(\BA)$
	given by
	$$
	f_\xi(z)=\int_{\mathrm{D}_{2,1}^{1,n+1}(F)\backslash
		\mathrm{D}_{2,1}^{1,n+1}(\BA)}\xi(vx(z))\psi_{\mathrm{D}_{2,1}^{1,n+1}}^{-1}(v)dv.
	$$
	Then this function is $A_r(F)$-invariant. We will show that in the
	Fourier expansion of $f_\xi$ along $A_r(F)\backslash A_r(\BA)$,
	all Fourier coefficients corresponding to non-trivial characters are
	zero. Such Fourier coefficients have the form
	\begin{equation}\label{8.17.1}
	\int_{A_r(F)\backslash A_r(\BA)}f_\xi(z)\psi^{-1}(tr(zL))dz,
	\end{equation}
	where $L\in M_{r\times r}(F)$, such that ${}^t(w_rL)=\delta_H(w_rL)$.
	Substituting in \eqref{8.17.1} the definition of $f_\xi$, we get
	\begin{equation}\label{8.18}
	\int_{A_r(F)\backslash A_r(\BA)}\int_{\mathrm{D}_{2,1}^{1,n+1}(F)\backslash
		\mathrm{D}_{2,1}^{1,n+1}(\BA)}\xi(vx(z))\psi_{\mathrm{D}_{2,1}^{1,n+1}}^{-1}(v)\psi^{-1}(tr(zL))dvdz.
	\end{equation}
	Let us consider the following inner integration of \eqref{8.18}. Let
	$E$ be the subgroup generated by the unipotent radical $U_{r^n}$ and
	$\mathrm{X}_{1,3}^{n+1,n-1}$. Consider the following character
	$\psi_{E,L}$ of $E({\BA})$. Write $v\in E({\BA})$ in the form
	$v=ux(z)$, where $u\in U_{r^n}(\BA)$, and $x(z)\in
	\mathrm{X}_{1,3}^{n+1,n-1}(\BA)$, as above. Write $u$ as in
	\eqref{8.4}, \eqref{8.5}. As before, we divide the rows of
	$B$ into $n$ block rows, each one of size $r$. Denote the first
	$r\times r$ block in the last block row of $B$ by $x_n$. Then
	\begin{equation}\label{8.19}
	\psi_{E,L}(v)=\psi(tr(x_1+x_2+\cdots +x_n)+tr(zL)).
	\end{equation}
	Then the following is an inner integral of \eqref{8.18} (as we let $\xi$ vary in \\
	$A(\Delta(\tau,m)\gamma_\pi^{(\epsilon)},\eta,k)$)	                  
	\begin{equation}\label{8.20}
	f_{\xi,L}=\int_{E_F\backslash E_{\BA}}\xi(v)\psi^{-1}_{E,L}(v)dv.
	\end{equation}
	We will prove that $f_{\xi,L}=0$, for all $\xi\in A(\Delta(\tau,m)\gamma_\pi^{(\epsilon)},\eta,k)$ and
	all nonzero $L$. The character $\psi_{E,L}$ corresponds, in the
	sense of \cite{MW87}, to the nilpotent element $Y$ in
	$Lie(H_{2nm})(F)$ and to the one parameter subgroup
	$\varphi(s)$,
	\begin{equation}\label{8.21}
	Y=\begin{pmatrix}\mathcal{U}\\
	\mathcal{V}&\mathcal{L}\\0_{nr\times
		nr}&\mathcal{V}'&-\mathcal{U}\end{pmatrix},
	\end{equation}
	where $\mathcal{U}$ is as in \eqref{8.13};
	$$
	\mathcal{V}=\begin{pmatrix}0&I_r\\0&0\end{pmatrix}\in
	M_{(2nm-2nr)\times nr}(F);
	$$
	$$
	\mathcal{L}=\begin{pmatrix}0\\0&0\\L&0&0\end{pmatrix}\in
	Lie(H_{2nm-2nr})(F);
	$$
	$$
	\varphi(s)=\diag(s^{2n+1}I_r,s^{2n-1}I_r,...,sI_r,
	I_{2nm-2(n+1)r},I_r,s^{-1}I_r,...,s^{-2n-1}I_r).
	$$
	A simple calculation shows that the nilpotent orbit of $Y$
	corresponds to the partition $((2n+2)^\ell,
	(n+1)^{2(r-\ell)},1^{2nm-2(n+1)r})$, where $\ell=rank(L)$. By
	Propositions \ref{prop 3.1},  \ref{prop 3.2}, we obtain that if $\ell>0$, then
	$f_{\xi,L}=0$, for all $\xi\in A(\Delta(\tau,m)\gamma_\pi^{(\epsilon)},\eta,k)$. In case $H_{2nm}$ is symplectic, we also need to use Lemma 1.1 in \cite{GRS03}.  This proves the proposition.
\end{proof}

Now, we can exchange, in the right hand side of \eqref{8.16.1},
\begin{equation}\label{8.22}
\mathrm{Y}_{2,1}^{n-1,n+2}\mathrm{Y}_{2,1}^{n,n+2}\mathrm{Y}_{2,1}^{n+1,n+2}\mathrm{Y}_{2,1}^{n+2,n+2}
\end{equation}
with
\begin{equation}\label{8.23}
\mathrm{X}_{1,2}^{n+1,n-1}\mathrm{X}_{1,2}^{n+1,n}\mathrm{X}_{1,2}^{n+1,n+1}\mathrm{X}_{1,2}^{n+1,n+2}.
\end{equation}
Note that these are subgroups, modulo  $\mathrm{Y}_{3,1}^{n-2,n+2}$, $\mathrm{X}_{1,3}^{n+1,n-1}$, respectively.
 Note that both $\mathrm{Y}_{3,1}^{n-2,n+2},
\mathrm{X}_{1,3}^{n+1,n-1}$ are subgroups of
$\mathbf{D}_{2,1}^{1,n+1}$. Again the verification of the set up of
Lemma 7.1 and Corollary 7.1 in \cite{GRS11} is straightforward and
we can exchange $Y$ with $X$ in the integral, in \eqref{8.16.1}. Thus,
this integral is trivial if and only if the following Fourier
coefficient is identically zero ,
\begin{equation}\label{8.24}
\widetilde{\epsilon}_{2,1}^{1,n+1}(\xi)(g)=\int_{\widetilde{\mathrm{D}}_{2,1}^{1,n+1}(F)\backslash
	\widetilde{\mathrm{D}}_{2,1}^{1,n+1}(\BA)}\xi(vg)\psi_{\widetilde{\mathrm{D}}_{2,1}^{1,n+1}}^{-1}(v)dv,
\end{equation}
where
\begin{equation}\label{8.25}
\widetilde{\mathrm{D}}_{2,1}^{1,n+1}=\widetilde{\mathrm{Y}}_{3,1}(n-2,n+2)\mathrm{N}\mathrm{X}_{1,2}(n+1,n-1),
\end{equation}
and $\widetilde{\mathrm{Y}}_{3,1}(n-2,n+2)$ is the subgroup of
elements in $\mathrm{Y}_{3,1}(n-2,n+2)$ of the form \eqref{6.10},
such that their corresponding blocks $Y_{2,1}^{i,n+2},\ n-1\leq
i\leq n+2$, are zero; $\mathrm{X}_{1,2}(n+1,n-1)$ is the subgroup
generated by $\mathrm{X}_{1,3}(n-1,n+1)$ and
$\mathrm{X}_{1,2}^{n+1,i},\ n-1\leq i\leq n+2$. Finally,
$\psi_{\widetilde{\mathrm{D}}_{2,1}^{1,n+1}}$ is the character of
$\widetilde{\mathrm{D}}_{2,1}^{1,n+1}(\BA)$ obtained by extending
$\psi_N$ trivially. Similarly, from Propositions \ref{prop 8.3}, \ref{prop 8.4},
	for all $\xi\in A(\Delta(\tau,m)\gamma_\pi^{(\epsilon)},\eta,k)$, we have
\begin{equation}\label{8.25.1}
\epsilon_{n,m,r}(\xi)(h)=\int_{\widetilde{\mathrm{Y}}^{n+1}_\BA}\tilde{\epsilon}_{2,1}^{1,n+1}(\xi)(yh)dy,
\end{equation}
where $\widetilde{\mathrm{Y}}^{n+1}$ is the subgroup of $\mathrm{Y}$ generated by $\mathrm{Y}^{n+1}$ and $\mathrm{Y}_{2,1}^{i,n+1}$, $n-1\leq i\leq n+2$.

Now, we repeat the steps as in \eqref{7.7} and Proposition \ref{prop
	7.3}. That is we exchange in $\widetilde{\epsilon}_{2,1}^{1,n+1}$ in
\eqref{8.24}, $\mathrm{Y}_{3,1}^{n-2,n+2}$ "into"
$\mathrm{X}_{1,3}^{n+1,n-2}$. Note again that
$\mathrm{Y}_{3,1}^{n-2,n+2}$ is isomorphic to the space of $r\times r$
matrices $z$, such that ${}^t(w_rz)=\delta_H(w_rz)$, while $\mathrm{X}_{1,3}^{n+1,n-2}$ is
isomorphic to $M_{r\times r}$. More precisely, let
$\mathrm{Y}_{2,1}(n-2,n+2)$ be the subgroup of elements in
$\widetilde{\mathrm{Y}}_{3,1}(n-2,n+2)$, written in the form
\eqref{6.10}, such that the block $Y_{3,1}^{n-2,n+2}$ is zero. Let
$\mathfrak{X}_{1,3}^{n+1,n-2}$ be the subgroup of elements in
$\mathrm{X}_{1,3}^{n+1,n-2}$, written as in \eqref{6.9}, such that
the block $X_{1,3}^{n+1,n-2}$ has the property that $ {}^t(w_r
X_{1,3}^{n+1,n-2})=\delta_H(w_rX_{1,3}^{n+1,n-2})$. Let
$\mathfrak{X}_{1,3}(n+1,n-2)$ be the subgroup generated by
$\mathrm{X}_{1,2}(n+1,n-1)$ and $\mathfrak{X}_{1,3}^{n+1,n-2}$.
Define
\begin{equation}\label{8.26}
\mathfrak{D}_{3,1}^{n-2,n+2}=\mathrm{Y}_{2,1}(n-2,n+2)\mathrm{N}\mathfrak{X}_{1,3}(n+1,n-2).
\end{equation}
This is a unipotent subgroup of $H_{2nm}$. Denote, as usual,
by $\psi_{\mathfrak{D}_{3,1}^{n-2,n+2}}$ the character of
$\mathfrak{D}_{3,1}^{n-2,n+2}(\BA)$ obtained by extending
$\psi_N$ trivially. Then, by Corollary 7.1 in \cite{GRS11},
$\widetilde{\epsilon}_{2,1}^{1,n+1}$ is trivial if and only if the
following Fourier coefficient is identically zero,
\begin{equation}\label{8.27}
\varepsilon_{3,1}^{n-2,n+2}(\xi)(g)=\int_{\mathfrak{D}_{3,1}^{n-2,n+2}(F)\backslash
	\mathfrak{D}_{3,1}^{n-2,n+2}(\BA)}\xi(vg)\psi_{\mathfrak{D}_{3,1}^{n-2,n+2}}^{-1}(v)dv.
\end{equation}
Let $\hat{\mathrm{Y}}^{n+1}$ be the subgroup generated by $\widetilde{\mathrm{Y}}^{n+1}$ and $\mathrm{Y}_{3,1}^{n-2,n+2}$. As in Proposition \ref{prop 7.3}, we get that for all $\xi\in A(\Delta(\tau,m)\gamma_\pi^{(\epsilon)},\eta,k)$, we have
	\begin{multline}\label{8.27.1}
	\epsilon_{n,m,r}(\xi)(h)=\int_{\hat{\mathrm{Y}}^{n+1}_\BA}\varepsilon_{3,1}^{n-2,n+2}(\xi)(yh)dy
	=\\
	=\int_{\hat{\mathrm{Y}}^{n+1}_\BA}\int_{\mathfrak{D}_{3,1}^{n-2,n+2}(F)\backslash \mathfrak{D}_{3,1}^{n-2,n+2}(\BA)}\xi(vyh)\psi^{-1}_{\mathfrak{D}_{3,1}^{n-2,n+2}}(v)dvdy.
	\end{multline}
	Moreover, $\epsilon_{n,m,r}$ is trivial on 	$A(\Delta(\tau,m)\gamma_\pi^{(\epsilon)},\eta,k)$ if and only if 
	$\varepsilon_{3,1}^{n-2,n+2}$ is trivial on $A(\Delta(\tau,m)\gamma_\pi^{(\epsilon)},\eta,k)$.

As in Proposition \ref{prop 8.1}, the Fourier coefficient
$\varepsilon_{3,1}^{n-2,n+2}$ remains unchanged when we replace
$\mathfrak{X}_{1,3}^{n+1,n-2}$ by $\mathrm{X}_{1,3}^{n+1,n-2}$. The
proof follows exactly the same steps. The main point is that
$A(\Delta(\tau,m)\gamma_\pi^{(\epsilon)},\eta,k)$ does not support a Fourier coefficient corresponding
to the nilpotent element $Y$ in $Lie(H_{2nm})(F)$ and to the
one parameter subgroup $\varphi(s)$, as in \eqref{8.13}, with the block $\mathcal{U}$ of size 
$(n+1)r$ (instead of $nr$), namely
\begin{equation}\label{8.27.2}
Y=\begin{pmatrix}\mathcal{U}\\
\mathcal{V}&0_{(2nm-2(n+1)r)\times
	(2nm-2(n+1)r)}\\0_{(n+1)r\times
	(n+1)r}&\mathcal{V}'&-\mathcal{U}\end{pmatrix},
\end{equation}
where
$$
\mathcal{U}=\begin{pmatrix}0\\I_r&0\\&I_r\\
&&\cdots\\&&&I_r&0\end{pmatrix}\in M_{(n+1)r\times (n+1)r}(F);
$$
$$
\mathcal{V}=\begin{pmatrix}0_{r\times
	nr}&I_r\\0_{(2nm-2(n+2)r)\times nr}&0_{(2nm-2(n+2)r)\times
	r}\\0_{r\times nr}&L\end{pmatrix};
$$
$$
\varphi(s)=\diag(s^{2n+2}I_r,s^{2n}I_r,...,s^2I_r, I_r,
I_{2nm-2(n+1)r},I_r,s^{-2}I_r,...,s^{-2n-2}I_r);
$$
$L\in M_{r\times r}(F)$ is such that ${}^t(w_rL)=-\delta_Hw_rL$. The
nilpotent orbit of $Y$ corresponds to the partition $((2n+3)^\ell,
(n+2)^{2(r-\ell)},1^{2nm-2(n+2)r+\ell})$, where $\ell=rank(L)$.
Now, we apply Propositions \ref{prop 3.1}, \ref{prop 3.2}. We get that
\begin{multline}\label{8.28'}
\varepsilon_{3,1}^{n-2,n+2}(\xi)(g)=\epsilon_{3,1}^{n-2,n+2}(\xi)(g)=\\
\int_{\mathrm{D}_{3,1}^{n-2,n+2}(F)\backslash
	\mathrm{D}_{3,1}^{n-2,n+2}(\BA)}\xi(vg)\psi_{\mathrm{D}_{3,1}^{n-2,n+2}}^{-1}(v)dv,
\end{multline}
where
$$
\mathrm{D}_{3,1}^{n-2,n+2}=\mathfrak{D}_{3,1}^{n-2,n+2}\mathrm{X}_{1,3}^{n+1,n-2}.
$$
By \eqref{8.27.1}, we get
	\begin{multline}\label{8.27.3}
\epsilon_{n,m,r}(\xi)(h)=\int_{\hat{\mathrm{Y}}^{n+1}_\BA}\epsilon_{3,1}^{n-2,n+2}(\xi)(yh)dy
=\\
=\int_{\hat{\mathrm{Y}}^{n+1}_\BA}\int_{\mathrm{D}_{3,1}^{n-2,n+2}(F)\backslash \mathrm{D}_{3,1}^{n-2,n+2}(\BA)}\xi(vyh)\psi^{-1}_{\mathrm{D}_{3,1}^{n-2,n+2}}(v)dvdy.
\end{multline}

At this point, we can repeat the process of Proposition \ref{prop
	8.2} and "fill in" (in the right order) the rest of block row $n+1$
both in $\mathrm{X}_{1,3}$, $\mathrm{X}_{1,2}$ in exchange of the
remaining blocks in block column $n+2$ of $Y_{3,1}$, $Y_{2,1}$. Let $\mathrm{Y}^{n+2}$ be the subgroup of $\mathrm{Y}$ generated by $\hat{\mathrm{Y}}^{n+1}$ and the subgroups $\mathrm{Y}_{3,1}^{i,n+2}$, $1\leq i\leq n-3$, $\mathrm{Y}_{2,1}^{i',n+2}$, $1\leq i'\leq n-2$.
As in Proposition \ref{prop 8.3}, we get
\begin{prop}\label{prop 8.3.1}
	For all $\xi\in A(\Delta(\tau,m)\gamma_\pi^{(\epsilon)},\eta,k)$, we have
	\begin{multline}\label{8.27.4}
	\epsilon_{n,m,r}(\xi)(h)=\int_{\mathrm{Y}^{n+2}_\BA}\epsilon_{2,1}^{1,n+2}(\xi)(yh)dy
	=\\
	=\int_{\mathrm{Y}^{n+2}_\BA}\int_{\mathrm{D}_{2,1}^{1,n+2}(F)\backslash \mathrm{D}_{2,1}^{1,n+2}(\BA)}\xi(vyh)\psi^{-1}_{\mathrm{D}_{2,1}^{1,n+2}}(v)dvdy.
	\end{multline}
	Moreover, $\epsilon_{n,m,r}$ is trivial on
	$A(\Delta(\tau,m)\gamma_\pi^{(\epsilon)},\eta,k)$ if and only if $\epsilon_{2,1}^{1,n+2}$ is trivial on 	$A(\Delta(\tau,m)\gamma_\pi^{(\epsilon)},\eta,k)$.
\end{prop}

We go on and prove
the analog of Proposition \ref{prop 8.4}, namely "fill $\mathrm{X}_{1,3}^{n+2,n-3}$ in $\epsilon^{1,n+2}$", using Fourier
expansions and Propositions \ref{prop 3.1}, \ref{prop 3.2}.
Then we exchange (the subgroup
generated by) $\mathrm{Y}_{2,1}^{i,n+3}$, $n-2\leq i\leq n+3$
(modulo $\mathrm{Y}_{3,1}^{n-3,n+3}$) with (the subgroup generated
by) $\mathrm{X}_{1,2}^{n+2,j}$, $n-2\leq j\leq n+3$ (modulo
$\mathrm{X}_{1,3}^{n+2,n-2}$); next, exchange $Y_{3,1}^{n-3,n+3}$ with
$\mathfrak{X}_{1,3}^{n+2,n-3}$ (analogous notation); prove the
analog of Proposition \ref{prop 8.1} to "fill in" all of
$\mathrm{X}_{1,3}^{n+2,n-3}$; repeat the process of Proposition
\ref{prop 8.2} and "fill in" the rest of block row $n+2$ both in
$\mathrm{X}_{1,3}$, $\mathrm{X}_{1,2}$, and so on, until we "use up"
all columns of $\mathrm{Y}_{3,1}$ and $\mathrm{Y}_{2,1}$. Let us
write this more precisely.

Define, for $n+1\leq j\leq 2n-2$,
\begin{equation}\label{8.28}
\mathrm{D}_{2,1}^{1,j}=\mathrm{Y}_{3,1}(2n-j-1,j+1)\mathrm{N}\mathrm{X}_{1,2}(j-1,1),
\end{equation}
where $\mathrm{Y}_{3,1}(2n-j-1,j+1)$ is the subgroup of
$\mathrm{Y}_{3,1}(n-2,n+2)$ of all elements, written in the form
\eqref{6.10}, such that their corresponding blocks $Y_{2,1}^{i',j'},
\ Y_{3,1}^{i',j'}$ are zero, for all $n+1\leq j'\leq j$ and all
$i'$; $\mathrm{X}_{1,2}(j-1,1)$ is the subgroup generated by
$\mathrm{X}_{1,2}(n,1)$ and the subgroups
$\mathrm{X}_{1,2}^{i',j'},\ \mathrm{X}_{1,3}^{i',j'}$, with $n+1\leq
i'\leq j-1$ and all $j'$. $\mathrm{D}_{2,1}^{1,j}$ is an
$F$-unipotent subgroup of $H_{2nm}$. Consider the character
$\psi_{\mathrm{D}_{2,1}^{1,j}}$, defined as usual. Assume that we
have already carried all root exchanges described above, so that
\begin{equation}\label{8.28.1}
\epsilon_{n,m,r}(\xi)(h)=\int_{\mathrm{Y}^j_\BA}\epsilon_{2,1}^{1,j}(\xi)(yh)dy,
\end{equation}
where
\begin{equation}\label{8.29}
\epsilon_{2,1}^{1,j}(\xi)(g)=\int_{\mathrm{D}_{2,1}^{1,j}(F)\backslash
	\mathrm{D}_{2,1}^{1,j}(\BA)}\xi(vg)\psi_{\mathrm{D}_{2,1}^{1,j}}^{-1}(v)dv,
\end{equation}
and $\mathrm{Y}^j$ is the subgroup of $\mathrm{Y}$ generated by by the $\mathrm{Y}_{2,1}^{i',j'}, \mathrm{Y}_{3,1}^{i'',j''}\subset \mathrm{Y}$, such that $j',j''\leq j$, and, also, $\epsilon_{n,m,r}$ is trivial on $A(\Delta(\tau,m)\gamma_\pi^{(\epsilon)},\eta,k)$ if and only if the
following $\epsilon_{2,1}^{1,j}$ is trivial on $A(\Delta(\tau,m)\gamma_\pi^{(\epsilon)},\eta,k)$.
Let
\begin{equation}\label{8.30}
\mathbf{D}_{2,1}^{1,j}=\mathrm{D}_{2,1}^{1,j}\mathrm{X}_{1,3}^{j,2n-j}=
\mathrm{Y}_{3,1}(2n-j-1,j+1)\mathrm{N}\mathrm{X}_{1,3}(j,2n-j),
\end{equation}
and let $\psi_{\mathbf{D}_{2,1}^{1,j}}$ be the character of
$\mathbf{D}_{2,1}^{1,j}(\BA)$ defined in the usual way. Then
\begin{prop}\label{prop 8.5}
	For all $\xi\in A(\Delta(\tau,m)\gamma_\pi^{(\epsilon)},\eta,k)$, $g\in H(\BA)$,
	\begin{equation}\label{8.31}
	\epsilon_{2,1}^{1,j}(\xi)(g)=\int_{\mathbf{D}_{2,1}^{1,j}(F)\backslash
		\mathbf{D}_{2,1}^{1,j}(\BA)}\xi(vg)\psi_{\mathbf{D}_{2,1}^{1,j}}^{-1}(v)dv.
	\end{equation}
\end{prop}
The proof is the same as that of Proposition \ref{prop 8.4},
replacing $nr$ by $(j-1)r$. The main point in the proof is that, by
Propositions \ref{prop 3.1}, \ref{prop 3.2}, $A(\Delta(\tau,m)\gamma_\pi^{(\epsilon)},\eta,k)$ does not support Fourier
coefficients corresponding to the partition 
$$
((2j)^\ell,j^{2(r-\ell)},1^{2nm-2rj}), 
$$
for any $\ell>0$, $j>n$.

Next, we exchange in the integral of \eqref{8.31}, the group
generated by $\mathrm{Y}_{2,1}^{i,j+1}$, $2n-j\leq i\leq j+1$, modulo
$\mathrm{Y}_{3,1}^{2n-j-1,j+1}$, with the group generated by
$\mathrm{X}_{1,2}^{j,i}$, $2n-j\leq i\leq j+1$, modulo
$\mathrm{X}_{1,3}^{j,2n-j}$. The setup of Lemma 7.1 in \cite{GRS11} can
be verified in a straightforward manner. We get that the integral in
the left hand side of \eqref{8.31} is trivial on $A(\Delta(\tau,m)\gamma_\pi^{(\epsilon)},\eta,k)$ if
and only if the following integral is trivial on $A(\Delta(\tau,m)\gamma_\pi^{(\epsilon)},\eta,k)$
\begin{equation}\label{8.32}
\widetilde{\epsilon}_{2,1}^{1,j}(\xi)(g)=\int_{\widetilde{\mathrm{D}}_{2,1}^{1,j}(F)\backslash
	\widetilde{\mathrm{D}}_{2,1}^{1,j}(\BA)}\xi(vg)\psi_{\widetilde{\mathrm{D}}_{2,1}^{1,j}}^{-1}(v)dv,
\end{equation}
where
\begin{equation}\label{8.33}
\widetilde{\mathrm{D}}_{2,1}^{1,j}=\widetilde{\mathrm{Y}}_{3,1}(2n-j-1,j+1)\mathrm{N}\mathrm{X}_{1,2}(j,2n-j),
\end{equation}
and $\widetilde{\mathrm{Y}}_{3,1}(2n-j-1,j+1)$ is the subgroup of
elements in $\mathrm{Y}_{3,1}(2n-j-1,j+1)$ of the form \eqref{6.10},
such that their corresponding blocks $Y_{2,1}^{i,j+1}$ are zero, for
$2n-j\leq i\leq j+1$; $\mathrm{X}_{1,2}(j,2n-j)$ is the subgroup
generated by $\mathrm{X}_{1,3}(j,2n-j)$ and
$\mathrm{X}_{1,2}^{j,i}$, for $2n-j\leq i\leq j+1$; the character
$\psi_{\widetilde{\mathrm{D}}_{2,1}^{1,j}}$ is defined as usual.
Now we exchange in $\widetilde{\epsilon}_{2,1}^{1,j}$ in
\eqref{8.32}, $\mathrm{Y}_{3,1}^{2n-j-1,j+1}$ "into"
$\mathrm{X}_{1,3}^{j,2n-j-1}$. So, let
$\mathrm{Y}_{2,1}(2n-j-1,j+1)$ be the subgroup of elements in
$\widetilde{\mathrm{Y}}_{3,1}(2n-j-1,j+1)$, written in the form
\eqref{6.10}, such that the block $Y_{3,1}^{2n-j-1,j+1}$ is zero.
Let $\mathfrak{X}_{1,3}^{j,2n-j-1}$ be the subgroup of elements in
$\mathrm{X}_{1,3}^{j,2n-j-1}$, written as in \eqref{6.9}, such that
the block $X_{1,3}^{j,2n-j-1}$ has the property that ${}^t(w_r
X_{1,3}^{j,2n-j-1})=\delta_H(w_rX_{1,3}^{j,2n-j-1})$. Let
$\mathfrak{X}_{1,3}(j,2n-j-1)$ be the subgroup generated by
$\mathrm{X}_{1,2}(j,2n-j)$ and $\mathfrak{X}_{1,3}^{j,2n-j-1}$.
Define
\begin{equation}\label{8.34}
\mathfrak{D}_{3,1}^{2n-j-1,j+1}=\mathrm{Y}_{2,1}(2n-j-1,j+1)\mathrm{N}\mathfrak{X}_{1,3}(j,2n-j-1).
\end{equation}
Define the character $\psi_{\mathfrak{D}_{3,1}^{2n-j-1,j+1}}$ in the
usual way. Then, by Corollary 7.1 in \cite{GRS11},
$\widetilde{\epsilon}_{2,1}^{1,j}$ is trivial if and only if the
following Fourier coefficient is identically zero,
\begin{equation}\label{8.35}
\varepsilon_{3,1}^{2n-j-1,j+1}(\xi)(g)=\int_{\mathfrak{D}_{3,1}^{2n-j-1,j+1}(F)\backslash
	\mathfrak{D}_{3,1}^{2n-j-1,j+1}(\BA)}\xi(vg)\psi_{\mathfrak{D}_{3,1}^{2n-j-1,j+1}}^{-1}(v)dv.
\end{equation}
The Fourier coefficient $\varepsilon_{3,1}^{2n-j-1,j+1}$ remains
unchanged when we replace $\mathfrak{X}_{1,3}^{j,2n-j-1}$ by
$\mathrm{X}_{1,3}^{j,2n-j-1}$. Denote by
$\mathrm{X}_{1,3}(j,2n-j-1)$ the group generated by
$\mathrm{X}_{1,2}(j,2n-j)$ and $\mathrm{X}_{1,3}^{j,2n-j-1}$, and
let
\begin{equation}\label{8.36}
\mathrm{D}_{3,1}^{2n-j-1,j+1}=\mathrm{Y}_{2,1}(2n-j-1,j+1)\mathrm{N}\mathrm{X}_{1,3}(j,2n-j-1),
\end{equation}
with the corresponding character of
$\mathrm{D}_{3,1}^{2n-j-1,j+1}(\BA)$,
$\psi_{\mathrm{D}_{3,1}^{2n-j-1,j+1}}$.
\begin{prop}\label{prop 8.6}
	Let $\xi\in A(\Delta(\tau,m)\gamma_\pi^{(\epsilon)},\eta,k)$. Then
	\begin{equation}\label{8.37}
	\varepsilon_{3,1}^{2n-j-1,j+1}(\xi)(g)=
	\end{equation}
	$$
	\int_{\mathrm{D}_{3,1}^{2n-j-1,j+1}(F)\backslash
		\mathrm{D}_{3,1}^{2n-j-1,j+1}(\BA)}\xi(vg)\psi_{\mathrm{D}_{3,1}^{2n-j-1,j+1}}^{-1}(v)dv:=\epsilon_{3,1}^{2n-j-1,j+1}(\xi)(g).
	$$
\end{prop}
The proof is the same as that of Proposition \ref{prop 8.1},
replacing $nr$ by $jr$. The main point, as in \eqref{8.28'}, is
that, by Propositions \ref{prop 3.1}, \ref{prop 3.2}, $A(\Delta(\tau,m)\gamma_\pi^{(\epsilon)},\eta,k)$ does not support
Fourier coefficients corresponding to the partition $((2j+1)^\ell,
(j+1)^{2(r-\ell)},1^{2nm-2(j+1)r+\ell})$, since $j>n$.

Now the road is open to apply the process of Proposition \ref{prop
	7.1}. Define, for $1\leq i\leq 2n-j-1$ (and $n+1\leq j\leq 2n-2$),
\begin{equation}\label{8.38}
\begin{array}{rcl}
\mathrm{D}_{3,1}^{i,j+1}&=&\mathrm{Y}_{2,1}(i,j+1)\mathrm{N}\mathrm{X}_{1,3}(j,i);\\
\mathrm{D}_{2,1}^{i,j+1}&=&\mathrm{Y}_{3,1}(i-1,j+1)\mathrm{N}\mathrm{X}_{1,2}(j,i)\qquad
(i>1),
\end{array}
\end{equation}
and define the characters $\psi_{\mathrm{D}_{3,1}^{i,j+1}},\
\psi_{\mathrm{D}_{2,1}^{i,j+1}}$ as before. Note that
$\mathrm{D}_{2,1}^{1,j+1}$ is already defined in \eqref{8.28}
$$
\mathrm{D}_{2,1}^{1,j+1}=\mathrm{Y}_{3,1}(2n-j-2,j+2)\mathrm{N}\mathrm{X}_{1,2}(j,1)
$$
When $j=2n-2$, we define
$$
\mathrm{D}_{2,1}^{1,2n-1}=\mathrm{N}\mathrm{X}_{1,2}(2n-2,1).
$$
Define, for $\xi \in A(\Delta(\tau,m)\gamma_\pi^{(\epsilon)},\eta,k)$,  as in \eqref{7.1}, the Fourier coefficients $\epsilon_{2,1}^{i,j+1}(\xi)(g)$ and
$\epsilon_{3,1}^{i,j+1}(\xi)(g)$ of $\xi$ with respect
$\psi_{\mathrm{D}_{2,1}^{i,j+1}}^{-1}$ and
$\psi_{\mathrm{D}_{3,1}^{i,j+1}}^{-1}$, respectively. Then we have the same
statement as that of Proposition \ref{prop 7.1}, with $j+1$
replacing $j$, with the same proof, for the groups defined in
\eqref{8.38}. We conclude
\begin{prop}\label{prop 8.7}
	Let $n+1\leq j\leq 2n-2$ and $1\leq i\leq 2n-j-1$. We have the following identities, for $\xi\in A(\Delta(\tau,m)\gamma_\pi^{(\epsilon)},\eta,k)$, and $h\in H(\BA)$.\\
	{\bf 1.} 
	$$
	\epsilon_{3,1}^{i,j+1}(\xi)(h)=\int_{\mathrm{Y}_{2,1}^{i,j+1}(\BA)}\epsilon_{2,1}^{i,j+1}(\xi)(yh)dy.
	$$
	Moreover, $\epsilon_{3,1}^{i,j+1}$ is zero on $A(\Delta(\tau,m)\gamma_\pi^{(\epsilon)},\eta,k)$ if and
	only if $\epsilon_{2,1}^{i,j+1}$ is zero on $A(\Delta(\tau,m)\gamma_\pi^{(\epsilon)},\eta,k)$.\\
	{\bf 2.} Assume that $i>1$. 
	$$
	\epsilon_{2,1}^{i,j+1}(\xi)(h)=\int_{\mathrm{Y}_{3,1}^{i-1,j+1}(\BA)}\epsilon_{3,1}^{i-1,j+1}(\xi)(yh)dy.
	$$
	Moreover, $\epsilon_{2,1}^{i,j+1}$ is zero on
	$A(\Delta(\tau,m)\gamma_\pi^{(\epsilon)},\eta,k)$ if and only if $\epsilon_{3,1}^{i-1,j+1}$ is zero on 	$A(\Delta(\tau,m)\gamma_\pi^{(\epsilon)},\eta,k)$.
\end{prop}
A repeated application of Proposition \ref{prop 8.7}, Propositions
\ref{prop 8.5}, \ref{prop 8.6}, what we explained before, and \eqref{8.28.1} prove
\begin{prop}\label{prop 8.8}
	Let $n+1\leq j\leq 2n-2$. Then 
\begin{equation}\label{8.39}
\epsilon_{n,m,r}(\xi)(h)=\int_{\mathrm{Y}^{j+1}_\BA}\epsilon_{2,1}^{1,j+1}(\xi)(yh)dy.
\end{equation}	
Moreover, $\epsilon_{n,m,r}$ is trivial on $A(\Delta(\tau,m)\gamma_\pi^{(\epsilon)},\eta,k)$ if and only if $\epsilon_{2,1}^{1,j+1}$ is trivial on $A(\Delta(\tau,m)\gamma_\pi^{(\epsilon)},\eta,k)$. Hence
\begin{equation}\label{8.40}
\epsilon_{n,m,r}(\xi)(h)=\int_{\mathrm{Y}_\BA}\epsilon_{2,1}^{1,2n-1}(\xi)(yh)dy,
\end{equation}
and	$\epsilon_{n,m,r}$ is trivial on $A(\Delta(\tau,m)\gamma_\pi^{(\epsilon)},\eta,k)$ if and only if $\epsilon_{2,1}^{1,2n-1}$ is trivial on $A(\Delta(\tau,m)\gamma_\pi^{(\epsilon)},\eta,k)$. 
\end{prop}
This completes the proof of Theorem \ref{thm 6.1}, since
$\mathrm{D}_{2,1}^{1,2n-1}$ is the group $\mathcal{D}'_{n,m,r}$, and
so, we have
$$
\epsilon_{2,1}^{1,2n-1}=\epsilon'_{n,m,r}.
$$
From \eqref{6.7.4} and Theorem \ref{thm 6.1}, we have, for all $\xi\in A(\Delta(\tau,m)\gamma_\pi^{(\epsilon)},\eta,k)$, and $h\in H(\BA)$,
\begin{equation}\label{8.41}
\mathcal{F}_\psi^r(\xi)(h)=\int_{\mathbf{Y}_1(\BA)}\cdots \int_{\mathbf{Y}_{n-1}(\BA)}\int_{\mathrm{Y}(\BA)}\epsilon'_{n,m,r}(\xi)(yw_0y_{n-1}\cdots y_1h)dydy_{n-1}\cdots dy_1.
\end{equation}
Denote by $\mathbf{Y}$ the subgroup generated by $\mathrm{Y}$ and $w_0\mathbf{Y}_iw_0^{-1}$, $i=1,...,n-1$. See \eqref{4.9}, \eqref{6.10}. Then \eqref{8.41} becomes
\begin{equation}\label{8.42}
\mathcal{F}_\psi^r(\xi)(h)=\int_{\mathbf{Y}(\BA)}\epsilon'_{n,m,r}(\xi)(yw_0h)dy.
\end{equation}
Let us describe the subgroup $\mathbf{Y}$. As in \eqref{6.10}, its elements have the form
\begin{equation}\label{8.43}
\begin{pmatrix}I_{(2n-1)r}\\Y_{2,1}&I_{2n(m-2r)+2r}\\
Y_{3,1}&Y'_{2,1}&I_{(2n-1)r}\end{pmatrix}\in H_{2nm},
\end{equation}
where the blocks are described as follows. The description is similar to that which appears right after \eqref{6.5.4}. Write the block $Y_{3,1}$ as a $(2n-1)\times (2n-1)$ matrix of blocks
$Y_{3,1}^{i,j}$, each one of size $r\times r$ ($1\leq i,j\leq
2n-1$). Then $Y_{3,1}^{i,j}=0$, for all $i\geq j$. (In \eqref{6.10}, we required that $Y_{3,1}^{i,j}=0$, for all $i\geq j-1$.) Thus, as a
$(2n-1)\times (2n-1)$ matrix of $r\times r$ blocks, $Y_{3,1}$ has an
upper triangular shape, with zero blocks on the diagonal.

The block $Y_{2,1}$ has $2n-1$ block columns, each one of size $r$. It has $2n+1$ block rows, the size of the first and last $n-1$ block rows is $m-2r$ and the sizes of
the three middle block rows are $[\frac{m}{2}],2([\frac{m+1}{2}]-r),[\frac{m}{2}]$. Denote by
$\widetilde{Y}_{2,1}$ the matrix obtained from $Y_{2,1}$ by deleting
block rows number $n, n+2$ (each one is of size $[\frac{m}{2}]$). Then
$\widetilde{Y}_{2,1}$ has a shape of an upper triangular matrix of blocks, where all the diagonal blocks are zero. When $m$ is even, the blocks in the second upper diagonal of $\widetilde{Y}_{2,1}$ are arbitrary, except the middle block in row $n$ and column $n+1$, which is zero. When $m=2m'-1$ is odd, the blocks in the second upper diagonal of $\widetilde{Y}_{2,1}$ are arbitrary, except the middle block in row $n$ and column $n+1$, which has the form
\begin{equation}\label{8.44}
\begin{pmatrix}0_{(m'-1-r)\times r}\\x\\y\\0_{(m'-1-r)\times r}\end{pmatrix},
\end{equation}
where $x$, $y$ are rows of $r$ coordinates. (Compare with \eqref{6.6}.) The rows number $n, n+2$ of $Y_{2,1}$ are rows of blocks of size $[\frac{m}{2}]\times r$, and their first $n$ blocks are all zero. The block $Y_{2,1}'$ has a form "dual" to that of $Y_{2,1}$ and is determined by $Y_{2,1}$. It has $2n-1$ block rows, each one of size $r$. It has $2n+1$ block
columns, the first and last $n-1$ columns are of size
$m-2r$ each and the three middle columns are of sizes $[\frac{m}{2}],2([\frac{m+1}{2}]-r),[\frac{m}{2}]$.
When we delete columns number $n, n+2$, we obtain
$\widetilde{Y'}_{2,1}$, a $(2n-1)\times (2n-1)$ matrix of blocks, where each block
is of size $r\times (m-2r)$, except the blocks of the middle column, which are each of size $r\times 2([\frac{m+1}{2}]-r)$. Then $\widetilde{Y}'_{2,1}$ has a shape of an upper triangular matrix of blocks, where all the diagonal blocks are zero. When $m$ is even, the blocks in the second upper diagonal are arbitrary, except the middle block in row $n$ and column $n+1$, which is zero. When $m=2m'-1$ is odd, the blocks in the second upper diagonal of $\tilde{Y}_{2,1}'$ are arbitrary, except the middle block in row $n$ and column $n+1$, which has the form
\begin{equation}\label{8.45}
(0_{r\times (m'-1-r)},w_{m'-1-r}{}^ty,w_{m'-1-r}{}^tx,0_{r\times (m'-1-r)}),
\end{equation}
where $x$, $y$ are as in \eqref{8.44}. The columns number $n, n+2$ of $Y'_{2,1}$ are columns of blocks of size $r\times [\frac{m}{2}]$, and their last $n$ blocks are all zero. 

Let us record \eqref{8.42} in the following theorem
\begin{thm}\label{thm 8.10}
Let $0\leq r\leq [\frac{m}{2}]$. For all $\xi\in A(\Delta(\tau,m)\gamma_\pi^{(\epsilon)},\eta,k)$, and $h\in H(\BA)$,  
$$
\mathcal{F}_\psi^r(\xi)(h)=\int_{\mathbf{Y}(\BA)}\epsilon'_{n,m,r}(\xi)(yw_0h)dy,
$$
where $\mathbf{Y}$ is the subgroup generated by $\mathrm{Y}$ and $w_0\mathbf{Y}_iw_0^{-1}$, $i=1,...,n-1$.
Moreover, $\mathcal{F}_\psi^r$ is identically zero on $A(\Delta(\tau,m)\gamma_\pi^{(\epsilon)},\eta,k)$ if and only if $\epsilon'_{n,m,r}$
is identically zero on $A(\Delta(\tau,m)\gamma_\pi^{(\epsilon)},\eta,k)$.
\end{thm}

\section{Fourier expansions II}

Let $E_1$ be the subgroup of $\mathcal{D}'_{n,m,r}$, consisting of
the elements in \eqref{6.8'}, such that the $V$ block in $\mathrm{N}$ is the identity. Let
$\psi_{E_1}$ be the restriction of $\psi_{\mathcal{D}'_{n,m,r}}$ to $E_1(\BA)$. Then
\begin{multline}\label{9.1}
\epsilon'_{n,m,r}(\xi)(h)=\\
\int_{U'_{{(m-2r)}^{n-1}}(F)\backslash
	U'_{{(m-2r)}^{n-1}}(\BA)}\int_{E_1(F)\backslash E_1(\BA)}\xi(uvh)\psi^{-1}_{E_1}(u)\psi^{-1}_{U'_{{(m-2r)}^{n-1}}}(v)dudv,
\end{multline}
where $U'_{{(m-2r)}^{n-1}}$ is the unipotent subgroup of the elements $V$ in \eqref{6.5}, identified as the subgroup of $N$ consisting of the elements \eqref{6.8} with $U=I_{(2n-1)r}$;
$\psi_{U'_{(m-2r)^{n-1}}}$ is the restriction of $\psi_N$ to
$U'_{(m-2r)^{n-1}}(\BA)$. Let
\begin{equation}\label{9.2}
x_r(y,c)=\begin{pmatrix}I_{(2n-2)r}\\&I_r&y&c\\&&I_{2(nm-(2n-1)r)}&y'\\&&&I_r\\
&&&&I_{(2n-2)r}\end{pmatrix},
\end{equation}
and $c$ is any matrix, such that the unipotent matrix
in \eqref{9.2} is in $H_{2nm}$.
Consider, for a given $h\in H(\BA)$, the following
smooth function on
$M_{r\times (2(nm-(2n-1)r))}(\BA)$,
\begin{multline}\label{9.3}
y\mapsto f(A(\Delta(\tau,m)\gamma_\psi^{(\epsilon)},\eta, k)(h)\xi)(y)=\epsilon'_{n,m,r}(\xi)(x_r(y,c)h)\\
\int_{U'_{{(m-2r)}^{n-1}}(F)\backslash
	U'_{{(m-2r)}^{n-1}}(\BA)}\int_{E_1(F)\backslash E_1(\BA)}\xi(uvx_r(y,c)h)\psi^{-1}_{E_1}(u)\psi^{-1}_{U'_{{(m-2r)}^{n-1}}}(v)dudv.
\end{multline}
Note that the integral in
\eqref{9.3} is independent of the choice of $c$. Our next goal will be to try to show that \eqref{9.3} is constant in $y$. For this to hold, we will need that $m, n, k$ satisfy certain relations. These will be listed in Cor. \ref{cor 10.3}. These include the cases of functoriality. In order to help the reader keep track of our goals, we state this in the following theorem.
\begin{thm}\label{thm 9.0}
Let $\tau$ and $H_m^{(\epsilon)}$ be as in Theorem \ref{thm 2.2}. For all $\xi\in A(\Delta(\tau,m)\gamma_\psi^{(\epsilon)},\eta, 1)$, and all $y\in M_{r\times (2(nm-(2n-1)r))}(\BA)$,
$$
f(\xi)(y)=f(\xi)(0),
$$
that is, for all $\xi\in A(\Delta(\tau,m)\gamma_\psi^{(\epsilon)},\eta, 1)$ and all $y\in M_{r\times (2(nm-(2n-1)r))}(\BA)$,
$$
\epsilon'_{n,m,r}(\xi)(x_r(y,c))=\epsilon'_{n,m,r}(\xi)(1).
$$

\end{thm}
The proof of this theorem will be concluded in the end of Sec. 11. As we just mentioned, we will prove this theorem in a more general framework.

Denote, for short,
$$
\lambda=2(nm-(2n-1)r)=2(n-1)(m-2r)+2(m-r).
$$
Let us write $y\in M_{r\times \lambda}$ in \eqref{9.2}, in the form
\begin{equation}\label{9.10}
y=(y_1,...,y_{n-1},y_n,y_{n+1},y_{n+2}, y_{n+3},...,y_{2n+1}),
\end{equation}
where, for $1\leq i\leq n-1$, or $n+3\leq i\leq 2n+1$, $y_i\in M_{r\times (m-2r)}$; $y_n, y_{n+2}\in M_{r\times [\frac{m}{2}]}$, and $y_{n+1}\in M_{r\times (2([\frac{m+1}{2}])-r)}$. Let $1\leq i \leq 2n+1$. We will denote, for the element $y$ in \eqref{9.10},
\begin{equation}\label{9.10.1}
y^{(i)}=(0,...,0, y_{2n+2-i},...,y_{2n+1}).
\end{equation} 
Since $\mathrm{X}_{1,2}^{2n-1,2n+1}\subset E_1$,
and $\psi_{E_1}$ is trivial on $\mathrm{X}_{1,2}^{2n-1,2n+1}(\BA)$, the function \eqref{9.3} satisfies
\begin{equation}\label{9.4}
f(A(\Delta(\tau,m)\gamma_\psi^{(\epsilon)},\eta, k)(h)\xi)(y^{(1)})=f(A(\Delta(\tau,m)\gamma_\psi^{(\epsilon)},\eta, k)(h)\xi)(0),
\end{equation}
for all $\xi\in A(\Delta(\tau,m)\gamma_\psi^{(\epsilon)},\eta, k)$, and all $y^{(1)}$ with adelic coordinates.
Note that
\begin{equation}\label{9.5}
\epsilon'_{n,m,r}(\xi)(h)=f(A(\Delta(\tau,m)\gamma_\psi^{(\epsilon)},\eta, k)(h)\xi)(0)=
f(A(\Delta(\tau,m)\gamma_\psi^{(\epsilon)},\eta, k)(vxh)\xi)(0),
\end{equation}
for all $x\in \mathrm{X}_{1,2}^{2n-1,2n+1}(\BA)$. 
Our main goal in this section is to prove
\begin{thm}\label{thm 9.1}
For all $\xi\in A(\Delta(\tau,m)\gamma_\psi^{(\epsilon)},\eta, k)$, and all $y^{(n-1)}$ with adele coordinates,
\begin{equation}\label{9.6}
f(\xi)(y^{(n-1)})=f(\xi)(0).
\end{equation}
\end{thm}

\begin{proof} Let $1\leq i\leq n-1$. Consider the subgroup $U^{(i)}$ of $U'_{{(m-2r)}^{n-1}}$, consisting of the elements of the form $diag(I_{(2n-1)r}, V, I_{2n-1)r})$, where $V$ is of the form
\begin{equation}\label{9.10.2}
\begin{pmatrix}I_{(m-2r)(i-2)}&0&\ast&\ast&\ast&\ast&\ast\\
&I_{m-2r}&v&\ast&\ast&\ast&\ast&\\&&I_{m-2r}&0&0&\ast&\ast\\&&&I_{\lambda-2i(m-2r)}&0&\ast&\ast\\&&&&I_{m-2r}&v'&\ast\\&&&&&I_{m-2r}&0\\&&&&&&I_{(m-2r)(i-2)}\end{pmatrix}.
\end{equation}
The restriction of the character $\psi_{U'_{{(m-2r)}^{n-1}}}$ to the elements \eqref{9.10.2} (with adele coordinates) is given by $\psi(tr(v))$. Denote this restriction by $\psi_{U^{(i)}}$.
Consider, for a given $h\in H(\BA)$, the following
smooth function on
$M_{r\times (m-2r)i}(\BA)$,
\begin{multline}\label{9.10.3}
y\mapsto f^{(i)}(A(\Delta(\tau,m)\gamma_\psi^{(\epsilon)},\eta, k)(h)\xi)(y^{(i)})=\\
\int_{U^{(i)}(F)\backslash
	U^{(i)}(\BA)}\int_{E_1(F)\backslash E_1(\BA)}\xi(uvx_r(y^{(i)},0)h)\psi^{-1}_{E_1}(u)\psi^{-1}_{U^{(i)}}(v)dudv.
\end{multline}
This is an inner integral of \eqref{9.3}. It will suffice to prove, by induction on $i$, that, for
$1\leq i\leq n-1$, and  $y_{2n+2-i},y_{2n+3-i},...,y_{2n+1}\in M_{r\times(m-2r)}(\BA)$,
\begin{equation}\label{9.11}
f^{(i)}(\xi)(0,...,0,y_{2n+2-i},...,y_{2n+1})=f^{(i)}(\xi)(0).
\end{equation}
This will prove the theorem. Note that the case $i=1$ is \eqref{9.5}. Thus, we may assume that $2\leq i\leq n-1$, and proceed by induction, assuming that the theorem is proved for $i-1$. Define, for $z\in M_{r\times (m-2r)}(\BA)$,
$$
f_i(\xi)(z)=f^{(i)}(\xi)(0,...,0,z,0,...,0),
$$
where $z$ is coordinate number $2n+2-i$. This is a smooth function on\\ 
$M_{r\times (m-2r)}(F)\backslash M_{r\times (m-2r)}(\BA)$. Note that by the induction assumption, for all $y_{2n+3-i},...,y_{2n+1}\in M_{r\times(m-2r)}(\BA)$,
\begin{equation}\label{9.12}
f_i(\xi)(z)=f^{(i)}(\xi)(0,...,0,z, y_{2n+3-i},...,y_{2n+1}).
\end{equation}
Write the Fourier expansion of $f_i$. A character of $M_{r\times (m-2r)}(F)\backslash M_{r\times (m-2r)}(\BA)$ has the form $\psi(tr(Lz))$, where $L\in M_{(m-2r)\times r}(F)$. The corresponding Fourier coefficient is
\begin{equation}\label{9.13}
f_i^{\psi,L}(\xi)=\int_{M_{r\times (m-2r)}(F)\backslash M_{r\times (m-2r)}(\BA)}f_i(\xi)(z)\psi^{-1}(tr(Lz))dz.
\end{equation}
We will show that the coefficient \eqref{9.13} is trivial on $A(\Delta(\tau,m)\gamma_\psi^{\epsilon)},\eta,k)$, for all nonzero $L$, and this will prove the theorem. Let $E_i$ be the subgroup generated by $E_1$ and $\mathrm{X}_{1,2}^{2n-1,2n}$, $\mathrm{X}_{1,2}^{2n-1,2n-1}$,...,$\mathrm{X}_{1,2}^{2n-1,2n+2-i}$. Let $\psi_{E_i,L}$ be the character of $E_i(\BA)$, which is $\psi_{E_1}$ on $E_1(\BA)$, trivial on $\mathrm{X}_{1,2}^{2n-1,2n}(\BA)$, $\mathrm{X}_{1,2}^{2n-1,2n-1}(\BA)$,...,$\mathrm{X}_{1,2}^{2n-1,2n+3-i}(\BA)$, and on $\mathrm{X}_{1,2}^{2n-1,2n+2-i}(\BA)$, it is given by $\psi(tr(Lz))$. Then by \eqref{9.12},
\begin{multline}\label{9.14}
f_i^{\psi,L}(\xi)=
\int_{U^{(i)}(F)\backslash
	U^{(i)}(\BA)}\int_{E_i(F)\backslash E_i(\BA)}\xi(uv)\psi^{-1}_{E_i,L}(u)\psi^{-1}_{U^{(i)}}(v)dudv.
\end{multline}
Assume that the rank of $L$ is $\ell\geq 1$. Let $a\in \GL_{m-2r}(F)$ and $b\in \GL_r(F)$ be such that 
$$
aLb=\begin{pmatrix}I_\ell&0\\0&0\end{pmatrix}:=A_\ell\in M_{(m-2r)\times r}(F).
$$
Define 
$$
d_{a,b}=diag(a^{-1},...,a^{-1},b^*,...,b^*,I_{\lambda-2k(m-2r)},b,...,b,(a^{-1})^*,...,(a^{-1})^*),
$$
where $a^{-1}$ appears $2n-1$ times, and $b$ appears $i$ times. Then \eqref{9.14} becomes
\begin{multline}\label{9.15}
f_i^{\psi,L}(\xi)=\\
\int_{U^{(i)}(F)\backslash
	U^{(i)}(\BA)}\int_{E_i(F)\backslash E_i(\BA)}\xi(uvd_{a,b})\psi^{-1}_{E_i,A_\ell}(u)\psi_{U^{(i)}}^{-1}(v)dudv.
\end{multline}
Hence, we may assume that $L=A_\ell$. Note that $\ell\leq m-2r, r$. Denote $E^{(i)}=U^{(i)}E_i$. Denote by $\psi_{E^{(i)},\ell}$ the character of $E^{(i)}(\BA)$ obtained by the product of $\psi_{E_i,A_\ell}$ and $\psi_{U^{(i)}}$. Then, up to the right translation by $d_{a,b}$, \eqref{9.15} becomes
\begin{equation}\label{9.15.1}
f_i^{\psi,\ell}(\xi)=
\int_{E^{(i)}(F)\backslash
	E^{(i)}(\BA)}\xi(v)\psi^{-1}_{E^{(i)},\ell}(v)dv. 
\end{equation}
Thus, we want to prove that when $\ell$ is positive, $f_i^{\psi,\ell}$ is trivial on $A(\Delta(\tau,m)\gamma_\psi^{\epsilon)},\eta,k)$. The proof is similar to that of Theorem \ref{thm 6.1}, once we apply a conjugation by a Weyl element $\epsilon_0$, which we describe now. It is similar to the one in \eqref{6.1}.	
\begin{equation}\label{9.16}
\epsilon_0=\begin{pmatrix}A_1&A_2&0\\A_3&0&0\\0&A_4&0\\0&0&A_5\\0&A_6&A_7\end{pmatrix}.
\end{equation}
The block $A_1$ has the following form. It has $2n+1$ block rows,
each one of size $\ell$; the last two block rows are zero. It has
$2n-1$ block columns, each one of size $r$.
\begin{equation}\label{9.17}
A_1=\begin{pmatrix}a\\&a\\&&a\\&&&\cdots\\&&&&a\\0&&&\cdots
&0\\ 0&&&\cdots &0\end{pmatrix},\quad
a=\begin{pmatrix}I_\ell&0_{\ell\times (r-\ell)}\end{pmatrix}.
\end{equation}
The block $A_3$ has $2n-1$ block rows, each one of size $r-\ell$. It has the same block column division as $A_1$.
\begin{equation}\label{9.19}
A_3=\begin{pmatrix} b\\&b\\&&b\\&&&\cdots\\&&&&b\end{pmatrix},
\quad b=\begin{pmatrix}0_{(r-\ell)\times \ell}&I_{r-\ell}\end{pmatrix}.
\end{equation}
The block $A_4$ is as follows
\begin{equation}\label{9.18}
A_4=\begin{pmatrix}I_{(m-2r)(i-2)}\\&c\\&&c\\&&&I_{\lambda-2i(m-2r)}\\&&&&c'\\&&&&&c'\\
&&&&&&I_{(m-2r)(i-2)}\end{pmatrix},
\end{equation}
where 
$$
c=\begin{pmatrix}I_{m-2r-\ell}&0_{(m-2r-\ell)\times \ell}\end{pmatrix}, \quad c'=\begin{pmatrix}0_{(m-2r-\ell)\times \ell}&I_{m-2r-\ell}\end{pmatrix}.
$$
The block $A_2$ has the same block row division as $A_1$, and the same block column division as $A_4$. The first $2n-1$ block rows are all zero. The last two block rows of $A_2$ have the following form
\begin{equation}\label{9.20}
\begin{pmatrix}0&0&0&0&d&0&0\\0&0&0&0&0&d&0\end{pmatrix},
\end{equation}
where $d=(I_\ell,0_{\ell\times (m-2r-\ell)})$. The block column division in \eqref{9.20} is as that of $A_4$.
The blocks $A_5, A_6, A_7$ are already determined by the other blocks.

Conjugating inside \eqref{9.15.1} by $\epsilon_0$, we see that we need to prove that, for $\ell$ positive, the following integral is identically zero on $A(\Delta(\tau,m)\gamma_\psi^{\epsilon)},\eta,k)$,
\begin{equation}\label{9.21}
\varphi_{i,\ell}^\psi(\xi)=
\int_{\mathcal{M}_{i,\ell}(F)\backslash
	\mathcal{M}_{i,\ell}(\BA)}\xi(v)\psi^{-1}_{\mathcal{M}_{i,\ell}}(v)dv, 
\end{equation}
where $\mathcal{M}_{i,\ell}=\epsilon_0E^{(i)}\epsilon_0^{-1}$, and $\psi_{\mathcal{M}_{i,\ell}}$ is the character of $\mathcal{M}_{i,\ell}(\BA)$ defined by $\psi_{\mathcal{M}_{i,\ell}}(x)=\psi_{E^{(i)},\ell}(\epsilon_0^{-1}x\epsilon_0)$. Let us describe these. The subgroup $\mathcal{M}_{i,\ell}$ consists of elements of the form
\begin{equation}\label{9.22}
v=\begin{pmatrix}U_1&X_{1,2}&X_{1,3}&X_{1,4}&X_{1,5}\\Y_{2,1}&U_2&X_{2,3}&X_{2,4}&X'_{1,4}\\
Y_{3,1}&0&V&X'_{2,3}&X'_{1,3}\\0&0&0&U_2'&X'_{1,2}\\Y_{5,1}&0&Y'_{3,1}&Y'_{2,1}&U'_1\end{pmatrix}.
\end{equation}
We now describe the blocks in \eqref{9.22}. The block $U_1$ has $(2n+1)\times (2n+1)$ blocks, all of size $\ell$, and has the form
\begin{equation}\label{9.23}
U_1=\begin{pmatrix}I_\ell&U^1_1&*&\cdots&*&*\\
&I_\ell&U^1_2&\cdots&*&*\\
& &I_\ell&\cdots &* & * \\
& & & \cdots& \cdots&\cdots &\\
& & &       &I_\ell&U^1_{2n}\\
& & &       & &I_\ell\end{pmatrix}.
\end{equation}
The block $U_1'$ is of the same size as $U_1$ and has the form
\eqref{9.23}. The block $U_2$ has $(2n-1)\times (2n-1)$ blocks, all of size $r-\ell$, and has the form
\begin{equation}\label{9.24}
U_2=\begin{pmatrix}I_{r-\ell}&U^2_1&*&\cdots&*&*\\
&I_{r-\ell}&U^2_2&\cdots&*&*\\
& &I_{r-\ell}&\cdots &* & * \\
& & & \cdots& \cdots&\cdots &\\
& & &       &I_{r-\ell}&U^2_{2n-2}\\
& & &       & &I_{r-\ell}\end{pmatrix}.
\end{equation}
The block $U_2'$ is of the same size as $U_2$ and has the form
\eqref{9.24}. The block $V$ has the form
\begin{equation}\label{9.25}
V=\begin{pmatrix}I_{\mu(i-2)}&0&\ast&\ast&\ast&\ast&\ast\\&I_{\mu-\ell}&x&\ast&\ast&\ast&\ast\\&&I_{\mu-\ell}&0&0&\ast&\ast\\&&&I_{\lambda-2i\mu}&0&\ast&\ast\\&&&&I_{\mu-\ell}&x'&\ast\\&&&&&I_{\mu-\ell}&0\\&&&&&&I_{\mu(i-2)}\end{pmatrix},
\end{equation}
where we abbreviated $m-2r=\mu$. The block $X_{1,2}$ is a $(2n+1)\times (2n-1)$ matrix of $\ell\times (r-\ell)$ blocks. The last two block rows are zero. The first $2n-1$ block rows form a $(2n-1)\times (2n-1)$ matrix of $\ell\times (r-\ell)$ blocks, which has an upper triangular shape, with all the blocks along the diagonal being zero. (Thus, the last three block rows of $X_{1,2}$ are zero.) As before, we denote by $X_{1,2}^{j,t}$, the block of $X_{1,2}$ lying in position $(j,t)$. We denote by $\mathrm{X}_{1,2}^{j,t}$ the corresponding abelian unipotent subgroup, exactly as we did right after \eqref{6.9}. The matrix $X'_{1,2}$ has a dual shape. It is a $(2n-1)\times (2n+1)$ matrix of $(r-\ell)\times \ell$ blocks. The first two block columns are zero. The last $2n-1$ block columns form a $(2n-1)\times (2n-1)$ matrix of $(r-\ell)\times \ell$ blocks, which has an upper triangular shape, with all the blocks along the diagonal being zero. The block $X_{1,3}$ has $2n+1$ block rows, each one containing $\ell$ rows. It has seven block columns with sizes as for $V$. It has the following form,
\begin{equation}\label{9.26}
X_{1,3}=\begin{pmatrix}\ast&\ast&\ast&\ast&\ast&\ast&\ast\\&\vdots&&&&\vdots&&\\\ast&\ast&\ast&\ast&\ast&\ast&\ast\\0&0&0&0&\ast&\ast&\ast\\0&0&0&0&0&\ast&\ast\\0&0&0&0&0&0&0\end{pmatrix}.
\end{equation}
The matrix $X'_{1,3}$ has a dual form. It has $2n+1$ block columns, each one containing $\ell$ columns. It has seven block rows, with the same row division as that of $V$. Its first block column is zero. Its second block column is zero except the two first blocks, and its third block columns is zero, except the first three block columns. The matrix $X_{1,4}$ is a $(2n+1)\times (2n-1)$ matrix of $\ell\times (r-\ell)$ blocks, all of whose blocks are arbitrary, except the last two blocks in the first block column, which are zero. Thus, it has the form
\begin{equation}\label{9.27}
X_{1,4}=\begin{pmatrix}\ast&\ast&\cdots&\ast&\ast\\
&\vdots&&\vdots\\\ast&\ast&\cdots&\ast&\ast\\0&\ast&\cdots&\ast&\ast\\0&\ast&\cdots&\ast&\ast\end{pmatrix}.
\end{equation}
Similarly, the blocks of $X'_{1,4}$ are arbitrary, except the first two blocks in the last row, which are zero. Of course, the blocks above are arbitrary up to the requirement that the matrix $v$ in \eqref{9.22} lies in $H$. The matrix $X_{1,5}$ is a $(2n+1)\times (2n+1)$ matrix of $\ell\times \ell$ blocks of the form
\begin{equation}\label{9.28}
X_{1,5}=\begin{pmatrix}\ast&\ast&\ast&\ast&\cdots&\ast&\ast\\
&\vdots&&\vdots\\\ast&\ast&\ast&\ast&\cdots&\ast&\ast\\0&0&\ast&\ast&\cdots&\ast&\ast\\0&0&0&\ast&\cdots&\ast&\ast\\0&0&0&\ast&\cdots&\ast&\ast\end{pmatrix}.
\end{equation}
The matrix $Y_{2,1}$ is a $(2n-1)\times (2n+1)$ matrix of $(r-\ell)\times\ell$ blocks. Its last two block columns are arbitrary, and its first $2n-1$ block columns form a $(2n-1)\times (2n-1)$ matrix of $(r-\ell)\times \ell$ blocks, which has an upper triangular shape, with all the blocks along the diagonal being zero. We denote by $Y_{2,1}^{j,t}$, the block of $Y_{2,1}$ lying in position $(j,t)$. We denote, as above, the corresponding unipotent subgroup by $\mathrm{Y}_{2,1}^{j,t}$. The matrix $Y'_{2,1}$ has a dual shape. It is a $(2n+1)\times (2n-1)$ matrix of $\ell\times (r-\ell)$ blocks. The first two block rows are arbitrary. The last $2n-1$ block rows form a $(2n-1)\times (2n-1)$ matrix of $\ell\times (r-\ell)$ blocks, which has an upper triangular shape, with all the blocks along the diagonal being zero. The matrix $X_{2,3}$ has $2n-1$ block rows, each one containing $r-\ell$ rows. It has seven block columns, with the same division as that of $V$. All its blocks are arbitrary, except the first four blocks in the last row, which are zero.
\begin{equation}\label{9.29}
X_{2,3}=\begin{pmatrix}\ast&\ast&\ast&\ast&\ast&\ast&\ast\\&\vdots&&&&\vdots&&\\\ast&\ast&\ast&\ast&\ast&\ast&\ast\\0&0&0&0&\ast&\ast&\ast\end{pmatrix}.
\end{equation}
The matrix $X'_{2,3}$ has a dual shape. It has seven block rows, with the same division as that of $V$. It has $2n-1$ block columns, each one containing $r-\ell$ columns. It has the form
\begin{equation}\label{9.30}
X'_{2,3}=\begin{pmatrix}\ast&\ast&\cdots&\ast\\\ast&\ast&&\ast\\\ast&\ast&&\ast\\0&\ast&\cdots&\ast\\0&\ast&&\ast\\0&\ast&&\ast\\0&\ast&\cdots&\ast\end{pmatrix}.
\end{equation}
The matrix $X_{2,4}$ is arbitrary (as long as $v$ lies in $H$). The matrix $Y_{3,1}$ has seven block rows, with the same division as that of $V$. It has $2n+1$ block columns, each one containing $\ell$ columns. Its first $2n-1$ block columns are zero. Its last block column is such that its last two blocks are zero, and its one before last column is such that its last five block columns are zero.
\begin{equation}\label{9.31}
Y_{3,1}=\begin{pmatrix}0&&0&\ast&\ast\\0&\cdots&0&\ast&\ast\\
0&&0&0&\ast\\0&&0&0&\ast\\0&\cdots&0&0&\ast\\0&&0&0&0\\0&&0&0&0\end{pmatrix}.
\end{equation}
The matrix $Y'_{3,1}$ has a dual shape
\begin{equation}\label{9.32}
\begin{pmatrix} 0&0&\ast&\ast&\ast&\ast&\ast\\0&0&0&0&0&\ast&\ast\\0&0&0&0&0&0&0\\&\vdots&&&&\vdots\\0&0&0&0&0&0&0\end{pmatrix}.
\end{equation}
Finally, $Y_{5,1}$ as a $(2n+1)\times (2n+1)$ matrix of $\ell\times \ell$ blocks is such that all its blocks are zero except the last two blocks in the first block row, and the last block in the second block row.
\begin{equation}\label{9.33}
Y_{5,1}=\begin{pmatrix}0&\cdots&0&\ast&\ast\\0&&0&0&\ast\\0&&0&0&0\\\vdots&&&&\vdots\\0&\cdots&0&0&0\end{pmatrix}.
\end{equation}
The character $\psi_{\mathcal{M}_{i,\ell}}(v)$ of the element $v\in \mathcal{M}_{i,\ell}(\BA)$ in \eqref{9.22}, described above, is given by
\begin{multline}\label{9.34}
 \psi_{\mathcal{M}_{i,\ell}}(v)=\prod_{i=1}^{n-1}\psi(tr(U^1_i)+tr(U^2_i))\psi^{-1}(tr(U^1_{n-1+i})+tr(U^2_{n-1+i}))\\
 \cdot\psi(tr(U^1_{2n-1}))\psi^{-1}(tr(U^1_{2n}))\psi(tr(x)),
 \end{multline}
 where we used the notation in \eqref{9.23} - \eqref{9.25}.
 
 Now we perform a sequence of roots exchange, exactly as ip Prop. \ref{prop 7.1} and Prop. \ref{prop 7.2}. We assume that $\ell$ is positive. We start with the subgroup $\mathrm{Y}_{2,1}^{1,2}$ and exchange it with $\mathrm{X}_{1,2}^{1,1}$. Then we exchange $\mathrm{Y}_{2,1}^{2,3}$ with $\mathrm{X}_{1,2}^{2,2}$, and $\mathrm{Y}_{2,1}^{1,3}$ with $\mathrm{X}_{1,2}^{2,1}$, and so on. We exchange column $t$ of $\mathrm{Y}_{2,1}$, $2\leq t\leq 2n-1$, $\mathrm{Y}_{2,1}^{t-1,t}$, $\mathrm{Y}_{2,1}^{t-2,t}$,...,$\mathrm{Y}_{2,1}^{1,t}$, with row $t-1$ of $\mathrm{X}_{1,2}$, $\mathrm{X}_{1,2}^{t-1,t-1}$, $\mathrm{X}_{1,2}^{t-1,t-2}$,..., $\mathrm{X}_{1,2}^{t-1,1}$, in this order. The proof is as in Prop. \ref{prop 7.1} (even simpler). We conclude that $\varphi_{i,\ell}^\psi$ is identically zero on $A(\Delta(\tau,m)\gamma_\psi^{\epsilon)},\eta,k)$, if and only if the following integral is identically zero on $A(\Delta(\tau,m)\gamma_\psi^{\epsilon)},\eta,k)$, 
 \begin{equation}\label{9.35}
 \tilde{\varphi}_{i,\ell}^\psi(\xi)=
 \int_{\tilde{\mathcal{M}}_{i,\ell}(F)\backslash
 	\tilde{\mathcal{M}}_{i,\ell}(\BA)}\xi(v)\psi^{-1}_{\tilde{\mathcal{M}}_{i,\ell}}(v)dv, 
 \end{equation}
 where $\tilde{\mathcal{M}}_{i,\ell}$ is the subgroup of elements $v$ written as in \eqref{9.22}, where $X_{1,2}$ is such that its last three block rows are zero, and all of its other blocks are arbitrary, and $Y_{2,1}$ is such that its first $2n-1$ block columns are zero. (Similarly, in the dual block $X'_{1,2}$, the first three block columns are zero, and all other blocks are arbitrary, provided, of course, that $v$ lies in $H$, and in $Y'_{2,1}$, the last $2n-1$ block rows are zero.) The shape of all other blocks of $v$ remains as before. Note that $\tilde{\mathcal{M}}_{i,\ell}$ lies in the standard parabolic subgroup $Q^H_{\ell^{2n-1}}$.  The character $\psi_{\tilde{\mathcal{M}}_{i,\ell}}$ is still given by \eqref{9.34}. Thus, we want to prove that, for $\ell$ positive, $\tilde{\varphi}_{i,\ell}^\psi$ is identically zero on $A(\Delta(\tau,m)\gamma_\psi^{(\epsilon)},\eta,k)$. 
 
 Write 
 \begin{equation}\label{9.36}
 \tilde{\mathcal{M}}_{i,\ell}=\mathcal{M}'_{i,\ell}\rtimes\mathcal{M}''_{i,\ell},
 \end{equation}
 where $\mathcal{M}'_{i,\ell}$ is the intersection of $\tilde{\mathcal{M}}_{i,\ell}$ with the unipotent radical $U^H_{\ell^{2n-1}}$ of $Q^H_{\ell^{2n-1}}$, and $\mathcal{M}''_{i,\ell}$ is the intersection of $\tilde{\mathcal{M}}_{i,\ell}$ with the Levi part $M^H_{\ell^{2n-1}}$ of $Q^H_{\ell^{2n-1}}$. We have
 \begin{equation}\label{9.37}
 \tilde{\varphi}_{i,\ell}^\psi(\xi)=
 \int_{\mathcal{M}''_{i,\ell}(F)\backslash
 	\mathcal{M}''_{i,\ell}(\BA)}\int_{\mathcal{M}'_{i,\ell}(F)\backslash
 	\mathcal{M}'_{i,\ell}(\BA)}\xi(v'v'')\psi^{-1}_{\tilde{\mathcal{M}}_{i,\ell}}(v'v'')dv'dv''. 
 \end{equation}
 We consider first the inner $dv'$- integration, where we replace the right $v''$- translate of $\xi$ with $\xi$, that is
 \begin{equation}\label{9.38}
 (\varphi')_{i,\ell}^\psi(\xi)=
 \int_{\mathcal{M}'_{i,\ell}(F)\backslash
 	\mathcal{M}'_{i,\ell}(\BA)}\xi(v')\psi^{-1}_{\tilde{\mathcal{M}}_{i,\ell}}(v')dv'. 
 \end{equation}
 Note that
 \begin{equation}\label{9.37.1}
 \tilde{\varphi}_{i,\ell}^\psi(\xi)=
 \int_{\mathcal{M}''_{i,\ell}(F)\backslash
 	\mathcal{M}''_{i,\ell}(\BA)}(\varphi')_{i,\ell}^\psi(v''\cdot \xi)\psi^{-1}_{\tilde{\mathcal{M}}_{i,\ell}}(v'')dv'',
 \end{equation}
 where we denote for short, $v''\cdot\xi= A(\Delta(\tau,m)\gamma_\psi^{(\epsilon)},\eta,k)(v'')\xi$, i.e. the right translation of $\xi$ by $v''$.  Consider the matrix $x_\ell(y,c)$, that is
 \begin{equation}\label{9.39}
 x_\ell(y,c)=\begin{pmatrix}I_{(2n-2)\ell}\\&I_\ell&y&c\\&&I_{2(nm-\ell(2n-1))}&y'\\&&&I_\ell\\&&&&I_{(2n-2)\ell}\end{pmatrix},
 \end{equation}
 (where $c$ is such that $x_\ell(y,c)\in H_{2nm}$). Write $y$ in \eqref{9.39} in the form
 \begin{equation}\label{9.40}
 y=(y_1,...,y_5),
 \end{equation}
 where $y_1,y_2\in M_{\ell\times\ell}$, $y_3\in M_{\ell\times ((2n-1)(r-\ell)+\lambda-\mu i-2\ell)}$, $y_4\in M_{\ell\times (\mu i-2\ell+(2n-1)(r-\ell))}$, $y_5\in M_{\ell\times 2\ell}$. Let $y$ be as in \eqref{9.40}, with adele coordinates, and $c$, such that $x'(y,c)\in H_{2nm}(\BA)$. Denote
 $$
 (\varphi')_{i,\ell}^\psi(\xi)(y)=(\varphi')_{i,\ell}^\psi((x_\ell(y,c)\cdot\xi).
 $$
 There is no dependence on $c$. Note that $(\varphi')_{i,\ell}^\psi(\xi)=(\varphi')_{i,\ell}^\psi(\xi)(0)$. Assume that $y$ is such that $y_3=0$, $y_5=0$. Then
 \begin{equation}\label{9.41}
 (\varphi')_{i,\ell}^\psi(\xi)(y)=\psi(tr(y_1)) (\varphi')_{i,\ell}^\psi(\xi).
 \end{equation}
 Let us write the Fourier expansion of $(\varphi')_{i,\ell}^\psi(\xi)(0,0,y_3,0,y_5)$, along
  \begin{equation}\label{9.42}
 M_{\ell\times ((2n-1)(r-\ell)+\lambda-\mu i-2\ell)}(F)\backslash M_{\ell\times ((2n-1)(r-\ell)+\lambda-\mu i-2\ell)}(\BA)\times M_{\ell\times 2\ell}(F)\backslash M_{\ell\times 2\ell}(\BA).
 \end{equation} 
 A typical Fourier coefficient of the last function has the following form
 \begin{equation}\label{9.43}
 (\varphi')_{i,\ell}^{\psi;a,b}(\xi)=\int (\varphi')_{i,\ell}^\psi(\xi)(0,0,y_3,0,y_5)\psi^{-1}(tr(y_3a)+tr(y_5b))dy_3dy_5,
 \end{equation}
 where the integration is over \eqref{9.42}, and $a\in  M_{((2n-1)(r-\ell)+\lambda-\mu i-2\ell)\times\ell}(F)$, $b\in M_{2\ell\times \ell}(F)$.
 Using \eqref{9.41}, it is easy to see that the Fourier coefficient \eqref{9.43} corresponds 
in the sense of \cite{MW87}, to the nilpotent element $Y$ in $Lie(H_{2nm})(F)$
and to the one parameter subgroup $\varphi(s)$, where
\begin{equation}\label{9.44}
Y=\begin{pmatrix}\mathcal{U}\\
\mathcal{V}&0_{(2nm-2(n-1)\ell)\times (2nm-2(n-1)\ell)}\\0_{2(n-1)\ell\times
	2(n-1)\ell}&\mathcal{V}'&-\mathcal{U}\end{pmatrix},
\end{equation}
and
$$
\mathcal{U}=\begin{pmatrix}0\\I_\ell&0\\&I_\ell\\
&&\cdots\\&&&I_\ell&0\end{pmatrix}\in M_{2(n-1)\ell\times 2(n-1)\ell}(F);
$$
$$
\mathcal{V}=\begin{pmatrix}0_{\ell\times
	(2n-3)\ell}&I_\ell\\0&0_{\ell\times\ell}\\0&a\\0&0_{((2n-1)(r-\ell)+\mu i-2\ell)\times\ell}\\0&b\end{pmatrix};
$$
$$
\varphi(s)=\diag(s^{4n-2}I_\ell,s^{4n-4}I_\ell,...,s^2I_\ell, 
I_{2nm-(4n-2)\ell},s^{-2}I_\ell,...,s^{2-4n}I_\ell).
$$
Write
$$
a=\begin{pmatrix}a_1\\a_2\end{pmatrix},\ \ b=\begin{pmatrix}b_1\\b_2\end{pmatrix},
$$
where $a_1\in M_{((2n-1)(r-\ell)+\mu i-2\ell)\times \ell}(F)$, $a_2\in M_{(\lambda-2\mu i)\times\ell}(F)$, and $b_1,b_2\in M_{\ell\times\ell}(F)$. Let $a_1'=-w_\ell{}^ta_1w_{(2n-1)(r-\ell)+\mu i-2\ell}$, $a_2'=-w_\ell{}^ta_2J_{H_{\lambda-2\mu i}}$, $b'_1=-\delta_Hw_\ell{}^tb_1w_\ell$, $b'_2=-\delta_Hw_\ell{}^tb_2w_\ell$. 
The nilpotent orbit of $Y$ corresponds to the partition\\  $((4n-1)^{\ell_0},(2n)^{2(\ell-\ell_0)},1,1,...)$, where $\ell_0=rank(b_2-b_2'-a_2'a_2)$.
By Propositions \ref{prop 3.1}, \ref{prop 3.2}, we obtain that if $\ell_0>0$, then
$(\varphi')_{i,\ell}^{\psi;a,b}=0$, for all $\xi\in A(\Delta(\tau,m)\gamma_\psi^{(\epsilon)},\eta,k)$. Hence, for this Fourier coefficient to be nontrivial, we must have that $b_2-b_2'-a_2'a_2=0$. This condition is equivalent to the fact that
\begin{equation}\label{9.45}
z(a,b)=\begin{pmatrix}I_{(2n-1)\ell}\\&I_\ell\\&0&I_\ell\\&a_1&&I\\&a_2&&&I_{\lambda-2\mu i}\\&0&&&&I\\&b_1&&&&&I_\ell\\&b_2&b'_1&0&a_2'&a_1'&0&I_\ell\\&&&&&&&&I_{(2n-1)\ell}\end{pmatrix}\in H_{2nm}(F).
\end{equation}
Here, the fourth and the sixth identity blocks on the diagonal are each of size $(2n-1)(r-\ell)+\mu i-2\ell$. Thus, in this case, we can rewrite \eqref{9.43} as
\begin{equation}\label{9.46}
(\varphi')_{i,\ell}^{\psi;a,b}(\xi)=(\varphi')_{i,\ell}^{\psi;0,0}(z(a,b)^{-1}\cdot\xi)=\int (\varphi')_{i,\ell}^\psi(z(a,b)^{-1}\cdot \xi)(0,0,y_3,0,y_5)dy_3dy_5,
\end{equation} 
We then have
\begin{equation}\label{9.47} 
(\varphi')_{i,\ell}^\psi(\xi)=\sum_{z(a,b)\in H_{2nm}(F)}(\varphi')_{i,\ell}^{\psi;0,0}(z(a,b)^{-1}\cdot\xi).
\end{equation}
Let
\begin{equation}\label{9.48}
\zeta_\ell(y,c)=\begin{pmatrix}I_{(2n-1)\ell}\\&I_\ell&y&c\\&&I_{2n(m-2\ell)}&y'\\&&&I_\ell\\
&&&&I_{(2n-1)\ell}\end{pmatrix}\in H_{2nm}.
\end{equation}
We claim that, for all $\xi\in A(\Delta(\tau,m)\gamma_\psi^{(\epsilon)},\eta,k)$, and all $\zeta_\ell(y,c)\in H_{2nm}(\BA)$,
\begin{equation}\label{9.49}
(\varphi')_{i,\ell}^{\psi;0,0}(\zeta_\ell(y,c)\cdot\xi)=(\varphi')_{i,\ell}^\psi(\xi).
\end{equation} 
Again, the proof is by examining Fourier coefficients. First, we show the claim when $y=0$ and $c$ such that ${}^t(w_\ell c)=-\delta_H(w_\ell c)$. (In the beginning of the proof of Prop. \ref{prop 8.1}, we denoted the space of such $c$ by $S_\ell(\BA)$). We consider the Fourier expansion of the function $c\mapsto (\varphi')_{i,\ell}^{\psi;0,0}(\zeta_\ell(0,c)\cdot\xi)$. It is a function on $S_\ell(F)\backslash S_\ell(\BA)$. A typical Fourier coefficient has the form
\begin{equation}\label{9.50}
\int_{S_\ell(F)\backslash S_\ell(\BA)}(\varphi')_{i,\ell}^{\psi;0,0}(\zeta(0,c)\cdot\xi)\psi^{-1}(tr(cA))dc,
\end{equation}
where $A\in S_\ell(F)$. If $A$ is of rank $\ell_0>0$, then the Fourier coefficient \eqref{9.50} defines, via \eqref{9.46}, a Fourier coefficient on $\xi\in A(\Delta(\tau,m)\gamma_\psi^{(\epsilon)},\eta,k)$, which corresponds to a partition whose first $\ell_0$ terms are $((4n)^{\ell_0},...)$. By Propositions \ref{prop 3.1}, \ref{prop 3.2}, this Fourier coefficient must be zero. 
(Again, in case $H_{2nm}$ is symplectic, we also need to use Lemma 1.1 in \cite{GRS03}.) Hence, we must have $\ell_0=0$, that is $A=0$. Thus, $(\varphi')_{i,\ell}^{\psi;0,0}(\zeta_\ell(y,c)\cdot\xi)$ is independent of $c$. Next, we consider its Fourier expansion as a function of $y\in M_{\ell\times 2n(m-2\ell)}(F)\backslash M_{\ell\times 2n(m-2\ell)}(\BA)$. A typical Fourier coefficient of this function has the form
\begin{equation}\label{9.50.1}
\int_{M_{\ell\times 2n(m-2\ell)}(F)\backslash M_{\ell\times 2n(m-2\ell)}(\BA)}(\varphi')_{i,\ell}^{\psi;0,0}(\zeta_\ell(y,c)\cdot\xi)\psi^{-1}(tr(yB))dy,
\end{equation}
where $B\in M_{2n(m-2\ell)\times \ell}(F)$. If $B$ is of rank $\ell_0>0$, then the Fourier coefficient \eqref{9.50.1} defines, via \eqref{9.46}, a Fourier coefficient on $\xi\in A(\Delta(\tau,m)\gamma_\psi^{(\epsilon)},\eta,k)$, which corresponds to a partition whose first $\ell_0$ terms are $((4n+1)^{\ell_0},...)$, and, again, by Propositions \ref{prop 3.1}, \ref{prop 3.2}, this Fourier coefficient is zero. 

Using \eqref{9.49}, we now get that
\begin{equation}\label{9.51}
(\varphi')_{i,\ell}^{\psi;0,0}(\xi)=(\xi^{U_{\ell^{2n}}})^{\psi_{V_{\ell^{2n}}}},
\end{equation}
where $\xi^{U_{\ell^{2n}}}$ is the constant term of $\xi$ along $U_{\ell^{2n}}$, and the upper $\psi_{V_{\ell^{2n}}}$ indicates that we further take the Fourier coefficient along the unipotent radical $V_{\ell^{2n}}$ inside $\GL_{2\ell n}$, with respect to the character $\psi_{V_{\ell^{2n}}}$ (see \eqref{1.10.3.1}). From \eqref{9.37.1}, \eqref{9.47}, \eqref{9.51},
\begin{equation}\label{9.52}
\tilde{\varphi}_{i,\ell}^\psi(\xi)=
\int_{\mathcal{M}''_{i,\ell}(F)\backslash
	\mathcal{M}''_{i,\ell}(\BA)}\sum_{z(a,b)\in H_{2nm}(F)}(\xi^{U_{\ell^{2n}}})^{\psi_{V_{\ell^{2n}}}}(z(a,b)^{-1}v'')
\psi^{-1}_{\tilde{\mathcal{M}}_{i,\ell}}(v'')dv''.
\end{equation}
This integral is identically zero, since we may consider the subgroup of $	\mathcal{M}''_{i,\ell}(\BA)$ consisting of the elements \eqref{9.48} $v''(y_1)=\zeta_\ell(y_1,0,...,0)$, with $y_1\in M_\ell(\BA)$. We note that $v''(y_1)\in U_{\ell^{2n}}(\BA)$, and that $\psi_{\tilde{\mathcal{M}}_{i,\ell}}(v''(y_1))=\psi(tr(y_1))$. Now, it is a simple exercise to show that the \eqref{9.52} is identically zero. This proves that when $\ell$ is positive, $f_i^{\psi,\ell}$ (in \eqref{9.15.1}) is trivial on $A(\Delta(\tau,m)\gamma_\psi^{(\epsilon)},\eta,k)$. This completes the proof of Theorem \ref{thm 9.1}. 
 	 
\end{proof}

Going back to the function $f(\xi)(y)$ in \eqref{9.3}, Theorem \ref{thm 9.1} shows that for all $\xi\in A(\Delta(\tau,m)\gamma_\psi^{(\epsilon)},\eta,k)$, and all $y\in M_{r\times \lambda}(\BA)$,
\begin{multline}\label{9.54}
f(\xi)(y)=\\
\int_{U'_{{(m-2r)}^{n-1}}(F)\backslash
	U'_{{(m-2r)}^{n-1}}(\BA)}\int_{E_{n-1}(F)\backslash E_{n-1}(\BA)}\xi(uvx_r(y,c))\psi^{-1}_{E_{n-1}}(u)\\
\quad\quad\quad\psi^{-1}_{U'_{{(m-2r)}^{n-1}}}(v)dudv,
\end{multline}
and hence, in the notation of \eqref{9.10}, $f(\xi)(y)$ is independent of $y_{n+3},...,y_{2n+1}\in M_{r\times(m-2r)}(\BA)$, so that
\begin{equation}\label{9.54.1}
f(\xi)(y)=f(\xi)(y_1,...,y_{n+2},0,...,0).
\end{equation}
Recall the notation of \eqref{9.10.1}, 
\begin{equation}\label{9.55}
y^{(n+2)}=(0,...,0, y_n,...,y_{2n+1}),
\end{equation}
where $y_{n+3},...,y_{2n+1}\in M_{r\times(m-2r)}(\BA)$, $y_n,y_{n+2}\in M_{r\times [\frac{m}{2}]}(\BA)$, and\\
 $y_{n+1}\in M_{r\times 2([\frac{m+1}{2}]-r)}(\BA)$.

\section{A space of double cosets and further Fourier expansions}

Our next step is to consider the function $f(\xi)(y^{(n+2)})$. By \eqref{9.54.1}, we know that it is independent of $y_{n+3},...,y_{2n+1}$, and hence we may denote it by $f_n(\xi)(z)$, where $z=(y_n,y_{n+1},y_{n+2})$. As before, consider the Fourier expansion of this function. A character of $M_{r\times 2(m-r)}(F)\backslash M_{r\times 2(m-r)}(\BA)$ has the form $\psi(tr(Lz))$, where $L\in M_{2(m-r)\times r}(F)$, and the corresponding Fourier coefficient is 
\begin{equation}\label{9.56}
f_n^{\psi,L}(\xi)=\int_{M_{r\times 2(m-r)}(F)\backslash M_{r\times 2(m-r)}(\BA)}f_n(\xi)(z)\psi^{-1}(tr(Lz))dz.
\end{equation}
Consider the right action of $\GL_r(F)\times H_{2(m-r)}(F)$ on $M_{2(m-r)\times r}(F)$ given by $L\cdot (a,h)=h^{-1}La$. In the next lemma, we describe a set of representatives of the orbits of this action.
\begin{lem}\label{lem 9.3}
	1. Under the action of $\GL_r(F)\times \SO_{2(m-r)}(F)$, each $L\in
	M_{2(m-r)\times r}(F)$ is equivalent to a matrix of the form
	$$
	L_{\ell_1,\ell_2;\delta}=\begin{pmatrix}I_{\ell_1}&0&0\\0&I_{\ell_2}&0\\0&0&0\\0&\delta
	w_{\ell_2}&0\\0_{\ell_1\times \ell_1}&0&0\end{pmatrix},
	$$
	where $\delta=\frac{1}{2}diag (d_1,...,d_{\ell_2})$ is a diagonal $\ell_2\times\ell_2$ matrix, and $0\leq
	\ell=\ell_1+\ell_2\leq r$. The numbers $\ell_1, \ell_2$ are determined uniquely, and the class of the quadratic form $d_1x_1^2+d_2x_2^2+\cdots+d_{\ell_2}x_{\ell_2}^2$ is determined uniquely. \\
	2. Under the action of $\GL_r(F)\times \Sp_{2(m-r)}(F)$, each $L\in
	M_{2(m-r)\times r}(F)$ is equivalent to a matrix of the form
	$$
	L_{\ell_1,\ell_2}=\begin{pmatrix}I_{\ell_1}&0&0&0\\0&I_{\ell'_2}&0&0\\0&0&0&0\\0&0&I_{\ell'_2}&0\\0_{\ell_1\times \ell_1}&0&0&0\end{pmatrix}.
	$$
	where $0\leq \ell=\ell_1+2\ell'_2\leq r$, and $\ell_1, \ell_2=2\ell_2'$ are determined uniquely.
	
	In both cases, if $\ell_2=0$, then ignore the corresponding
	block columns. In this case, we will denote the matrix by
	$L_{\ell_1}=L_\ell$. 
\end{lem}
\begin{proof}
	Let $L\in M_{2(m-r)\times r}(F)$ be of rank $\ell$; $0\leq \ell\leq
	r$. Then we can find $a_1\in \GL_r(F)$, such that the first
	$\ell$ columns $v_1,...,v_\ell$ of $L_1=La_1$ are linearly
	independent, and its last $r-\ell$ columns are zero. Consider the
	Gram matrix of $v_1,...,v_\ell$. We can find $a_2 \in
	\GL_\ell(F)$, such that the Gram matrix of
	$(v_1,...,v_\ell)a_2=(u_1,...,u_\ell)$ is of the form
	$$
	\begin{pmatrix}0_{\ell_1\times\ell_1}&0_{\ell_1\times\ell_2}\\
	0_{\ell_2\times\ell_1}&S\end{pmatrix},
	$$
	where $\ell_1+\ell_2=\ell$, and the matrix $S$ is non-degenerate and symmetric (resp. anti-symmetric). In case $H$ is orthogonal, we may take $S$ to be diagonal, $S=\diag(d_1,...,d_{\ell_2})$, $d_i\in F^*$. In case $H$ is symplectic, $\ell_2=2\ell_2'$ must be even, and we may take $S=J_{2\ell_2}$. Let
	$$
	L_2=L_1\diag (a_2,I_{r-\ell}):=La.
	$$
	By Witt's theorem, we can
	find $h\in H_{2(m-r)}(F)$, such that $h^{-1}
	u_1,...,h^{-1} u_\ell$ are the first $\ell$ columns of
	$L_{\ell_1,\ell_2,\delta}$ (resp. $L_{\ell_1,\ell_2}$). Clearly, the rank $\ell$ of $L$ is determined uniquely by the orbit of $L$. When we consider the space spanned by the columns of $L$, $V_L$, as a space equipped with a the symmetric (resp. anti-symmetric) form given by restriction of the one on the column space with $2mn$ coordinates, then $\ell_1$ is the dimension of its radical, $V_L^0$, so that $\ell_1$ (and hence $\ell_2$) is determined uniquely by the orbit of $L$. The corresponding form on the space $V^0_L\backslash V_L$ is non-degenerate. Its isomorphism class, as a space with a non-degenerate symmetric (resp. anti-symmetric) form is determined uniquely by the orbit of $L$. (Of course, in case the form is anti-symmetric, there is only one such class.)
\end{proof}
In order to unify our notation, denote, in case $H$ is symplectic, $L_{\ell_1,\ell_2,I_{\ell_2}}=L_{\ell_1,\ell_2}$, so that in this case $\delta$ is the identity matrix. 
\begin{prop}\label{prop 9.4}
	Let $L\in M_{2(m-r)\times r}(F)$ be in the orbit of $L_{\ell_1,\ell_2,\delta}$. If $\ell_2>0$, then $f_n^{\psi,L}(\xi)=0$, for all $\xi\in A(\Delta(\tau,m)\gamma_\psi^{(\epsilon)},\eta,k)$.
	\end{prop}
\begin{proof}
For $z=(y_n,y_{n+1},y_{n+2})$, consider the matrices \eqref{9.2}
$$
e(y_n,y_{n+1},y_{n+2};c)=x_r((0_{r\times (m-2r)},...,0_{r\times (m-2r)},z,0_{r\times (m-2r)},...,0_{r\times (m-2r)}),c),
$$
where $0_{r\times (m-2r)}$ is repeated $n-1$ times on each side of $z$.	
It is enough to show that the following integral is zero for all $\xi\in A(\Delta(\tau,m)\gamma_\psi^{(\epsilon)},\eta,k)$.
\begin{equation}\label{9.57}
\int_{M_{r\times 2(m-r)}(F)\backslash M_{r\times 2(m-r)}(\BA)}\int_{E_{n-1}(F)\backslash E_{n-1}(\BA)}\xi(ue(z;c))\psi^{-1}_{E_{n-1}}(u)\psi^{-1}(tr(Lz))dudz.
\end{equation}	
This is an inner integration of $f_n^{\psi,L}$. Note that there is no dependence on $c$ in \eqref{9.57}. Let $a\in \GL_r(F)$, $\gamma\in H_{2nm}(F)$ be such that $L=\gamma L_{\ell_1,\ell_2;\delta}a^{-1}$.	Then the integral \eqref{9.57} is equal to 
\begin{multline}\label{9.58}
f_n^{\psi;\ell_1,\ell_2;\delta}(d_{a,\gamma}^{-1}\cdot\xi)=\\
\int_{M_{r\times 2(m-r)}(F)\backslash M_{r\times 2(m-r)}(\BA)}\int_{E_{n-1}(F)\backslash E_{n-1}(\BA)}\xi(ue(z;c)d^{-1}_{a,\gamma})\psi^{-1}_{E_{n-1}}(u)\\
\psi^{-1}(tr(L_{\ell_1,\ell_2;\delta}z))dudz,
\end{multline}	
where $d_{a,\gamma}=diag(a,...,a,\gamma,a^*,...,a^*)$, with $a$ repeated $2n-1$ times. Thus, we need to show that $f_n^{\psi;\ell_1,\ell_2;\delta}$ is trivial on $A(\Delta(\tau,m)\gamma_\psi^{(\epsilon)},\eta,k)$. Consider the subgroup of elements $x_r((y_1,...,y_{n-1},0...,0),0)$ and the function 
$$
f_n^{\psi;\ell_1,\ell_2;\delta}(y_1,...,y_{n-1})=
f_{n,\xi}^{\psi;\ell_1,\ell_2;\delta}(x_r((y_1,...,y_{n-1},0...,0),0)\cdot\xi).
$$
This is a smooth function on $M_{r\times (m-2r)(n-1)}(F)\backslash M_{r\times (m-2r)(n-1)}(\BA)$. A typical Fourier coefficient of this function is of the form
\begin{equation}\label{9.59}
f_{n,\xi}^{\psi;\ell_1,\ell_2;\delta;A}=\int_{M_{r\times (m-2r)(n-1)}(F)\backslash M_{r\times (m-2r)(n-1)}(\BA)}f_n^{\psi;\ell_1,\ell_2;\delta}(y)\psi^{-1}(tr(Ay))dy,
\end{equation}
where $A\in M_{(m-2r)(n-1)\times r}(F)$. Substituting \eqref{9.58}, we see that the Fourier coefficient \eqref{9.59} corresponds 
in the sense of \cite{MW87}, to the nilpotent element $Y$ in $Lie(H_{2nm})(F)$
and to the one parameter subgroup $\varphi(s)$, where
$$
Y=\begin{pmatrix}\mathcal{U}\\
\mathcal{V}&0_{(2nm-(2n-1)r)\times (2nm-(2n-1)r)}\\0_{(2n-1)r\times
	(2n-1)r}&\mathcal{V}'&-\mathcal{U}\end{pmatrix},
$$
and
$$
\mathcal{U}=\begin{pmatrix}0\\I_r&0\\&I_r\\
&&\cdots\\&&&I_r&0\end{pmatrix}\in M_{(2n-1)r\times (2n-1)\ell}(F);
$$
$$
\mathcal{V}=\begin{pmatrix}0_{(m-2r)(n-1)\times 2r(n-1)}&A\\0&L_{\ell_1,\ell_2;\delta}\\0&0_{(m-2r)(n-1)\times r}\\0&0\end{pmatrix};
$$
$$
\varphi(s)=\diag(s^{4n-2}I_r,s^{4n-4}I_r,...,s^2I_r, 
I_{2nm-(4n-2)r},s^{-2}I_r,...,s^{2-4n}I_r).
$$	
Now, one checks that, independently of $A$, the nilpotent orbit of $Y$ corresponds to a partition $(4n-1)^{\ell_2}, (2n)^{2\ell_1},...)$. By Propositions \ref{prop 3.1}, \ref{prop 3.2}, we get that if $\ell_2>0$, then the Fourier coefficient \eqref{9.59} is zero on $A(\Delta(\tau,m)\gamma_\psi^{(\epsilon)},\eta,k)$, for all $A$, and hence  $f_n^{\psi;\ell_1,\ell_2;\delta}$ is trivial on $A(\Delta(\tau,m)\gamma_\psi^{(\epsilon)},\eta,k)$.	
\end{proof}
Let $\mathcal{L}^0$ be the variety over $F$ of all matrices $L\in M_{2(m-r)\times r}$, such that the column space $V_L$ is an isotropic subspace with respect to the symmetric (resp. anti-symmetric form) defined by $w_{2(m-r)}$ (resp. $J_{2(m-r)}$). We conclude that
\begin{equation}\label{9.60}
f_n(\xi)=\sum_{L\in \mathcal{L}^0(F)}f_n^{\psi,L}(\xi).
\end{equation}
Let $0\leq \ell\leq r$. Denote by $\mathcal{L}^0_\ell$ the sub-variety of elements of $\mathcal{L}^0$ of rank $\ell$. Now, we consider the action of $H_{2[\frac{m}{2}]}(F)\times H_{2([\frac{m+1}{2}]-r)}(F)$ on $\mathcal{L}^0_\ell(F)$. We embed $H_{2[\frac{m}{2}]}\times H_{2([\frac{m+1}{2}]-r)}$ inside $H_{2(m-r)}$ by the following embedding $j_r$. Let $g=\begin{pmatrix}g_1&g_2\\g_3&g_4\end{pmatrix}\in H_{2[\frac{m}{2}]}$, where $g_1,...,g_4$ are $[\frac{m}{2}]\times [\frac{m}{2}]$ matrices, and let $h=\begin{pmatrix}h_1&h_2\\h_3&h_4\end{pmatrix}\in H_{2([\frac{m+1}{2}]-r)}$, where $h_1,...,h_4$ are $([\frac{m+1}{2}]-r)\times ([\frac{m+1}{2}]-r)$ matrices. Then
\begin{equation}\label{9.61}
j_r(g,h)=\begin{pmatrix}g_1&&&g_2\\&h_1&h_2\\&h_3&h_4\\g_3&&&g_4\end{pmatrix}.
\end{equation} 
Let us describe a set of representatives for 
$$
\mathcal{L}^0_\ell(F)/\GL_r(F) \times j_r(H_{2[\frac{m}{2}]}(F)\times H_{2([\frac{m+1}{2}]-r)}(F)).
$$
\begin{lem}\label{lem 9.5}
	The following elements form a set of representatives for the action of $\GL_r(F) \times j_r(H_{2[\frac{m}{2}](F)}\times H_{2([\frac{m+1}{2}]-r)}(F))$ on $\mathcal{L}^0_\ell(F)$.
	\begin{equation}\label{9.62}
L_{\ell,c,d,e}=\begin{pmatrix}0_{d\times d}&0&0&0& 0_{d\times (r-\ell)}\\0&I_{\ell-(c+d+e)}&0&0&0\\0_{([\frac{m+1}{2}]-r-\ell+c+e)\times d}&0&A_1&0&0\\0&0&0&I_c&0\\0_{([\frac{m}{2}]-[\frac{m+1}{2}]+r-c)\times d}&0&0&0&0\\I_d&0&0&0&0\\0&I_{\ell-(c+d+e)}&0&0&0\\0&0&A_1&0&0\\0_{([\frac{m+1}{2}]-r-\ell+c+e)\times d}&0&A_2&0&0\\0_{([\frac{m}{2}]-[\frac{m+1}{2}]+r+\ell-(c+e))\times d}&0&0&0&0\\

0&0&-A_2&0&0\\0_{(\ell-(c+e))\times d}&0&0&0&0
\end{pmatrix},
\end{equation}
where $c,d,e,\ell-(c+d+e), [\frac{m+1}{2}]-r-\ell+c+e, [\frac{m}{2}]-[\frac{m+1}{2}]+r-c\geq 0$, are determined uniquely, and $A=\begin{pmatrix}A_1\\A_2\end{pmatrix}$ is chosen as follows.\\
1. Assume that $H$ is orthogonal and $[\frac{m+1}{2}]-r-\ell+c\geq 0$, so that $e\leq [\frac{m+1}{2}]-r-\ell+c+e$. Then
$$
A=\begin{pmatrix}I_e\\0_{([\frac{m+1}{2}]-r-\ell+c)\times e}\\0_{([\frac{m+1}{2}]-r-\ell+c)\times e}\\\delta w_e\end{pmatrix},
$$
where $\delta=\frac{1}{2}diag(b_1,...,b_e)$, $b_i\in F^*$.\\
2. Assume that $H$ is orthogonal and $[\frac{m+1}{2}]-r-\ell+c< 0$. Then $ 2([\frac{m+1}{2}]-r-\ell+c)+e\geq 0$, and we may choose
$$
A=\begin{pmatrix}I_{2([\frac{m+1}{2}]-r-\ell+c)+e}&0&0\\0&I_{r+\ell-c-[\frac{m+1}{2}]}&0\\0&0&I_{r+\ell-c-[\frac{m+1}{2}]}\\\delta w_{2([\frac{m+1}{2}]-r-\ell+c)+e}&0&0\end{pmatrix},
$$
where $\delta=\frac{1}{2}diag(b_1,...,b_{2([\frac{m+1}{2}]-r-\ell+c)+e})$, $b_i\in F^*$. In the last two cases, the class of the quadratic form $b_1x_1^2+b_2x_2^2+\cdots$ is determined uniquely. \\
3. Assume that $H$ is symplectic. Then $e=2e'$ is even, $[\frac{m+1}{2}]-r-\ell+c+e'\geq 0$, and
$$
A=\begin{pmatrix}I_{e'}&0\\0_{([\frac{m+1}{2}]-r-\ell+c+e')\times e'}&0\\0_{([\frac{m+1}{2}]-r-\ell+c+e')\times e'}&0\\0&I_{e'}\end{pmatrix}.
$$
\end{lem}
\begin{proof}
Denote by $V_{2(m-r)}(F)$ the subspace spanned over $F$ by 
\begin{equation}\label{9.63}
\{e_{(n-1)m+r+1},e_{(n-1)m+r+2},...,e_{-(n-1)m-r-2},e_{-(n-1)m-r-1}\}. 
\end{equation}
Decompose $V_{2(m-r)}(F)=W_{2[\frac{m}{2}]}(F)\oplus U_{2([\frac{m+1}{2}]-r)}(F)$, where $W_{2[\frac{m}{2}]}(F)$ is spanned by the first $[\frac{m}{2}]$ vectors and the last $[\frac{m}{2}]$ vectors of \eqref{9.63}, and $U_{2([\frac{m+1}{2}]-r)}(F)$ is spanned by the middle $2([\frac{m+1}{2}]-r)$ vectors of \eqref{9.63}. We are interested in the orbits of the action of $H_{2[\frac{m}{2}]}(F)\times H_{2([\frac{m+1}{2}]-r)}(F)$ (via $j_r$) on the variety of the $\ell$-dimensional isotropic subspaces of $V_{2(m-r)}(F)$. Let $X$ be an $\ell$-dimensional isotropic subspace of $V_{2(m-r)}(F)$. Denote $c=c_X=dim(X\cap W_{2[\frac{m}{2}]}(F))$, $d=d_X=dim(X\cap U_{2([\frac{m+1}{2}]-r)}(F))$. Clearly, $c,d$ are invariants of the orbit of $X$. Let $B_1=\{x_1,...,x_c\}$, $B_2=\{x_{c+1},...,x_{c+d}\}$ be bases of $X\cap W_{2[\frac{m}{2}]}(F)$, $X\cap U_{2([\frac{m+1}{2}]-r)}(F)$ respectively. Let $B_{-1}=\{x_{-c},x_{-c+1},...,x_{-1}\}$ be a dual set to $B_1$, spanning an isotropic subspace of $W_{2[\frac{m}{2}]}(F)$ (i.e. $(x_i,x_{-j})=\delta_{i,j}$, $(x_{-i},x_{-j})=0$, for $1\leq i,j\leq c$). Similarly, let $B_{-2}=\{x_{-c-d},...,x_{-c-1}\}$ be a dual set to $B_2$, spanning an isotropic subspace of $U_{2([\frac{m+1}{2}]-r)}(F)$. Let $W'$ be the ortho-complement of $X\cap W_{2[\frac{m}{2}]}(F)+Span(B_{-1})$ inside $W_{2[\frac{m}{2}]}(F)$, and let $U'$ be the ortho-complement of $X\cap U_{2([\frac{m+1}{2}]-r)}(F)+Span(B_{-2})$ inside $U_{2([\frac{m+1}{2}]-r)}(F)$. Let $B'=\{w'_i+u'_i+f_i\  |\ c+d+1\leq i\leq \ell \}$ be a completion of $B_1\cup B_2$ to a basis of $X$, where $w'_i\in W'$, $u'_i\in U'$, $f_i\in Span(B_{-1}\cup B_{-2})$. It is easy to see that we must have $f_i=0$, for all $c+d+1\leq i\leq \ell$, and that $B_3=\{w'_{c+d+1},...,w'_\ell\}$, $B_5=\{u'_{c+d+1},...,u'_\ell\}$ are linearly independent. Note that $(w'_i,w'_j)=-(u'_i,u'_j)$, for all $c+d+1\leq i\leq \ell$, so that the Gram matrix of $B_3$ is the negative of the Gram matrix of $B_5$. Let $B_4$ (resp. $B_6$) be a completion of $B_3$ (resp. of $B_5$) to a basis of $W'$ (resp. of $U'$). Let $e$ be the rank of the Gram matrix of $B_3$. Assume that $H$ is orthogonal. We may choose $B_3$, such that 
$$
Gram(w'_{c+d+1},...w'_\ell)=diag(b_1,...,b_e,0,...,0),\ b_i\in F^*.
$$
The elements $b_1,...,b_e$ are determined up to the action of $\GL_{\ell-d-c}(F)$ on the Gram matrix, and $0$ is repeated $\ell-(c+d+e)$ times. Now, we may choose $B_4$, such that
$$
Gram(B_4)=diag(0,...,0,a_1,...,a_{\tilde{e}}),\ a_i\in F^*,
$$
where $0$ is repeated $\ell-(c+d+e)$ times,
$$
Gram(B_3\cup B_4)= diag(b_1,...,b_e,w_{2(\ell-(c+d+e))},a_1,...,a_{\tilde{e}}),
$$
and the class of the quadratic form defined by $diag(b_1,...,b_e,a_1,...,a_{\tilde{e}})$ is the same as that of the quadratic form defined by $w_{2([\frac{m}{2}]-(\ell-d-e))}$.
Assume that $H$ is symplectic. Then $e=2e'$ is even, and we may then choose $B_3$, so that
$$
Gram(w'_{c+d+1},...w'_\ell)=\begin{pmatrix}J_{2e'}&0\\0&0\end{pmatrix}.
$$
Similarly, we may choose $B_4$, such that
$$
Gram(B_4)=\begin{pmatrix}0&0\\0&J_{2\tilde{e'}}\end{pmatrix},
$$
and
$$
Gram(B_3\cup B_4)=diag(J_{2e'},J_{2(\frac{m}{2}-(\ell-d-e))},J_{2\tilde{e'}}).
$$
Finally, we can determine $Gram(B_6)$ from $Gram (B_5)$ in a similar way. Recall that $Gram(B_5)=-Gram(B_3)$.
Note that $B_1\cup B_3\cup B_4\cup B_{-1}$ is a basis of $W_{2[\frac{m}{2}]}(F)$, $B_2\cup B_5\cup B_6\cup B_{-2}$ is a basis of $U_{2([\frac{m+1}{2}]-r)}(F)$, and $B_1\cup B_2\cup B'_{3,5}$ is a basis of $X$, where
$$
B'_{3,5}=\{w'_{c+d+1}+u'_{c+d+1},...,w'_\ell+u'_\ell\}.
$$  
Denote $B_i=B_i^X$, $i=\pm 1, \pm 2, 3,4,5,6$. Let $Y$ be another $\ell$-dimensional isotropic subspace of $V_{2(m-r)}(F)$, such that $c_Y=c_X$, $d_Y=d_X$. Find bases $B_i^Y$, as above, for $i=\pm 1, \pm 2, 3,4,5,6$. Note that $|B_i^X|=|B_i^Y|$, for $i=\pm 1, \pm 2, 3,4,5,6$. Assume that the Gram matrices of $B_3^X$ and $B_3^Y$ are equivalent under the $\GL_{\ell-(c+d)}(F)$-action. Then we may assume that they are equal, and what we explained before shows that we can find a linear $T$ isomorphism of $V_{2(m-r)}(F)$, which sends the bases $B_i^X$ to $B_i^Y$, in such a way that the Gram matrix of $T(B^X_1\cup B^X_3\cup B^X_4\cup B^X_{-1})$ (resp. $T(B^X_2\cup B^X_5\cup B^X_6\cup B^X_{-2})$) is equal to the Gram matrix of $B^Y_1\cup B^Y_3\cup B^Y_4\cup B^Y_{-1}$ (resp. $B^Y_2\cup B^Y_5\cup B^Y_6\cup B^Y_{-2}$). Thus, $T\in H_{2[\frac{m}{2}]}(F)\times H_{2([\frac{m+1}{2}]-r)}(F)$, and $T(X)=Y$, that is $Y$ is in the orbit of $X$ under the action of $H_{2[\frac{m}{2}]}(F)\times H_{2([\frac{m+1}{2}]-r)}(F)$. Let us describe this in an invariant way. Given the isotropic subspace $X$, consider the quotient $X\cap W_{2[\frac{m}{2}]}(F)+X\cap U_{2([\frac{m+1}{2}]-r)}(F)\backslash X$. Let $p$ be the projection of this quotient to $X\cap W_{2[\frac{m}{2}]}(F)\backslash W_{2[\frac{m}{2}]}(F)$, deduced from the restriction to $X$ of projection of $V_{2(m-r)}(F)$ onto $W_{2[\frac{m}{2}]}(F)$. Note that $p$ is injective. Let $W_X$ denote the image of $p$, and we consider the isometry class of $Rad(W_X)\backslash W_X$. When $H$ is symplectic, this class is determined by the dimension of of $Rad(W_X)\backslash W_X$. What we proved before is that the orbit of $X$ under the action of $H_{2[\frac{m}{2}]}(F)\times H_{2([\frac{m+1}{2}]-r)}(F)$ is uniquely determined by $dim (X\cap W_{2[\frac{m}{2}]}(F))$, $dim (X\cap U_{2([\frac{m+1}{2}]-r)}(F))$ and the isometry class of $Rad(W_X)\backslash W_X$. Now the assertions of the lemma are clear.	
\end{proof}
Let $0\leq \ell\leq r$, and let $L\in\mathcal{L}^0_\ell(F)$. In the notation of Lemma \ref{lem 9.5}, let $\alpha\in \GL_r(F)$, $\beta\in H_{2[\frac{m}{2}]}(F)$, $\gamma\in H_{2([\frac{m+1}{2}]-r)}(F)$, such that $L=j_r(\beta,\gamma)^{-1}\cdot L_{\ell,c,d,e}\cdot \alpha$. Assume that $m=2m'$ is even. Then $2([\frac{m+1}{2}]-r)=m-2r$. Let
\begin{equation}\label{9.64.1}
d_{\alpha,\gamma,\beta,\gamma}=[diag(\alpha,...,\alpha,\gamma,...,\gamma,j_r(\beta,\gamma),\gamma^*,...,\gamma^*,\alpha^*,...,\alpha^*)]^{-1}.
\end{equation}
Here $\alpha$ and $\alpha^*$ are repeated each $2n-1$ times, and $\gamma$, $\gamma^*$ are repeated each $n-1$ times. We then have (see \eqref{9.54}-- \eqref{9.56})
\begin{equation}\label{9.64}
f_n^{\psi,L}(\xi)=f_n^{\psi,L_{\ell,c,d,e}}(d_{\alpha,\gamma,\beta,\gamma}\cdot\xi).
\end{equation}
Note that conjugation by $d_{\alpha,\gamma,\beta,\gamma}$ inside the integral \eqref{9.54} preserves the character $\psi_{U'_{(m-2r)^{n-1}}}$. Assume that $m=2m'-1$. Then \eqref{9.64} is not correct, since the conjugation above by $d_{\alpha,\gamma,\beta,\gamma}$ takes $\psi_{U'_{(m-2r)^{n-1}}}$, in the notation of \eqref{6.7.1}, to the character
\begin{equation}\label{9.65}
	\psi(tr(S_1+\cdots+S_{n-2}))\psi(tr(S_{n-1}\gamma^{-1}A'_H)\gamma).
\end{equation}
See \eqref{6.7.2}. Note that in this case $2([\frac{m+1}{2}]-r)=m-2r+1$. Let $\eta\in \GL_{m-2r}(F)$. Denote
\begin{equation}\label{9.65.1}
d_{\alpha,\eta,\beta,\gamma}=[diag(\alpha,...,\alpha,\eta,...,\eta,j_r(\beta,\gamma),\eta^*,...,\eta^*,\alpha^*,...,\alpha^*)]^{-1},
\end{equation}
$\alpha$ and $\alpha^*$ are repeated each $2n-1$ times, and $\eta$, $\eta^*$ are repeated each $n-1$ times. Conjugation as above by $d_{\alpha,\eta,\beta,\gamma}$ takes $\psi_{U'_{(m-2r)^{n-1}}}$ to the character of the form \eqref{9.65}, with $\gamma^{-1}A'_H\gamma$ replaced by $\gamma^{-1}A'_H\eta$.
\begin{lem}\label{lem 9.6}
Let $m=2m'-1$ be odd. Given $\gamma\in \SO_{m-2r+1}(F)$, there are $\eta\in \GL_{m-2r}(F)$, an integer $0\leq t_\gamma\leq m'-r$, and a row vector $b_\gamma$, such that
$$
\gamma^{-1}A'_H\eta=\begin{pmatrix}I_{m'-r-t_\gamma-1}&0\\0&b_\gamma\\0&I_{m'-r+t_\gamma}\end{pmatrix}:=A'_\gamma.
$$ 
\end{lem}
\begin{proof}
The rank of $\gamma^{-1}A'_H$ is $m-2r$. The bottom row of this matrix cannot be zero. Otherwise, we get that the last row of $\gamma^{-1}$ has the form $(0,a_1,a_2,0)$, where $0$ stands for the zero vector with $m'-r-1$ coordinates, and $a_1,a_2\in F$ satisfy $a_1+\frac{1}{2}a_2=0$. Since $\gamma\in \SO_{m-2r+1}(F)$, we also have $2a_1a_2=0$, and then $a_1=a_2=0$, which is impossible. We conclude that there is $\eta_1\in \GL_{m-2r}(F)$, such that the last row of $\gamma^{-1}A'_H\eta_1$ is $(0,...,0,1)$. Let $1\leq m'-r-1$. Assume, by induction, that there is $\eta_\nu\in \GL_{m-2r}(F)$, such that the last $\nu$ rows of $\gamma^{-1}A'_H\eta_\nu$ are $(0_{\nu\times (m-2r-\nu)}, I_\nu)$. We claim that the $\nu+1$-th row from the bottom of $\gamma^{-1}A'_H\eta_\nu$ is not of the form $(0,...,0,x_\nu,...,x_1)$, where $x_1,...,x_\nu\in F$. Assume otherwise, and denote the last $\nu$ rows of $\gamma^{-1}$ by $(\gamma_{i,1},a_i,b_i,\gamma_{i,2})$, where $\gamma_{i,1}, \gamma_{i,2}$ are row vectors with $m'-r-1$ coordinates, and $a_i,b_i\in F$. We let the last row have index 1, the row before last- index 2, and so on. We then have that the $(\nu+1)$-th row from the bottom of $\gamma^{-1}A'_H$ has the form 
$$
(\sum_{i=1}^\nu x_i\gamma_{i,1},a_{\nu+1},b_{\nu+1}, \sum_{i=1}^\nu x_i\gamma_{i,2}),
$$  
and 
\begin{equation}\label{9.66}
a_{\nu+1}+\frac{1}{2}b_{\nu+1}=\sum_{i=1}^\nu x_i(a_i+\frac{1}{2}b_i).
\end{equation} 
Multiply $\gamma^{-1}$ from the left by the following matrix, lying in $\SO_{2(m-r)}(F)$,
$$
diag(\begin{pmatrix}1&&&x_1\\&\ddots\\&&&x_\nu\\&&&1\end{pmatrix},I_{2(m'-r-\nu)},\begin{pmatrix}1&-x_\nu&\cdots&-x_1\\&1\\&&\ddots\\&&&1\end{pmatrix}).
$$
We get a matrix in $\SO_{2(m-r)}(F)$, whose $(\nu+1)$-th row from bottom is
$$
(0,a_{\nu+1}-\sum_{i=1}^\nu x_ia_i,b_{\nu+1}-\sum_{i=1}^\nu x_ib_i,0).
$$
Since this a row of a matrix in $\SO_{2(m'-r)}(F)$, we get that the product of the two middle coordinates is zero. By \eqref{9.66}, they are both zero, and this is impossible. By induction, there is $\eta_{m'-r}\in \GL_{m-2r}(F)$, such that the last $m'-r$ rows of $\gamma^{-1}A'_H\eta_{m'-r}$ are $(0_{(m'-r)\times (m'-r-1)},I_{m'-r})$. Thus, the reduced column-echelon form of $\gamma^{_1}A'_H$ is as in the statement of the lemma.
\end{proof} 
Assume that $m=2m'-1$ is odd, and let $L=j_r(\beta,\gamma)^{-1}L_{\ell,c,d,e}\alpha$. Assume also that $\gamma^{-1}A'_H\eta=A'_\gamma$. Then we get
\begin{multline}\label{9.67}
f_n^{\psi,L}(\xi)=\int_{M_{r\times 2(m-r)}(F)\backslash M_{r\times 2(m-r)}(\BA)}\\
\int_{U'_{{(m-2r)}^{n-1}}(F)\backslash
	U'_{{(m-2r)}^{n-1}}(\BA)}\int_{E_{n-1}(F)\backslash E_{n-1}(\BA)}\xi(uvx_zd_{\alpha,\eta,\beta,\gamma})\psi^{-1}_{E_{n-1}}(u)\\
\quad\quad\quad\psi^{-1}_{U'_{(m-2r)^{n-1}},A'_\gamma}(v)\psi^{-1}(tr(L_{\ell,c,d,e}\cdot z))dudv dz\\
:=f_n^{\psi,L_{\ell,c,d,e},A'_\gamma}(d_{\alpha,\eta,\beta,\gamma}\cdot\xi).
\end{multline}
Here, for $z=(y_n,y_{n+1},y_{n+2})$ (see \eqref{9.55}), $x_z$ is any element of the form $x(y^{(n+2)},c)$. The character $\psi_{{U'_{{(m-2r)}^{n-1}},A'_\gamma}}$ of $U'_{(m-2r)^{n-1}}(\BA)$ is defined, in the notation of \eqref{6.5.1}, \eqref{6.5.2},
\begin{equation}\label{9.68}
\psi_{U'_{(m-2r)^{n-1}},A'_\gamma}(v)=\psi(tr(S_1+\cdots+S_{n-2}))\psi(tr(S_{n-1}A'_\gamma)).
\end{equation}
\begin{thm}\label{thm 9.7}
Let $L\in \mathcal{L}_\ell^0(F)$ be in the orbit of $L_{\ell,c,d,e}$. Let $\ell'=\ell-c$ in Case 1 of Lemma \ref{lem 9.5}, let $\ell'=[\frac{m+1}{2}]-r$ in Case 2 of Lemma \ref{lem 9.5}, and let $\ell'=\ell-c-e'$ in Case 3 of Lemma \ref{lem 9.5}. If $\ell'>0$, then $f_n^{\psi,L}(\xi)=0$, for all $\xi\in A(\Delta(\tau,m)\gamma_\psi^{(\epsilon)},\eta,k)$. 
\end{thm}
\begin{proof}
By \eqref{9.64} and \eqref{9.67}, we need to show that for all $\xi\in A(\Delta(\tau,m)\gamma_\psi^{(\epsilon)},\eta,k)$, $f_n^{\psi,L_{\ell,c,d,e}}(\xi)=0$, when $m$ is even, and $f_n^{\psi,L_{\ell,c,d,e},A'_\gamma}(\xi)=0$, when $m$ is odd. Let us unify our notation, and let $A'_\gamma=I_{m-2r}$, when $m=2m'$ is even. Denote then $f_n^{\psi,L_{\ell,c,d,e}}=f_n^{\psi,L_{\ell,c,d,e},I_{m-2r}}$. In both cases, the integral expression of $f_n^{\psi,L_{\ell,c,d,e},A'_\gamma}(\xi)$ is given by the right hand side of \eqref{9.67}, with $d_{\alpha,\eta,\beta,\gamma}\cdot\xi$ replaced by $\xi$. Consider the inner integration
\begin{multline}\label{9.69}
\tilde{f}_n^{\psi,L_{\ell,c,d,e},A'_\gamma}(\xi)=\int_{M_{r\times 2(m-r)}(F)\backslash M_{r\times 2(m-r)}(\BA)}\\
\int_{U'_{(m-2r)(n-2),m-2r}(F)\backslash
	U'_{(m-2r)(n-2),m-2r}(\BA)}\int_{E_{n-1}(F)\backslash E_{n-1}(\BA)}\xi(uvx_z)\psi^{-1}_{E_{n-1}}(u)\\
\quad\quad\quad\psi^{-1}_{U'_{(m-2r)(n-2),m-2r},A'_\gamma}(v)\psi^{-1}(tr(L_{\ell,c,d,e}\cdot z))dudv dz,
\end{multline}
where $\psi_{U'_{(m-2r)(n-2),m-2r},A'_\gamma}$ is the restriction of $\psi_{U'_{(m-2r)^{n-1}},A'_\gamma}$ to\\ 
$U'_{(m-2r)(n-2),m-2r}(\BA)$ (the unipotent radical of the parabolic subgroup\\ $Q_{(m-2r)(n-2),m-2r}^{H_{2n(m-2r)+2r}}$).
Denote by $E^{(n+2)}$ the subgroup generated by $E_{n-1}$, the elements $x_z$ and $U'_{(m-2r)(n-2),m-2r}$. Note that the subgroup generated by $E_{n-1}$ and the elements $x_z$ is the subgroup $E_{n+2}$. Denote, for short, by $\psi_{E^{(n+2)}}$ the character of $E^{(n+2)}(\BA)$, such that, in the notation of the integrand in \eqref{9.69},
$$
\psi_{E^{(n+2)}}(uvx_z)=\psi_{E_{n-1}}(u)\psi_{U'_{(m-2r)(n-2),m-2r},A'_\gamma}(v)\psi(tr(L_{\ell,c,d,e}\cdot z)).
$$
Then \eqref{9.69} is
\begin{equation}\label{9.69.1}
\int_{E^{(n+2)}(F)\backslash E^{(n+2)}(\BA)}\xi(v)\psi^{-1}_{E^{(n+2)}}(v)dv.
\end{equation}
It suffices to show that the integral \eqref{9.69.1} is zero on $A(\Delta(\tau,m)\gamma_\psi^{(\epsilon)},\eta,k)$.  The proof is similar to that of Theorem \ref{thm 6.1}, and Theorem \ref{thm 9.1}. We first apply a conjugation by a Weyl element $\epsilon'_0$, which we describe now. It is similar to the ones in \eqref{6.1}, \eqref{9.16}.
Consider the following Weyl element	
\begin{equation}\label{9.70}
\epsilon'_0=\begin{pmatrix}A'_1&A'_2&0\\A'_3&0&0\\0&A'_4&0\\0&0&A'_5\\0&A'_6&A'_7\end{pmatrix}.
\end{equation}
The block $A'_1$ (resp.$A'_3$) has the same form as \eqref{9.17} (resp. \eqref{9.19}), with $\ell'$ instead of $\ell$. The block $A'_4$ is as follows
\begin{equation}\label{9.71}
A'_4=diag(I_{(m-2r)(n-2)},\begin{pmatrix}c_1\\&c_2\\&&I_{2([\frac{m+1}{2}]-r-\ell')}\\&&&c'_2\\&&&&c'_1\end{pmatrix},I_{(m-2r)(n-2)}),
\end{equation}
where $c_1=(I_{m-2r-\ell'},0_{(m-2r-\ell')\times\ell'})$, $c_2=(I_{[\frac{m}{2}]},0_{[\frac{m}{2}]\times\ell'})$, $c'_2=(0_{[\frac{m}{2}]\times\ell'},I_{[\frac{m}{2}]})$, $c'_1=(0_{(m-2r-\ell')\times\ell'},I_{m-2r-\ell'})$.
The block $A'_2$ has the same block row division as $A'_1$, and the same block column division as $A'_4$. The first $2n-1$ block rows are all zero. The last two block rows of $A'_2$ have the following form
\begin{equation}\label{9.72}
\begin{pmatrix}0&0&d_1&0&0&0&0\\
0&0&0&0&0&d_2&0\end{pmatrix},
\end{equation}
where $d_1=(0_{\ell'\times [\frac{m}{2}]},I_{\ell'})$, $d_2=(I_{\ell'},0_{\ell'\times (m-2r-\ell')})$. The block column division in \eqref{9.72} is as that of $A'_4$.
The blocks $A'_5, A'_6, A'_7$ are already determined by the other blocks.

As in \eqref{9.21}, we conjugate inside \eqref{9.69.1} by $\epsilon'_0$. Then we need to prove that, for $\ell'>0$, the following integral is identically zero on $A(\Delta(\tau,m)\gamma_\psi^{\epsilon)},\eta,k)$,
\begin{equation}\label{9.73}
\varphi_{n+2}^{\psi,\ell,c,d,e,\gamma}(\xi)=
\int_{\mathcal{M}_{n+2,\ell'}(F)\backslash
	\mathcal{M}_{n+2,\ell'}(\BA)}\xi(v)\psi^{-1}_{\mathcal{M}_{n+2,\ell'}}(v)dv, 
\end{equation}
where $\mathcal{M}_{n+2,\ell'}=\epsilon'_0E^{(n+2)}(\epsilon'_0)^{-1}$, and $\psi_{\mathcal{M}_{n+2,\ell'}}$ is the character of $\mathcal{M}_{n+2,\ell'}(\BA)$ defined by $\psi_{\mathcal{M}_{n+2,\ell'}}(x)=\psi_{E^{(n+2)},\ell}((\epsilon'_0)^{-1}x\epsilon'_0)$. The subgroup $\mathcal{M}_{n+2,\ell'}$ consists of elements of the form \eqref{9.22}, with the blocks described as follows. The block $U_1$ (resp. $U_2$) has the form \eqref{9.23} (resp. \eqref{9.24}) with $\ell'$ instead of $\ell$, and similarly with $U_1'$ (resp. $U'_2$). The block $V$ has the form
\begin{equation}\label{9.74}
V=\begin{pmatrix}I_{\mu(n-2)}&0&\ast&\ast&\ast&\ast&\ast\\&I_{\mu-\ell'}&\ast&v_{2,4}&\ast&\ast&\ast\\&&I_{[\frac{m}{2}]}&0&0&\ast&\ast\\&&&I_{2([\frac{m+1}{2}]-r-\ell')}&0&v'_{2,4}&\ast\\&&&&I_{[\frac{m}{2}]}&\ast&\ast\\&&&&&I_{\mu-\ell'}&0\\&&&&&&I_{\mu(n-2)}\end{pmatrix}.
\end{equation}
Recall that sometimes we abbreviate $m-2r=\mu$. The blocks $X_{1,2}$, $X'_{1,2}$, $X_{1,3}$, $X'_{1,3}$ have the same shape as the shape described right after \eqref{9.25} and \eqref{9.26}, with $\ell'$ instead of $\ell$. For the description of the character $\psi_{\mathcal{M}_{n+2,\ell'}}$, we need to specify some blocks of $X_{1,3}$. It has $2n+1$ block rows, each one containing $\ell'$ rows. It has seven block columns with sizes as for $V$. It has the following form,
\begin{equation}\label{9.75}
X_{1,3}=\begin{pmatrix}\ast&\ast&\ast&\ast&\ast&\ast&\ast\\&\vdots&&&&\vdots&&\\\ast&\ast&\ast&\ast&\ast&\ast&\ast\\0&0&(x_{1,3})_{2n-1,3}&(x_{1,3})_{2n-1,4}&(x_{1,3})_{2n-1,5}&\ast&\ast\\0&0&0&0&0&(x_{1,3})_{2n,6}&\ast\\0&0&0&0&0&0&0\end{pmatrix}.
\end{equation}
The matrix $X_{1,4}$ is a $(2n+1)\times (2n-1)$ matrix of $\ell'\times (r-\ell')$ blocks, all of whose blocks are arbitrary, except the last block in the first block column, which is zero. Thus, it has the form
\begin{equation}\label{9.76}
X_{1,4}=\begin{pmatrix}\ast&\ast&\cdots&\ast&\ast\\
&\vdots&&\vdots\\\ast&\ast&\cdots&\ast&\ast\\(x_{1,4})_{2n,1}&\ast&\cdots&\ast&\ast\\0&\ast&\cdots&\ast&\ast\end{pmatrix}.
\end{equation}
Similarly, the blocks of $X'_{1,4}$ are arbitrary, except the first block in the last row, which is zero. As always, the blocks above are arbitrary up to the requirement that the matrix $v$ as in \eqref{9.22} lies in $H$. The matrix $X_{1,5}$ is a $(2n+1)\times (2n+1)$ matrix of $\ell'\times \ell'$ blocks of the form
\begin{equation}\label{9.77}
X_{1,5}=\begin{pmatrix}\ast&\ast&\ast&\ast&\cdots&\ast&\ast\\
&\vdots&&\vdots\\\ast&\ast&\ast&\ast&\cdots&\ast&\ast\\0&(x_{1,5})_{2n-1,2}&\ast&\ast&\cdots&\ast&\ast\\0&0&(x_{1,5})'_{2n-1,2}&\ast&\cdots&\ast&\ast\\0&0&0&\ast&\cdots&\ast&\ast\end{pmatrix}.
\end{equation}
The matrices $Y_{2,1}$ and $Y'_{2,1}$ have the same shape as the shape described right after \eqref{9.28}, with $\ell'$ instead of $\ell$. Again, for the description of the character $\psi_{\mathcal{M}}$, we need to specify more. $Y_{2,1}$ is a $(2n-1)\times (2n+1)$ matrix of $(r-\ell')\times\ell$ blocks. Its last two block columns are arbitrary, and its first $2n-1$ block columns form a $(2n-1)\times (2n-1)$ matrix of $(r-\ell')\times \ell'$ blocks, which has an upper triangular shape, with all the blocks along the diagonal being zero. The matrix $Y'_{2,1}$ has a dual shape. 
\begin{equation}\label{9.78}
Y_{2,1}=\begin{pmatrix}0&\ast&\ast&\cdots&\ast&\ast&\ast\\0&0&\ast&\cdots&\ast&\ast&\ast\\\vdots&&\vdots&&\vdots&\vdots&\vdots\\0&0&0&\cdots&\ast&\ast&\ast\\0&0&0&\cdots&0&(y_{2,1})_{2n-1,2n}&\ast\end{pmatrix}.
\end{equation}
The matrix $X_{2,3}$ has $2n-1$ block rows, each one containing $r-\ell'$ rows. It has seven block columns, with the same division as that of $V$. All its blocks are arbitrary, except the first two blocks in the last row, which are zero.
\begin{equation}\label{9.79}
X_{2,3}=\begin{pmatrix}\ast&\ast&\ast&\ast&\ast&\ast&\ast\\&\vdots&&&&\vdots&&\\\ast&\ast&\ast&\ast&\ast&\ast&\ast\\0&0&(x_{2,3})_{2n-1,3}&(x_{2,3})_{2n-1,4}&(x_{2,3})_{2n-1,5}&\ast&\ast\end{pmatrix}.
\end{equation}
The matrix $X'_{2,3}$ has a dual shape (as in \eqref{9.30}). 
The matrix $X_{2,4}$ is arbitrary (as long as $v$ lies in $H$). The matrix $Y_{3,1}$ has the same shape as in \eqref{9.31}, with $\ell'$ instead of $\ell$. 
\begin{equation}\label{9.80}
Y_{3,1}=\begin{pmatrix}0&&0&\ast&\ast\\0&\cdots&0&(y_{3,1})_{2,2n}&\ast\\
0&&0&0&\ast\\0&&0&0&(y_{3,1})_{4,2n+1}\\0&\cdots&0&0&\ast\\0&&0&0&0\\0&&0&0&0\end{pmatrix}.
\end{equation}
The matrix $Y'_{3,1}$ has a dual shape (as in \eqref{9.32}).
Finally, as in \eqref{9.33}, $Y_{5,1}$, as a $(2n+1)\times (2n+1)$ matrix of $\ell'\times \ell'$ blocks, is such that all its blocks are zero except the last two blocks in the first block row, and the last block in the second block row.
\begin{equation}\label{9.81}
Y_{5,1}=\begin{pmatrix}0&\cdots&0&(y_{5,1})_{1,2n}&\ast\\0&&0&0&(y_{5,1})'_{1,2n}\\0&&0&0&0\\\vdots&&&&\vdots\\0&\cdots&0&0&0\end{pmatrix}.
\end{equation}
The character $\psi_{\mathcal{M}_{n+2,\ell'}}(v)$ of the element $v\in \mathcal{M}_{n+2,\ell'}(\BA)$ in \eqref{9.22}, described above, is given by
\begin{equation}\label{9.82}
\psi_{\mathcal{M}_{n+2,\ell'}}(v)=\prod_{i=1}^{n-1}\psi(tr(U^1_i)+tr(U^2_i)-tr(U^1_{n-1+i})-tr(U^2_{n-1+i}))\tilde{\psi}_{\mathcal{M}_{n+2,\ell'}}(v),
\end{equation}
where we used the notation in \eqref{9.23} - \eqref{9.25}; $\tilde{\psi}_{\mathcal{M}_{n+2,\ell'}}(v)$ is described as follows. Let
\begin{equation}\label{9.82.1}
A_v=\begin{pmatrix}(x_{1,3})_{2n-1,3}&U_{2n-1}^1&(x_{1,3})_{2n-1,4}&(x_{1,5})_{2n-1,2}&(x_{1,3})_{2n-1,5}\\(x_{2,3})_{2n-1,3}&(y_{2,1})_{2n-1,2n}&(x_{2,3})_{2n-1,4}&-w_{\ell'}{}^t(x_{1,4})_{2n,1}w_{r-\ell'}&(x_{2,3})_{2n-1,5}\end{pmatrix},
\end{equation}
and
\begin{equation}\label{9.82.2}
B_v=\begin{pmatrix} (y_{3,1})_{2,2n}&v_{2,4}&-w_{\mu-\ell'}{}^t(x_{1,3})_{2n,6}w_{\ell'}\\(y_{5,1})_{1,2n}&-w_{\ell'}{}^t(y_{3,1})_{4,2n+1}{}^tJ_{H_{2([\frac{m+1}{2}]-r-\ell')}}&-w_{\ell'}{}^tU_{2n}^1w_{\ell'}\end{pmatrix}.
\end{equation}
Then
\begin{equation}\label{9.83}
\tilde{\psi}_{\mathcal{M}_{n+2,\ell'}}(v)=\psi(tr(A_vL_{\ell,c,d,e}))\psi(tr(B_v A'_\gamma)).
\end{equation}
For example, when $H_{2nm}$ is symplectic (and then $m=2m'$ is even, by assumption),
\begin{multline}\label{9.83.1}
tr(A_vL_{\ell,c,d,e})=tr(U_{2n-1}^1)+tr((x_{1,3})_{2n-1,3}\begin{pmatrix}0_{d\times d}&0\\0&I_{\ell'-d}\\0_{(m'-\ell')\times d}&0\end{pmatrix})+\\
+tr((x_{2,3})_{2n-1,3}\begin{pmatrix} 0_{(m'-r)\times d}&0&0\\0&I_c&0_{c\times (r-\ell)}\\0_{(r-c)\times d}&0&0\end{pmatrix})-tr(x_{1,4}\begin{pmatrix}0_{(r-\ell'-e')\times (\ell'-e')}&0\\0&I_{e'}\end{pmatrix})-\\
-tr((x_{2,3})_{2n-1,5}\begin{pmatrix}0_{(m'-\ell')\times e'}&0\\I_{e'}&0_{e'\times(r-\ell'-e')}\\0_{(\ell'-e')\times e'}&0\end{pmatrix});\\
tr(B_v A'_\gamma)=tr(B_v)=-tr(U_{2n}^1)+tr(\begin{pmatrix}(y_{3,1})_{2,2n}&v_{2,4}\end{pmatrix}).
\end{multline}
As in the proof of Theorem \ref{thm 9.1}, right after \eqref{9.34}, we perform the similar sequence of roots exchange, exactly as in Prop. \ref{prop 7.1} and Prop. \ref{prop 7.2}. We assume that $\ell'$ is positive. We start with the subgroup $\mathrm{Y}_{2,1}^{1,2}$ and exchange it with $\mathrm{X}_{1,2}^{1,1}$. Then we exchange $\mathrm{Y}_{2,1}^{2,3}$ with $\mathrm{X}_{1,2}^{2,2}$, and $\mathrm{Y}_{2,1}^{1,3}$ with $\mathrm{X}_{1,2}^{2,1}$, and so on. We exchange column $t$ of $\mathrm{Y}_{2,1}$, $2\leq t\leq 2n-1$, $\mathrm{Y}_{2,1}^{t-1,t}$, $\mathrm{Y}_{2,1}^{t-2,t}$,...,$\mathrm{Y}_{2,1}^{1,t}$, with row $t-1$ of $\mathrm{X}_{1,2}$, $\mathrm{X}_{1,2}^{t-1,t-1}$, $\mathrm{X}_{1,2}^{t-1,t-2}$,..., $\mathrm{X}_{1,2}^{t-1,1}$, in this order. We conclude that $\varphi_{n+2}^{\psi,\ell,c,d,e,\gamma}$ is identically zero on $A(\Delta(\tau,m)\gamma_\psi^{\epsilon)},\eta,k)$, if and only if the following integral is identically zero on $A(\Delta(\tau,m)\gamma_\psi^{\epsilon)},\eta,k)$, 
\begin{equation}\label{9.84}
\tilde{\varphi}_{n+2}^{\psi,\ell,c,d,e,\gamma}(\xi)=
\int_{\tilde{\mathcal{M}}_{n+2,\ell'}(F)\backslash
	\tilde{\mathcal{M}}_{n+2,\ell'}(\BA)}\xi(v)\psi^{-1}_{\tilde{\mathcal{M}}_{n+2},\ell'}(v)dv, 
\end{equation}
where, as in the proof of Theorem \ref{thm 9.1}, $\tilde{\mathcal{M}}_{n+2,\ell'}$ is the subgroup of elements $v$ written as in \eqref{9.22}, with $\ell'$ instead of $\ell$, plus the description starting right after \eqref{9.74} up to \eqref{9.81}, where $X_{1,2}$ is such that its last three block rows are zero, and all of its other blocks are arbitrary, and $Y_{2,1}$ is such that its first $2n-1$ block columns are zero. Note that $\tilde{\mathcal{M}}_{n+2,\ell'}$ lies in the standard parabolic subgroup $Q^H_{(\ell')^{2n-1}}$.  The character $\psi_{\tilde{\mathcal{M}}_{n+2,\ell'}}$ is still given by \eqref{9.83}. Thus, we want to prove that, for $\ell'$ positive, $\tilde{\varphi}_{n+2}^{\psi,\ell,c,d,e,\gamma}$ is identically zero on $A(\Delta(\tau,m)\gamma_\psi^{(\epsilon)},\eta,k)$. We continue as in \eqref{9.36}, with $n+2$ instead of $i$ and $\ell'$ instead of $\ell$. Thus,
$$
\tilde{\mathcal{M}}_{n+2,\ell'}=\mathcal{M}'_{n+2,\ell'}\rtimes\mathcal{M}''_{n+2,\ell'},
$$
where $\mathcal{M}'_{n+2,\ell'}$ is the intersection of $\tilde{\mathcal{M}}_{n+2,\ell'}$ with the unipotent radical $U^H_{(\ell')^{2n-1}}$ of $Q^H_{(\ell')^{2n-1}}$, and $\mathcal{M}''_{n+2,\ell'}$ is the intersection of $\tilde{\mathcal{M}}_{n+2,\ell'}$ with the Levi part $M^H_{(\ell')^{2n-1}}$ of $Q^H_{(\ell')^{2n-1}}$. We have the analogue of \eqref{9.37}, and consider first its inner $dv'$- integration, 
\begin{equation}\label{9.85}
(\varphi')_{n+2}^{\psi,\ell,c,d,e,\gamma}(\xi)=
\int_{\mathcal{M}'_{n+2,\ell'}(F)\backslash
	\mathcal{M}'_{n+2,\ell'}(\BA)}\xi(v')\psi^{-1}_{\tilde{\mathcal{M}}_{n+2,\ell'}}(v')dv', 
\end{equation}
so that
\begin{equation}\label{9.85.1}
\tilde{\varphi}_{n+2}^{\psi,\ell,c,d,e,\gamma}(\xi)=
\int_{\mathcal{M}''_{n+2,\ell'}(F)\backslash
	\mathcal{M}''_{n+2,\ell'}(\BA)}(\varphi')_{n+2}^{\psi,\ell,c,d,e,\gamma}(v''\cdot \xi)\psi^{-1}_{\tilde{\mathcal{M}}_{n+2,\ell'}}(v'')dv'',
\end{equation}
Consider the matrices $x_{\ell'}(y,z)$ (see \eqref{9.2}), with 
\begin{equation}\label{9.86}
y=(y_1,...,y_9),
\end{equation}
where $y_1,y_2,y_8,y_9\in M_{\ell'\times\ell'}$, $y_3,y_7\in M_{\ell'\times ((2n-1)(r-\ell')+\mu(n-1)-\ell')}$, $y_4, y_6\in M_{\ell'\times [\frac{m}{2}]}$, $y_5\in M_{\ell'\times 2(m'-r-\ell')}$. For such $y$, with adele coordinates, and $z$, such that $x_{\ell'}(y,z)\in H_{2nm}(\BA)$, denote
$$
(\varphi')_{n+2}^{\psi,\ell,c,d,e,\gamma}(\xi)(y)=(\varphi')_{n+2}^{\psi,\ell,c,d,e,\gamma}(\xi)(x_{\ell'}(y,z)\cdot\xi).
$$
Assume that $y$ is such that $y_3=0$, $y_8=0$. Then
\begin{equation}\label{9.87}
(\varphi')_{n+2}^{\psi,\ell,c,d,e,\gamma}(\xi)(y)=\psi(tr(y_1))\psi(tr(y_4b_1)+tr(y_6b_2)+tr(y_9b_3)) (\varphi')_{n+2}^{\psi,\ell,c,d,e,\gamma}(\xi),
\end{equation}
where $b_1, b_2\in M_{[\frac{m}{2}]\times \ell'}(F)$, $b_3\in M_{\ell'\times\ell'}(F)$ can be read from \eqref{9.83}, \eqref{9.83.1}. We describe them in detail. 
$$
b_1=\begin{pmatrix}0_{d\times d}&0\\0&I_{\ell'-d}\\0_{([\frac{m}{2}]-\ell')\times d}&0\end{pmatrix},\ b_2=-\begin{pmatrix}0_{([\frac{m}{2}]-\ell')\times \ell'}\\b_3\end{pmatrix}.
$$
When $H_{2nm}$ is symplectic, $b_3=0$ (and hence $b_2=0$).
In Case 1 of Lemma \ref{lem 9.5},
$$
b_3=\begin{pmatrix}0_{e\times (\ell'-e)}&\delta w_e\\0&0_{(\ell'-e)\times e}\end{pmatrix}.
$$ 
Consider Case 2 of Lemma \ref{lem 9.5}. Recall that now $\ell'=[\frac{m+1}{2}]-r$ and $\ell'-\ell+c<0$. Let, in this case,
$$
T=\begin{pmatrix}0&0_{(\ell-c-\ell')\times(\ell-c-\ell')}\\\delta w_{2(\ell'-\ell+c)+e}&0\end{pmatrix}.
$$
Then 
$$
b_3=\begin{pmatrix}0&T\\0_{(\ell-c-e)\times (\ell-c-e)}&0\end{pmatrix}.
$$
Let $b'_1=-w_{\ell'}{}^tb_1w_{[\frac{m}{2}]}$, $b'_2=-w_{\ell'}{}^tb_2w_{[\frac{m}{2}]}$, $b'_3=-\delta_Hw_{\ell'}{}^tb_3w_{\ell'}$. Note that in all cases we have $b'_3=-b_3$. It is straightforward to check that we have
	\begin{equation}\label{9.87.1}
	b'_3-b_3+b'_1b_2+b'_2b_1=0.
	\end{equation}
As in \eqref{9.42}, we write the Fourier expansion of $(\varphi')_{n+2}^{\psi,\ell,c,d,e,\gamma}(\xi)(0,0,y_3,0,0,0,0,y_8,0)$, which we abbreviate to $\beta(\xi)(y_3,y_8)$.
A typical Fourier coefficient of the last function has the following form
\begin{equation}\label{9.88}
\int \beta(\xi)(y_3,y_8)\psi^{-1}(tr(y_3a_1)+tr(y_8a_2))dy_3dy_8,
\end{equation}
where $a_1\in  M_{((2n-1)(r-\ell')+\mu(n-1)-\ell')\times\ell'}(F)$, $a_2\in M_{\ell'\times \ell'}(F)$. From \eqref{9.87.1}, the following matrix lies in $H_{2nm}(F)$,
$$
z'(a,b)=diag(I_{(2n-1)\ell'},\begin{pmatrix}I_{\ell'}\\0&I_{\ell'}\\a_1&&I\\b_1&&&I_{[\frac{m}{2}]}\\0&&&&I\\b_2&&&&&I_{[\frac{m}{2}]}\\0&&&&&&I\\a_2&&&&&&&I_{\ell'}\\b_3&a'_2&0&b'_2&0&b'_1&a_1'&0&I_{\ell'}\end{pmatrix},I_{(2n-1)\ell'}).
$$
Here, the third and the seventh identity blocks on the diagonal are each of size $(2n-1)(r-\ell')+(m-2r)(n-1)-\ell'$, and the middle identity block is of size $2([\frac{m+1}{2}]-r-\ell')$. Now, we can express the Fourier coefficient \eqref{9.88} in the following simpler form. When we substitute $\beta(\xi)(y_3,y_8)$, using \eqref{9.85}, the domain of integration in \eqref{9.88} is $U_{(\ell')^{2n-1}}(F)\backslash U_{(\ell')^{2n-1}}(\BA)$. Consider the the following character $\psi_{U_{(\ell')^{2n-1}}}$ of $U_{(\ell')^{2n-1}}(\BA)$. On the adele points of the intersection of $U_{(\ell')^{2n-1}}$ and $\tilde{\mathcal{M}}'_{n+2,\ell'}$, such that the coordinates $y_1,...,y_9$ are all zero, the character is $\psi_{\tilde{\mathcal{M}}'_{n+2,\ell'}}$, and on the elements $x_{\ell'}(y,z)$, the character is $\psi(tr(y_1))$. Then the Fourier coefficient \eqref{9.88} is equal to
\begin{equation}\label{9.89}
\varphi_{(\ell')^{2n-1}}^\psi(z'(a,b))^{-1}\cdot \xi)=\int_{U_{(\ell')^{2n-1}}(F)\backslash U_{(\ell')^{2n-1}}(\BA)}\xi(v(z'(a,b))^{-1})\psi^{-1}_{U_{(\ell')^{2n-1}}}(v)dv.
\end{equation}
We have reached the same point as \eqref{9.47}, in the proof of Theorem \ref{thm 9.1} . Indeed, we have
\begin{equation}\label{9.90} 
(\varphi')_{n+2}^{\psi,\ell,c,d,e,\gamma}(\xi)=\sum_{ z'(a,b)\in H_{2nm}(F)}\varphi_{(\ell')^{2n-1}}^\psi((z'(a,b))^{-1}\cdot\xi).
\end{equation}
We continue with the same steps as the ones from \eqref{9.48} till the end of the proof of Theorem \ref{thm 9.1}, with $\ell'$ instead of $\ell$, and we finish the proof of Theorem \ref{thm 9.7}.
\end{proof}

\begin{cor}\label{cor 9.8}
	Let $L\in \mathcal{L}_\ell^0(F)$ be in the orbit of $L_{\ell,c,d,e}$. Assume that $f_n^{\psi,L}$ is nontrivial on $A(\Delta(\tau,m)\gamma_\psi^{(\epsilon)},\eta,k)$. Then Case 2 of Lemma \ref{lem 9.5} is impossible, and we have $\ell=c$. In particular, $d=e=0$.
\end{cor}
\begin{proof}
By Theorem \ref{thm 9.7}, we must have $\ell'=0$. In Case 2 of Lemma \ref{lem 9.5}, we have $\ell'=[\frac{m+1}{2}]-r$. We then have
$$
\ell'=d+(\ell-(c+d+e))+(r+\ell-c-[\frac{m+1}{2}])+(2([\frac{m+1}{2}]+c-\ell-r)+e).
$$
All four summands are non-negative, and the third summand is positive. Since $\ell'=0$, we get a contradiction. Thus, Case 2 of Lemma \ref{lem 9.5} is impossible. In Case 1 of Lemma \ref{lem 9.5}, $\ell'=\ell-c$. In this case, 
$$
\ell'=d+(\ell-(c+d+e))+e,
$$
All three summands are non-negative. Since $\ell'=0$, we get that $\ell=c$ and $d=e=0$.	In Case 3 of Lemma \ref{lem 9.5}, $\ell'=\ell-c-e'$. (Recall that $e=2e'$.) In this case,
$$
\ell'=d+(\ell-(c+d+e))+e'.
$$
All three summands are non-negative. Since $\ell'=0$, we get that $\ell=c$ and $d=e=0$.	
\end{proof}
It remains to analyze the Fourier coefficients $f_n^{\psi,L_{c,c,0,0,A'_\gamma}}$. (See \eqref{9.64}, \eqref{9.67}.) It will be convenient to replace $L_{c,c,0,0}$ with the following matrix which lies in the same $H_{2[\frac{m}{2}]}(F)\times H_{2([\frac{m+1}{2}]-r}(F)$-orbit. 
\begin{equation}\label{9.91}
L_c=\begin{pmatrix}I_c&0_{c\times(r-c)}\\0_{(2(m-r)-c)\times c}&0\end{pmatrix}.
\end{equation}
Thus, we consider
\begin{multline}\label{9.92}
f_n^{\psi,c}(\xi)=\int_{M_{r\times 2(m-r)}(F)\backslash M_{r\times 2(m-r)}(\BA)}\\
\int_{U'_{{(m-2r)}^{n-1}}(F)\backslash
	U'_{{(m-2r)}^{n-1}}(\BA)}\int_{E_{n-1}(F)\backslash E_{n-1}(\BA)}\xi(uvx_z)\psi^{-1}_{E_{n-1}}(u)\\
\quad\quad\quad\psi^{-1}_{U'_{(m-2r)^{n-1}},A'_0}(v)\psi^{-1}(tr(L_c\cdot z))dudv dz,
\end{multline}
where $A'_0=I_{m-2r}$, when $m$ is even, and $A'_0$ is as in Lemma \ref{lem 9.6}, when $m=2m'-1$ is odd (and then $H_{2nm}$ is orthogonal). 

\section{Fourier expansions III}

In this section we analyze $f_n^{\psi,c}(\xi)$ in \eqref{9.92}. A consequence of the following theorem will be that in the cases of functoriality (the cases in Theorem \ref{thm 2.2}), for $c>0$, $f_n^{\psi,c}$ is trivial.

\begin{thm}\label{thm 10.1}
Assume that $c>0$ and $f_n^{\psi,c}$  is nontrivial on $A(\Delta(\tau,m)\gamma_\psi^{(\epsilon)},\eta,k)$. Then $n$ divides $r-c$.
\end{thm}
\begin{proof}
Of course, the theorem is clear when $c=r$, but we will allow this case in the proof, as we are going to prove more. The proof of the theorem is similar to that of Theorem \ref{thm 6.1}, \ref{thm 9.1}, \ref{thm 9.7}. 
Let $E^{n+2}$ be the subgroup generated by $E_{n-1}$, the elements $x_r(y^{(n+2)},t)$ (see \eqref{9.2}, \eqref{9.55}), and $U'_{(m-2r)^{n-1}}$. Then \eqref{9.92} can be rewritten as
\begin{equation}\label{10.5}
f_n^{\psi,c}(\xi)=\int_{E^{n+2}(F)\backslash E^{n+2}(\BA)}\xi(v)\psi^{-1}_{E^{n+2},c}(v)dv,
\end{equation}
where $\psi_{E^{n+2},c}$ is the character of $E^{n+2}(\BA)$, given, in the notation of \eqref{9.92}, by 
$$
\psi_{E^{n+2},c}(uvx_z)=\psi_{E_{n-1}}(u)\psi(tr(L_cz))\psi_{U'_{(m-2r)^{n-1}}}(v).
$$
We apply a conjugation by a Weyl element $\epsilon'_0$, which we describe now. It is similar to the ones in \eqref{6.1}, \eqref{9.16}, \eqref{9.70}.
Consider the following Weyl element	
\begin{equation}\label{10.1}
\epsilon'_0=\begin{pmatrix}B_1&B_2&0\\B_3&0&0\\0&B_4&0\\0&0&B_5\\0&B_6&B_7\end{pmatrix}.
\end{equation}
The block $B_1$ has the following form. It has $2n$ block rows,
each one of size $c$; the last block row is zero. It has
$2n-1$ block columns, each one of size $r$.
\begin{equation}\label{10.2}
B_1=\begin{pmatrix}a\\&a\\&&a\\&&&\cdots\\&&&&a\\0&&&\cdots
&0\end{pmatrix},\quad
a=\begin{pmatrix}I_c&0_{c\times (r-c)}\end{pmatrix}.
\end{equation}
The block $B_3$ has the same form as \eqref{9.19}, with $c$ replacing $\ell$. The block $B_4$ is as follows
\begin{equation}\label{10.3}
B_4=\begin{pmatrix}I_{(m-2r)(n-1)}\\&0_{2(m-r-c)\times c}&I_{2(m-r-c)}\\&&&0_{(m-2r)(n-1)\times c}&I_{(m-2r)(n-1)} \end{pmatrix}.
\end{equation}
The block $B_2$ has the same block row division as $B_1$, and the same block column division as $B_4$. The first $2n-1$ block rows are all zero. The last block row of $B_2$ have the following form
\begin{equation}\label{10.4}
\begin{pmatrix}0_{c\times (m-2r)(n-1)}& I_c&0_{c\times 2(m-r-c)}&0_{c\times c}&0_{c\times (m-2r)(n-1)}\end{pmatrix}
\end{equation}
The blocks $B_5, B_6, B_7$ are already determined by the other blocks.

Conjugating inside \eqref{9.92} (or \eqref{10.5}) by $\epsilon'_0$, we see that we need to consider, for $c$ positive, the following integral on $A(\Delta(\tau,m)\gamma_\psi^{\epsilon)},\eta,k)$,
\begin{equation}\label{10.6}
\varphi_n^{\psi,c}(\xi)=
\int_{\mathcal{N}_{n+2,c}(F)\backslash
	\mathcal{N}_{n+2,c}(\BA)}\xi(v)\psi^{-1}_{\mathcal{N}_{n+2,c}}(v)dv, 
\end{equation}
where $\mathcal{N}_{n+2,c}=\epsilon'_0E^{n+2}(\epsilon')_0^{-1}$, and $\psi_{\mathcal{N}_{n+2,c}}$ is the character of $\mathcal{N}_{n+2,c}(\BA)$ defined by $\psi_{\mathcal{N}_{n+2,c}}(x)=\psi_{E^{n+2},c}((\epsilon')_0^{-1}x\epsilon'_0)$. Let us describe these.\\ 
The subgroup $\mathcal{N}_{n+2,c}$ consists of elements of the form \eqref{9.22}, with $Y_{5,1}=0$, that is
\begin{equation}\label{10.7}
v=\begin{pmatrix}U_1&X_{1,2}&X_{1,3}&X_{1,4}&X_{1,5}\\Y_{2,1}&U_2&X_{2,3}&X_{2,4}&X'_{1,4}\\
Y_{3,1}&0&V&X'_{2,3}&X'_{1,3}\\0&0&0&U_2'&X'_{1,2}\\0&0&Y'_{3,1}&Y'_{2,1}&U'_1\end{pmatrix}.
\end{equation}
We now describe the blocks in \eqref{10.7}. The block $U_1$ has $(2n)\times (2n)$ blocks, all of size $c\times c$, and has the form
\begin{equation}\label{10.8}
U_1=\begin{pmatrix}I_c&U^1_1&*&\cdots&*&*\\
&I_c&U^1_2&\cdots&*&*\\
& &I_c&\cdots &* & * \\
& & & \cdots& \cdots&\cdots &\\
& & &       &I_c&U^1_{2n-1}\\
& & &       & &I_c\end{pmatrix}.
\end{equation}
The block $U_1'$ is of the same size as $U_1$ and has the form
\eqref{10.8}. The block $U_2$ has the form \eqref{9.24}, with $c$ instead of $\ell$, 
\begin{equation}\label{10.9}
U_2=\begin{pmatrix}I_{r-c}&U^2_1&*&\cdots&*&*\\
&I_{r-c}&U^2_2&\cdots&*&*\\
& &I_{r-c}&\cdots &* & * \\
& & & \cdots& \cdots&\cdots &\\
& & &       &I_{r-c}&U^2_{2n-2}\\
& & &       & &I_{r-c}\end{pmatrix}.
\end{equation}
The block $U_2'$ is of the same size as $U_2$ and has the form
\eqref{10.9}. The block $V$ is of size $(2(m-2r)(n-1)+2(m-r-c))\times
(2(m-2r)(n-1)+2(m-r-c))$ and has a form similar to \eqref{6.5},
\begin{equation}\label{10.10}
V=\begin{pmatrix}
S&E_1&E_2&E_3&C\\&I_{[\frac{m}{2}]-c}&0&0&E'_3\\
&&I_{2([\frac{m+1}{2}]-r)}&0&E_2'\\&&&I_{[\frac{m}{2}]-c}&E'_1\\
& &&&S'\end{pmatrix},
\end{equation}
where $S$ is $(m-2r)(n-1)\times (m-2r)(n-1)$ upper unipotent matrix of the
form \eqref{6.5.1} ($S'$ has a similar form to that of $S$.) As in \eqref{6.5}, \eqref{6.5.2}, write the
$(m-2r)(n-1)\times
2([\frac{m+1}{2}]-r)$ matrix $E_2$ in \eqref{10.10} in the form $E_2=\begin{pmatrix}*\\
S_{n-1}\end{pmatrix}$, with $S_{n-1}$ of size $(m-2r)\times2([\frac{m+1}{2}]-r)$.\\
The block $X_{1,2}$ is a $2n\times (2n-1)$ matrix of $c\times (r-c)$ blocks. The last two block rows are zero. The first $2n-1$ block rows form a $(2n-1)\times (2n-1)$ matrix of $c\times (r-c)$ blocks, which has an upper triangular shape, with all the blocks along the diagonal being zero. As before, we denote by $X_{1,2}^{j,t}$, the block of $X_{1,2}$ lying in position $(j,t)$, and, similarly, we denote by $\mathrm{X}_{1,2}^{j,t}$ the corresponding abelian unipotent subgroup. The matrix $X'_{1,2}$ has a dual shape. It is a $(2n-1)\times 2n$ matrix of $(r-c)\times c$ blocks. The first two block columns are zero. The last $2n-1$ block columns form a $(2n-1)\times (2n-1)$ matrix of $(r-c)\times c$ blocks, which has an upper triangular shape, with all the blocks along the diagonal being zero. The block $X_{1,3}$ has $2n$ block rows, each one containing $c$ rows. It has five block columns with sizes as for $V$. It has the following form,
\begin{equation}\label{10.11}
X_{1,3}=\begin{pmatrix}\ast&\ast&\ast&\ast&\ast\\&&\vdots\\ \ast&\ast&\ast&\ast&\ast\\0&\ast&\ast&\ast&\ast
\\0&0&0&0&\ast\end{pmatrix}.
\end{equation}
The matrix $X'_{1,3}$ has a dual form. It has $2n$ block columns, each one containing $c$ columns. It has five block rows, with the same row division as that of $V$. Its first block column is zero except the first block. Its second block column is such that its last block is zero. The matrix $X_{1,4}$ is a $2n\times (2n-1)$ matrix of $c\times (r-c)$ blocks, all of whose blocks are arbitrary. Similarly, the blocks of $X'_{1,4}$ are arbitrary (blocks above are arbitrary (up to the requirement that the matrix $v$ in \eqref{10.7} lies in $H$.) The matrix $X_{1,5}$ is a $2n\times 2n$ matrix of $c\times c$ blocks, such that the block in position $(2n,1)$ is zero. Thus, it has the form
\begin{equation}\label{10.12}
X_{1,5}=\begin{pmatrix}\ast&\ast&\cdots&\ast&\ast\\
&\vdots&&\vdots\\\ast&\ast&\cdots&\ast&\ast\\0&\ast&\cdots&\ast&\ast\end{pmatrix}.
\end{equation}
The matrix $Y_{2,1}$ is a $(2n-1)\times 2n$ matrix of $(r-c)\times c$ blocks. Its last block column is arbitrary, and its first $2n-1$ block columns form a $(2n-1)\times (2n-1)$ matrix of $(r-c)\times c$ blocks, which has an upper triangular shape, with all the blocks along the diagonal being zero. We denote by $Y_{2,1}^{j,t}$, the block of $Y_{2,1}$ lying in position $(j,t)$. We denote, as above, the corresponding unipotent subgroup by $\mathrm{Y}_{2,1}^{j,t}$. The matrix $Y'_{2,1}$ has a dual shape. It is a $2n\times (2n-1)$ matrix of $c\times (r-c)$ blocks. The first block row is arbitrary. The last $2n-1$ block rows form a $(2n-1)\times (2n-1)$ matrix of $c\times (r-c)$ blocks, which has an upper triangular shape, with all the blocks along the diagonal being zero. The matrix $X_{2,3}$ has $2n-1$ block rows, each one containing $r-c$ rows. It has five block columns, with the same division as that of $V$. It has the following form\begin{equation}\label{10.13}
X_{2,3}=\begin{pmatrix}\ast&\ast&\ast&\ast&\ast\\&&\vdots\\ \ast&\ast&\ast&\ast&\ast\\0&\ast&\ast&\ast&\ast\end{pmatrix}.
\end{equation}
The matrix $X'_{2,3}$ has a dual shape. It has five block rows, with the same division as that of $V$. It has $2n-1$ block columns, each one containing $r-c$ columns. It has the form
\begin{equation}\label{10.14}
X'_{2,3}=\begin{pmatrix}\ast&\ast&\cdots&\ast\\ \ast&\ast&&\ast\\ \ast&\ast&&\ast\\\ast&\ast&\cdots&\ast\\0&\ast&\cdots&\ast\end{pmatrix}.
\end{equation}
The matrix $X_{2,4}$ is arbitrary (as long as $v$ lies in $H$). The matrix $Y_{3,1}$ has five block rows, with the same division as that of $V$. It has $2n$ block columns, each one containing $cl$ columns. All of its blocks are zero, except the block in position $(1,5)$.
\begin{equation}\label{10.15}
Y_{3,1}=\begin{pmatrix}0&&0&\ast\\0&\cdots&0&0\\
0&&0&0\\0&\cdots&0&0\\0&&0&0\end{pmatrix}.
\end{equation}
The matrix $Y'_{3,1}$ has a dual shape
\begin{equation}\label{10.16}
\begin{pmatrix} 0&0&0&0&\ast\\0&0&0&0&0\\&&\vdots\\0&0&0&0&0\end{pmatrix}.
\end{equation}
The character $\psi_{\mathcal{N}_{n+2,c}}(v)$ of the element $v\in \mathcal{N}_{n+2,c}(\BA)$ in \eqref{10.7}, described above, is given by
\begin{multline}\label{10.17}
 \psi_{\mathcal{N}_{n+2,c}}(v)=\prod_{i=1}^{n-1}\psi(tr(U^1_i)+tr(U^2_i))\psi^{-1}(tr(U^1_{n-1+i})+tr(U^2_{n-1+i}))\\
 \cdot\psi(tr(U^1_{2n-1}))\prod_{i=1}^{n-2}\psi(tr(S_i))\psi(tr(S_{n-1}A'_0)).
 \end{multline}
 We used the notation in \eqref{10.8} - \eqref{10.10}, and \eqref{6.5.1}.
 
 Now we perform a sequence of roots exchange, exactly as in Prop. \ref{prop 7.1} and Prop. \ref{prop 7.2} (and as in Theorem \ref{thm 9.1}, right after \eqref{9.34}, and Theorem \ref{thm 9.7}, right after \eqref{9.83.1}). We assume that $c$ is positive. We start with the subgroup $\mathrm{Y}_{2,1}^{1,2}$ and exchange it with $\mathrm{X}_{1,2}^{1,1}$. Then we exchange $\mathrm{Y}_{2,1}^{2,3}$ with $\mathrm{X}_{1,2}^{2,2}$, and $\mathrm{Y}_{2,1}^{1,3}$ with $\mathrm{X}_{1,2}^{2,1}$, and so on. We exchange column $t$ of $\mathrm{Y}_{2,1}$, $2\leq t\leq 2n-1$, $\mathrm{Y}_{2,1}^{t-1,t}$, $\mathrm{Y}_{2,1}^{t-2,t}$,...,$\mathrm{Y}_{2,1}^{1,t}$, with row $t-1$ of $\mathrm{X}_{1,2}$, $\mathrm{X}_{1,2}^{t-1,t-1}$, $\mathrm{X}_{1,2}^{t-1,t-2}$,..., $\mathrm{X}_{1,2}^{t-1,1}$, in this order. Finally, one can check that we may continue the root exchanges, and exchange $\mathrm{Y}_{3,1}^{1,2n}$ with $\mathrm{X}_{1,3}^{2n-1,1}$. We have to do this step by step, as follows. Write the block $Y_{3,1}^{1,2n}$ as a column of $n-1$ matrices, each one of size $(m-2r)\times c$. Denote them by $T_1$,...,$T_{n-1}$. Similarly, write a matrix in $\mathrm{X}_{1,3}^{2n-1,1}$ as a row of $n-1$ matrices, each one of size $c\times (m-2r)$. Denote them by $R_1,...,R_{n-1}$. Then we can exchange $T_{n-1}$ with $R_{n-1}$, and then $T_{n-2}$ with $R_{n-2}$, and so on. Next, we exchange $\mathrm{Y}_{2,1}^{2n-1,2n}$ with $\mathrm{X}_{1,2}^{2n-1,2n-1}$, then  $\mathrm{Y}_{2,1}^{2n-2,2n}$ with $\mathrm{X}_{1,2}^{2n-1,2n-2}$,..., $\mathrm{Y}_{2,1}^{1,2n}$ with $\mathrm{X}_{1,2}^{2n-1,1}$, in this order. 
 
 We conclude that $\varphi_n^{\psi,c}$ is identically zero on $A(\Delta(\tau,m)\gamma_\psi^{\epsilon)},\eta,k)$, if and only if the following integral is identically zero on $A(\Delta(\tau,m)\gamma_\psi^{\epsilon)},\eta,k)$, 
\begin{equation}\label{10.18}
 \tilde{\varphi}_n^{\psi,c}(\xi)=
 \int_{\tilde{\mathcal{N}}_{n+2,c}(F)\backslash
 	\tilde{\mathcal{N}}_{n+2,c}(\BA)}\xi(v)\psi^{-1}_{\tilde{\mathcal{N}}_{n+2,c}}(v)dv, 
 \end{equation}
 where $\tilde{\mathcal{N}}_{n+2,c}$ is the subgroup of elements $v$ written as in \eqref{10.7}, where $X_{1,2}$ is such that its last block row is zero, and all of its other blocks are arbitrary, $Y_{2,1}=0$ and $Y_{3,1}=0$. (Similarly, in the dual block $X'_{1,2}$, the first column is zero, and all other blocks are arbitrary, provided, of course, that $v$ lies in $H$, $Y'_{2,1}=0$ and $Y_{3,1}=0$.) The shape of all other blocks of $v$ remains as before. Note that $\tilde{\mathcal{N}}_{n+2,c}$ lies in the standard parabolic subgroup $Q^H_{c^{2n-1}}$, and it contains its unipotent radical $U_{c^{2n-1}}$.  The character $\psi_{\tilde{\mathcal{N}}_{n+2,c}}$ is still given by \eqref{10.17}.  
 Write 
 \begin{equation}\label{10.19}
\tilde{\mathcal{N}}_{n+2,c}=U_{c^{2n-1}}\rtimes\mathcal{N}''_{n+2,c},
 \end{equation}
 where $\mathcal{N}''_{n+2,c}$ is the intersection of $\tilde{\mathcal{N}}_{n+2,c}$ with the Levi part $M^H_{c^{2n-1}}$ of $Q^H_{c^{2n-1}}$. We have
 \begin{equation}\label{10.20}
 \tilde{\varphi}_n^{\psi,c}(\xi)=
 \int_{\mathcal{N}''_{n+2,c}(F)\backslash
 	\mathcal{N}''_{n+2,c}(\BA)}\int_{U_{c^{2n-1}}(F)\backslash
 	U_{c^{2n-1}}(\BA)}\xi(v'v'')\psi^{-1}_{\tilde{\mathcal{N}}_{n+2,c}}(v'v'')dv'dv''. 
 \end{equation}
 As in \eqref{9.38}, we consider first the inner $dv'$- integration, where we replace the right $v''$- translate of $\xi$ with $\xi$, that is
 \begin{equation}\label{10.21}
 (\varphi')_n^{\psi,c}(\xi)=
\int_{U_{c^{2n-1}}(F)\backslash
 	U_{c^{2n-1}}(\BA)}\xi(v')\psi^{-1}_{\tilde{\mathcal{N}}_{n+2,c}}(v')dv'. 
 \end{equation}
As in \eqref{9.49}, with $c$ replacing $\ell$, we conclude that
\begin{equation}\label{10.22}
(\varphi')_n^{\psi,c}(\zeta_c(y,t)\cdot\xi)=(\varphi')_n^{\psi,c}(\xi),
\end{equation} 
for all $\zeta_c(y,t)\in H_{2nm}(\BA)$. Thus, as in \eqref{9.51}, 
\begin{equation}\label{10.23}
(\varphi')_n^{\psi,c}(\xi)=(\xi^{U_{c^{2n}}})^{\psi_{V_{c^{2n}}}}.
\end{equation}
We conclude that
\begin{equation}\label{10.24}
 \tilde{\varphi}_n^{\psi,c}(\xi)=\int_{\tilde{U}(F)\backslash \tilde{U}(\BA)}
(\xi^{U_{c^{2n}}})^{\psi_{V_{c^{2n}}}}(u)\tilde{\psi}^{-1}(u)du,
\end{equation}
where $\tilde{U}$ is the  the subgroup of $\mathcal{N}_{n+2,c}$, consisting of the elements
\begin{equation}\label{10.25}
u=\begin{pmatrix}I_{2nc}\\&U_2&X_{2,3}&X_{2,4}\\&&V&X'_{2,3}\\&&&U'_2\\&&&&I_{2nc}\end{pmatrix}\in H_{2nm},
\end{equation}
where $U_2$, $X_{2,3}$, $X_{2,4}$ and $V$ are as in \eqref{10.7}, and similarly with the character $\tilde{\psi}$; it is the restriction of
of $\psi_{\mathcal{N}_{n+2,c}}$ to $\tilde{U}(\BA)$. Note that $U_{c^{2n}}\rtimes \tilde{U}:=\tilde{\mathcal{N}}_{n+2,c}$ is quite "close" to the unipotent radical $U^H_{c^{2n},(r-c)^{n-1},(m-2r)^{n-1}}:=\mathcal{U}_{r,c}$. (By \eqref{10.13}, $X_{2,3}$ has a zero block in the position $(2n-1,1)$.) Let  $\psi_{\tilde{\mathcal{N}}_{n+2,c}}$ be the character of $\tilde{\mathcal{N}}_{n+2,c}(\BA)$ obtained from $\psi_{\mathcal{N}_{n+2,c}}$ by the trivial extension. To finish the proof of Theorem \ref{thm 10.1}, we will show that when we view \eqref{10.24} as a Fourier coefficient of $\xi$ along $\tilde{\mathcal{N}}_{n+2,c}$, with respect to the character $\psi_{\tilde{\mathcal{N}}_{n+2,c}}$, we may replace this unipotent group by the full unipotent radical $U^H_{c^{2n},(r-c)^{n-1},(m-2r)^{n-1}}$ (with trivial extension of $\psi_{\tilde{\mathcal{N}}_{n+2,c}}$). For this, consider the subgroup $\mathrm{X}_{2,3}^{2n-1,1}$. Then we will show that, for all $x\in \mathrm{X}_{2,3}^{2n-1,1}(\BA)$,
\begin{multline}\label{10.26}
\int_{\tilde{\mathcal{N}}_{n+2,c}(F)\backslash \tilde{\mathcal{N}}_{n+2,c}(\BA)}\xi(vx)\psi^{-1}_{\tilde{\mathcal{N}}_{n+2,c}}(v)dv=\\
\int_{\tilde{\mathcal{N}}_{n+2,c}(F)\backslash \tilde{\mathcal{N}}_{n+2,c}(\BA)}\xi(v)\psi^{-1}_{\tilde{\mathcal{N}}_{n+2,c}}(v)dv.
\end{multline}
Note that there is nothing to prove when $c=r$. Thus, we need to prove \eqref{10.26} only when $c<r$. The proof of \eqref{10.26} can be copied from the proof appearing in the next section, where we replace $r$ there by $r-c$ here, and $[\frac{m}{2}]$ there by $[\frac{m}{2}]-c$ here. In the next section, we continue our series of Fourier expansions for the case $c=0$, and the unipotent group there is the subgroup \eqref{10.25}, with $c=0$. See the remark after the proof of Theorem \ref{thm 11.1}. Now, we conclude from \eqref{10.24}, \eqref{10.26}, that
\begin{equation}\label{10.27}
 \tilde{\varphi}_n^{\psi,c}(\xi)=\int_{\mathcal{U}_{r,c}(F)\backslash \mathcal{U}_{r,c}(\BA)}
 \xi(v)\psi^{-1}_{\mathcal{U}_{r,c}}dv,
 \end{equation}
 where $\psi_{\mathcal{U}_{r,c}}$ is the character of $\mathcal{U}_{r,c}(\BA)$ obtained from $\psi_{\tilde{\mathcal{N}}_{n+2,c}}$ (or from $\psi_{\mathcal{N}_{n+2,c}}$) by the trivial extension. We note that the right hand side of \eqref{10.27} is the composition of the constant term $\xi^{U_{2nc+(2n-1)(r-c)}}$ and the Fourier coefficient on the intersection of  $\mathcal{U}_{r,c}$ with the Levi subgroup $M^H_{2nc+(2n-1)(r-c)}$, with respect to the character given by \eqref{10.17}, that is 
 \begin{multline}\label{10.28}
 \begin{pmatrix}U_1&X\\&U_2\\&&V\\&&&U'_2&X'\\&&&&U'_1\end{pmatrix}\mapsto \\
 \prod_{i=1}^{n-1}\psi(tr(U^1_i)+tr(U^2_i))\psi^{-1}(tr(U^1_{n-1+i})+tr(U^2_{n-1+i}))\\
 \cdot\psi(tr(U^1_{2n-1}))\prod_{i=1}^{n-2}\psi(tr(S_i))\psi(tr(S_{n-1}A'_0)).
\end{multline}
Here, $U_1$, $U_2$, $V$ are as in \eqref{10.7}, with the same notation; also, $X$ is an arbitrary matrix. We will denote the composition above of  $\xi^{U_{2nc+(2n-1)(r-c)}}$ with the character \eqref{10.28} by $(\xi^{U_{2nc+(2n-1)(r-c)}})^\psi$. Thus,
 \begin{equation}\label{10.29}
 \tilde{\varphi}_n^{\psi,c}(\xi)= 
 (\xi^{U_{2nc+(2n-1)(r-c)}})^\psi.
 \end{equation}
  We proved that $\varphi_n^{\psi,c}$ is identically zero on $A(\Delta(\tau,m)\gamma_\psi^{\epsilon)},\eta,k)$, if and only if \\
   $(\xi^{U_{2nc+(2n-1)(r-c)}})^\psi $ is identically zero on $A(\Delta(\tau,m)\gamma_\psi^{\epsilon)},\eta,k)$. In particular, if  $\varphi_n^{\psi,c}$ is nonzero on $A(\Delta(\tau,m)\gamma_\psi^{\epsilon)},\eta,k)$ , then the constant term  $\xi^{U_{2nc+(2n-1)(r-c)}}$ is not identically zero on $A(\Delta(\tau,m)\gamma_\psi^{\epsilon)},\eta,k)$.  Since $\tau$ is cuspidal, $2nc+(2n-1)(r-c)$ must be an integer multiple of $n$, and hence $n$ must divide $r-c$. This completes the proof of Theorem \ref{thm 10.1}.
   \end{proof}
   
\begin{cor}\label{cor 10.2}
Consider a case of functoriality for $\tau$ and $H_m^{(\epsilon)}$ (the four cases before Theorem \ref{thm 2.2}). If $c>0$, then $\varphi_n^{\psi,c}$ is identically zero on $A(\Delta(\tau,m)\gamma_\psi^{\epsilon)},\eta,1)$.
\end{cor}
\begin{proof}
Note that in the cases of functoriality, $k=1$. Assume that $c>0$, and that $\varphi_n^{\psi,c}$ is nontrivial on $A(\Delta(\tau,m)\gamma_\psi^{\epsilon)},\eta,1)$. Then, by the previous proof (\eqref{10.29}), $(\xi^{U_{2nc+(2n-1)(r-c)}})^\psi $ is nontrivial on $A(\Delta(\tau,m)\gamma_\psi^{(\epsilon)},\eta,1)$. The Fourier coefficient on $A(\Delta(\tau,m)\gamma_\psi^{\epsilon)},\eta,1)$, given by $(\xi^{U_{2nc+(2n-1)(r-c)}})^\psi $, corresponds to the partition\\
 $((2n)^{2c},(2n-1)^{m-2c},1^{m-2c})$. Now, let us use Prop. \ref{prop 3.1}, \ref{prop 3.2}. In the first two cases of functoriality, $L(\tau,\wedge^2,s)$ has a pole at $s=1$, so that by  Prop. \ref{prop 3.1}(2), we must have
\begin{equation}\label{10.30}
((2n)^{2c},(2n-1)^{m-2c},1^{m-2c})\leq ((2n)^{m-1},n^2).
\end{equation}
This forces $2c\leq m-1$, and 
\begin{equation}\label{10.31}
(m-2c)(2n-1)\leq (m-2c-1)\cdot (2n)+n.
\end{equation}
When $H^{(\epsilon)}_m=\SO_m$, $m=2m'-1$, $n=m-1=2m'-2$. When $H_m^{(\epsilon)}=\Sp_m^{(2)} $, $m=2m'$, $n=m=2m'$. We conclude from \eqref{10.31}, that $2c\leq 1, 0$, respectively. Hence $c=0$, which is a contradiction. 
In the third case of functoriality, $L(\tau,\vee^2,s)$ has a pole at $s=1$, $H_m^{(\epsilon)}=\SO_m$, $m=2m'$, $n=m=2m'$. By Prop. \ref{prop 3.2}(3),
$$
((2n)^{2c},(2n-1)^{m-2c},1^{m-2c})\leq ((2n)^{m-2},2n-1,n+1,n-1,1).
$$
Hence $c\leq m'-1$, and 
$$
(m-2c)(2n-1)\leq (m-2c-2)\cdot(2n)+(n+1)+(n-1).
$$
Since $n=m$, we get that $c=0$.
In the last case of functoriality, $L^S(\tau,\vee^2,s)$ has a pole at $s=1$, $H_m^{(\epsilon)}=\Sp_m$, $m=2m'-2$, $n=m+1=2m'-1$. By Prop. \ref{prop 3.2}(1),
$$
((2n)^{2c},(2n-1)^{m-2c},1^{m-2c})\leq ((2n)^{m-1},n+1,n-1).
$$ 
Hence $c\leq m'-2$, and 
$$
(m-2c)(2n-1)\leq (m-2c-1)\cdot(2n)+(n+1).
$$
Since $n=m+1$, we get that $c=0$.
 
\end{proof}

\begin{rmk}\label{rmk 10.2.1}
Corollary \ref{cor 10.2} has a simpler proof when $c<r$.
\end{rmk}
\begin{proof}.
By the last theorem, $n$ divides $r-c$. We know that $r-c<r\leq [\frac{m}{2}]$, and since we are in a case of functoriality, $m=n, n\pm 1$.
Thus, $r-c<[\frac{n+1}{2}]$, while $r-c$ is a multiple of $n$, which is impossible. Thus, we need the proof of Cor. \ref{cor 10.2} mainly for the case $c=r$.
\end{proof}

\begin{cor}\label{cor 10.3} 
Let $n>1$. Assume that $c>0$. Then $\varphi_n^{\psi,c}$ is identically zero on $A(\Delta(\tau,m)\gamma_\psi^{\epsilon)},\eta,k_0)$ in the following cases.
\begin{enumerate}
\item In Case (1) of Prop. \ref{prop 3.1} and Case (2) of Prop. \ref{prop 3.2}, for $n$ even, when (in both cases) $m$ is an even multiple of $n$, and $m=2k_0n$.\\
\item In Case (4) of Prop. \ref{prop 3.2}, for $n$ even, when $m-1$ is an even multiple of $n$, and   $m-1=2k_0n$.\\
\item In Case (2) of Prop. \ref{prop 3.2}, for $n$ odd, when $m$ is an even multiple of $n-1$, and $m=2k_0(n-1)$.\\
\item In Case (4) of Prop. \ref{prop 3.2}, for $n$ odd, when $m-1$ is an even multiple of $n-1$, and $m-1=2k_0(n-1)$.\\
\item In Case (2) of Prop. \ref{prop 3.1}, for $H=\Sp_{2nm}^{(2)}$, and Cases (1), (3)  of Prop. \ref{prop 3.2}, for $n$ even, when (in all three cases) $m$ is an odd multiple of $n$, and $m=(2k_0-1)n$.\\
\item In Case (2) of Prop. \ref{prop 3.1}, for $H=\SO_{2nm}$ ($m$ odd), when $m-1$ is an odd multiple of $n$, and $m-1=(2k_0-1)n$.\\
\item In Cases (1), (3) of Prop. \ref{prop 3.2}, for $n$ odd, when $m$ is an odd multiple of $n-1$, and $m=(2k_0-1)(n-1)$.
\end{enumerate}
\end{cor}
\begin{proof}
Note, first, that Prop. \ref{prop 3.3} implies the inequality $m\geq \mu_0n$, $m\geq \mu_0(n-1)$, $m-1\geq \mu_0n$, or $m-1\geq \mu_0(n-1)$ according to the case at hand, where $\mu_0$ is the indicated multiple ($2k_0$ or $2k_0\pm 1$). In all cases of the corollary, the proof shows that if the reverse inequality is satisfied then $c$ must be zero. The proofs for the various cases of the corollary are all similar to the proof of Cor. \ref{cor 10.2}. Let us show (6), for example. Assume that $\varphi_n^{\psi,c}$ is nontrivial on $A(\Delta(\tau,m)\gamma_\psi^{\epsilon)},\eta,k_0)$. In this case, by Prop. \ref{prop 3.1}(2), we have, as in \eqref{10.30},
$$
((2n)^{2c},(2n-1)^{m-2c},1^{m-2c})\leq ((2n)^{m-2k_0+1},n^{4k_0-2}).
$$
This implies that $2c\leq m-2k_0+1$, and, as in \eqref{10.31},
$$
(m-2c)(2n-1)\leq (m-2c-2k_0+1)\cdot (2n)+(2k_0-1)n.
$$
We get that $(2k_0-1)n\leq m-2c$, and since $m$ is odd, $(2k_0-1)n\leq m-1-2c$. Our assumption is that $m-1=(2k_0-1)n$, and so $c=0$.

\end{proof}

Going back to the function $f(\xi)$ in \eqref{9.3}, we conclude that in the cases of the last two corollaries, 
\begin{equation}\label{10.32}
f(\xi)(y^{(n+2)})=f(\xi)(0),
\end{equation}
for all  $y^{n+2}$ with adele coordinates. See \eqref{9.55} and \eqref{9.10.1}. Thus, by \eqref{9.54}, we get, in the cases of the last corollaries,
\begin{multline}\label{10.33}
f(\xi)(y)=\\
\int_{U'_{{(m-2r)}^{n-1}}(F)\backslash
	U'_{{(m-2r)}^{n-1}}(\BA)}\int_{E_{n+2}(F)\backslash E_{n+2}(\BA)}\xi(uvx_r(y,c))\psi^{-1}_{E_{n+2}}(u)\\
\quad\quad\quad\psi^{-1}_{U'_{{(m-2r)}^{n-1}}}(v)dudv.
\end{multline}
Recall that $E_{n+2}$ is the subgroup generated by $E_{n-1}$, $\mathrm{X}_{1,2}^{2n-1,n+2}$, $\mathrm{X}_{1,2}^{2n-1,n+1}$, $\mathrm{X}_{1,2}^{2n-1,n}$, and $\psi_{E_{n+2}}$ is obtained from $\psi_{E_{n-1}}$ by the trivial extension.

\section{Fourier expansions IV, and the conclusion of the proof of cuspidality of $\mathcal{D}\mathcal{D}_\psi(\tau)$ in Theorem \ref{thm 4.1}}

In this section, we consider $f(\xi)$ in the cases of functoriality, as in Cor. \ref{cor 10.2}, and a little more generally in the cases of Cor. \ref{cor 10.3}. Note that Cor. \ref{cor 10.3} contains the cases of functoriality, with $k_0=1$. We will prove

\begin{thm}\label{thm 11.1}
In the cases of Cor. \ref{cor 10.3}, for all $\xi$ and all $y^{(2n+1)}$ with adele coordinates,
\begin{equation}\label{11.1}
f(\xi)(y^{(2n+1)})=f(\xi)(0).
\end{equation}
This means that
$$
f(\xi)(y_1,...,y_{n-1},y_n,y_{n+1},y_{n+2}, y_{n+3},...,y_{2n+1})=f(\xi)(0),
$$
where, for $1\leq i\leq n-1$, or $n+3\leq i\leq 2n+1$, $y_i\in M_{r\times (m-2r)}(\BA)$; $y_n, y_{n+2}\in M_{r\times [\frac{m}{2}]}(\BA)$, and $y_{n+1}\in M_{r\times (2([\frac{m+1}{2}])-r)}(\BA)$. In particular, we get Theorem \ref{thm 9.0}.
\end{thm}
\begin{proof}
We will prove, by induction, that, for all $n+3\leq i \leq 2n+1$, $f(\xi)(y^{(i)})=f(\xi)(0)$. The proofs are similar to the proofs of the previous theorems. We start with $i=n+3$.\\
Consider the subgroup  $U^{(n+3)}$ of $U'_{(m-2r)^{n-1}}$, consisting of elements of the form $diag(I_{(2n-1)r},V, I_{(2n-1)r})$, where $V$ has the form
\begin{equation}\label{11.2}
\begin{pmatrix}I_{(m-2r)(n-2)}&S&\ast&\ast&\ast&\ast&\ast\\
&I_{m-2r}&\ast&S_{n-1}&\ast&\ast&\ast&\\&&I_{[\frac{m}{2}]}&0&0&\ast&\ast\\&&&I_{2([\frac{m+1}{2}]-r)}&0&S'_{n-1}&\ast\\&&&&I_{[\frac{m}{2}]}&\ast&\ast\\&&&&&I_{m-2r}&S'\\&&&&&&I_{(m-2r)(n-2)}\end{pmatrix}.
\end{equation}
Denote the restriction of the character $\psi_{U'_{(m-2r)^{n-1}}}$ to the elements \eqref{11.2} (with adele coordinates) by $\psi_{U^{(n+3)}}$. It is given, as follows. Write $S$ in \eqref{11.2} as a column of $n-2$ matrices of size $(m-2r)\times (m-2r)$ each. Denote the last matrix by $S_{n-2}$. Then $\psi_{U^{(n+3)}}$ assigns the value $\psi(tr(S_{n-2})+tr(S_{n-1}A'_H))$. See \eqref{6.7.1.1}, \eqref{6.7.2}. As in \eqref{9.10.3},  \eqref{9.12}, consider, for a given $h\in H(\BA)$, the following
smooth function on $M_{r\times (m-2r)}(\BA)$,
\begin{multline}\label{11.3}
 f_{n+3}(\xi)(z)=\\
\int_{U^{(n+3)}(F)\backslash
	U^{(n+3)}(\BA)}\int_{E_{n+2}(F)\backslash E_{n+2}(\BA)}\xi(uvx_r(y_{n-1}(z),0))\psi^{-1}_{E_{n+2}}(u)\psi^{-1}_{U^{(n+3)}}(v)dudv,
\end{multline}
where $y_{n-1}(z)=(0,...,0,z,0,...,0)$, and $z$ is in coordinate number $n-1$. As in \eqref{9.13}, we consider the Fourier coefficients of $f_{n+3}$,  
\begin{equation}\label{11.4}
f_{n+3}^{\psi,L}(\xi)=\int_{M_{r\times (m-2r)}(F)\backslash M_{r\times (m-2r)}(\BA)}f_{n+3}(\xi)(z)\psi^{-1}(tr(Lz))dz,
\end{equation}
where $L\in M_{(m-2r)\times r}(F)$. Again, we will show that the coefficient \eqref{11.3} is trivial on $A(\Delta(\tau,m)\gamma_\psi^{\epsilon)},\eta,k)$, for all nonzero $L$. Let $E_{n+3}$ be the subgroup generated by $E_{n+2}$ and $\mathrm{X}_{1,2}^{2n-1,n-1}$. Let $\psi_{E_{n+3},L}$ be the character of $E_{n+3}(\BA)$, which is $\psi_{E_{n+2}}$ on $E_{n+2}(\BA)$, and on $\mathrm{X}_{1,2}^{2n-1,n-1}(\BA)$ is given by $\psi(tr(Lz))$. Then we can rewrite \eqref{11.4} as
\begin{multline}\label{11.5}
f_{n+3}^{\psi,L}(\xi)=\\
\int_{U^{(n+3)}(F)\backslash
	U^{(n+3)}(\BA)}\int_{E_{n+3}(F)\backslash E_{n+3}(\BA)}\xi(uv)\psi^{-1}_{E_{n+3},L}(u)\psi^{-1}_{U^{(n+3)}}(v)dudv=\\
\int_{E^{(n+3)}(F)\backslash
	E^{(n+3)}(\BA)}\xi(v)\psi^{-1}_{E^{(n+3)},L}(v)dv	,
\end{multline}
where $E^{(n+3)}$ is the subgroup generated by $E_{n+3}$ and $U^{(n+3)}$; $\psi_{E^{(n+3)},L}$ is the character of $E^{(n+3)}(\BA)$ obtained from $\psi_{E_{n+3},L}$, $\psi_{U^{(n+3)}} $.

Before we continue, we need to examine representatives for the matrices $L$ under the following action. We consider the conjugation inside \eqref{11.5} by the inverse of
$$
d_{a,b,h}=diag(a^{\Delta_{2n-1}}, b^{\Delta_{n-1}} , I_{[\frac{m}{2}]},h, I_{[\frac{m}{2}]}, (b^*)^{\Delta_{n-1}}, (a^*)^{\Delta_{2n-1}}),
$$
where $a\in \GL_r(F)$, $b\in \GL_{m-2r}(F)$, $h\in H_{2([\frac{m+1}{2}]-r)}(F)$. Changing variables in \eqref{11.5} and replacing $\xi$ by $d_{a,b,h}\cdot\xi$, we may replace $L$ by $b^{-1}La$, while in $\psi_{E^{(n+3)},L}$, $A'_H$ is replaced by $h^{-1}A'_Hb$. See \eqref{6.7.1.1}, \eqref{6.7.2}. 

Recall the element $\omega_0$ in \eqref{6.0'}.
We modify $A'_H$ in case $H_{2nm}$ is orthogonal, so that the matrix $w_2^{r(n-1)}$, which appears in the definition of $A'_H$, is replaced by $I_2$. For this, we need to apply an outer conjugation inside the integral \eqref{11.5}, by the element $\omega_0^{r(n-1)}$, when $r(n-1)$ is odd, and replace the representation $A(\Delta(\tau,m)\gamma_\psi^{(\epsilon)},\eta,k_0)$ by its outer conjugate 
 $A^{\omega_0^{r(n-1)}}(\Delta(\tau,m)\gamma_\psi^{(\epsilon)},\eta,k_0)$, which also satisfies Prop. \ref{prop 3.1}, Prop. \ref{prop 3.2}.

Now, let us require that $h^{-1}A'_Hb=A'_H$, i.e. $A'_Hb=hA'_H$. In case $H_{2nm}$ is symplectic, or  in case it is orthogonal with $m$ even, this means that $b=h$, and hence $b\in H_{m-2r}(F)$. In case $H_{2nm}$ is orthogonal, and $m=2m'-1$ is odd, let
$$
A''_H=\begin{pmatrix}I_{m'-r-1}\\&\frac{1}{2}&1\\&&&I_{m'-r-1}\end{pmatrix}.
$$
Then $A''_HA'_H=I_{m-2r}$, and we get that 
\begin{equation}\label{11.6}
b=A''_HhA'_H.
\end{equation} 
Now, it is easy to check that if  
\begin{equation}\label{11.6.1}
h\begin{pmatrix}0_{m'-r-1}\\1\\-\frac{1}{2}\\0_{m'-r-1}\end{pmatrix}=\begin{pmatrix}0_{m'-r-1}\\1\\-\frac{1}{2}\\0_{m'-r-1}\end{pmatrix},
\end{equation}
then $b$, defined by \eqref{11.6}, lies in $\SO_{m-2r}(F)$. Thus, we will take, in this case, $h\in \SO_{m+1-2r}(F)$, satisfying \eqref{11.6.1}, and $b$ as in \eqref{11.6}. Now, we may take $L$ in a set of representatives for 
 the right action of $\GL_r(F)\times H_{m-2r}(F)$ on $M_{(m-2r)\times r}(F)$ given by $L\cdot (a,b)=b^{-1}La$. Representatives are given exactly as in Lemma \ref{lem 9.3}, except that we replace there $m$ by $m'$, when $m=2m'$ is even. We note that the similar result holds also when $m=2m'-1$ is odd, except that we need to add one more type of a representative, when $\ell_1+\ell_2=m'-r-1$. Here, the representatives are given by
	$$
	L_{\ell_1,\ell_2;\delta}=\begin{pmatrix}I_{\ell_1}&0&0\\0&I_{\ell_2}&0\\0&0&0\\0&\delta
	w_{\ell_2}&0\\0_{\ell_1\times \ell_1}&0&0\end{pmatrix},
	$$
	where $\delta=\frac{1}{2}diag (d_1,...,d_{\ell_2})$ is a diagonal $\ell_2\times\ell_2$ matrix, $0\leq
	\ell_1+\ell_2\leq r, m'-r-1$. The numbers $\ell_1, \ell_2$ are determined uniquely, and the class of the quadratic form $d_1x_1^2+d_2x_2^2+\cdots+d_{\ell_2}x_{\ell_2}^2$ is determined uniquely. \\
Similarly, if $\ell_1+\ell_2=m'-r-1$, we also have the representatives	 
 \begin{equation}\label{11.7}
 \L_{\ell_1;\delta}^1=\begin{pmatrix}I_{\ell_1}&0&0&0\\0&I_{\ell_2}&0&0\\0&0&1&0\\0&\delta w_{\ell_2}&0&0\\0_{\ell_1\times \ell_1}&0&0&0\end{pmatrix}.
 \end{equation}
 We assume that in\eqref{11.5}, $L$ is one of the representatives above.
 Let $\ell$ be defined as follows. If $H_{2nm}$ is symplectic, then $\ell=\ell_1+\ell'_2$. (Recall that in this case, we always assume that $m=2m'$ is even.) If $H_{2nm}$ is orthogonal, then $\ell=\ell_1+\ell_2$, except in case \eqref{11.7}. In this case, $\ell=\ell_1+\ell_2+1=m'-r$. 
Now, we apply an appropriate variant of our proofs of Theorems  \ref{thm 6.1}, \ref{thm 9.1}, \ref{thm 9.7}. We apply a conjugation, in \eqref{11.5}, by a Weyl element, similar to the ones used before, exchange roots, and carry out Fourier expansions. Consider the following Weyl element, similar to \eqref{9.70},	
\begin{equation}\label{11.8}
\epsilon'_0=\begin{pmatrix}B'_1&B'_2&0\\B'_3&0&0\\0&B'_4&0\\0&0&B'_5\\0&B'_6&B'_7\end{pmatrix}.
\end{equation}
The block $B'_1$ (resp.$B'_3$) has the same form as \eqref{9.17} (resp. \eqref{9.19}). The block $B'_4$ is as follows
\begin{equation}\label{11.9}
B'_4=diag(I_{(m-2r)(n-2)},\begin{pmatrix}c\\&I_{2([\frac{m+1}{2}]-r-\ell)}\\&&c\end{pmatrix},I_{(m-2r)(n-2)}),
\end{equation}
where 
$$
c=\begin{pmatrix}0_{(m-2r-\ell+[\frac{m}{2}])\times \ell}&I_{m-2r-\ell+[\frac{m}{2}]}&0_{(m-2r-\ell+[\frac{m}{2}])\times \ell}\end{pmatrix}
$$
The block $B'_2$ has the same block row division as $B'_1$, and the same block column division as $B'_4$. The first $2n-1$ block rows are all zero. The last two block rows of $B'_2$ have the following form
\begin{equation}\label{11.10}
\begin{pmatrix}0_{\ell\times (m-2r)(n-2)}& I_\ell & 0_{\ell\times (m-2r-\ell+[\frac{m}{2}])}& 0&0_{\ell\times ([\frac{m}{2}]+\ell+(m-2r)(n-1))}\\                        0& 0&       0& I_\ell & 0 \end{pmatrix},
\end{equation}
The blocks $B'_5, B'_6, B'_7$ are already determined by the other blocks. We need to prove that, for $\ell>0$, the following integral is identically zero on $A^{\omega_0^{r(n-1)}}(\Delta(\tau,m)\gamma_\psi^{\epsilon)},\eta,k_0)$,
\begin{equation}\label{11.11}
\varphi_{n+3}^{\psi,\ell,\ell_1}(\xi)=
\int_{\mathcal{M}_{n+3,\ell}(F)\backslash
	\mathcal{M}_{n+3,\ell}(\BA)}\xi(v)\psi^{-1}_{\mathcal{M}_{n+3,\ell}}(v)dv, 
\end{equation}
where $\mathcal{M}_{n+3,\ell}=\epsilon'_0E^{(n+3)}(\epsilon'_0)^{-1}$, and $\psi_{\mathcal{M}_{n+3,\ell}}$ is the character of $\mathcal{M}_{n+3,\ell}(\BA)$ defined by $\psi_{\mathcal{M}_{n+3,\ell}}(v)=\psi_{E^{(n+3)},L}((\epsilon'_0)^{-1}v\epsilon'_0)$. (Recall that $L$ is the representative corresponding to $\ell$, $\ell_1$.) The subgroup $\mathcal{M}_{n+3,\ell}$ consists of elements $v$ of the form 
\begin{equation}\label{11.11.1}
v=\begin{pmatrix}U_1&X_{1,2}&X_{1,3}&X_{1,4}&X_{1,5}\\Y_{2,1}&U_2&X_{2,3}&X_{2,4}&X'_{1,4}\\
Y_{3,1}&0&V&X'_{2,3}&X'_{1,3}\\0&0&0&U_2'&X'_{1,2}\\0&0&Y'_{3,1}&Y'_{2,1}&U'_1\end{pmatrix}.
\end{equation}

The blocks are described as follows. The block $U_1$ (resp. $U_2$) has the form \eqref{9.23} (resp. \eqref{9.24}),  
and similarly with $U_1'$ (resp. $U'_2$). The block $V$ has the form
\begin{equation}\label{11.12}
V=\begin{pmatrix}I_{\mu(n-2)}&v_{1,2}&\ast&\ast&\ast&\ast&\ast\\&I_{\mu-\ell}&\ast&v_{2,4}&\ast&\ast&\ast\\&&I_{[\frac{m}{2}]}&0&0&\ast&\ast\\&&&I_{2([\frac{m+1}{2}]-r-\ell)}&0&v'_{2,4}&\ast\\&&&&I_{[\frac{m}{2}]}&\ast&\ast\\&&&&&I_{\mu-\ell}&v'_{1.2}\\&&&&&&I_{\mu(n-2)}\end{pmatrix}.
\end{equation}
Recall that sometimes we abbreviate $m-2r=\mu$. The block $X_{1,2}$, has the same shape as the shape described right after \eqref{9.25}. 
The block $X_{1,3}$ has $2n+1$ block rows, each one containing $\ell$ rows. It has seven block columns with sizes as for $V$. It has the following form,
\begin{equation}\label{11.13}
X_{1,3}=\begin{pmatrix}\ast&\ast&\ast&\ast&\ast&\ast&\ast\\&\vdots&&&&\vdots&&\\\ast&\ast&\ast&\ast&\ast&\ast&\ast\\0&(x_{1,3})_{2n-1,2}&\ast&\ast&\ast&\ast&\ast\\0&0&\ast&\ast&\ast&\ast&\ast\\0&0&0&0&0&(x_{1,3})_{2n+1,6}&\ast\end{pmatrix}.
\end{equation}
The matrix $X'_{1,3}$ has a dual form. The matrix $X_{1,4}$ is a $(2n+1)\times (2n-1)$ matrix of $\ell\times (r-\ell)$ blocks, all of whose blocks are arbitrary.  Similarly, the blocks of $X'_{1,4}$ are arbitrary. (As usual, the blocks above are arbitrary up to the requirement that the matrix $v$ lies in $H$.) The matrix $X_{1,5}$ is a $(2n+1)\times (2n+1)$ matrix of $\ell\times \ell$ blocks, with zero block at the position $(2n+1,1)$. Thus, it has the form

\begin{equation}\label{11.14}
X_{1,5}=\begin{pmatrix}\ast&\ast&\cdots&\ast&\ast\\
&\vdots&&\vdots\\\ast&\ast&\cdots&\ast&\ast\\\ast&\ast&\cdots&\ast&\ast\\0&\ast&\cdots&\ast&\ast\end{pmatrix}.
\end{equation}
The block $X_{2,3}$ is a matrix with $2n-1$ block rows, each one of size $r-\ell$. It has seven block columns, with division as in $V$. It has the form
\begin{equation}\label{11.15}
X_{2,3}=\begin{pmatrix}\ast&\ast&\ast&\ast&\ast&\ast&\ast\\&\vdots&&&&\vdots&&\\\ast&\ast&\ast&\ast&\ast&\ast&\ast\\0&(x_{2,3})_{2n-1,2}&\ast&\ast&\ast&\ast&\ast\end{pmatrix}.
\end{equation}
The matrix $Y_{2,1}$ (as well as $Y'_{2,1}$) has the same shape described right after \eqref{9.28}. Finally, the matrix $Y_{3,1}$ has seven block rows, with the same division as that of $V$. It has $2n+1$ block columns, each one containing $\ell$ columns. It has the form\begin{equation}\label{11.16}
Y_{3,1}=\begin{pmatrix}0&&0&(y_{3,1})_{1,2n}&\ast\\0&\cdots&0&0&\ast\\
0&&0&0&0\\0&&0&0&0\\0&\cdots&0&0&0\\0&&0&0&0\\0&&0&0&0\end{pmatrix}.
\end{equation}
The matrix $Y'_{3,1}$ has a dual shape. Let us write $(y_{3,1})_{1,2n}$ as a column of $n-2$ matrices of size $(m-2r)\times \ell$ each. Denote the last matrix by $(y_{3,1})_{1,2n}^{n-2}$. Similarly, write $v_{1,2}$ in \eqref{11.12} as a column of $n-2$ matrices of size $(m-2r)\times (m-2r-\ell)$ each, and denote the last one by $v_{1,2}^{n-2}$.\\
The character $\psi_{\mathcal{M}_{n+3,\ell}}(x)$ on $\mathcal{M}_{n+3,\ell}(\BA)$ is described in terms of the blocks as in \eqref{9.22} of the forms \eqref{9.23}, \eqref{9.24}, \eqref{11.12} - \eqref{11.16}, as follows. Except when $m=2m'-1$ is odd, and $\ell=m'-r$,
\begin{multline}\label{11.17}
 \psi_{\mathcal{M}_{n+3,\ell}}(v)=\prod_{i=1}^{n-1}\psi(tr(U^1_i)+tr(U^2_i))\psi^{-1}(tr(U^1_{n-1+i})+tr(U^2_{n-1+i}))\cdot \\
 \cdot\psi(tr(U^1_{2n-1}))\psi^{-1}(tr(U^1_{2n}))\psi(tr((y_{3,1})_{1,2n}^{n-2},\ (v_{1,2})^{n-2}))\cdot \\
 \cdot \psi(tr(\begin{pmatrix}v_{2,4},&-\delta_H w_\ell {}^t(x_{1,3})_{2n+1,6}w_{m-2r-\ell}\end{pmatrix}A'_H))\tilde{\psi}(v),
 \end{multline}
where if $H_{2nm}$ is symplectic, then
\begin{equation}\label{11.18}
\tilde{\psi}(v)=\psi((tr((x_{2,3})_{2n-1,2}\begin{pmatrix}0_{(m-2r-2\ell)\times \ell'_2}&0\\I_{\ell'_2}&0\\0&0_{\ell_1\times (r-\ell-\ell'_2)}\end{pmatrix}));
\end{equation}
If $H_{2nm}$ is orthogonal,
\begin{equation}\label{11.19}
\tilde{\psi}(v)=\psi(tr((x_{1,3})_{2n-1,2}\begin{pmatrix}0_{\ell_1\times (m-2r-2\ell)}&0\\0&\delta w_{\ell_2}\\0_{\ell_1\times \ell_1}&0\end{pmatrix})).
\end{equation}
If $H_{2nm}$ is orthogonal, $m=2m'-1$ is odd, and $\ell=m'-r$, 
\begin{multline}\label{11.20}
 \psi_{\mathcal{M}_{n+3,\ell}}(v)=\prod_{i=1}^{n-1}\psi(tr(U^1_i)+tr(U^2_i))\psi^{-1}(tr(U^1_{n-1+i})+tr(U^2_{n-1+i}))\cdot \\
 \cdot\psi(tr(U^1_{2n-1}))\psi^{-1}(tr(U^1_{2n}\begin{pmatrix}I_{m'-r-2}&0&0\\0&1&0\\0&0&\frac{1}{2}\end{pmatrix}))\psi(tr((y_{3,1})_{1,2n}^{n-2},\ (v_{1,2})^{n-2}))\cdot \\
 \cdot \psi(tr(- w_\ell {}^t(x_{1,3})_{2n+1,6}w_{m-2r-\ell}A'_H))
 \psi(tr((x_{1,3})_{2n-1,2}\begin{pmatrix}0&\delta w_{\ell_2}&0_{\ell_2\times 1}\\0_{\ell_1\times \ell_1}&0&0\end{pmatrix})).  
 \end{multline}
Now, we continue as in the proofs of Theorem \ref{thm 6.1}, \ref{thm 9.1},  \ref{thm 9.7}, \ref{10.1}
We exchange $\mathrm{Y}_{2,1}^{1,2}$ with $\mathrm{X}_{1,2}^{1,1}$, $\mathrm{Y}_{2,1}^{2,3}$ with $\mathrm{X}_{1,2}^{2,2}$,  
$\mathrm{Y}_{2,1}^{1,3}$ with $\mathrm{X}_{1,2}^{2,1}$, and so on, until we exchange $\mathrm{Y}_{2,1}^{1,2n-1}$, with $\mathrm{X}_{1,2}^{2n-2,1}$. Next, as we did right after \eqref{10.17}, we may continue the root exchanges, and exchange $\mathrm{Y}_{3,1}^{1,2n}$ with $\mathrm{X}_{1,3}^{2n-1,1}$, $\mathrm{Y}_{2,1}^{2n-1,2n}$ with $\mathrm{X}_{1,2}^{2n-1,2n-1}$, then  $\mathrm{Y}_{2,1}^{2n-2,2n}$ with $\mathrm{X}_{1,2}^{2n-1,2n-2}$,..., $\mathrm{Y}_{2,1}^{1,2n}$ with $\mathrm{X}_{1,2}^{2n-1,1}$. We remark here, that although the character $\psi_{\mathcal{M}_{n+3,\ell}}$ is nontrivial on $Y_{3,1}^{1,2n}(\BA)$, the root exchange works, as usual; the character on $X_{1,3}^{2n-1,1}(\BA)$, after the root exchange, remains trivial. Thus, $\varphi_{n+3}^{\psi,\ell,\ell_1}$ is trivial on $A^{\omega_0^{r(n-1)}}(\Delta(\tau,m)\gamma_\psi^{(\epsilon)},\eta,k_0)$, if and only if the following integral is identically zero on $A^{\omega_0^{r(n-1)}}(\Delta(\tau,m)\gamma_\psi^{(\epsilon)},\eta,k_0)$, 
 \begin{equation}\label{11.21}
 \tilde{\varphi}_{n+3}^{\psi,\ell,\ell_1}(\xi)=
 \int_{\tilde{\mathcal{M}}_{n+3,\ell}(F)\backslash
 	\tilde{\mathcal{M}}_{n+3,\ell}(\BA)}\xi(v)\psi^{-1}_{\tilde{\mathcal{M}}_{n+3,\ell}}(v)dv, 
 \end{equation}
 where $\tilde{\mathcal{M}}_{n+3,\ell}$ is the subgroup of elements $v$ written as in \eqref{11.11.1}, where $X_{1,2}$ is such that its last two block rows are zero, and all of its other blocks are arbitrary, and $Y_{2,1}$ is such that its first $2n$ block columns are zero. (Similarly, in the dual block $X'_{1,2}$, the first two block columns are zero, and all other blocks are arbitrary, provided, of course, that $v$ lies in $H$, and in $Y'_{2,1}$, the last $2n$ block rows are zero.) The shape of all other blocks of $v$ remains as before. Note that $\tilde{\mathcal{M}}_{n+3,\ell}$ lies in the standard parabolic subgroup $Q^H_{\ell^{2n-1}}$.  The character $\psi_{\tilde{\mathcal{M}}_{n+3,\ell}}$ is still given by \eqref{11.17} - \eqref{11.20}. Thus, we want to prove that, for $\ell$ positive, $\tilde{\varphi}_{n+3}^{\psi,\ell,\ell_1}$ is identically zero on $A^{\omega_0^{r(n-1)}}(\Delta(\tau,m)\gamma_\psi^{(\epsilon)},\eta,k)$.  
 Write, as in \eqref{9.36},
 \begin{equation}\label{11.22}
 \tilde{\mathcal{M}}_{n+3,\ell}=\mathcal{M}'_{n+3,\ell}\rtimes\mathcal{M}''_{n+3,\ell},
 \end{equation}
 where $\mathcal{M}'_{n+3,\ell}$ is the intersection of $\tilde{\mathcal{M}}_{n+3,\ell}$ with the unipotent radical $U^H_{\ell^{2n-1}}$ of $Q^H_{\ell^{2n-1}}$, and $\mathcal{M}''_{n+3,\ell}$ is the intersection of $\tilde{\mathcal{M}}_{n+3,\ell}$ with the Levi part $M^H_{\ell^{2n-1}}$ of $Q^H_{\ell^{2n-1}}$. We have
 \begin{equation}\label{11.23}
 \tilde{\varphi}_{n+3}^{\psi,\ell,\ell_1}(\xi)=
 \int_{\mathcal{M}''_{n+3,\ell}(F)\backslash
 	\mathcal{M}''_{n+3,\ell}(\BA)}\int_{\mathcal{M}'_{n+3,\ell}(F)\backslash
 	\mathcal{M}'_{n+3,\ell}(\BA)}\xi(v'v'')\psi^{-1}_{\tilde{\mathcal{M}}_{n+3,\ell}}(v'v'')dv'dv''. 
 \end{equation}
 As in \eqref{9.37}, we consider first the inner $dv'$- integration, where we replace the right $v''$- translate of $\xi$ with $\xi$, that is
 \begin{equation}\label{11.24}
  \tilde{\varphi'}_{n+3}^{\psi,\ell,\ell_1}(\xi)=
 \int_{\mathcal{M}'_{n+3,\ell}(F)\backslash
 	\mathcal{M}'_{n+3,\ell}(\BA)}\xi(v')\psi^{-1}_{\tilde{\mathcal{M}}_{n+3,\ell}}(v')dv'. 
 \end{equation}
 We are now at the same point as in \eqref{9.48} in the proof of Theorem \ref{thm 9.1}, and we get the analog of \eqref{9.51}, with the same notation,
 \begin{equation}\label{11.25}
 \tilde{\varphi'}_{n+3}^{\psi,\ell,\ell_1}(\xi)=(\xi^{U_{\ell^{2n}}})^{\psi_{V_{\ell^{2n}}}}.
\end{equation}
From \eqref{11.23}, 
\begin{equation}\label{11.26}
\tilde{\varphi}_{n+3}^{\psi,\ell,\ell_1}(\xi)=
\int_{\mathcal{M}''_{n+3,\ell}(F)\backslash
	\mathcal{M}''_{n+3,\ell}(\BA)}(\xi^{U_{\ell^{2n}}})^{\psi_{V_{\ell^{2n}}}}(v'')
\psi^{-1}_{\tilde{\mathcal{M}}_{n+3,\ell}}(v'')dv''.
\end{equation}
This integral is identically zero, since we may consider the subgroup of $\mathcal{M}''_{n+3,\ell}(\BA)$ consisting of the elements \eqref{9.48} $v''(y_1)=\zeta_\ell(y_1,0,...,0)\in U_{\ell^{2n}}(\BA)$, with $y_1\in M_\ell(\BA)$. We note that $\psi_{\tilde{\mathcal{M}}_{n+3,\ell}}(v''(y_1))=\psi(tr(y_1))$. It now follows that \eqref{11.26} is identically zero. This completes the first step of our induction, namely that in the notation of the theorem
\begin{equation}\label{11.27}
f(\xi)(y^{(n+3)})=f(\xi)(0),
\end{equation} 
 for all $y^{(n+3)}$ with adele coordinates, and all $\xi\in A(\Delta(\tau,m)\gamma_\psi^{(\epsilon)},\eta,k)$.
 
 Assume by induction that, for a given $n+4\leq i\leq 2n+1$, $f(\xi)(y^{(i-1)})=f(\xi)(0)$, for all $y^{(i-1)}$ with adele coordinates, and all $\xi\in A(\Delta(\tau,m)\gamma_\psi^{(\epsilon)},\eta,k_0)$. We need to prove that $f(\xi)(y^{(i)})=f(\xi)(0)$, for all $y^{(i)}$ with adele coordinates, and all $\xi\in A(\Delta(\tau,m)\gamma_\psi^{(\epsilon)},\eta,k_0)$. The proof  is, once again, similar to that of \eqref{11.27} and the previous theorems.
 
 Consider the subgroup  $U^{(i)}$ of $U'_{(m-2r)^{n-1}}$, consisting of elements of the form $diag(I_{(2n-1)r},V, I_{(2n-1)r})$, where $V$ has the form
\begin{equation}\label{11.28}
\begin{pmatrix}I_{(m-2r)(2n+1-i)}&S&\ast&\ast&\ast&\ast&\ast\\
&I_{m-2r}&S_{2n+2-i}&\ast&\ast&\ast&\ast\\&&I_{m-2r}&0&0&\ast&\ast\\&&&I&0&\ast&\ast\\&&&&I_{m-2r}&S'_{2n+2-i}&\ast\\&&&&&I_{m-2r}&S'\\&&&&&&I_{(m-2r)(n-2)}\end{pmatrix},
\end{equation}
where the middle identity block is of size $2((i-n-4)(m-2r)+(m-r))$. Denote the restriction of the character $\psi_{U'_{(m-2r)^{n-1}}}$ to the elements \eqref{11.28} (with adele coordinates) by $\psi_{U^{(i)}}$. It is given, as follows. Write $S$ in \eqref{11.28} as a column of $2n+1-i$ matrices of size $(m-2r)\times (m-2r)$ each. Denote the last matrix by $S_{2n+1-i}$. Then $\psi_{U^{(i)}}$ assigns the value $\psi(tr(S_{2n+1-i})+tr(S_{2n+2-i}))$.  Consider the following
smooth function on $M_{r\times (m-2r)}(\BA)$,
\begin{multline}\label{11.29}
 f_i(\xi)(z)=\\
\int_{U^{(i)}(F)\backslash
	U^{(i)}(\BA)}\int_{E_{i-1}(F)\backslash E_{i-1}(\BA)}\xi(uvx_r(y_{2n+2-i}(z),0))\psi^{-1}_{E_{i-1}}(u)\psi^{-1}_{U^{(i)}}(v)dudv,
\end{multline}
where $y_{2n+2-i}(z)=(0,...,0,z,0,...,0)$, and $z$ is in coordinate number $2n+2-i$. Now, we consider the Fourier coefficients of $f_i(\xi)$,  
\begin{equation}\label{11.30}
f_{i}^{\psi,L}(\xi)=\int_{M_{r\times (m-2r)}(F)\backslash M_{r\times (m-2r)}(\BA)}f_{i}(\xi)(z)\psi^{-1}(tr(Lz))dz,
\end{equation}
where $L\in M_{(m-2r)\times r}(F)$. Again, we will show that the coefficient \eqref{11.30} is trivial on $A(\Delta(\tau,m)\gamma_\psi^{\epsilon)},\eta,k)$, for all nonzero $L$. Consider the unipotent subgroup $E_i$. This is the subgroup generated by $E_{i-1}$ and $\mathrm{X}_{1,2}^{2n-1,2n+2-i}$. Let $\psi_{E_i,L}$ be the character of $E_i(\BA)$, which is $\psi_{E_{i-1}}$ on $E_{i-1}(\BA)$, and on $\mathrm{X}_{1,2}^{2n-1,2n+2-i}(\BA)$ is given by $\psi(tr(Lz))$. Then we can rewrite \eqref{11.30} as
\begin{multline}\label{11.31}
f_i^{\psi,L}(\xi)=
\int_{U^{(i)}(F)\backslash
	U^{(i)}(\BA)}\int_{E_i(F)\backslash E_i(\BA)}\xi(uv)\psi^{-1}_{E_i,L}(u)\psi^{-1}_{U^{(i)}}(v)dudv=\\
=\int_{E^{(i)}(F)\backslash
	E^{(i)}(\BA)}\xi(v)\psi^{-1}_{E^{(i)},L}(v)dv	,
\end{multline}
where $E^{(i)}$ is the subgroup generated by $E_i$ and $U^{(i)}$; $\psi_{E^{(i)},L}$ is the character of $E^{(i)}(\BA)$ obtained from $\psi_{E_i,L}$, $\psi_{U^{(i)}} $.

Assume that $L$ is of rank $\ell>0$. As in \eqref{9.15}, we may assume that 
\begin{equation}\label{11.31.1}
L=A_\ell=\begin{pmatrix}I_\ell&0_{\ell\times (r-\ell)}\\0_{(m-2r-\ell)\times\ell}&0\end{pmatrix}.
\end{equation}
Denote $f_i^{\psi,A_\ell} =f_i^{\psi,\ell}$. As in \eqref{11.8}, we apply a conjugation, in \eqref{11.31}, by the following Weyl element,	
\begin{equation}\label{11.32}
\epsilon_0=\begin{pmatrix}B_1&B_2&0\\B_3&0&0\\0&B_4&0\\0&0&B_5\\0&B_6&B_7\end{pmatrix}.
\end{equation}
The block $B_1$ (resp.$B_3$) has the same form as \eqref{9.17} (resp. \eqref{9.19}). The block $B_4$ is as follows (recall that $\mu=m-2r$)
\begin{equation}\label{11.33}
B_4=diag(I_{\mu(2n+1-i)},\begin{pmatrix}c\\&c\\&&I_{2(\mu (i-n-4)+(m-r))}\\&&&c'\\&&&&c'\end{pmatrix},I_{\mu (2n+1-i)}),
\end{equation}
where 
$$
c=(0_{(\mu-\ell)\times \ell},I_{\mu-\ell}), \quad c'=(I_{\mu-\ell},0_{(\mu-\ell)\times \ell}).
$$
The block $B_2$ has the same block row division as $B_1$, and the same block column division as $B_4$. The first $2n-1$ block rows are all zero. The last two block rows of $B_2$ have the form
\begin{equation}\label{11.34}
\begin{pmatrix}0_{\ell\times \mu(2n+1-i)}& I_\ell & 0_{\ell\times (\mu-\ell)}& 0&0_{\ell\times (\mu(i-4)+2(m-r)-\ell)}\\                        
0& 0& 0&  I_\ell & 0 \end{pmatrix},
\end{equation}
The blocks $B_5, B_6, B_7$ are determined by the other blocks. We need to prove that, for $\ell>0$, the following integral is identically zero on $A(\Delta(\tau,m)\gamma_\psi^{\epsilon)},\eta,k_0)$,
\begin{equation}\label{11.35}
\varphi_i^{\psi,\ell}(\xi)=
\int_{\mathcal{M}_{i,\ell}(F)\backslash
	\mathcal{M}_{i,\ell}(\BA)}\xi(v)\psi^{-1}_{\mathcal{M}_{i,\ell}}(v)dv, 
\end{equation}
where $\mathcal{M}_{i,\ell}=\epsilon_0E^{(i)}(\epsilon_0)^{-1}$, and $\psi_{\mathcal{M}_{i,\ell}}$ is the character of $\mathcal{M}_{i,\ell}(\BA)$ defined by $\psi_{\mathcal{M}_{i,\ell}}(v)=\psi_{E^{(i)},A_\ell}((\epsilon_0)^{-1}v\epsilon_0)$. The subgroup $\mathcal{M}_{i,\ell}$ consists of elements $v$ of the form 
\begin{equation}\label{11.36}
v=\begin{pmatrix}U_1&X_{1,2}&X_{1,3}&X_{1,4}&X_{1,5}\\Y_{2,1}&U_2&X_{2,3}&X_{2,4}&X'_{1,4}\\
Y_{3,1}&0&V&X'_{2,3}&X'_{1,3}\\0&0&0&U_2'&X'_{1,2}\\0&0&Y'_{3,1}&Y'_{2,1}&U'_1\end{pmatrix}.
\end{equation}
The blocks are described as follows. The block $U_1$ (resp. $U_2$) has the form \eqref{9.23} (resp. \eqref{9.24}),  
and similarly with $U_1'$ (resp. $U'_2$). The block $V$ has the form
\begin{equation}\label{11.37}
V=\begin{pmatrix}I_{\mu(2n+1-i)}&v_{1,2}&\ast&\ast&\ast&\ast&\ast\\&I_{\mu-\ell}&v_{2,3}&\ast&\ast&\ast&\ast\\&&I_{\mu-\ell}&0&0&\ast&\ast\\&&&I_{2((i-n-4)\mu-(m-r))}&0&\ast&\ast\\&&&&I_{\mu-\ell}&v'_{2,3}&\ast\\&&&&&I_{\mu-\ell}&v'_{1,2}\\&&&&&&I_{\mu(2n+1-i)}\end{pmatrix}.
\end{equation}
The block $X_{1,2}$, has the same shape as the shape described right after \eqref{9.25}. 
The block $X_{1,3}$ has $2n+1$ block rows, each one containing $\ell$ rows. It has seven block columns with sizes as for $V$ in \eqref{11.37}, and then it has the form as in \eqref{11.13}, except that we won't need to pay special attention to $(x_{1,3})_{2n-1,2}$ or $(x_{1,3})_{2n+1,6}$, as in \eqref{11.13}. The matrix $X_{1,4}$ is a $(2n+1)\times (2n-1)$ matrix of $\ell\times (r-\ell)$ blocks, all of whose blocks are arbitrary.  Similarly, the blocks of $X'_{1,4}$ are arbitrary, up to the requirement that the matrix $v$ lies in $H$. The matrix $X_{1,5}$ has the form
as in \eqref{11.14}. The block $X_{2,3}$ is as in \eqref{11.15}. It has seven block columns with the same division as for $V$ in \eqref{11.37}. We don't need to pay special attention to $(x_{2,3})_{2n-1,2}$, as in \eqref{11.15}.
The matrix $Y_{2,1}$ (as well as $Y'_{2,1}$) has the same shape described right after \eqref{9.28}, and the matrix $Y_{3,1}$ has the form
as in \eqref{11.16}. It has seven block rows with the same division as for $V$ in \eqref{11.37}. Here we do need to pay special attention to $(y_{3,1})_{1,2n}$, as in \eqref{11.16}. The matrix $Y'_{3,1}$ has a dual shape. Let us write $(y_{3,1})_{1,2n}$ as a column of $2n+1-i$ matrices of size $(m-2r)\times \ell$ each. Denote the last matrix by $(y_{3,1})_{1,2n}^{2n+1-i}$. Similarly, write $v_{1,2}$ in \eqref{11.37} as a column of $2n+1-i$ matrices of size $(m-2r)\times (m-2r-\ell)$ each, and denote the last one by $v_{1,2}^{2n+1-i}$.\\
The character $\psi_{\mathcal{M}_{i,\ell}}(x)$ on $\mathcal{M}_{i,\ell}(\BA)$ is described in terms of the blocks as in \eqref{11.36} of the forms \eqref{9.23}, \eqref{9.24}, \eqref{11.36}, \eqref{11.37}, by 
\begin{multline}\label{11.38}
 \psi_{\mathcal{M}_{n+3,\ell}}(v)=\prod_{i=1}^{n-1}\psi(tr(U^1_i)+tr(U^2_i))\psi^{-1}(tr(U^1_{n-1+i})+tr(U^2_{n-1+i}))\cdot \\
 \cdot\psi(tr(U^1_{2n-1}))\psi^{-1}(tr(U^1_{2n}))\psi(tr((y_{3,1})_{1,2n}^{2n+1-i},\ (v_{1,2})^{2n+1-i}))\psi(tr(v_{2,3})).
  \end{multline}

Now, we continue as in the proofs of Theorem \ref{thm 6.1}, \ref{thm 9.1},  \ref{thm 9.7}, \ref{10.1} and as in the first case of our induction.
We perform exactly the same root exchanges: $\mathrm{Y}_{2,1}^{1,2}$ with $\mathrm{X}_{1,2}^{1,1}$, $\mathrm{Y}_{2,1}^{2,3}$ with $\mathrm{X}_{1,2}^{2,2}$, until we exchange $\mathrm{Y}_{2,1}^{1,2n-1}$, with $\mathrm{X}_{1,2}^{2n-2,1}$. We continue, as in the first case of induction, and exchange $\mathrm{Y}_{3,1}^{1,2n}$ with $\mathrm{X}_{1,3}^{2n-1,1}$, $\mathrm{Y}_{2,1}^{2n-1,2n}$ with $\mathrm{X}_{1,2}^{2n-1,2n-1}$, ..., $\mathrm{Y}_{2,1}^{1,2n}$ with $\mathrm{X}_{1,2}^{2n-1,1}$. Thus, $\varphi_{i}^{\psi,\ell}$ is trivial on $A(\Delta(\tau,m)\gamma_\psi^{(\epsilon)},\eta,k_0)$, if and only if the following integral is identically zero on $A(\Delta(\tau,m)\gamma_\psi^{(\epsilon)},\eta,k_0)$, 
 \begin{equation}\label{11.39}
 \tilde{\varphi}_{i}^{\psi,\ell}(\xi)=
 \int_{\tilde{\mathcal{M}}_{i,\ell}(F)\backslash
 	\tilde{\mathcal{M}}_{i,\ell}(\BA)}\xi(v)\psi^{-1}_{\tilde{\mathcal{M}}_{i,\ell}}(v)dv, 
 \end{equation}
 where $\tilde{\mathcal{M}}_{i,\ell}$ is the subgroup of elements $v$ written as in \eqref{11.36}, where $X_{1,2}$ is such that its last two block rows are zero, and all of its other blocks are arbitrary, and $Y_{2,1}$ is such that its first $2n$ block columns are zero. (Similarly, in the dual block $X'_{1,2}$, the first two block columns are zero, and all other blocks are arbitrary, provided that $v$ lies in $H$, and in $Y'_{2,1}$, the last $2n$ block rows are zero.) The shape of all other blocks of $v$ remains as before. Note that $\tilde{\mathcal{M}}_{i,\ell}$ lies in the standard parabolic subgroup $Q^H_{\ell^{2n-1}}$.  The character $\psi_{\tilde{\mathcal{M}}_{i,\ell}}$ is still given by \eqref{11.38}. Thus, we want to prove that, for $\ell$ positive, $\tilde{\varphi}_{i}^{\psi,\ell}$ is identically zero on $A(\Delta(\tau,m)\gamma_\psi^{(\epsilon)},\eta,k_0)$.  
 Write, as in \eqref{9.36}, \eqref{11.22},
 \begin{equation}\label{11.40}
 \tilde{\mathcal{M}}_{i,\ell}=\mathcal{M}'_{i,\ell}\rtimes\mathcal{M}''_{i,\ell},
 \end{equation}
 where $\mathcal{M}'_{i,\ell}$ is the intersection of $\tilde{\mathcal{M}}_{i,\ell}$ with the unipotent radical $U^H_{\ell^{2n-1}}$ of $Q^H_{\ell^{2n-1}}$, and $\mathcal{M}''_{i,\ell}$ is the intersection of $\tilde{\mathcal{M}}_{i,\ell}$ with the Levi part $M^H_{\ell^{2n-1}}$ of $Q^H_{\ell^{2n-1}}$. We have, as in \eqref{11.24},
 \begin{equation}\label{11.41}
 \tilde{\varphi}_{i}^{\psi,\ell}(\xi)=
 \int_{\mathcal{M}''_{i,\ell}(F)\backslash
 	\mathcal{M}''_{i,\ell}(\BA)}\int_{\mathcal{M}'_{i,\ell}(F)\backslash
 	\mathcal{M}'_{i,\ell}(\BA)}\xi(v'v'')\psi^{-1}_{\tilde{\mathcal{M}}_{i,\ell}}(v'v'')dv'dv''. 
 \end{equation}
 As in \eqref{9.37}, \eqref{11.24}, we consider first the inner $dv'$- integration, where we replace the right $v''$- translate of $\xi$ with $\xi$, that is
 \begin{equation}\label{11.42}
  \tilde{\varphi'}_{i}^{\psi,\ell}(\xi)=
 \int_{\mathcal{M}'_{i,\ell}(F)\backslash
 	\mathcal{M}'_{i,\ell}(\BA)}\xi(v')\psi^{-1}_{\tilde{\mathcal{M}}_{i,\ell}}(v')dv'. 
 \end{equation}
 We are now at the same point as in \eqref{9.48} in the proof of Theorem \ref{thm 9.1}, and we get the analog of \eqref{9.51} (as well as \eqref{11.25}), with the same notation,
 \begin{equation}\label{11.43}
 \tilde{\varphi'}_{i}^{\psi,\ell}(\xi)=(\xi^{U_{\ell^{2n}}})^{\psi_{V_{\ell^{2n}}}}.
\end{equation}
As in \eqref{11.26}, we get that the integral \eqref{11.41} is identically zero on\\
 $A(\Delta(\tau,m)\gamma_\psi^{(\epsilon)},\eta,k_0)$. this completes the proof of Theorem \ref{thm 11.1} (and of Theorem \ref{thm 9.0}).
 
\end{proof}

{\bf Proof of \eqref{10.26}:} Since there is nothing to prove when $c=r$, we assume that $c<r$. Consider \eqref{10.24}, and denote by $\mathrm{X}^{2n-1}_{2,3}$ the subgroup of $\mathrm{X}_{2,3}$, in the notation there, generated by $\mathrm{X}^{2n-1,j}_{2,3}$, for $1\leq j\leq 2n+1$. In the notation of \eqref{10.24}, consider the following function on $\mathrm{X}^{2n-1}_{2,3}(F)\backslash \mathrm{X}^{2n-1}_{2,3}(\BA)$,
\begin{equation}\label{11.44}
  x\mapsto \tilde{\varphi}_n^{\psi,c}(\rho(x)\xi)=\int_{\tilde{U}(F)\backslash \tilde{U}(\BA)}
(\xi^{U_{c^{2n}}})^{\psi_{V_{c^{2n}}}}(ux)\tilde{\psi}^{-1}(u)du.
\end{equation}
Write an element of
$\mathrm{X}^{2n-1}_{2,3}$ in the form
\begin{multline}\label{11.45}
x'_{r-c}(b,z)=\\
diag(I_{2nc}, I_{(2n-2)(r-c)}, \begin{pmatrix}I_{r-c}&b&z\\&I_{2(\mu(n-1)+(m-r-c))}&b'\\&&I_{r-c}\end{pmatrix}, I_{(2n-2)(r-c)},I_{2nc}).
\end{multline}
As in \eqref{9.10}, write $b$ in \eqref{11.45} as
\begin{equation}\label{11.46}
b=(b_1,b_2,...,b_{2n+1}),
\end{equation}
where for $1\leq i\leq n-1$, or $n+3\leq i\leq 2n+1$, $b_i\in M_{(r-c)\times \mu}$, $b_n, b_{n+2}\in M_{(r-c)\times ([\frac{m}{2}]-c)}$, and $b_{n+1}\in M_{(r-c)\times (2([\frac{m+1}{2}]-r))}$. Denote, as in \eqref{9.10.1}, for $1\leq i\leq 2n+1$,
$$
b^{(i)}=(0,...,0,b_{2n+2-i},...,b_{2n+1}).
$$
By the definition of the character $\tilde{\psi}$ in \eqref{10.24}, it follows that \eqref{11.44}, as a function of $x'_{r-c}(b,z)$ is independent of $z$. Denote the function \eqref{11.44} by $\tilde{f}_n^{\psi,c}(\xi)(b)$. By the definition of the character $\tilde{\psi}$ in \eqref{10.24}, it follows that 
\begin{equation}\label{11.47}
\tilde{f}_n^{\psi,c}(\xi)(b^{(n+2)})=\tilde{f}_n^{\psi,c}(\xi)(0),
\end{equation}
for all $b$ in adelic coordinates and all $\xi\in A(\Delta(\tau,m)\gamma_\psi^{(\epsilon)},\eta,k)$. In order to prove \eqref{10.26}, we need to show the analog of Theorem \ref{11.1}, namely that $\tilde{f}_n^{\psi,c}(\xi)(b^{(2n+1)})=\tilde{f}_n^{\psi,c}(\xi)(0)$. For this, we repeat the steps of the proof of Theorem \ref{thm 11.1}, starting with $i=n+3$ and continuing by induction on $i$, for $n+3\leq i\leq 2n+1$. We repeat these steps for the  functions $(\xi^{U_{c^{2n}}})^{\psi_{V_{c^{2n}}}}$ viewed as automorphic functions on $H_{2n(m-2c)}^{(\epsilon)}(\BA)$, by identifying $H_{2n(m-2c)}$ as a subgroup of $H_{2mn}$ through
$h\mapsto diag(I_{2nc},h,I_{2nc})$. We remark here that the proof of Theorem \ref{thm 11.1} is valid also for $k\neq k_0$, provided we know that the function $f(\xi)$ satisfies \eqref{10.32}, which we have here, due to \eqref{11.47}. Thus, here $k$ need not be $k_0$. Let us sketch now, for example, the induction step parallel to the induction step, right after \eqref{11.29}. Thus,
assume by induction that, for a given $n+4\leq i\leq 2n+1$, $\tilde{f}_n^{\psi,c}(\xi)(b^{(i-1)})=\tilde{f}_n^{\psi,c}(\xi)(0)$, for all $b^{(i-1)}$ with adele coordinates, and all $\xi\in A(\Delta(\tau,m)\gamma_\psi^{(\epsilon)},\eta,k)$. Consider the subgroup  $\tilde{U}^{(i)}$ of consisting of elements of the form $diag(I_{2nc+(2n-1)(r-c)},V, I_{2nc+(2n-1)(r-c)})$, where $V$ has the form \eqref{11.28}, except that the middle identity block should be of size $2(i-n-4)+2(m-r-c)$. We also have the character $\psi_{\tilde{U}^{(i)}}$, analogous to the character $\psi_{U^{(i)}}$, defined right after \eqref{11.28}. Let $\tilde{E}_{i-1}$ be the subgroup generated by the subgroup $\tilde{U}$ described in \eqref{10.25} and the subgroups (in the notation there) $\mathrm{X}_{2,3}^{2n-1,2n+2-i},...,\mathrm{X}_{2,3}^{2n-1,n-1}$. Denote by $\psi_{\tilde{E}_{i-1}}$ the character of $\tilde{E}_{i-1}(\BA)$ obtained from $\tilde{\psi}$ in \eqref{10.24} by the trivial extension. Then we have the function on $M_{r\times (m-2r)}(\BA)$, analogous to \eqref{11.29},
\begin{multline}\label{11.48}
 \tilde{f}_{n,i}^{\psi,c}(\xi)(z)=\\
\int_{\tilde{U}^{(i)}(F)\backslash
	\tilde{U}^{(i)}(\BA)}\int_{\tilde{E}_{i-1}(F)\backslash \tilde{E}_{i-1}(\BA)}(\xi^{U_{c^{2n}}})^{\psi_{V_{c^{2n}}}}(uvx'_{r-c}(b_{2n+2-i}(z),0))\\
	\psi^{-1}_{\tilde{E}_{i-1}}(u)\psi^{-1}_{\tilde{U}^{(i)}}(v)dudv,
\end{multline}
where $b_{2n+2-i}(z)=(0,...,0,z,0,...,0)$, and $z$ is in coordinate number $2n+2-i$. Now, we consider the Fourier coefficients of $ \tilde{f}_{n,i}^{\psi,c}(\xi)$, as we did in \eqref{11.30} -  \eqref{11.31}. Again, it is enough to consider the following Fourier coefficients, for $0\leq \ell\leq r, m-2r$,
\begin{equation}\label{11.49}
 \tilde{f}_{n,i}^{\psi,c,\ell}(\xi)=\int_{M_{r\times (m-2r)}(F)\backslash M_{r\times (m-2r)}(\BA)} \tilde{f}_{n,i}^{\psi,c}(\xi)(z)\psi^{-1}(tr(A_\ell z))dz,
\end{equation}
where $A_\ell$ is the matrix in \eqref{11.31.1}. Now, we follow the same steps as in \eqref{11.32} - \eqref{11.43}, and we get that if $\ell>0$, then $ \tilde{f}_{n,i}^{\psi,c}$ is identically zero on $A(\Delta(\tau,m)\gamma_\psi^{(\epsilon)},\eta,k)$. The main point to observe is the point right after \eqref{11.42}, where said that we reached the same point as in \eqref{9.48}. Indeed, the proofs of the analogs of \eqref{9.50} and \eqref{9.50.1} in this case are exactly the same. We still get orbits for $A(\Delta(\tau,m)\gamma_\psi^{(\epsilon)},\eta,k)$, which correspond to partitions of the forms $((4n)^{\ell_0},...)$, $((4n+1)^{\ell_0},...)$, respectively, with $\ell_0>0$. This proves \eqref{10.26}

Let us go back to Theorem \ref{thm 8.10}. In the next theorem we use the notation of Theorem \ref{thm 8.10} and \eqref{9.1}. As a corollary, we obtain

\begin{thm}\label{thm 11.2}
In the cases of Cor. \ref{cor 10.3}, for all $h\in H(\BA)$, $\xi\in A(\Delta(\tau,m)\gamma_\psi^{(\epsilon)},\eta,k_0)$ and 
$r\leq [\frac{m}{2}]$,
\begin{multline}\label{11.50}
\mathcal{F}^r_\psi(\xi)(h)=\\
\int_{\mathbf{Y}(\BA)} \int_{U'_{{(m-2r)}^{n-1}}(F)\backslash
	U'_{{(m-2r)}^{n-1}}(\BA)}(\xi^{U_{(2n-1)r}})^{\psi_{V_{r^{2n-1}}}}(uyw_0h)\psi^{-1}_{U'_{{(m-2r)}^{n-1}}}(u)dudy.
\end{multline}
We conclude that in cases (3) - (7) of Cor. \ref{cor 10.3}, with $k_0=1$, for all $0<r\leq [\frac{m}{2}]$, $h\in H(\BA)$ and $\xi\in A(\Delta(\tau,m)\gamma_\psi^{(\epsilon)},\eta,1)$,
\begin{equation}\label{11.51}
\mathcal{F}^r_\psi(\xi)(h)=0.
\end{equation}
In particular, in the cases of functoriality for $\tau$ and $H_m^{(\epsilon)}$ (the four cases before Theorem \ref{thm 2.2}), $D^1_\psi(\tau)$ is a cuspidal module over $H^{(\epsilon)}_m(\BA)$ (see Theorem \ref{thm 4.1}).
\end{thm}
\begin{proof}
By Theorem \ref{thm 8.10}, we have, with the notation there, 
\begin{equation}\label{11.52}
\mathcal{F}_\psi^r(\xi)(h)=\int_{\mathbf{Y}(\BA)}\epsilon'_{n,m,r}(\xi)(yw_0h)dy.
\end{equation}
Recall from Theorem \ref{thm 6.1} that
\begin{equation}\label{11.52}
\epsilon'_{n,m,r}(\xi)(h)=\int_{\mathcal{D'}_{n,m,r}(F)\backslash
	\mathcal{D'}_{n,m,r}(\BA)}\xi(vh)\psi_{\mathcal{D'}_{n,m,r}}^{-1}(v)dv,
\end{equation}
where $\mathcal{D}'_{n,m,r}=\mathrm{N}\widetilde{\mathrm{X}}$ is defined in \eqref{6.8'}, and the unipotent subgroups $\mathrm{N},\ \widetilde{\mathrm{X}}$ are defined in \eqref {6.8}, \eqref{6.8''}. Recall from \eqref{9.3} that we denoted, for $\xi\in A(\Delta(\tau,m)\gamma_\psi^{(\epsilon)},\eta, k)$, $f(\xi)(y)=\epsilon'_{n,m,r}(\xi)(x_r(y,c))$, where $x_r(y,c)$ is defined in \eqref{9.2}. By Theorem \ref{thm 11.1}, we get that, for all $r\leq [\frac{m}{2}]$, $\xi\in A(\Delta(\tau,m)\gamma_\psi^{(\epsilon)},\eta, k_0), h\in H(\BA)$, and $x_r(y,c)$ with adele coordinates,
\begin{equation}\label{11.53}
 \epsilon'_{n,m,r}(\xi)(x_r(y,c)h)=\epsilon'_{n,m,r}(\xi)(h).
 \end{equation}
 Now note that the group generated by $\widetilde{\mathrm{X}}$ and the subgroup of elements $x_r(y,c)$ is the unipotent radical $U_{(2n-1)r}+U_{(2n-1)r}^{H_{2nm}}$. By \eqref{11.53}, we see that in \eqref{11.52} we may replace $\mathcal{D}'_{n,m,r}$ by $\mathrm{N}U_{(2n-1)r}$ (we take the measure of $F\backslash \BA$ to be $1$). Also, the character $\psi_{\mathcal{D}'_{n,m,r}}$ is extended to be trivial on the elements $x_r(y,c)$. Note that this extended character is trivial on $U_{(2n-1)r}(\BA)$, and writing $\mathrm{N}$ as the direct product of $\hat{V}_{r^{2n-1}}$ and $U'_{(m-2r)^{n-1}}$, we see that the restrictions of $\psi_{\mathcal{D}'_{n,m,r}}$ to $\hat{V}_{r^{2n-1}}(\BA)$ and $U'_{(m-2r)^{n-1}}(\BA)$ are $\psi_{\hat{V}_{r^{2n-1}}}$ and $\psi_{U'_{(m-2r)^{n-1}}}$, respectively. (See \eqref{1'.15}, \eqref{1.10.3.1} for notation.) Thus, we may rewrite \eqref{11.52} as
 \begin{multline}\label{11.54}
 \epsilon'_{n,m,r}(\xi)(h)=\\
 \int_{U'_{(m-2r)^{n-1}}(F)\backslash U'_{(m-2r)^{n-1}}(\BA)}\int_{V_{r^{2n-1}}(F)\backslash V_{r^{2n-1}}(\BA)}\int_{U_{(2n-1)r}(F)\backslash U_{(2n-1)r}(\BA)}\xi(u\hat{v}u'h)\\
 \psi_{\hat{V}_{r^{2n-1}}}^{-1}(v)\psi_{U'_{(m-2r)^{n-1}}}^{-1}(u')dudvdu'=\\
  \int_{U'_{(m-2r)^{n-1}}(F)\backslash U'_{(m-2r)^{n-1}}(\BA)}(\xi^{U_{(2n-1)r}})^{\psi_{\hat{V}_{r^{2n-1}}}}(u'h)\psi_{U'_{(m-2r)^{n-1}}}^{-1}(u')du'.
  \end{multline}
 Substituting \eqref{11.54} in \eqref{11.51}, we get \eqref{11.50}. Assume that $k_0=1$ and we are in Cases (3) - (7) of Cor. \ref{cor 10.3}. Now, if $\mathcal{F}^r_\psi$ is nontrivial  on $A(\Delta(\tau,m)\gamma_\psi^{(\epsilon)},\eta, 1)$, then the constant term $\xi^{U_{(2n-1)r}}$ must be nontrivial on $A(\Delta(\tau,m)\gamma_\psi^{(\epsilon)},\eta, 1)$. Since $\tau$ is cuspidal, $(2n-1)r$ must be divisible by $n$, and hence $r$ is divisible by $n$. Since $r\leq [\frac{m}{2}]$, we see that in Cases (3) - (7) of Cor. \ref{cor 10.3}, with $k_0=1$, this forces $r=0$. This proves the theorem.                                      
\end{proof}

At this point, we have completed the proof of the cuspidality part of the double descent in Theorem \ref{thm 2.4}. Recall that right after the statement of the theorem, we explained how to prove this cuspidality, once we have Theorem \ref{thm 4.1}. In order to complete the proof of Theorem \ref{thm 2.4}, it remains to prove the unramified correspondence at almost all places, namely that
for each irreducible, automorphic, cuspidal representation $\sigma$ in \eqref{2.4}, $\sigma_v$ lifts  $\tau_v$, for almost all finite places $v$, where $\sigma_v$ and $\tau_v$ are unramified. This we do in the remaining two sections of this paper.

\section{The correspondence at unramified places I}

Our next and final goal in this paper is to prove
\begin{thm}\label{thm 12.1}
Let $\tau$ and $H^{(\epsilon)}_m$ be as in Theorem \ref{thm 2.2}. Consider the direct sum decomposition \eqref{2.2.1} of the double descent $\mathcal{D}\mathcal{D}_\psi(\tau)$ into irreducible, automorphic, cuspidal representations of $H^{(\epsilon)}_m(\BA)\times H^{(\epsilon)}_m(\BA)$
\begin{equation}\label{12.1}
\mathcal{D}\mathcal{D}_\psi(\tau)=\oplus_\sigma (\sigma\otimes \bar{\sigma}^\iota).
\end{equation}
Then each representation $\sigma$ appearing in \eqref{12.1} lifts at almost all finite unramified places to $\tau$.
\end{thm}

Let us explain our strategy of proof in the four cases listed before Theorem \ref{thm 2.2}. Let $\sigma$ be a representation appearing in \eqref{12.1}, and let $v$ be a finite place of $F$. Then
\begin{equation}\label{12.2}
\Hom_{H^{(\epsilon)}_m(F_v)\times H^{(\epsilon)}_m(F_v)}(\sigma_v\otimes \hat{\sigma_v}^\iota,J_{(\psi_v)_{U_{m^{n-1}}(F_v)}}(\rho_{\tau_v,m,\gamma_{\psi_v}^{(\epsilon)};\eta,1}))\neq 0,
\end{equation}
where $J_{(\psi_v)_{U_{m^{n-1}}(F_v)}}$ denotes the (twisted) Jacquet functor with respect to the character $(\psi_v)_{U_{m^{n-1}}(F_v)}$ of $U_{m^{n-1}}(F_v)$, which is the component at $v$ of the character \eqref{1.3}. The representation (of $H(F_v)$) $\rho_{\tau_v,m,\gamma_{\psi_v}^{(\epsilon)};\eta,1}$ is the unramified constituent of the component at $v$ of the representation $\rho_{\Delta(\tau,m)\gamma_\psi^{(\epsilon)}, e_1^H(\eta)}$. Recall that $e_1^H(\eta)$ is the first positive pole of our Eisenstein series corresponding to $\rho_{\Delta(\tau,m)\gamma_\psi^{(\epsilon)}, s}$.

We fix $S$, a finite set of places of $F$, containing the Archimedean places, outside which $\tau$ is unramified, and $\psi$ is normalized. Consider the four cases of functoriality listed before Theorem \ref{thm 2.2}. In all these cases, $\tau$ is self-dual and has a trivial central character. Assume first that $n=2n'$ is even. Let $v\notin S$. Then we may write, as in \eqref{3.1},
\begin{equation}\label{12.2.1}
\tau_v=\chi_1\times\cdots\times \chi_{n'}\times\chi^{-1}_{n'}\times\cdots\times\chi^{-1}_1,
\end{equation}
where $\chi_i$ are unramified characters of $F_v^*$. Then in the two cases where
$L(\tau, \wedge^2,s)$ has a pole at $s=1$, we wrote in \eqref{3.4} that $\rho_{\tau_v,m, \gamma_{\psi_v}^{(\epsilon)};\wedge^2,1}$ is the unramified constituent of the parabolic induction,\\
\\
$\rho^H_{\chi,\gamma_{\psi_v}^{(\epsilon)},\wedge^2,1}=$ 
\begin{equation}\label{12.3}
\Ind^{H(F_v)}_{Q^{(\epsilon)}_{(m+1)^{n'},(m-1)^{n'}}(F_v)}[(\otimes_{i=1}^{n'}\chi_i\gamma_{\psi_v}^{(\epsilon)}\circ det_{\GL_{m+1}}) \otimes (\otimes_{i=1}^{n'}\chi_i\gamma_{\psi_v}^{(\epsilon)}\circ det_{\GL_{m-1}})].
\end{equation}
Recall that in this case either $H^{(\epsilon)}_m=\SO_m$, $H=\SO_{2nm}$, with $m=2m'-1=n+1$, or, when also $L(\tau,\frac{1}{2})\neq 0$, $H_m^{(\epsilon)}=\Sp_m^{(2)}$, $H=\Sp^{(2)}_{2nm}$, with $m=2m'=n$. Next, in the case where $L(\tau, \vee^2,s)$ has a pole at $s=1$, $H_m=\SO_m$, $H=\SO_{2mn}$, with $m=2m'=n$, $\rho_{\tau_v,m,;\vee^2,1}$ is the unramified constituent of the parabolic induction,
\begin{equation}\label{12.4}
\rho^H_{\chi,\vee^2,1}=\Ind^{H(F_v)}_{Q_{(m+2)^{n'},(m-2)^{n'}}(F_v)}[(\otimes_{i=1}^{n'}\chi_i\circ det_{\GL_{m+2}}) \otimes (\otimes_{i=1}^{n'}\chi_i\circ det_{\GL_{m-2}})].
\end{equation}
Finally, in the fourth case of functoriality, $n=2n'+1$ is odd. Write, as in \eqref{12.2.1},
\begin{equation}\label{12.5}
\tau_v=\chi_1\times\cdots\times \chi_{n'}\times 1\times\chi^{-1}_{n'}\times\cdots\times\chi^{-1}_1,
\end{equation}
where $\chi_i$ are unramified characters of $F_v^*$. In this case, $H_m=\Sp_m$, $H=\Sp_{2mn}$, with $m=2m'=n-1$, and $\rho_{\tau_v,m,;\vee^2,1}$ is the unramified constituent of the parabolic induction, \\
\\
$\rho^H_{\chi,\vee^2,1}=$
\\
\begin{equation}\label{12.6}
\Ind^{H(F_v)}_{Q_{(m+1)^{n'},m,(m-1)^{n'}}(F_v)}[(\otimes_{i=1}^{n'}\chi_i\circ det_{\GL_{m+1}}) \otimes |det_{\GL_m}|^{\frac{1}{2}}\otimes (\otimes_{i=1}^{n'}\chi_i\circ det_{\GL_{m-1}})]. 
\end{equation}   
Then we will prove 
\begin{thm}\label{thm 12.2}
Let $v\notin S$, and consider the unramified representation $\tau_v$ as in \eqref{12.2.1}, or \eqref{12.5}. Let $\rho^H_{\chi,\gamma_{\psi_v}^{(\epsilon)},\eta,1}$ be the unramified representation of $H(F_v)$, as in \eqref{12.3}, \eqref{12.4}, or \eqref{12.6}. Then
\begin{multline}\label{12.6.1}
J_{(\psi_v)_{U_{m^{n-1}}(F_v)}}(\rho_{\chi,\gamma_{\psi_v}^{(\epsilon)},\eta,1})\cong\\ (\Ind_{B_{H_m}^{(\epsilon)}(F_v)}^{H_m^{(\epsilon)}(F_v)}\chi_1\gamma_{\psi_v}^{(\epsilon)}\otimes\cdots\otimes\chi_{n'}\gamma_{\psi_v}^{(\epsilon)})\otimes (\Ind_{B_{H_m}^{(\epsilon)}(F_v)}^{H_m^{(\epsilon)}(F_v)}\chi_1\gamma_{\psi_v}^{(\epsilon)}\otimes\cdots\otimes\chi_{n'}\gamma_{\psi_v}^{(\epsilon)}),
\end{multline}
where $B_{H_m}=Q_{1^{[\frac{m}{2}]}}^{H_m}$ is the standard Borel subgroup of $H_m$. In particular, if $\pi_1$, $\pi_2$ are two irreducible, unramified representations of $H_m^{(\epsilon)}(F_v)$, such that
\begin{equation}\label{12.6.2}
\Hom_{H_m^{(\epsilon)}(F_v)\times H_m^{(\epsilon)}(F_v)}(\pi_1\otimes \pi_2,J_{(\psi_v)_{U_{m^{n-1}}(F_v)}}(\rho_{\chi,\gamma_{\psi_v}^{(\epsilon)},\eta,1}))\neq 0,
\end{equation}
then $\pi_1\cong \pi_2$ is the unramified constituent of $\Ind_{B_{H_m}^{(\epsilon)}(F_v)}^{H_m^{(\epsilon)}(F_v)}\chi_1\gamma_{\psi_v}^{(\epsilon)}\otimes\cdots\otimes\chi_{n'}\gamma_{\psi_v}^{(\epsilon)}$, and hence lifts to $\tau_v$. We conclude that if $\sigma$ satisfies \eqref{12.1}, and $\sigma_v$ is unramified, then $\sigma_v$ lifts to $\tau_v$.
\end{thm}		

We will prove a more general theorem, which will yield Theorem \ref{thm 12.2} as a special case. We will place ourselves in the more general cases of Cor. \ref{cor 10.3}. In order to formulate our general theorem, we need some preparation. We let $\tau_v$ be the unramified representation of $\GL_n(F_v)$ as in \eqref{12.2.1}, or \eqref{12.5}. We will divide the cases of Cor. \ref{cor 10.3} into four main cases, according to the parities of $n$ and $m$.\\
\\
{\bf Case 1:} $n=2n'$, $m=2m'$ are even; $m=\mu_0n$,
\begin{multline}\label{12.7}
\rho^H_{\chi,\gamma_{\psi_v}^{(\epsilon)},\eta,k_0}=\\
\Ind^{H(F_v)}_{Q^{(\epsilon)}_{(m+\mu_0)^{n'},(m-\mu_0)^{n'}}(F_v)}[(\otimes_{i=1}^{n'}\chi_i\gamma_{\psi_v}^{(\epsilon)}\circ det_{\GL_{m+\mu_0}}) \otimes (\otimes_{i=1}^{n'}\chi_i\gamma_{\psi_v}^{(\epsilon)}\circ det_{\GL_{m-\mu_0}})],
\end{multline}
where $\mu_0=2k_0-1$ in the following cases
\begin{enumerate}
	\item $H=\Sp_{2nm}^{(2)}$, $\eta=\wedge^2$,
	\item $H=\Sp_{2nm}$, $\eta=\vee^2$,
	\item $H=\SO_{2nm}$, $\eta=\vee^2$,
\end{enumerate}
and $\mu_0=2k_0$ in the following cases
\begin{enumerate}
	\item $H=\Sp_{2nm}^{(2)}$, $\eta=\vee^2$,
	\item $H=\Sp_{2nm}$, $\eta=\wedge^2$,
	\item $H=\SO_{2nm}$, $\eta=\wedge^2$,
\end{enumerate}
In each of the last six cases, the nilpotent orbit corresponding to the representation \eqref{12.7} is given by the partition
\begin{equation}\label{12.8}
\underline{P_\rho}=((2n)^{m-\mu_0},n^{2\mu_0}).
\end{equation}
These are the partitions mentioned in Prop. \ref{prop 3.1} and Prop. \ref{prop 3.2}.\\
{\bf Case 2:} $n=2n'$ is even, $m=2m'-1$ is odd; $m-1=\mu_0n$.
\begin{enumerate}
	\item $\mu_0=2k_0-1$, $H=\SO_{2nm}$, $\eta=\wedge^2$.
	\begin{equation}\label{12.11}
	\rho^H_{\chi,\gamma_{\psi_v}^{(\epsilon)},\eta,k_0}=\Ind^{H(F_v)}_{Q_{(m+\mu_0)^{n'},(m-\mu_0)^{n'}}(F_v)}[(\otimes_{i=1}^{n'}\chi_i\circ det_{\GL_{m+\mu_0}}) \otimes (\otimes_{i=1}^{n'}\chi_i\circ det_{\GL_{m-\mu_0}})],
	\end{equation}
	$$
	\underline{P_\rho}=((2n)^{m-\mu_0},n^{2\mu_0}).
	$$
	\item $\mu_0=2k_0$, $H=\SO_{2nm}$, $\eta=\vee^2$.
	\begin{equation}\label{12.12}
	\rho^H_{\chi,\gamma_{\psi_v}^{(\epsilon)},\eta,k_0}=\Ind^{H(F_v)}_{Q_{(m+\mu_0)^{n'},(m-\mu_0)^{n'}}(F_v)}[(\otimes_{i=1}^{n'}\chi_i\circ det_{\GL_{m+\mu_0}}) \otimes (\otimes_{i=1}^{n'}\chi_i\circ det_{\GL_{m-\mu_0}})],
	\end{equation}
	$$
		\underline{P_\rho}=((2n)^{m-\mu_0-1},2n-1, n+1, n^{2\mu_0-2},n-1,1).
	$$
	\end{enumerate}
We wrote the partitions $\underline{P}$ corresponding to the representations $\rho^H_{\chi,\gamma_{\psi_v}^{(\epsilon)},\eta,k_0}$ in \eqref{12.11}, \eqref{12.12} as in Prop. \ref{prop 3.1}, Prop. \ref{prop 3.2}. \\
{\bf Case 3:} $n=2n'+1$ is odd, $m=2m'$ is even; $m=\mu_0(n-1)$,
\begin{multline}\label{12.14}
\rho^H_{\chi,\gamma_{\psi_v}^{(\epsilon)},\eta,k_0}=
\Ind^{H(F_v)}_{Q^{(\epsilon)}_{(m+\mu_0)^{n'},m,(m-\mu_0)^{n'}}(F_v)}\chi\\
\chi=[(\otimes_{i=1}^{n'}\chi_i\gamma_{\psi_v}^{(\epsilon)}\circ det_{\GL_{m+\mu_0}}) \otimes |det_{\GL_m}|^{\frac{\mu_0}{2}}\otimes (\otimes_{i=1}^{n'}\chi_i\gamma_{\psi_v}^{(\epsilon)}\circ det_{\GL_{m-\mu_0}})],
\end{multline}
where $\mu_0=2k_0-1$ in the following cases, where we indicate the partition corresponding to the representation \eqref{12.14} (as in Prop. \ref{prop 3.2}),
\begin{enumerate}
	\item $H=\Sp_{2nm}$, $\eta=\vee^2$, 
	$$
	\underline{P_\rho}=((2n)^{m-\mu_0}, (n+1)^{\mu_0},(n-1)^{\mu_0}),
	$$
	\item $H=\SO_{2nm}$, $\eta=\vee^2$,
 $$
 \underline{P_\rho}=((2n)^{m-2k_0},2n-1,n+2,(n+1)^{2k_0-2},(n-1)^{2k_0-2},n-2,1).
 $$
\end{enumerate}
and $\mu_0=2k_0$ in the case $H=\Sp_{2nm}^{(2)}$, $\eta=\vee^2$,
$$ 
\underline{P_\rho}=((2n)^{m-\mu_0}, (n+1)^{\mu_0},(n-1)^{\mu_0}).
$$
{\bf Case 4:} $n=2n'+1$, $m=2m'-1$ odd, $m-1=\mu_0(n-1)$,
\begin{multline}\label{12.16}
\rho^H_{\chi,\gamma_{\psi_v}^{(\epsilon)},\eta,k_0}=
\Ind^{H(F_v)}_{Q_{(m+\mu_0)^{n'},m,(m-\mu_0)^{n'}}(F_v)}\chi\\
\chi=[(\otimes_{i=1}^{n'}\chi_i\circ det_{\GL_{m+\mu_0}}) \otimes |det_{\GL_m}|^{\frac{\mu_0}{2}}\otimes (\otimes_{i=1}^{n'}\chi_i\circ det_{\GL_{m-\mu_0}})],
\end{multline}
 where $\mu_0=2k_0$, $H=\SO_{2nm}$, $\eta=\vee^2$, and the partition corresponding to the representation \eqref{12.16} is (as in Prop. \ref{prop 3.2} (4))
 $$
\underline{P_\rho}=((2n)^{m-2k_0+1},2n-1,n+2,(n+1)^{2k_0-2},n^2,(n-1)^{2k_0-2},n-2,1).
$$
Denote in Cases 1, 2,
\begin{equation}\label{12.9}
\tau'=\Ind^{\GL_{nm}(F_v)}_{P_{(m+\mu_0)^{n'},(m-\mu_0)^{n'}}(F_v)}[(\otimes_{i=1}^{n'}\chi_i\circ det_{\GL_{m+\mu_0}}) \otimes (\otimes_{i=1}^{n'}\chi_i\circ det_{\GL_{m-\mu_0}})],
\end{equation}
$$
\underline{P_{\tau'}}=(n^{m-\mu_0}, (n')^{2\mu_0}).
$$
Denote
$$
\chi_{\tau'}=(\otimes_{i=1}^{n'}\chi_i\circ det_{\GL_{m+\mu_0}}) \otimes (\otimes_{i=1}^{n'}\chi_i\circ det_{\GL_{m-\mu_0}}),
$$
and think of it as a character of $P_{(m+\mu_0)^{n'},(m-\mu_0)^{n'}}(F_v)$.
In Cases 3, 4, denote
\begin{multline}\label{12.9.1}
\tau'=
\Ind^{\GL_{nm}(F_v)}_{P_{(m+\mu_0)^{n'},m,(m-\mu_0)^{n'}}(F_v)}[(\otimes_{i=1}^{n'}\chi_i\circ det_{\GL_{m+\mu_0}}) \otimes|det_{\GL_m}|^{\frac{\mu_0}{2}}\otimes\\ 
\otimes(\otimes_{i=1}^{n'}\chi_i\circ det_{\GL_{m-\mu_0}})],
\end{multline}
$$
\underline{P_{\tau'}}=(n^{m-\mu_0}, (n'+1)^{\mu_0}, (n')^{\mu_0}).
$$
Similarly, we denote here,
$$
\chi_{\tau'}=(\otimes_{i=1}^{n'}\chi_i\circ det_{\GL_{m+\mu_0}}) \otimes|det_{\GL_m}|^{\frac{\mu_0}{2}}\otimes\\ 
\otimes(\otimes_{i=1}^{n'}\chi_i\circ det_{\GL_{m-\mu_0}}),
$$
and think of it as a character of $P_{(m+\mu_0)^{n'},m,(m-\mu_0)^{n'}}(F_v)$.
In all cases, the nilpotent orbits which define nontrivial Jacquet modules of  $\tau'$ are all bounded by the nilpotent orbit corresponding to the partition $\underline{P_{\tau'}}$. We also have 
\begin{equation}\label{12.10}
\rho^H_{\chi,\gamma_{\psi_v}^{(\epsilon)},\eta,k_0}=\Ind^{H(F_v)}_{Q^{(\epsilon)}_{mn}(F_v)}\tau'\gamma_{\psi_v}^{(\epsilon)}.
\end{equation}
Our main goal now is to prove the following theorem, a special case of which is Theorem \ref{thm 12.2}.
\begin{thm}\label{thm 12.3}
	Let $v\notin S$, and consider the unramified representation $\tau_v$ as in \eqref{12.2.1}, or \eqref{12.5}. Then
	\begin{multline}\label{12.17}
	J_{(\psi_v)_{U_{m^{n-1}}(F_v)}}(\rho_{\chi,\gamma_{\psi_v}^{(\epsilon)},\eta,k_0})\cong\\ (\Ind_{Q_{\mu_0^{n'}}^{(\epsilon)}(F_v)}^{H_m^{(\epsilon)}(F_v)}\chi_1\gamma_{\psi_v}^{(\epsilon)}\circ det_{\GL_{\mu_0}}\otimes\cdots\otimes\chi_{n'}\gamma_{\psi_v}^{(\epsilon)}\circ det_{\GL_{\mu_0}})\otimes\\ \otimes(\Ind_{Q_{\mu_0^{n'}}^{(\epsilon)}(F_v)}^{H_m^{(\epsilon)}(F_v)}\chi_1\gamma_{\psi_v}^{(\epsilon)}\circ det_{\GL_{\mu_0}}\otimes\cdots\otimes\chi_{n'}\gamma_{\psi_v}^{(\epsilon)}\circ det_{\GL_{\mu_0}}),
	\end{multline}
	In particular, if $\pi_1$, $\pi_2$ are two irreducible, unramified representations of $H_m^{(\epsilon)}(F_v)$, such that
	\begin{equation}\label{12.18}
	\Hom_{H_m^{(\epsilon)}(F_v)\times H_m^{(\epsilon)}(F_v)}(\pi_1\otimes \pi_2,
	J_{(\psi_v)_{U_{m^{n-1}}(F_v)}}(\rho_{\chi,\gamma_{\psi_v}^{(\epsilon)},\eta,k_0})\neq 0,
	\end{equation}
	then $\pi_1\cong \pi_2$ is the unramified constituent of 
	$$
	\Ind_{Q_{\mu_0^{n'}}^{(\epsilon)}(F_v)}^{H_m^{(\epsilon)}(F_v)}(\chi_1\gamma_{\psi_v}^{(\epsilon)}\circ det_{\GL_{\mu_0}}\otimes\cdots\otimes\chi_{n'}\gamma_{\psi_v}^{(\epsilon)}\circ det_{\GL_{\mu_0}});
	$$
When $H$ is not symplectic, it lifts to the unramified constituent of
$$
\Ind_{P_{\mu_0^{2n'}}(F_v)}^{\GL_{2\mu_0n'}(F_v)}(\chi_1\circ det_{\GL_{\mu_0}}\otimes\cdots\otimes\chi_{n'}\circ det_{\GL_{\mu_0}}\otimes
\chi^{-1}_{n'}\circ det_{\GL_{\mu_0}}\otimes\cdots\otimes \chi^{-1}_1\circ det_{\GL_{\mu_0}}), 
$$
and when $H$ is symplectic, it lifts to the unramified constituent of
$$
\Ind_{P_{\mu_0^{2n'},1}(F_v)}^{\GL_{2\mu_0n'+1}(F_v)}(\chi_1\circ det_{\GL_{\mu_0}}\otimes\cdots\otimes\chi_{n'}\circ det_{\GL_{\mu_0}}\otimes
\chi^{-1}_{n'}\circ det_{\GL_{\mu_0}}\otimes\cdots\otimes \chi^{-1}_1\circ det_{\GL_{\mu_0}}\otimes 1). 
$$
\end{thm}

Our main task is to analyze the Jacquet module $J_{(\psi_v)_{U_{m^{n-1}}(F_v)}}(\rho_{\chi,\gamma_{\psi_v}^{(\epsilon)},\eta,k_0})$ as a $H_m^{(\epsilon)}(F_v)\times H_m^{(\epsilon)}(F_v)$ - module. We will do this using Mackey theory. There will be three parts to this analysis. By \eqref{12.10},
\begin{equation}\label{12.19}
J_{(\psi_v)_{U_{m^{n-1}}(F_v)}}(\rho_{\chi,\gamma_{\psi_v}^{(\epsilon)},\eta,k_0})\cong J_{(\psi_v)_{U_{m^{n-1}}(F_v)}}(\Ind^{H(F_v)}_{Q^{(\epsilon)}_{mn}(F_v)}\tau'\gamma_{\psi_v}^{(\epsilon)}).
\end{equation}
In the first and second parts, we analyze the Jacquet module on the r.h.s. of \eqref{12.19} in terms of $\tau'$. The first part is done very similarly to the step in the unfolding process of the doubling integrals for $H_m^{(\epsilon)}\times \GL_n$, where one shows that only the open double coset in $Q_{mn}\backslash H_{2mn}/Q_{m^{n-1}}$ contributes to the integral. See \cite{CFGK16}, Sec. 2.3. See also \cite{GS18}, Sec. 2, which we will follow here. A similar analysis was done in \cite{GJS15}, Sec. 3, which we will follow, as well. In the second part, we consider the contribution of the double coset above and analyze it as an $H_m^{(\epsilon)}(F_v)\times H_m^{(\epsilon)}(F_v)$ - module. This part will be carried out similarly to Sec. 3 in \cite{GS18}, only that there, it is the open double coset that contributes to the Jacquet module, while here it will be that only the closed double coset contributes. We will use for this part the fact that the nilpotent orbits which define nontrivial Jacquet modules of  $\rho_{\chi,\gamma_{\psi_v}^{(\epsilon)},\eta,k_0}$ are all bounded by the nilpotent orbit corresponding to the partition $\underline{P_\rho}$, as listed in Cases 1-4 above. This is the same partition as in the global case. We used it repeatedly in order to assert that various Fourier coefficients of $A(\Delta(\tau,m)\gamma_\psi^{(\epsilon)})$ are zero. At this point, we will get an expression of $J_{(\psi_v)_{U_{m^{n-1}}(F_v)}}(\Ind^{H(F_v)}_{Q^{(\epsilon)}_{mn}(F_v)}\tau'\gamma_{\psi_v}^{(\epsilon)})$ in terms of the Jacquet module of $\tau'$ with respect to a certain character of $V_{m^n}(F_v)$. In the third part, we will analyze this Jacquet module of $\tau'$. The analysis here is similar, in certain parts, to \cite{GJS15}, Sec. 4. Here, we will use the tool of exchanging roots, which is valid locally, and the fact that the nilpotent orbits which define 
nontrivial Jacquet modules of $\tau'$ are all bounded by the nilpotent orbit corresponding to the partition $\underline{P_{\tau'}}$.

Our first step is to analyze
\begin{equation}\label{12.20}
 \Res_{Q^{(\epsilon)}_{m(n-1)}(F_v)}(\Ind^{H(F_v)}_{Q^{(\epsilon)}_{mn}(F_v)}(\tau'\gamma_{\psi_v}^{(\epsilon)})).
\end{equation}
This is a representation of finite length, with subquotients parameterized by the double
cosets
\begin{equation}\label{12.21}
Q_{mn}(F_v)\backslash H_{2mn}(F_v)/Q_{m(n-1)}(F_v).
\end{equation}
We take the following set of representatives, as in \cite{GS18}, (2.2) (with $i=0$ there). For
$0\leq r\leq m(n-1)$,
\begin{equation}\label{12.22}
\alpha_r=\begin{pmatrix}I_r\\&\alpha'_r\\&&I_r\end{pmatrix}
\end{equation}
where
$$
\alpha'_r=\begin{pmatrix}0 &I_m&0&0\\
0&0 &0&I_{m(n-1)-r} \\
\delta_HI_{m(n-1)-r}&0 &0&0 \\
&0 &0 &I_m&0\end{pmatrix}\omega_0^{m(n-1)},
$$
where $\omega_0=I$ unless $H$ is orthogonal. When $H=\SO_{2mn}$, we take
\begin{equation}\label{12.23}
\omega_0=diag (I_{mn-1},w'_2,I_{mn-1}),
\end{equation}
where, for $m$ even, $w'_2=w_2=\begin{pmatrix}&1\\1\end{pmatrix}$, and when $m$ is odd,
$$
w'_2=\begin{pmatrix}&2\\ \frac{1}{2}\end{pmatrix}.
$$
As in \cite{GS18}, although $\omega_0$ depends on $H_m$ and $n$, we use the notation $\omega_0$ for simplicity. In the metaplectic case, we will identify $(\alpha_r,1)$ with $\alpha_r$.
Thus, up to semisimplification (which we denote by $\equiv$), we have 
\begin{equation}\label{12.24}
 \Res_{Q^{(\epsilon)}_{m(n-1)}(F_v)}(\Ind^{H(F_v)}_{Q^{(\epsilon)}_{mn}(F_v)}(\tau'\gamma_{\psi_v}^{(\epsilon)}))\equiv \oplus_{r=0}^{m(n-1)}\ind^{Q^{(\epsilon)}_{m(n-1)}(F_v)}_{Q^{(r),(\epsilon)}(F_v)}(\sigma_{(r)}\gamma^{(\epsilon)}_{\psi_v}),
\end{equation}
where $"\ind"$ denotes non-normalized compact induction,
$$
Q^{(r)}=Q_{m(n-1)}\cap \alpha_r^{-1}\cdot Q_{mn}\cdot \alpha_r,
$$
and
$\sigma_{(r)}$ is the representation of $Q^{(r),(\epsilon)}(F_v)$ given by
$$
x\mapsto
\delta_{Q_{mn}}^{\frac{1}{2}}\cdot\tau'\cdot \gamma^{(\epsilon)}_{\psi_v}(\alpha_rx\alpha_r^{-1}),\quad
x\in Q^{(r),(\epsilon)}(F_v).
$$
The elements of $Q^{(r)}$ are written in \cite{GS18}, (2.3), and have the form:
\begin{equation}\label{12.25}
\begin{pmatrix}
a_1&a_2&y_1&y_2&z_1&z_2\\
0&a_4&0&y_4&0&z_1'\\
&&c&v&y_4'&y_2'\\
&&0&c^*&0&y_1'\\
&&&&a_4^*&a_2'\\
&&&&0&a_1^*\end{pmatrix}^{\omega_0^{m(n-1)-r}},
\end{equation}
where $a_1\in\GL_r$, $a_4\in\GL_{m(n-1)-r}$, $c\in\GL_m$. The
representation $\sigma_{(r)}$ assigns to an element of the form
\eqref{12.25} (in the metaplectic case we consider $(\ast,1)$, where $\ast$ is \eqref{12.25}) the operator
\begin{equation}\label{12.26}
\left|\frac{\det(a_1)\cdot\det(c)}{\det(a_4)}\right|^{\frac{mn-\delta_H}{2}}\cdot\gamma^{(\epsilon)}_{\psi_v}(\det(a_1)\det(c)\det(a_4))\tau'(\begin{pmatrix}a_1&y_1&z_1\\
0&c&y'_4\\
0&0&a^*_4\end{pmatrix}).
\end{equation}
Since Jacquet functors are exact, \eqref{12.24} implies that
\begin{equation}\label{12.27}
J_{(\psi_v)_{U_{m^{n-1}}(F_v)}}(\Ind^{H(F_v)}_{Q^{(\epsilon)}_{mn}(F_v)}\tau'\gamma_{\psi_v}^{(\epsilon)})\equiv \oplus_{r=0}^{m(n-1)}J_{(\psi_v)_{U_{m^{n-1}}(F_v)}}(\ind^{Q^{(\epsilon)}_{m(n-1)}(F_v)}_{Q^{(r),(\epsilon)}(F_v)}(\sigma_{(r)})).
\end{equation}
Our first theorem is
\begin{thm}\label{thm 12.4}
For all $1\leq r\leq m(n-1)$,
$$
J_{(\psi_v)_{U_{m^{n-1}}(F_v)}}(\ind^{Q^{(\epsilon)}_{m(n-1)}(F_v)}_{Q^{(r),(\epsilon)}(F_v)}(\sigma_{(r)}))=0.
$$
\end{thm}
\begin{proof}
The proof is the exact local analog of the proof of Theorem 2.1 in \cite{GS18}. We will review the proof in terms of Jacquet modules (which parallel Fourier coefficients in the global set up). We will refer to the calculations done in \cite{GS18}. These are algebraic calculations, such as the exact form of intersections of algebraic subgroups, conjugations of algebraic subgroups by a given matrix, and so on. We will also refer to \cite{GJS15}, Sec. 3, where we proved the theorem in a very similar case. 

Assume that $1\leq r\leq m(n-1)$. Consider 
$$
\Res_{Q^{(\epsilon)}_{m^{n-1}}(F_v)}(\ind^{Q^{(\epsilon)}_{m(n-1)}(F_v)}_{Q^{(r),(\epsilon)}(F_v)}(\sigma_{(r)})),
$$
and analyze it by Mackey theory. Again, we need to consider the set
$$
Q^{(r)}(F_v)\backslash Q_{m(n-1)}(F_v)/ Q_{m^{n-1}}(F_v);
$$
it has the following set of representatives:
\begin{equation}\label{12.29}
\hat{a}=\diag(a,I_{2m},a^*),
\end{equation}
where $a$ varies in a set of Weyl elements, which form a set of
representatives for
$$
P_{r,m(n-1)-r}(F_v)\backslash\GL_{m(n-1)}(F_v)/P_{m^{n-1}}(F_v).
$$
(In the metaplectic case, we replace $\hat{a}$ with $(\hat{a},1)$.) Then, up to semi-simplification, we have
\begin{equation}\label{12.30}
\Res_{Q^{(\epsilon)}_{m^{n-1}}(F_v)}(\ind^{Q^{(\epsilon)}_{m(n-1)}(F_v)}_{Q^{(r),(\epsilon)}(F_v)}(\sigma_{(r)})) \equiv
\oplus_a\ind^{Q^{(\epsilon)}_{m^{n-1}}(F_v)}_{Q^{(\epsilon)}_{m^{n-1}}(F_v)\cap\hat{a}^{-1}Q^{(r),(\epsilon)}(F_v)\hat{a}}(\sigma_{(r)}^{\hat{a}}).
\end{equation}
The representation $\sigma_{(r)}^{\hat{a}}$ is obtained by composing
$\sigma_{(r)}$ with conjugation by $\hat{a}$ (on
${Q^{(\epsilon)}_{m^{n-1}}(F_v)}\cap\hat{a}^{-1}Q^{(r),(\epsilon)}(F_v)\hat{a}$). 
We will take the set of representatives $a$ in \eqref{12.29} to be the ones described in \cite{GS18}, (2.8) - (2.11). They are parametrized by sequences $r_1,...,r_{n-1}$ and
$\ell_1,...,\ell_{n-1}$ be non-negative integers, such that
\begin{equation}\label{12.31}
\sum_{i=1}^{n-1}r_i=r,\ \sum_{i=1}^{n-1}\ell_i=m(n-1)-r; \
r_i+\ell_i=m,\ i=1,...,n-1.
\end{equation}
Put $\bar{r}=(r_1,...,r_{n-1})$. Let $a_{\bar{r}}$ be the
permutation matrix in $\GL_{m(n-1)}(F)$ described in \cite{GS18}, (2.9) - (2.11). In \cite{GS18}, (2.12) - (2.18), we described presicely the subgroup $Q_{m^{n-1}}\cap \hat{a}_{\bar{r}}^{-1}Q^{(r)}\hat{a}_{\bar{r}}$. Next, let $L_{m^{n-1}}$ be the subgroup of the Levi part of $Q_{m^{n-1}}$ consisting of the elements
\begin{equation}\label{12.31.1}
diag(g_1,...,g_{n-1},I_{2m},g_{n-1}^*,...,g_1^*),\  g_i\in \GL_m.
\end{equation}
Denote
$$
R_{m^{n-1}}=L_{m^{n-1}}t(H_m\times H_m)U_{m^{n-1}}.
$$
Now, restrict each summand in \eqref{12.30} to $R^{(\epsilon)}_{m^{n-1}}(F_v)$. One more time, we need to consider the set
\begin{equation}\label{12.32}
Q_{m^{n-1}}(F_v)\cap \hat{a}_{\bar{r}}^{-1}Q^{(r)}(F_v)\hat{a}_{\bar{r}}\backslash Q_{m^{n-1}}(F_v)/ R_{m^{n-1}}(F_v).
\end{equation}
By \cite{GS18}, (2.19) - (2.25), we may take representatives for \eqref{12.32} of the following form
\begin{equation}\label{12.33}
h_\gamma=diag(I_{m(n-1)},\gamma,I_{m(n-1)}),
\end{equation}
where $\gamma$ is any coset representative in
\begin{equation}\label{12.34}
Q_m(F_v)^{\omega_0^{m(n-1)-r}}\backslash
H_{2m}(F_v)/j(H_m(F_v)\times H_m(F_v)). 
\end{equation}
(In the metaplectic case, we take $(h_\gamma,1)$ instead.) 
We will take the set of representatives of \eqref{12.34} given by \cite{GS18}, (2.25) (with $i=0$ there). Thus, let, for $0\leq e\leq [\frac{m}{2}]$,
\begin{equation}\label{12.35}
\gamma_e=diag(I_e,\begin{pmatrix}I_{[\frac{m}{2}]-e}\\0&I_{[\frac{m}{2}]-e}\\0&0&I_{2((m-2[\frac{m}{2}])+e)}\\
I_{[\frac{m}{2}]-e}&0&0&I_{[\frac{m}{2}]-e}\\0&-\delta_HI_{[\frac{m}{2}]-e}&0&0&I_{[\frac{m}{2}]-e}\end{pmatrix},I_e).
\end{equation}
Put 
$$
Q^{(\bar{r};e)}=h_{\gamma_e}^{-1}(Q_{m^{n-1}}\cap
\hat{a}_{\bar{r}}^{-1}Q^{(r)}\hat{a}_{\bar{r}})h_{\gamma_e},\ h_{a_{\bar{r}},\gamma_e}=\hat{a}_{\bar{r}}h_{\gamma_e}. 
$$
Thus, we get from \eqref{12.30}
\begin{equation}\label{12.36}
\Res_{R^{(\epsilon)}_{m^{n-1}}(F_v)}(\ind^{Q^{(\epsilon)}_{m(n-1)}(F_v)}_{Q^{(r),(\epsilon)}(F_v)}(\sigma_{(r)})) \equiv
\oplus_{\bar{r}}\oplus_{e=0}^{[\frac{m}{2}]}\ind^{R^{(\epsilon)}_{m^{n-1}}(F_v)}_{R^{(\epsilon)}_{m^{n-1}}(F_v)\cap Q^{(\bar{r};e),(\epsilon)}(F_v)}(\sigma_{(r)}^{h_{a_{\bar{r}},\gamma_e}}).
\end{equation}	
Note that $R_{m^{n-1}}\cap Q^{(\bar{r};e)}=R_{m^{n-1}}\cap h_{a_{\bar{r}},\gamma_e}^{-1}Q^{(r)}h_{a_{\bar{r}},\gamma_e}$. Applying our Jacquet functor to \eqref{12.30}, we get
\begin{multline}\label{12.37}
J_{(\psi_v)_{U_{m^{n-1}}(F_v)}}(\ind^{Q^{(\epsilon)}_{m(n-1)}(F_v)}_{Q^{(r),(\epsilon)}(F_v)}(\sigma_{(r)}))\equiv\\
 \oplus_{\bar{r}}\oplus_{e=0}^{[\frac{m}{2}]} J_{(\psi_v)_{U_{m^{n-1}}(F_v)}}(\ind^{R^{(\epsilon)}_{m^{n-1}}(F_v)}_{R^{(\epsilon)}_{m^{n-1}}(F_v)\cap h_{a_{\bar{r}},\gamma_e}^{-1}Q^{(r),(\epsilon)}(F_v)h_{a_{\bar{r}},\gamma_e}}(\sigma_{(r)}^{h_{a_{\bar{r}},\gamma_e}}).
\end{multline}
Assume that the summand corresponding to $\bar{r}$ and $e$ in \eqref{12.37} is nonzero. Then it follows that
\begin{equation}\label{12.38}
r_1\leq r_2\leq\cdots\leq r_{n-1}\leq e.
\end{equation} 
Indeed, if the condition \eqref{12.38} is not satisfied, then the element $\alpha_rh_{a_{\bar{r}},\gamma_e}$ is irrelevant with respect to the character $(\psi_v)_{U_{m^{n-1}}(F_v)}$. This means that for each $h\in L^{(\epsilon)}_{m^{n-1}}(F_v)$, there is $u\in U_{m^{n-1}}(F_v)\cap h^{-1}h_{a_{\bar{r}},\gamma_e}^{-1}Q^{(r),(\epsilon)}(F_v)h_{a_{\bar{r}},\gamma_e}h$, such that $(\psi_v)_{U_{m^{n-1}}(F_v)}\neq 1$, but $\sigma_{(r)}^{h_{a_{\bar{r}},\gamma_e}}(huh^{-1})=id$. The proof of Prop. 2.3 in \cite{GS18} shows exactly that if \eqref{12.38} is not satisfied, then $\alpha_rh_{a_{\bar{r}},\gamma_e}$ is irrelevant, and then it is clear that
$$
J_{(\psi_v)_{U_{m^{n-1}}(F_v)}}(\ind^{R^{(\epsilon)}_{m^{n-1}}(F_v)}_{R^{(\epsilon)}_{m^{n-1}}(F_v)\cap h_{a_{\bar{r}},\gamma_e}^{-1}Q^{(r),(\epsilon)}(F_v)h_{a_{\bar{r}},\gamma_e}}(\sigma_{(r)}^{h_{a_{\bar{r}},\gamma_e}}))=0.
$$ 	
See Prop. 3.2 in \cite{GJS15} for a very similar case. Now, assume that $\bar{r}$ and $e$ satisfy \eqref{12.38}. As in \cite{GJS15}, (3.37), the following map on\\ $\ind^{R^{(\epsilon)}_{m^{n-1}}(F_v)}_{R^{(\epsilon)}_{m^{n-1}}(F_v)\cap h_{a_{\bar{r}},\gamma_e}^{-1}Q^{(r),(\epsilon)}(F_v)h_{a_{\bar{r}},\gamma_e}}(\sigma_{(r)})^{h_{a_{\bar{r}},\gamma_e}}$
 induces an injective homomorphism on\\ $J_{(\psi_v)_{U_{m^{n-1}}(F_v)}}(\ind^{R^{(\epsilon)}_{m^{n-1}}(F_v)}_{R^{(\epsilon)}_{m^{n-1}}(F_v)\cap h_{a_{\bar{r}},\gamma_e}^{-1}Q^{(r),(\epsilon)}(F_v)h_{a_{\bar{r}},\gamma_e}}(\sigma_{(r)}^{h_{a_{\bar{r}},\gamma_e}})$,
\begin{multline}\label{12.39}
f\mapsto  (ht^{(\epsilon)}(g_1,g_2)\mapsto\\
 \int_{U'_{m^{n-1}}(F_v)\backslash U_{m^{n-1}}(F_v)}J_{(\psi_v)^{h,e}_{V_{\bar{r},m,\underline{\ell}}}}(f(uht^{(\epsilon)}(g_1,g_2)))
(\psi^{-1}_v)_{U_{m^{n-1}}(F_v)}(h^{-1}uh)du,\\
h\in L^{(\epsilon)}_{m^{n-1}}(F_v),\ g_1,g_2\in H_m^{(\epsilon)}(F_v).
\end{multline}	
Here, $U'_{m^{n-1}}=U_{m^{n-1}}\cap h_{a_{\bar{r}},\gamma_e}^{-1}Q^{(r)}h_{a_{\bar{r}},\gamma_e}$, and $V_{\bar{r},m,\underline{\ell}}$ denotes the projection to $\GL_{mn}$ of $\alpha_rh_{a_{\bar{r}},\gamma_e}U'_{m^{n-1}}h_{a_{\bar{r}},\gamma_e}^{-1}\alpha_r^{-1}$. This is the standard unipotent radical\\ $V_{r_1,...,r_{n-1},m.\ell_{n-1},...,\ell_1}$. See \cite{GS18}, (2.48). $J_{(\psi_v)^{h,e}_{V_{\bar{r},m,\underline{\ell}}}}$ denotes the Jacquet functor, applied to the space of $\tau'$, with respect to the character $(\psi_v)^{h,e}_{V_{\bar{r},m,\underline{\ell}}}$ of $V_{\bar{r},m,\underline{\ell}}(F_v)$ given as follows. Let $z\in V_{\bar{r},m,\underline{\ell}}(F_v)$. Let $u'\in \alpha_rh_{a_{\bar{r}},\gamma_e}U'_{m^{n-1}}(F_v)h_{a_{\bar{r}},\gamma_e}^{-1}\alpha_r^{-1}$	be any pull back of $\hat{z}$. Then
\begin{equation}\label{12.40}
(\psi_v)^{h,e}_{V_{\bar{r},m,\underline{\ell}}}(z)=(\psi_v)_{U_{m^{n-1}}(F_v)}(h^{-1}h_{a_{\bar{r}},\gamma_e}^{-1}\alpha_r^{-1}u'\alpha_rh_{a_{\bar{r}},\gamma_e}h).
\end{equation}
Exactly as in Prop. 3.3 in \cite{GJS15} (see also \cite{GS18}, (2.39), (2.43)), we get that the support in $h\in L^{(\epsilon)}_{m^{n-1}}(F_v)$ of the function \eqref{12.39} lies in 
$$
L^{(\epsilon)}_{m^{n-1}}(F_v)\cap \hat{a}_{\bar{r}}^{-1}Q^{(r),(\epsilon)}\hat{a}_{\bar{r}}=L^{(\epsilon)}_{m^{n-1}}(F_v)\cap h_{a_{\bar{r}},\gamma_e}^{-1}Q^{(r),(\epsilon)}h_{a_{\bar{r}},\gamma_e}.
$$	
See \cite{GS18}, (2.12), for the description of the last subgroup. Note that for such $h$ in the support, the function (2.39) is uniquely determined by its restriction to $t^{(\epsilon)}(H^{(\epsilon)}_m(F_v)\times H^{(\epsilon)}_m(F_v))$,
\begin{multline}\label{12.41}
f\mapsto  (\ t^{(\epsilon)}(g_1,g_2)\mapsto \\
\int_{U'_{m^{n-1}}(F_v)\backslash U_{m^{n-1}}(F_v)}J_{(\psi^e_v)_{V_{\bar{r},m,\underline{\ell}}}}(f(ut^{(\epsilon)}(g_1,g_2)))
(\psi^{-1}_v)_{U_{m^{n-1}}(F_v)}(u)du \ ).
\end{multline}
In this case, the character \eqref{12.40}, which we now denote $(\psi^e_v)_{V_{\bar{r},m,\underline{\ell}}}$ can be read from \cite{GS18}, (2.49), where we take there $\eta_1=\cdots \eta_{n-1}=I_m$ (and $i=0$). We repeat it here. First, write an element $z\in V_{\bar{r},m,\underline{\ell}}(F_v)$ as
\begin{equation}\label{12.42}
z=\begin{pmatrix}I_{r_1}&x_{1,2}&\cdots&x_{1,n-1}&b_1&\ast
&&\cdots&\ast\\&I_{r_2}&\cdots&x_{2,n-1}&b_2&\ast&&\cdots&\ast\\&&\ddots\\&&&I_{r_{n-1}}&b_{n-1}&\ast&&\cdots&\ast\\
&&&&I_m&c_{n-1}&&\cdots&c_1\\&&&&&I_{\ell_{n-1}}&y_{n-2,n-1}&\cdots&y_{1,n-2}\\&&&&&&\ddots\\&&&&&&&&y_{1,2}\\
&&&&&&&&I_{\ell_1}\end{pmatrix}.
\end{equation}
Then
\begin{multline}\label{12.43}
(\psi^e_v)_{V_{\bar{r},m,\underline{\ell}}}(z)=\\
\prod_{i=1}^{n-2}\psi(tr(R_ix_{i,i+1}))\psi^{-1}(tr(L_iy_{i,i+1}))\psi(trR_{n-1}b_{n-1}))\psi^{-1}(tr(Sc_{n-1})),
\end{multline}
where, for $1\leq i\leq n-1$, $R_i=\begin{pmatrix}I_{r_i}\\0_{(r_{i+1}-r_i)\times
		r_i}\end{pmatrix}$ (taking $r_n=m$);\\\ for $ 1\leq i\leq n-2$,\\
	\\
$L_i=\begin{pmatrix}I_{\ell_{i+1}}\\0_{(\ell_i-\ell_{i+1})\times
		\ell_{i+1}}\end{pmatrix}$. When $m$ is even,
	$$
	S=\begin{pmatrix}I_{m-e}&0_{(m-e)\times
		e}\\0_{(e-r_{n-1})\times (m-e)}&0_{(e-r_{n-1})\times
		e}\end{pmatrix},
	$$
	and when $m=2m'-1$ is odd
	$$
	S=\begin{pmatrix}I_{m'-1}&0&0_{(m'-1)\times e}&0\\0&0&0&\frac{1}{2}\\0&I_{m'-1-e}&0&0\\0_{(e-r_{n-1})\times (m'-1)}&0&0&0\end{pmatrix}.
	$$
Now, we apply the analog of \cite{GS18}, Prop. 2.4. namely, we show that if $r_{n-1}>0$, then $J_{(\psi^e_v)_{V_{\bar{r},m,\underline{\ell}}}}(\tau')=0$. Indeed, the character $(\psi^e_v)_{V_{\bar{r},m,\underline{\ell}}}$ in \eqref{12.43} corresponds to the nilpotent orbit attached to a partition of $mn$ of the form $((2n-1)^{r_1},...)$. An element in this orbit is written in \cite{GS18}, (2.54).
This is an $mn\times mn$ lower nilpotent matrix, written as a $(2n-1)\times (2n-1)$ matrix of blocks. The blocks along the diagonal are of sizes $r_1\times r_1$,..., $r_{n-1}\times r_{n-1}$, $m\times m$, $\ell_{n-1}\times \ell_{n-1}$,...,$\ell_1\times \ell_1$. The blocks in the second lower diagonal are $R_1$,..., $R_{n-1}$, $S$, $L_{n-2}$,..., $L_2$. 
By \eqref{12.9}, \eqref{12.13}, \eqref{12.9.1}, we must have $r_1=0$. Note that $2n-1>n$, i.e. $n\geq 2$, since we assume that $r_{n-1}>0$. Now it is clear that the proof of Prop. 2.4 in \cite{GS18} works here, word for word. We conclude, by induction, that if $J_{(\psi^e_v)_{V_{\bar{r},m,\underline{\ell}}}}(\tau')$ is nonzero, then $r_1=r_2=\cdots=r_{n-1}=0$. This completes the proof of the theorem.
\end{proof}

With $r=0$, the character \eqref{12.43} is a character of $V_{m^n}(F_v)$, which we denote by $(\psi_{v,e})_{V_{m^n}}$. For
$$
z=\begin{pmatrix}I_m&c_{n-1}&&\cdots&c_1\\&I_m&y_{n-2,n-1}&\cdots&y_{1,n-2}\\&&\ddots\\&&&&y_{1,2}\\
&&&&I_m\end{pmatrix},
$$
\begin{equation}\label{12.44}
(\psi_{v,e})_{V_{m^n}}(z)
=\prod_{i=1}^{n-2}\psi^{-1}(tr(y_{i,i+1}))\psi^{-1}(tr(Sc_{n-1})),
\end{equation}
where in case $m$ is even,
$$
S=\begin{pmatrix}I_{m-e}&0_{(m-e)\times e}\\0_{e\times
	(m-e)}&0\end{pmatrix},
$$
and in case $m=2m'-1$ is odd,
$$
S=\begin{pmatrix}I_{m'-1}&0&0_{(m'-1)\times
	e}&0\\0&0&0&\frac{1}{2}\\0&I_{m'-1-e}&0&0\\0_{e\times (m'-1)}&0&0&0\end{pmatrix}.
$$
Note that for $r=0$, $a_{\bar{r}}=I_{m(n-1)}$. We will denote from now on $h_{I_{m(n-1)},\gamma_e}=h_{\gamma_e}$.

We conclude from \eqref{12.27} and Theorem \ref{thm 12.4} that
as $t^{(\epsilon)}(H^{(\epsilon)}_m(F_v)\times H^{(\epsilon)}_m(F_v))$ - modules, 
\begin{multline}\label{12.45}
J_{(\psi_v)_{U_{m^{n-1}}(F_v)}}(\Ind^{H(F_v)}_{Q^{(\epsilon)}_{mn}(F_v)}\tau'\gamma_{\psi_v}^{(\epsilon)})\cong\\ \oplus_{e=0}^{[\frac{m}{2}]}J_{(\psi_v)_{U_{m^{n-1}}(F_v)}}(\ind^{R^{(\epsilon)}_{m^{n-1}}(F_v)}_{R^{(\epsilon)}_{m^{n-1}}(F_v)\cap h_{\gamma_e}^{-1}Q^{(0),(\epsilon)}(F_v)h_{\gamma_e}}(\sigma_{(0)}^{h_{\gamma_e}})\cong \\ \oplus_{e=0}^{[\frac{m}{2}]}ind^{t^{(\epsilon)}(H^{(\epsilon)}_m(F_v)\times H^{(\epsilon)}_m(F_v))}_{t^{(\epsilon)}(H^{(\epsilon)}_m(F_v)\times H^{(\epsilon)}_m(F_v))\cap h_{\gamma_e}^{-1}Q^{(0),(\epsilon)}(F_v)h_{\gamma_e}}\sigma_{[\tau',e]}.
\end{multline}
The second isomorphism follows from the last part of the proof of Theorem \ref{thm 12.4}, exactly as in \cite{GJS15}, (3.41).	
We now describe the data entering in the last induced representation, corresponding to $e$. First, the (algebraic) subgroup $t(H_m\times H_m)\cap h_{\gamma_e}^{-1}Q^{(0)}h_{\gamma_e}$ can be read from \cite{GS18}, (3.7). It cosists of the elements
\begin{equation}\label{12.46}
t(\begin{pmatrix}a&x&y\\&{}^eb&x'\\&&a^*\end{pmatrix},\hat{\beta}_e\begin{pmatrix}A&X&Y\\&b&X'\\&&A^*\end{pmatrix}\hat{\beta}_e^{-1})\in
t(H_m\times H_m),
\end{equation}
where $a, A\in \GL_e$ and, for $b\in H_{m-2e}$,
\begin{equation}\label{12.47}
{}^eb=J^{-1}_e b J_e,
\end{equation}
and $J_e=J_{e;m,n}$ is defined as follows. If $m(n-1)$ is even (so that $\omega_0^{m(n-1)}=I$), then
\begin{equation}\label{12.48} 
J_e=\begin{pmatrix}&&I_{[\frac{m}{2}]-e}\\&I_{m-2[\frac{m}{2}]}\\-\delta_H
I_{[\frac{m}{2}]-e}\end{pmatrix};
\end{equation}
and if $m(n-1)$ is odd, so that $m=2m'-1$ is odd and $n$ is even, then
\begin{equation}\label{12.49}
J_e=\begin{pmatrix}&&-I_{m'-1-e}\\&1\\I_{m'-1-e}\end{pmatrix}.
\end{equation}
Also,
\begin{equation}\label{12.50}
\beta_e=\begin{pmatrix}&I_{[\frac{m}{2}]-e}\\I_e\end{pmatrix}.
\end{equation}
Note that $b\mapsto {}^e b$ is an outer conjugation of $H_m$. Next, the action of the representation $\sigma_{[\tau',e]}$ can be expressed in terms of $J_{(\psi_{v,e})_{V_{m^n}}}(\tau')$. In the linear case, it takes the element \eqref{12.46}, with coordinates in $F_v$, to an operator of the form
\begin{equation}\label{12.51}
|\det(aA)|^{\frac{m(2-n)-\delta_H}{2}}J_{(\psi_{v,e})_{V_{m^n}}}(\tau')(diag(M(a,A,X), \begin{pmatrix}a&x&y\\&{}^eb&x'\\&&a^*\end{pmatrix}^{\Delta_{n-1}})),
\end{equation}
where the upper index $\Delta_{n-1}$ denotes the repetition of the matrix $\begin{pmatrix}a&x&y\\&{}^eb&x'\\&&a^*\end{pmatrix}$ $n-1$ times; $M(a,A,X)$ is an $m \times m$ matrix with determinant $\det(aA)$ (It varies a little according to the parities of $m$ and $n$.) The repeated matrix appears $n-1$ times. In the metaplectic case there are slight modifications to \eqref{12.51}.
\begin{prop}\label{prop 12.5}
For all $0\leq e<[\frac{m}{2}]$,
$$
J_{(\psi_{v,e})_{V_{m^n}}}(\tau')=0,
$$
and hence, by \eqref{12.45} and \eqref{12.51},
\begin{multline}\label{12.52}
J_{(\psi_v)_{U_{m^{n-1}}(F_v)}}(\Ind^{H(F_v)}_{Q^{(\epsilon)}_{mn}(F_v)}\tau'\gamma_{\psi_v}^{(\epsilon)})\cong\\ ind^{t^{(\epsilon)}(H^{(\epsilon)}_m(F_v)\times H^{(\epsilon)}_m(F_v))}_{t^{(\epsilon)}(H^{(\epsilon)}_m(F_v)\times H^{(\epsilon)}_m(F_v))\cap Q^{(0),(\epsilon)}(F_v)}\sigma_{[\tau',[\frac{m}{2}]]}.
\end{multline}
\end{prop}
\begin{proof}	
The character $(\psi_{v,e})_{V_{m^n}}$ corresponds to the conjugacy class of the nilpotent matrix
$$
\begin{pmatrix}0\\S&0\\&I_m&0\\&&&\ddots\\&&&I_m&0\\&&&&I_m&0\end{pmatrix},
$$
whose class corresponds to the partition $(n^{m-e}, (n-1)^e, 1^e)$. Thus, if $J_{(\psi_{v,e})_{V_{m^n}}}(\tau')$ is nonzero, then  this partition must be majorized by $P_{\tau'}$. In Cases 1, 2 (listed after Theorem \ref{thm 12.2}), $P_{\tau'}$ is given in \eqref{12.9}. This majorization implies that $m-e\leq m-\mu_0$, i.e. $e\geq \mu_0$, and
	$$
	(m-e)n+e(n-1)\leq (m-\mu_0)n+\mu_0\frac{n}{2},
	$$
	from which we conclude that $e\geq \mu_0\frac{n}{2}=[\frac{m}{2}]$. Since $e<[\frac{m}{2}]$, we get a contradiction. Similarly, in Cases 3, 4, write $n=2n'+1$. Then using \eqref{12.9.1}, we get that $e\geq \mu_0$, and
	$$
	(m-e)n+e(n-1)\leq (m-\mu_0)n+\mu_0(n'+1),
	$$ 
which implies that $e\geq \mu_0n'=[\frac{m}{2}]$, contrary to our assumption.	
\end{proof}

It will convenient to conjugate the character \eqref{12.44}, with $e=[\frac{m}{2}]$, as follows. Consider the following matrix. When $m=2m'$ is even,
$$
a_0=diag(\begin{pmatrix}0&I_{m'}\\I_{m'}&0\end{pmatrix},I_{m(n-1)}).
$$
When $m=2m'-1$ is odd,
$$
a_0=diag(\begin{pmatrix}0&I_{m'-1}&0\\I_{m'-1}&0&0\\0&0&\frac{1}{2}\end{pmatrix},I_{m(n-1)}).
$$
For $z\in V_{m^n}(F_v)$, let
$$
(\psi'_v)_{V_{m^n}}(z)=(\psi_{v,[\frac{m}{2}]})_{V_{m^n}}(a_0^{-1}za_0).
$$
Then in the notation of \eqref{12.44},
\begin{equation}\label{12.52.1}
(\psi'_v)_{V_{m^n}}(z)
=\prod_{i=1}^{n-2}\psi^{-1}(tr(y_{i,i+1}))\psi^{-1}(tr(\begin{pmatrix}0&I_{[\frac{m+1}{2}]}\\0_{[\frac{m}{2}]\times [\frac{m}{2}]}&0\end{pmatrix}c_{n-1})),
\end{equation}
Conjugation by $\tau'(a_0)$ takes $J_{(\psi_{v,[\frac{m}{2}]})_{V_{m^n}}}(\tau')$
to $J_{(\psi'_v)_{V_{m^n}}}(\tau')$, and takes $\sigma_{[\tau',[\frac{m}{2}]]}$ to the following isomorphic representation, which we denote by $\sigma_{\tau'}$ and now write its action explicitly. See \eqref{12.51}.\\
Assume that $m=2m'$ is even and $H$ is linear. Then
\begin{multline}\label{12.53}
\sigma_{\tau'}(t(\begin{pmatrix}a&y\\&a^*\end{pmatrix}, \begin{pmatrix}A&Y\\&A^*\end{pmatrix}))=\\
|\det(aA)|^{\frac{m(2-n)-\delta_H}{2}} J_{(\psi'_v)_{V_{m^n}}}(\tau')(diag( \begin{pmatrix}A\\&a\end{pmatrix}, (\begin{pmatrix}a&y\\&a^*\end{pmatrix}^{\Delta_{n-1}}),
\end{multline}
where $a, A\in \GL_{m'}(F_v)$.\\
In the metaplectic case, we need to multiply by $\gamma_{\psi_v}(\det(a))\gamma_{\psi_v}(\det(A))$.\\
When $m=2m'-1$ is odd
\begin{multline}\label{12.55}
\sigma_{\tau'}(t(\begin{pmatrix}a&x&y\\&1&x'\\&&a^*\end{pmatrix}, \begin{pmatrix}A&X&Y\\&1&X'\\&&A^*\end{pmatrix}))=\\
|\det(aA)|^{\frac{m(2-n)-1}{2}} J_{(\psi'_v)_{V_{m^n}}}(\tau')(diag( \begin{pmatrix}A&0&X\\&a&x\\&&1\end{pmatrix}, (\begin{pmatrix}a&x&y\\&1&x'\\&&a^*\end{pmatrix}^{\Delta_{n-1}}),
\end{multline}
where $a, A\in \GL_{m'-1}(F_v)$.

Now comes the main part of this section, and the next, namely the analysis of the Jacquet module $J_{(\psi_v')_{V_{m^n}}}(\tau')$.  We will first show that $\sigma_{\tau'}\circ t^{(\epsilon)}$ is trivial on the unipotent radical $U_{[\frac{m}{2}]}^{H_m}\times U_{[\frac{m}{2}]}^{H_m}(F_v)$. Denote
\begin{multline}\label{12.57}
u(x,y)=diag(\nu_1(x),\nu_2(y)^{\Delta_{n-1}}),\ \textit{where}\\
\nu_1(x)=\begin{pmatrix}I_{[\frac{m}{2}]}&x\\&I_{[\frac{m+1}{2}]}\end{pmatrix},\
\nu_2(y)=\begin{pmatrix}I_{[\frac{m+1}{2}]}&y\\&I_{[\frac{m}{2}]}\end{pmatrix},
\end{multline}
for all $ x\in M_{[\frac{m}{2}]\times [\frac{m+1}{2}]}(F_v),\ y\in M_{[\frac{m+1}{2}]\times [\frac{m}{2}]}(F_v)$. In case, $m=2m'-1$ is odd, denote
\begin{equation}\label{12.56}
u(x)=diag(\begin{pmatrix}I_{m'-1}\\&a(x)\end{pmatrix}, \begin{pmatrix}a(x)\\&I_{m'-1}\end{pmatrix}^{\Delta_{n-1}});\ a(x)=\begin{pmatrix}I_{m'-1}&x\\&1\end{pmatrix}.
\end{equation}
\begin{prop}\label{prop 12.6}
The elements \eqref{12.57}, and in case $m$ is odd, also the elements \eqref{12.56}, act trivially on the Jacquet module $J_{(\psi_v')_{V_{m^n}}}(\tau')$. Thus, in \eqref{12.53}-\eqref{12.55}, we get that the action of the product of unipotent radicals $U_{[\frac{m}{2}]}^{H_m}\times U_{[\frac{m}{2}]}^{H_m}(F_v)$, via $\sigma_{\tau'}\circ t^{(\epsilon)}$, is trivial.
\end{prop}	
\begin{proof}
This follows from the local analog of Claim 8 in \cite{CFGK17}. See also \cite{GS18}, Prop. 3.2. We start with the elements \eqref{12.57}. Let $D_1$ be the subgroup of the elements $u(x,0)$. This is an abelian group which acts on $J_{(\psi_v')_{V_{m^n}}}(\tau')$. Thus, there is $A\in M_{[\frac{m+1}{2}]\times [\frac{m}{2}]}(F_v)$, such that the Jacquet module of $J_{(\psi_v')_{V_{m^n}}}(\tau')$, with respect to the character   $(u(x,0))\mapsto\psi_v^{-1}(tr(Ax))$, is nontrivial. Note that the group generatd by $V_{m^n}(F_v)$ and $D_1$ is the unipotent radical $V_{[\frac{m}{2}],[\frac{m+1}{2}],m^{n-1}}(F_v)$. If $A\neq 0$, say it is of rank $a\geq 1$, then we get a nontrivial Jacquet module of $\tau'$ with respect to the following character of $V_{[\frac{m}{2}],[\frac{m+1}{2}],m^{n-1}}(F_v)$,
\begin{multline}\label{12.57.1}
\begin{pmatrix}I_{[\frac{m}{2}]}&x_1&\ast&\ast&&\ast\\
&I_{[\frac{m+1}{2}]}&x_2&\ast&&\ast\\&&I_m&x_3&&\ast\\&&&\ddots\\&&&&I_m&x_n\\&&&&&I_m\end{pmatrix}\mapsto\\ \psi^{-1}_v(tr(Ax_1)+tr(\begin{pmatrix}I_{[\frac{m+1}{2}]}\\0_{[\frac{m}{2}]\times [\frac{m+1}{2}]}\end{pmatrix}x_2)+tr(x_3)+\cdots tr(x_n)).
\end{multline}
The character \eqref{12.57.1} corresponds to a partition of $mn$ of the form $((n+1)^a,...)$. Since $a\geq 1$, we get a contradiction to \eqref{12.9}, \eqref{12.9.1}. Thus, $A$ must be zero, and hence the action of the elements $u(x,0)$ on $J_{(\psi_v')_{V_{m^n}}}(\tau')$ is trivial. We conclude that
\begin{equation}\label{12.57.2}
J_{(\psi_v')_{V_{m^n}}}(\tau')\cong J_{(\psi_v')_{V_{[\frac{m}{2}],[\frac{m+1}{2}],m^{n-1}}}}(\tau'),
\end{equation}
where $(\psi_v')_{V_{[\frac{m}{2}],[\frac{m+1}{2}],m^{n-1}}}$ dnotes the extension by the trivial character of $(\psi'_v)_{V_{m^n}}$. The isomorphism \eqref{12.57.2} is that of modules over the the subgroup generated by $V_{[\frac{m}{2}],[\frac{m+1}{2}],m^{n-1}}$ and the stabilizer of $(\psi_v')_{V_{[\frac{m}{2}],[\frac{m+1}{2}],m^{n-1}}}$ in the Levi part of\\ $P_{[\frac{m}{2}],[\frac{m+1}{2}],m^{n-1}}(F_v)$.
Next, let $D_2$ be the subgroup of the elements $u(0,y)$. Again, this is an abelian group which acts on $J_{(\psi_v')_{V_{m^n}}}(\tau')$, and hence, there is a matrix $B\in M_{[\frac{m}{2}]\times [\frac{m+1}{2}]}(F_v)$, such that the Jacquet module of $J_{(\psi_v')_{V_{m^n}}}(\tau')$, with respect to the character of $D_2$ given by $\psi_{v,B}((u(0,y)))=\psi_v^{-1}(tr(By))$, is nontrivial. As before, we want to show that $B$ must be zero. The proof is the local analog of the proof given (in the global set up) in \cite{GS18}, Prop. 3.2, starting from (3.35) (replacing $e$ there by $[\frac{m}{2}]$). As we did there, assuming that $B\neq 0$, say it is of rank $b\geq 1$, we carry out the process of exchanging roots. The process of exchanging roots at the local level is explained in \cite{GRS99}, Sec. 2.2. It works in exactly the same way as in the global set up in \cite{GRS11}, Sec. 7.1. Write an element of $V_{m^n}(F_v)$ as
\begin{equation}\label{12.57'}
z=\begin{pmatrix}I_m&x_1&\ast&&\ast\\&I_m&x_2&&\ast\\&&I_m&&\ast\\&&&\ddots&\ast\\&&&I_m&x_{n-1}\\&&&&Im\end{pmatrix}.
\end{equation} 
Write in the matrix \eqref{12.57'}, for $2\leq
j\leq n-1$,
$$
x_j=\begin{pmatrix}x_{j,1}&x_{j,2}\\x_{j,3}&x_{j,4}\end{pmatrix},
$$
where $x_{j,1}$ is of size $[\frac{m+1}{2}]\times [\frac{m+1}{2}]$. Now it is easy to check that one can exchange (in the sense of Sec. 2.2 in
\cite{GRS99}) $x_{2,3}$ with
$diag(I_m,\nu_2(y_1),I_{(n-2)m})$. This means the following. Denote by $V^0$ the subgroup generated by $D_2$ and $V_{m^n}(F_v)$, and denote by $(\psi_v)_{V^0,B}$ the character of $V^0$ obtained by
extending $(\psi'_v)_{V_{m^n}}$ by $\psi_{v,B}$. Denote by $\tilde{V}^1$ the subgroup of $V^0$ generated by $D_2$ and the subgroup of $V_{m^n}(F_v)$ obtained by replacing, in \eqref{12.57'}, 
$x_{2,3}$ by zero. Denote by $(\psi_v)_{\tilde{V}^1,B}$ the restriction of $(\psi_v)_{V^0,B}$ to $\tilde{V}^1$. Finally, denote by $V^1$ the subgroup generated by $\tilde{V}^1$ and the subgroup of elements $diag(I_m,\nu_2(y_1),I_{(n-2)m})$, and let $(\psi_v)_{V^1,B}$ be the extension of $(\psi_v)_{\tilde{V}^1,B}$ by the trivial character (of the last subgroup). Then the meaning of the root exchange that we just did is that we have an isomorphism 
\begin{equation}\label{12.57.3}
J_{(\psi_v)_{V^0,B}}(\tau')\cong J_{(\psi_v)_{V^1,B}}(\tau'). 
\end{equation}
This is an isomorphism of $\tilde{V}^1$-modules. We conclude that $J_{(\psi_v)_{V^1,B}}(\tau')$
is non trivial. This is the way the analogy (with the global set up) of exchanging roots works. We continue with the same analogy, as we did right after (3.36) in \cite{GS18}. So, next, we exchange roots in the r.h.s. of \eqref{12.57.3}, $x_{3,3}$ with $diag(I_{2m},\nu_2(y_2),I_{(n-3)m})$, and then we exchange$x_{4,3}$ with
 $diag(I_{3m},\nu_2(y_3),I_{(n-4)m})$, and so on.\\ 
 In general, we exchange $x_{j,3}$ with
$diag(I_{jm},\nu_2(y_j),I_{(n-j-1)m})$, step by step, up to $j=n-2$, and get, in each step, an isomorphism of $V^{j-1}\cap V^j$-modules,
\begin{equation}\label{12.57.4}
J_{(\psi_v)_{V^{j-1},B}}(\tau')\cong J_{(\psi_v)_{V^j,B}}(\tau'),
\end{equation}
where $V^j$ is the subgroup generated by $D_2$, the subgroup of $V_{m^n}(F_v)$ obtained by replacing in \eqref{12.57'} $x_{2,3}, x_{3,3},...,x_{j+1,3}$ by zero, and the subgroup of elements $diag(I_m,\nu_2(y_1),...,\nu_2(y_j),I_{(n-j-1)m})$. The character $(\psi_v)_{V^j,B}$ of $V^j$ is trivial on the third out of the last three subgroups, and on the first two subgroups it is given by the restriction from $(\psi_v)_{V^0,B}$. Note that this sequence of root exchanges is similar to one of the sequences in the proof of Prop. \ref{prop 4.1}.
By \eqref{12.57.3}, \eqref{12.57.4}, we get
\begin{equation}\label{12.57.5}
J_{(\psi_v)_{V^0},B}(\tau')\cong J_{(\psi_v)_{V^{n-2},B}}(\tau'), 
\end{equation}
as $V^0\cap V^{n-2}$-modules. In particular, $J_{(\psi_v)_{V^{n-2},B}}(\tau')$ is nontrivial. Note that $V^{n-2}$ is the unipotent radical $V_{m,[\frac{m+1}{2}],m^{n-2},[\frac{m}{2}]}(F_v)$. The character $(\psi_v)_{V^{n-2},B}$ corresponds to a partition of $mn$ of the form $((n+1)^b,...)$, and hence we get a contradiction to \eqref{12.9}, \eqref{12.9.1}. We also conclude that
\begin{equation}\label{12.57.6}
J_{(\psi_v')_{V_{m^n}}}(\tau')\cong J_{(\psi_v)_{V^0},0}(\tau'),
\end{equation}
as $V^0$-modules. Let $V'$ be the subgroup generated by $V^0$ and $D_1$. This is the subgroup generated by $V_{m^n}(F_v)$, $D_1$ and $D_2$. Let $(\psi_v)_{V'}$ be the character of $V'$ which is trivial on $D_1$, $D_2$ and is $(\psi'_v)_{V_{m^n}}$ on $V_{m^n}$. Then \eqref{12.57.2} and \eqref{12.57.6} show that we have an isomorphism of $V'$-modules 
\begin{equation}\label{12.57.6'}
J_{(\psi_v')_{V_{m^n}}}(\tau')\cong J_{(\psi_v)_{V'}}(\tau').
\end{equation}
This proves the proposition for the elements $u(x,y)$. The proof for the elements $u(x)$ in \eqref{12.56} is very similar, and hence we will be brief. Recall that here we assume that $m=2m'-1$ is odd.
For this, we go back to the notation \eqref{12.57'}. Assume that $z\in V_{m^n}(F_v)\cap V_{m'-1,(m')^2,m^{n-2},m'-1}(F_v)$  and write the blocks $x_j$ in the following form. 
\begin{multline}\label{12.62}
x_1=\begin{pmatrix}x^{1,1}&x^{1,2}&x^{1,3}\\x^{1,4}&x^{1,5}&x^{1,6}\\x^{1,7}&x^{1,8}&x^{1,9}\end{pmatrix},\ x^{1,1}\in M_{m'-1}(F_v),\ x^{1,8}\in F_v\\
\textit{For}\ \ 2\leq j\leq n-1,\ \
x_j=\begin{pmatrix}x^{j,1}&x^{j,2}&x^{j,3}\\x^{j,4}&x^{j,5}&x^{j,6}\\0&0&x^{1,9}\end{pmatrix},\ x^{j,1}, x^{j,9}\in M_{m'-1}(F_v).
\end{multline}
Let $D_3$ denote the subgroup of elements $u(x)$, $x\in F_v^{m'-1}$.
Now we exchange the subgroup of elements\\
$z\in V_{m^n}(F_v)\cap V_{m'-1,(m')^2,m^{n-2},m'-1}(F_v)$ (in the notation of \eqref{12.57'} and \eqref{12.62}), such that $x_j=0$, for all $2\leq j\leq n-1$, and such that in \eqref{12.62}, $x^{1,i}=0$, except for $i=7$, with the subgroup of elements
$diag(\begin{pmatrix}I_{m'-1}\\&a(x_1)\end{pmatrix},I_{m(n-1)})$, and then we exchange, for $j=2,...,n-1$, the subgroup of elements $z\in V_{m^n}(F_v)\cap V_{m'-1,(m')^2,m^{n-2},m'-1}(F_v)$, such that $x_{j'}=0$, for all $j'\neq j$, and such that in \eqref{12.62}, $x^{j,i}=0$, except for $i=4$, with the subgroup of elements\\
$diag(I_{m(j-1)},\begin{pmatrix}a(x_j)\\&I_{m'-1}\end{pmatrix},I_{m(n-j)})$. Now, exactly as in the proof for the elements $u(x,y)$, we get that	$J_{(\psi_v)_{V_{m'-1,(m')^2,m^{n-2},m'-1}}}(\tau')$ is invariant under the action of $D_3$. This completes the proof of the proposition.
	
\end{proof}

Let $(\psi_v)_{V_{[\frac{m}{2}],[\frac{m+1}{2}]^2,m^{n-2},[\frac{m}{2}]}}$ denote the following character of\\
 $V_{[\frac{m}{2}],[\frac{m+1}{2}]^2,m^{n-2},[\frac{m}{2}]}(F_v)$,
\begin{multline}\label{12.58}
\begin{pmatrix}I_{[\frac{m}{2}]}&x_1&\ast&\ast&\ast&&\ast\\
&I_{[\frac{m+1}{2}]}&x_2&\ast&\ast&&\ast\\&&I_{[\frac{m+1}{2}]}&x_3&\ast&&\ast\\&&&I_m&x_4&&\ast
\\&&&&\ddots\\&&&&&I_m&x_{n+1}\\&&&&&&I_{[\frac{m}{2}]}\end{pmatrix}\mapsto\\ \psi^{-1}_v(tr(x_2)+tr(\begin{pmatrix}I_{[\frac{m+1}{2}]}\\0_{[\frac{m}{2}]\times [\frac{m+1}{2}]}\end{pmatrix}x_3)+tr(x_4+\cdots+x_n)+tr(\begin{pmatrix}I_{[\frac{m}{2}]}&0_{[\frac{m}{2}]\times [\frac{m+1}{2}]}\end{pmatrix}x_{n+1})).
\end{multline}

\begin{cor}\label{cor 12.7}
We have an isomorphism of $V'\cap V_{[\frac{m}{2}],[\frac{m+1}{2}]^2,m^{n-2},[\frac{m}{2}]}$-modules
\begin{equation}\label{12.57.7}
J_{(\psi_v')_{V_{m^n}}}(\tau')\cong J_{(\psi_v)_{V_{[\frac{m}{2}],[\frac{m+1}{2}]^2,m^{n-2},[\frac{m}{2}]}}}(\tau').
\end{equation}
Embed $\GL_{[\frac{m}{2}]}(F_v)\times \GL_{[\frac{m+1}{2}]}(F_v)\times \GL_{[\frac{m}{2}]}(F_v)$ inside $\GL_{mn}(F_v)$ by
\begin{equation}\label{12.59}
j(\alpha,\beta,\gamma)=diag(\begin{pmatrix}\alpha\\&\beta\end{pmatrix},\begin{pmatrix}\beta\\&\gamma\end{pmatrix}^{\Delta_{n-1}}). 
\end{equation}
Both Jacquet modules in \eqref{12.57.7} are modules over $j(\GL_{[\frac{m}{2}]}(F_v)\times \GL_{[\frac{m+1}{2}]}(F_v)\times \GL_{[\frac{m}{2}]}(F_v))$. The isomorphism \eqref{12.57.7} sends the action of $j(\alpha, \beta, \gamma)$ on $J_{(\psi_v')_{V_{m^n}}}(\tau')$ to the action of $j(\alpha, \beta, \gamma)$ on $J_{(\psi_v')_{V_{[\frac{m}{2}],[\frac{m+1}{2}]^2,m^{n-2},[\frac{m}{2}]}}}(\tau')$, twisted by the character
\begin{equation}\label{12.57.8} (|\det(\gamma)|^{[\frac{m+1}{2}]}|\det(\beta)|^{-[\frac{m}{2}]})^{n-2}.
\end{equation}
\end{cor}
\begin{proof}
As we noted in the last proof, $V^{n-2}$ is the unipotent radical\\ $V_{m,[\frac{m+1}{2}],m^{n-2},[\frac{m}{2}](F_v)}$. Also, the subgroup generated by $D_1$ and
 $V_{m,[\frac{m+1}{2}],m^{n-2},[\frac{m}{2}](F_v)}$ is $V_{[\frac{m}{2}],[\frac{m+1}{2}]^2,m^{n-2},[\frac{m}{2}]}$. The proof of Prop. \ref{prop 12.6} shows that the isomorphism \eqref{12.57.5} is valid when $B=0$. Thus, by \eqref{12.57.2} and \eqref{12.57.6}, we get the isomorphism \eqref{12.57.7}. The fact that the isomorphism is of $j(\GL_{[\frac{m}{2}]}(F_v)\times \GL_{[\frac{m+1}{2}]}(F_v)\times \GL_{[\frac{m}{2}]}(F_v))$-modules, up to the character \eqref{12.57.8}, follows from the following general fact on the isomorphism of Jacquet modules obtained when we exchange roots. To explain it, we will use the notation of \cite{GRS99}, Sec. 2.2. Let the unipotent subgroups (of $\GL_{mn}(F_v)$) $C, X, Y, B=CY, D=CX, A=BX=DY$ satisfy the assumptions (i) - (v) in the beginning of Sec. 2.2, \cite{GRS99}. Consider the character $\chi=\chi_C$ of $C$,  as in condition (v), and extend it by the trivial character of $Y$ (resp. $X$) to a character $\chi_B$ of $B$ (resp. $\chi_D$ of D). Let $\pi$ be a smooth representation of $A$, then it follows from \cite{GRS99}, Lemma 2.2, that the $C$-module isomorphism (where we exchange $Y$ with $X$)
 $$
 i: J_{B,\chi,B}(\pi)\widetilde{\rightarrow} J_{D,\chi_D}(\pi)
 $$
 satisfies the following property. Let $R$ be a subgroup (of $\GL_{mn}(F_v)$) which normalizes $C, Y, X$, and preserves $\chi_C$. For $r\in R$, let $\delta^Y(r)$ be the Jacobian of the conjugation action of $r$ on $Y$. Then, for all $r\in R$, we have
 \begin{equation}\label{12.57.9}
 i\circ \bar{\pi}_B(r)=\delta^Y(r)\bar{\pi}_D\circ i.
 \end{equation}
 We take the subgroup $R=j(\GL_{[\frac{m}{2}]}(F_v)\times \GL_{[\frac{m+1}{2}]}(F_v)\times \GL_{[\frac{m}{2}]}(F_v))$. In each step of the root exchanges in the proof of Prop. \ref{prop 12.6}, namely \eqref{12.57.4}, the Jacobian character \eqref{12.57.9} is $|\det(\gamma)|^{[\frac{m+1}{2}]}|\det(\beta)|^{-[\frac{m}{2}]}$, and we exchanged roots $n-2$ times. This proves \eqref{12.57.8}.  

\end{proof}	

\begin{rmk}\label{rmk 12.8}
	The unipotent radical $V_{[\frac{m}{2}],[\frac{m+1}{2}]^2,m^{n-2},[\frac{m}{2}]}$ is the subgroup generated by the following two subgroups: The subgroup $D_0$ of elements\\
	 $diag(\nu_1(y_1),\nu_2(y_2),...,\nu_2(y_n))$, and the subgroup $V^0_{m^n}(F_v)$ of $V_{m^n}(F_v)$ of the elements \eqref{12.57'}, with $x_{j,3}=0$, for $2\leq j\leq n-1$. The character $(\psi_v')_{V_{[\frac{m}{2}],[\frac{m+1}{2}]^2,m^{n-2},[\frac{m}{2}]}}$ is trivial on $D_0$ and on $V^0_{m^n}(F_v)$ it is given by the restriction of $(\psi'_v)_{V_{m^n}}$.
\end{rmk}

Let us go back to \eqref{12.53}-\eqref{12.55}. By the last proposition and corollary, we may replace $\sigma_{\tau'}$ by an isomorphic representation, which we will keep denoting $\sigma_{\tau'}$, such that when $m=2m'$ is even and $H$ is linear,
\begin{multline}\label{12.63}
\sigma_{\tau'}(t(\begin{pmatrix}a&y\\&a^*\end{pmatrix}, \begin{pmatrix}A&Y\\&A^*\end{pmatrix}))=|\det(A)|^{\frac{m(2-n)-\delta_H}{2}}|\det(a)|^{\frac{3m(2-n)-\delta_H}{2}}\\
 J_{(\psi'_v)_{V_{(m')^3,m^{n-2},m'}}}(\tau')(diag( \begin{pmatrix}A\\&a\end{pmatrix}, (\begin{pmatrix}a\\&a^*\end{pmatrix}^{\Delta_{n-1}}),
\end{multline}
where $a, A\in \GL_{m'}(F_v)$.\\
In the metaplectic case, we need to multiply by $\gamma_{\psi_v}(\det(a))\gamma_{\psi_v}(\det(A))$.
When $m=2m'-1$ is odd
\begin{multline}\label{12.63.1}
\sigma_{\tau'}(t(\begin{pmatrix}a&x&y\\&1&x'\\&&a^*\end{pmatrix}, \begin{pmatrix}A&X&Y\\&1&X'\\&&A^*\end{pmatrix}))=|\det(A)|^{\frac{m(2-n)-1}{2}}|\det(a)|^{\frac{3m(2-n)-1}{2}}\\
 J_{(\psi_v)_{V_{m'-1,(m')^2,m^{n-2},m'-1}}}(\tau')(diag( \begin{pmatrix}A\\&a\\&&1\end{pmatrix}, \begin{pmatrix}a\\&1\\&&a^*\end{pmatrix}^{\Delta_{n-1}}).
\end{multline}
where $a, A\in \GL_{m'-1}(F_v)$.

As in Sec. 5, we will now apply to $J_{(\psi_v)_{V_{[\frac{m}{2}],[\frac{m+1}{2}]^2,m^{n-2},[\frac{m}{2}]}}}(\tau')$ a conjugation by the following element, analogous to \eqref{6.1},
\begin{equation}\label{12.64}
\zeta_0=\begin{pmatrix}I_m\\&\alpha_0\\&&\ddots\\&&&\alpha_0\\0&\beta_0\\&&\ddots\\&&&\beta_0\end{pmatrix},
\end{equation}
where 
$$
\alpha_0=\begin{pmatrix}I_{[\frac{m+1}{2}]}&0_{[\frac{m+1}{2}]\times [\frac{m}{2}]}\end{pmatrix},\
\beta_0=\begin{pmatrix}0_{[\frac{m}{2}]\times [\frac{m+1}{2}]}&I_{[\frac{m}{2}]}\end{pmatrix};
$$
The number of $\alpha_0$'s and the number of $\beta_0$'s in \eqref{12.64} is $n-1$. Let\\ $\mathcal{D}_{n,m}=\zeta_0V_{[\frac{m}{2}],[\frac{m+1}{2}]^2,m^{n-2},[\frac{m}{2}]}(F_v)\zeta_0^{-1}$. Let $(\psi_v)_{\mathcal{D}_{n,m}}$ be the character of $\mathcal{D}_{n,m}$ given by
$$
(\psi_v)_{\mathcal{D}_{n,m}}(v)=(\psi_v)_{V_{[\frac{m}{2}],[\frac{m+1}{2}]^2,m^{n-2},[\frac{m}{2}]}}(\zeta_0^{-1}v\zeta_0).
$$
We get the Jacquet module $J_{(\psi_v)_{\mathcal{D}_{n,m}}}(\tau')$, and in \eqref{12.63}, \eqref{12.63.1}, $\sigma_{\tau'}$ is conjugated to the following representation $\pi_{\tau'}$, such that  when $m=2m'$ is even and $H$ is linear,
\begin{multline}\label{12.64.1}
\pi_{\tau'}(t(\begin{pmatrix}a&y\\&a^*\end{pmatrix}, \begin{pmatrix}A&Y\\&A^*\end{pmatrix}))=|\det(A)|^{\frac{m(2-n)-\delta_H}{2}}|\det(a)|^{\frac{3m(2-n)-\delta_H}{2}}\\
J_{(\psi_v)_{\mathcal{D}_{n,m}}}(\tau')(diag( A,a^{\Delta_n},(a^*)^{\Delta_{n-1}}),
\end{multline}
where $a, A\in \GL_{m'}(F_v)$.\\
In the metaplectic case, we need to multiply by $\gamma_{\psi_v}(\det(a))\gamma_{\psi_v}(\det(A))$.
When $m=2m'-1$ is odd
\begin{multline}\label{12.64.2}
\sigma_{\tau'}(t(\begin{pmatrix}a&x&y\\&1&x'\\&&a^*\end{pmatrix}, \begin{pmatrix}A&X&Y\\&1&X'\\&&A^*\end{pmatrix}))=|\det(A)|^{\frac{m(2-n)-1}{2}}|\det(a)|^{\frac{3m(2-n)-1}{2}}\\
J_{(\psi_v)_{\mathcal{D}_{n,m}}}(\tau')(diag(A, \begin{pmatrix}a\\&1\end{pmatrix}^{\Delta_n},(a^*)^{\Delta_{n-1}})).
\end{multline}
where $a, A\in \GL_{m'-1}(F_v)$. Our main goal now will be to analyze the Jacquet module
$J_{(\psi_v)_{\mathcal{D}_{n,m}}}(\tau')$, and through it the representation $\pi_{\tau'}\circ t^{(\epsilon)}$, viewed as a representation of $\GL^{(\epsilon)}_{[\frac{m}{2}]}(F_v)\times \GL^{(\epsilon)}_{[\frac{m}{2}]}(F_v)$. Let us describe first the group $\mathcal{D}_{n,m}$ and its character $(\psi_v)_{\mathcal{D}_{n,m}}$. As in \eqref{6.3}, the elements of $\mathcal{D}_{n,m}$ have the form 
\begin{equation}\label{12.65}
v=\begin{pmatrix}U&X_{1,2}\\Y_{2,1}&S\end{pmatrix},
\end{equation}
where $U\in V_{[\frac{m}{2}], [\frac{m+1}{2}]^n}(F_v)$; 
\begin{equation}\label{12.66}
U=\begin{pmatrix}I_{[\frac{m}{2}]}&*&*&\cdots&*&*\\
&I_{[\frac{m+1}{2}]}&U_2&\cdots&*&*\\
& &I_{[\frac{m+1}{2}]}&\cdots &* & * \\
& & & \cdots& \cdots&\cdots &\\
& & &       &I_{[\frac{m+1}{2}]}&U_n\\
& & &       & &I_{[\frac{m+1}{2}]}\end{pmatrix}.
\end{equation}
The matrix $S$ lies in $V_{[\frac{m}{2}]^{n-1}}(F_v)$;
\begin{equation}\label{12.67}
S=\begin{pmatrix}I_{[\frac{m}{2}]}&S_1&*&\cdots&*&*\\
&I_{[\frac{m}{2}]}&S_2&\cdots&*&*\\
& &I_{[\frac{m}{2}]}&\cdots &* & * \\
& & & \cdots& \cdots&\cdots &\\
& & &       &I_{[\frac{m}{2}]}&S_{n-2}\\
& & &       & &I_{[\frac{m}{2}]}\end{pmatrix}.
\end{equation}
Write $X_{1,2}$ as a matrix of blocks with block rows according to those of $U$ and block columns according to those of $S$, and similarly, write $Y_{2,1}$ as a matrix of blocks according to the block row division of $S$, and block column division of $U$. The first two block rows of $X_{1,2}$ are arbitrary, and the block matrix obtained from $X_{1,2}$ after deleting the first two block rows has an upper triangular shape. For example, for $n=5$,
$$
X_{1,2}=\begin{pmatrix}\ast&\ast&\ast&\ast\\\ast&\ast&\ast&\ast\\\ast&\ast&\ast&\ast\\0&\ast&\ast&\ast\\0&0&\ast&\ast\\0&0&0&\ast\end{pmatrix}.
$$
The first four block columns and the last two block rows of $Y_{1,2}$ are zero. If we delete the first two block columns of $Y_{1,2}$, then we obtain a matrix of blocks with upper triangular shape, such that the main block diagonal and the next one above it are zero. For example, for $n=6$,
$$
Y_{1,2}=\begin{pmatrix}0&0&0&0&\ast&\ast&\ast\\0&0&0&0&0&\ast&\ast\\0&0&0&0&0&0&\ast\\
0&0&0&0&0&0&0\\0&0&0&0&0&0&0\end{pmatrix}.
$$
For $1\leq i\leq n+1$, $1\leq j\leq n-1$, denote by $X_{1,2}^{(i,j)}$ the block of $X_{1,2}$ in block row $i$ and block column $j$, and let $\mathrm{X}_{1,2}^{i,j}$ denote the subgroup of matrices \eqref{12.65}, with $U$, $S$ being identity matrices, $Y_{2,1}=0$, and all blocks of $X_{1,2}$ are zero except $X_{1,2}^{(i,j)}$, which may be arbitrary. Similarly, for $1\leq i\leq n-1$, $1\leq j\leq n+1$, denote by $Y_{2,1}^{(i,j)}$ the block of $Y_{2,1}$ in block row $i$ and block column $j$, and let $\mathrm{Y}_{2,1}^{i,j}$ denote the subgroup of matrices \eqref{12.65}, with $U$, $S$ being identity matrices, $X_{1,2}=0$, and all blocks of $Y_{2,1}$ are zero except $Y_{2,1}^{(i,j)}$, which may be arbitrary. Finally, in the notation of \eqref{12.65}-\eqref{12.67},
\begin{equation}\label{12.68}
(\psi_v)_{\mathcal{D}_{n,m}}(v)=\psi_v^{-1}(tr(U_2+\cdots U_n)+tr(S_1+\cdots S_{n-2})).
\end{equation}
Consider the unipotent radical $V_{[\frac{m}{2}],[\frac{m+1}{2}]^n,[\frac{m}{2}]^{n-1}}(F_v)$. Write its elements as
$$
\begin{pmatrix}U&X\\0&S\end{pmatrix},
$$
where $U, S$ are as in \eqref{12.66}, \eqref{12.67}. Extend the character $(\psi_v)_{\mathcal{D}_{n,m}}$ to the character $(\psi_v)_{V_{[\frac{m}{2}],[\frac{m+1}{2}]^n,[\frac{m}{2}]^{n-1}}}$ given by the same formula as \eqref{12.68}. 
\begin{prop}\label{prop 12.9}
We have an isomorphism of $\mathcal{D}_{n,m}\cap V_{[\frac{m}{2}],[\frac{m+1}{2}]^n,[\frac{m}{2}]^{n-1}}(F_v)$-modules
\begin{equation}\label{12.69}
J_{(\psi_v)_{\mathcal{D}_{n,m}}}(\tau')\cong J_{(\psi_v)_{V_{[\frac{m}{2}],[\frac{m+1}{2}]^n,[\frac{m}{2}]^{n-1}}}}(\tau').
\end{equation}
Embed $\GL_{[\frac{m}{2}]}(F_v)\times \GL_{[\frac{m+1}{2}]}(F_v)\times \GL_{[\frac{m}{2}]}(F_v)$ inside $\GL_{mn}(F_v)$ by
\begin{equation}\label{12.70}
j'(\alpha,\beta,\gamma)=diag(\alpha, \beta^{\Delta_n},\gamma^{\Delta_{n-1}}). 
\end{equation}
Both Jacquet modules in \eqref{12.69} are modules over $j'(\GL_{[\frac{m}{2}]}(F_v)\times \GL_{[\frac{m+1}{2}]}(F_v)\times \GL_{[\frac{m}{2}]}(F_v))$. The isomorphism \eqref{12.69} sends the action of $j'(\alpha, \beta, \gamma)$ on $J_{(\psi_v)_{\mathcal{D}_{n,m}}}(\tau')$ to the action of $j'(\alpha, \beta, \gamma)$ on $J_{(\psi_v)_{V_{[\frac{m}{2}],[\frac{m+1}{2}]^n,[\frac{m}{2}]^{n-1}}}}(\tau')$, twisted by the character
\begin{equation}\label{12.71} (|\det(\gamma)|^{[\frac{m+1}{2}]}|\det(\beta)|^{-[\frac{m}{2}]})^{\frac{(n-3)(n-2)}{2}}.
\end{equation}
\end{prop}
\begin{proof}
We will exchange roots, as we did in Theorem \ref{thm 6.1} when we exchanged there columns of $Y_{3,1}$ with corresponding rows of $X_{1,3}$. The details are similar. We exchange in $\mathcal{D}_{n,m}$, in the notation of \eqref{12.65}, $\mathrm{Y}_{2,1}^{1,5}$ with $\mathrm{X}_{1,2}^{4,1}$, then $\mathrm{Y}_{2,1}^{2,6}$ with $\mathrm{X}_{1,2}^{5,2}$, $\mathrm{Y}_{2,1}^{1,6}$ with $\mathrm{X}_{1,2}^{5,1}$. We continue, column by column of $Y_{2,1}$. In (block) column $5\leq j$, we exchange $\mathrm{Y}_{2,1}^{i,j}$ with $\mathrm{X}_{1,2}^{j-1,i}$, for $i=j-4,j-5,...,1$, in this order, and we do this for $j=5,6,...,n+1$, in this order. We then reach the following subgroup ${}^0V$. It consists of all matrices 
\begin{equation}\label{12.71.1}
v=\begin{pmatrix}U&X_{1,2}\\0&S\end{pmatrix},
\end{equation}
where $U, S$ are as in \eqref{12.66}, \eqref{12.67}, and $X_{1,2}$ is such that $X_{1,2}^{(n+1,j)}=0$, for $1\leq j\leq n-2$, and all other blocks are arbitrary. Let $(\psi_v)_{{}^0V}$ be the character of ${}^0V$, given on the last matrix by the same formula as \eqref{12.68}. Then we conclude that 
\begin{equation}\label{12.72}
J_{(\psi_v)_{\mathcal{D}_{n,m}}}(\tau')\cong J_{(\psi_v)_{{}^0V}}(\tau'),
\end{equation}
as $\mathcal{D}_{n,m}\cap {}^0V$-modules. Applying \eqref{12.57.9} for each step of the last root exchanges, we get that the isomorphism \eqref{12.72} sends the action of $j'(\alpha, \beta, \gamma)$ on $J_{(\psi_v)_{\mathcal{D}_{n,m}}}(\tau')$ to the action of $j'(\alpha, \beta, \gamma)$ on $J_{(\psi_v)_{{}^0V}}(\tau')$ twisted by the character \eqref{12.71}. In order to complete the proof, we will show that all the subgroups $\mathrm{X}^{n+1,j}$, $1\leq j\leq n-2$ act trivially on $J_{(\psi_v)_{{}^0V}}(\tau')$.
For this, let us write $J_{(\psi_v)_{{}^0V}}(\tau')$ as a composition of of two Jacquet modules,
$$
J_{(\psi_v)_{{}^0V}}(\tau')\cong J_{(\psi_v)_{{}^2V}}(J_{(\psi_v)_{{}^1V}}(\tau')),
$$
where ${}^1V$ is the subgroup of ${}^0V$ consisting of the matrices \eqref{12.71.1}, such that $S$ is the identity, and ${}^2V$ is the subgroup of ${}^0V$ consisting of the matrices \eqref{12.71.1}, such that $U$ is the identity and $X_{1,2}=0$. The characters
$(\psi_v)_{{}^1V}$, $(\psi_v)_{{}^2V}$ are obtained by restriction of $(\psi_v)_{{}^0V}$.
Let $\mathrm{X}_{n+1}$ be the subgroup generated by $\mathrm{X}^{n+1,j}$, $1\leq j\leq n-2$. Its elements have the form
$$
v(x)=\begin{pmatrix}I_{[\frac{m}{2}]+(n-1)[\frac{m+1}{2}]}\\&I_{[\frac{m+1}{2}]}&x\\&&I_{(n-1)[\frac{m}{2}]}\end{pmatrix},
$$
where $x\in M_{[\frac{m+1}{2}]\times (n-1)[\frac{m}{2}]}(F_v)$. It is enough to show that $\mathrm{X}_{n+1}$ acts trivially on $J_{(\psi_v)_{{}^1V}}(\tau')$. Let $A\in M_{(n-1)[\frac{m}{2}]\times [\frac{m+1}{2}]}(F_v)$ be, such that the Jacquet module of $J_{(\psi_v)_{{}^1V}}(\tau')$, with respect to the character of $\mathrm{X}_{n+1}$ given by $\psi_{v,A}: v(x)\mapsto \psi_v(tr(xA))$, is nontrivial. We have to show that we must have $A=0$. Denote the rank of $A$ by $\ell$, and assume that $\ell\geq 1$. We obtain the nontrivial Jacquet module of $\tau'$ with respect to the character $(\psi'{_v,A})_{V_{[\frac{m}{2}], [\frac{m+1}{2}]^n,(n-1)[\frac{m}{2}]}}$ of $V_{[\frac{m}{2}], [\frac{m+1}{2}]^n,(n-1)[\frac{m}{2}]}(F_v)$ given by $(\psi_v)_{{}^0V}$ on ${}^0V$ and by $\psi_{v,A}$ on $\mathrm{X}_{n+1}$. Now, this character corresponds to the nilpotent orbit attached to a partion of the form $((n+1)^\ell,...)$ of $mn$, contradicting \eqref{12.9}, \eqref{12.9.1}. This completes the proof of the proposition.

\end{proof}

Let $(\psi_v)_{V_{[\frac{m}{2}],[\frac{m+1}{2}]^n,[\frac{m}{2}]^{n-1}}}$ be the character of $V_{[\frac{m}{2}],[\frac{m+1}{2}]^n,[\frac{m}{2}]^{n-1}}(F_v)$ corresponding to $A=0$ in the last proof. 
By Prop. \ref{prop 12.9}, $\pi_{\tau'}$ in \eqref{12.64.1}, \eqref{12.64.2}, is conjugated to to the following representation, which we keep denoting by $\pi_{\tau'}$, such that  when $m=2m'$ is even and $H$ is linear,
\begin{multline}\label{12.73}
\pi_{\tau'}(t(\begin{pmatrix}a&y\\&a^*\end{pmatrix}, \begin{pmatrix}A&Y\\&A^*\end{pmatrix}))=|\det(A)|^{\frac{m(2-n)-\delta_H}{2}}|\det(a)|^{\frac{mn(2-n)-\delta_H}{2}}\\
J_{(\psi_v)_{V_{[\frac{m}{2}],[\frac{m+1}{2}]^n,[\frac{m}{2}]^{n-1}}}}(\tau')(diag( A,a^{\Delta_n},(a^*)^{\Delta_{n-1}}),
\end{multline}
where $a, A\in \GL_{m'}(F_v)$.\\
In the metaplectic case, we need to multiply by $\gamma_{\psi_v}(\det(a))\gamma_{\psi_v}(\det(A))$.
When $m=2m'-1$ is odd
\begin{multline}\label{12.74}
\pi_{\tau'}(t(\begin{pmatrix}a&x&y\\&1&x'\\&&a^*\end{pmatrix}, \begin{pmatrix}A&X&Y\\&1&X'\\&&A^*\end{pmatrix}))=|\det(A)|^{\frac{m(2-n)-1}{2}}|\det(a)|^{\frac{mn(2-n)-1}{2}}\\
J_{(\psi_v)_{V_{[\frac{m}{2}],[\frac{m+1}{2}]^n,[\frac{m}{2}]^{n-1}}}}(\tau')(diag(A, \begin{pmatrix}a\\&1\end{pmatrix}^{\Delta_n},(a^*)^{\Delta_{n-1}})).
\end{multline}
where $a, A\in \GL_{m'-1}(F_v)$.

\section{The correspondence at unramified places II: Analysis of $J_{(\psi_v)_{V_{[\frac{m}{2}],[\frac{m+1}{2}]^n,[\frac{m}{2}]^{n-1}}}}(\tau')$}

We analyze $J_{(\psi_v)_{V_{[\frac{m}{2}],[\frac{m+1}{2}]^n,[\frac{m}{2}]^{n-1}}}}(\tau')$ by Mackey Theory. See \cite{GJS15}, Sec. 4, for a similar case. We start with the restriction 
$$
Res_{P_{[\frac{m}{2}],[\frac{m+1}{2}]^n,[\frac{m}{2}]^{n-1}}(F_v)}(\tau').
$$ 
By \eqref{12.9}, \eqref{12.9.1}, the decomposition of this restriction is parametrized by the double cosets
$$
P_{(m+\mu_0)^{[\frac{n}{2}]},m^{n-[\frac{n}{2}]}, (m-\mu_0)^{[\frac{n}{2}]}}(F_v)\backslash \GL_{mn}(F_v)/ P_{[\frac{m}{2}],[\frac{m+1}{2}]^n,[\frac{m}{2}]^{n-1}}(F_v).
$$
As in \cite{GJS15}, (4.2), the representatives are determined by matrices of non-negative integers $\underline{k}=(k_{i,j})_{\scriptsize{\begin{matrix}i\leq n\\ j\leq 2n\end{matrix}}}$, satisfying the following relations
\begin{equation}\label{13.1}
\sum_{i=1}^n k_{i,j}=[\frac{m}{2}],\ \ j=1,n+2,n+3,...,2n;
\end{equation}
$$
\sum_{i=1}^n k_{i,j}=[\frac{m+1}{2}],\ \ j=2,3,...,n+1;
$$
$$
\sum_{j=1}^{2n}k_{i,j}=m+\mu_0,\ \ i=1,2,...,[\frac{n}{2}];
$$
$$
\sum_{j=1}^{2n}k_{[\frac{n+2}{2}],j}=\beta;
$$
$$
\sum_{j=1}^{2n}k_{i,j}=m-\mu_0,\ \ i=[\frac{n}{2}]+2,[\frac{n}{2}]+3,...,n,
$$
where $\beta=m-\mu_0$, when $n$ is even, and $\beta=m$, when $n$ is odd. The corresponding representative is a permutation matrix $w_{\underline{k}}=(w_{i,j})_{\scriptsize{\begin{matrix}i\leq n\\j\leq 2n\end{matrix}}}$, and each $w_{i,j}=(w_{i,j}^{\ell,t})_{\scriptsize{\begin{matrix}\ell\leq 2n\\t\leq n\end{matrix}}}$ is a matrix of blocks $w_{i,j}^{\ell,t}$ of size $k_{i,\ell}\times k_{t,j}$, such that, for $(\ell,t)\neq (j,i)$, $w_{i,j}^{\ell,t}=0_{k_{i,\ell}\times k_{t,j}}$, and, for $w_{i,j}^{j,i}=I_{k_{i,j}}$. In particular, $w_{i,j}$ is zero if and only if $k_{i,j}=0$. Thus, up to semi-simplification
\begin{multline}\label{13.2}
J_{(\psi_v)_{V_{[\frac{m}{2}],[\frac{m+1}{2}]^n,[\frac{m}{2}]^{n-1}}}}(\tau')\\
\equiv \oplus_{\underline{k}}
J_{(\psi_v)_{V_{[\frac{m}{2}],[\frac{m+1}{2}]^n,[\frac{m}{2}]^{n-1}}}}(ind^{P_{[\frac{m}{2}],[\frac{m+1}{2}]^n,[\frac{m}{2}]^{n-1}}(F_v)}_{C_{\underline{k}}}(\delta^{\frac{1}{2}}\chi_{\tau'})^{w_{\underline{k}}}),
\end{multline}
where
$$
C_{\underline{k}}=P_{[\frac{m}{2}],[\frac{m+1}{2}]^n,[\frac{m}{2}]^{n-1}}(F_v)\cap w_{\underline{k}}^{-1}P_{(m+\mu_0)^{[\frac{n}{2}]},m^{n-[\frac{n}{2}]}, (m-\mu_0)^{[\frac{n}{2}]}}(F_v)w_{\underline{k}},
$$	  
$\delta$ is short for $\delta_{P_{(m+\mu_0)^{[\frac{n}{2}]},m^{n-[\frac{n}{2}]}, (m-\mu_0)^{[\frac{n}{2}]}}}$, and $(\delta^{\frac{1}{2}}\chi_{\tau'})^{w_{\underline{k}}}$ is the character of $C_{\underline{k}}$ obtained by composing $\delta^{\frac{1}{2}}\chi_{\tau'}$ with conjugation by $w_{\underline{k}}$. See \eqref{12.9}, \eqref{12.9.1}. Recall that $ind$ denotes non-normalized compact induction. 
\begin{prop}\label{prop 13.1}
Assume that $w_{\underline{k}}$ is relevant (for $J_{(\psi_v)_{V_{[\frac{m}{2}],[\frac{m+1}{2}]^n,[\frac{m}{2}]^{n-1}}}}$). Then the matrix $\underline{k}$ satisfies the following relations.\\
Let $2\leq j\leq n$, or $n+2\leq j\leq 2n-1$. Then
\begin{multline}\label{13.3}
k_{1,j}=0,\ \ k_{n,j+1}=0\\
k_{2,j}+k_{3,j}+\cdots+k_{i,j}\leq k_{1,j+1}+k_{2,j+1}+\cdots+k_{i-1,j+1},\ \ \ 2\leq i\leq n\\
k_{n-1,j+1}+k_{n-2,j+1}+\cdots+k_{i,j+1}\leq k_{n,j}+k_{n-1,j}+\cdots+k_{i+1,j},\ \ 1\leq i\leq n-1.
\end{multline}	
\end{prop}
\begin{proof}
The proof is very similar to that of Prop. 4.1 in \cite{GJS15}, and also to that of \cite{GS18}, Prop. 2.3, which yielded \eqref{12.38}.\\
 Let $2\leq j\leq n$. For a matrix $X\in M_{[\frac{m+1}{2}]\times [\frac{m+1}{2}]}(F_v)$, consider
$$
v(X)=diag(I_{[\frac{m}{2}]+(j-2)[\frac{m+1}{2}]},\begin{pmatrix}I_{[\frac{m+1}{2}]}&X\\&I_{[\frac{m+1}{2}]}\end{pmatrix},I_{(n-j)[\frac{m+1}{2}]+(n-1)[\frac{m}{2}]}).
$$
Write $X$ as a matrix of blocks $(X_{k_{i,j},k_{r,j+1}})_{\scriptsize{i,r\leq n}}$, where $X_{k_{i,j},k_{r,j+1}}$ is a $k_{i,j}\times k_{r,j+1}$ matrix. One checks that $w_{\underline{k}}v(X)w_{\underline{k}}^{-1}\in P_{(m+\mu_0)^{[\frac{n}{2}]},m^{n-[\frac{n}{2}]}, (m-\mu_0)^{[\frac{n}{2}]}}(F_v)$ iff $X$ has a block upper triangular shape. Since $w_{\underline{k}}$ is relevant, there are $g_j,g_{j+1}\in \GL_{[\frac{m+1}{2}]}(F_v)$, such that for all such $X$, $\psi_v(tr(g^{-1}_jXg_{j+1}))=1$. This forces the matrix $g_{j+1}g_j^{-1}$ to have the following form. Write it as a matrix of blocks\\ $(A_{k_{i,j+1},k_{r,j}})_{\scriptsize{i,r\leq n}}$, where $A_{k_{i,j+1},k_{r,j}}$ is a $k_{i,j+1}\times k_{r,j}$ matrix. Then the last matrix of blocks must have an upper triangular shape, with zero blocks along the diagonal. Since $g_{j+1}g_j^{-1}$ is invertible, we conclude that $k_{1,j}=0$, $k_{n,j+1}=0$, as well as the inequalities stated in \eqref{13.3}. In order to get the similar relations, for $n+2\leq j\leq 2n-1$, we repeat the similar argument, for the matrices 
$$
w(X)=diag(I_{(j-n-1)[\frac{m}{2}]+n[\frac{m+1}{2}]},\begin{pmatrix}I_{[\frac{m}{2}]}&X\\&I_{[\frac{m}{2}]}\end{pmatrix},I_{(2n-1-j)[\frac{m}{2}]}).
$$
\end{proof}	
\begin{prop}\label{prop 13.2}
The relations \eqref{13.3} and and the fact that $m, m-1=\mu_0n$, $m, m-1=\mu_0(n-1)$ (according to Cases 1-4, \eqref{12.7}-\eqref{12.16}) determine $\underline{k}$ uniquely, as the following matrix $\underline{k}^0=(k^0_{i,j})_{\scriptsize{\begin{matrix}i\leq n\\j\leq 2n\end{matrix}}}$.
\begin{enumerate}
	\item For $1\leq i< [\frac{n+1}{2}]$, $k^0_{i,1}=\mu_0$, and for $[\frac{n+1}{2}]+1\leq i\leq n$, $k^0_{i,1}=0$.\\
	\item For $i=[\frac{n+1}{2}]$, $k_{[\frac{n+1}{2}],1}=\nu_n\mu_0$, where $\mu_n=1$, when $n$ is even, and $\nu_n=0$, when $n$ is odd.
	\item Let $2\leq j\leq n+1$, and $1\leq i\leq n$. Then $k^0_{i,j}=\delta_{i,n+2-j}[\frac{m+1}{2}]$.\\
	\item Let $n+2\leq j\leq n+[\frac{n}{2}]$, and $1\leq i\leq n$. Then $k^0_{i,j}=0$, except when $i=2n+1-j$, or $i=2n+2-j$. In these two cases, we must have $[\frac{n+1}{2}]+1\leq i\leq n$, and then
	$$
	k^0_{i,2n+1-i}=(n-i)\mu_0,\ k^0_{i,2n+2-i}=[\frac{m}{2}]-(n-i+1)\mu_0.
	$$
	In particular, $k^0_{i,j}=0$, for all $1\leq i\leq [\frac{n+1}{2}]$ and $n+2\leq j\leq n+[\frac{n}{2}]$.\\
	\item Let $n+[\frac{n}{2}]+1\leq j\leq 2n$, and $1\leq i\leq n$. Then $k^0_{i,j}=\delta_{i,2n+1-j}[\frac{m}{2}]$. Note that if $i=2n+1-j$, then
	$1\leq i\leq [\frac{n+1}{2}]$. In particular, $k^0_{i,j}=0$, for all $[\frac{n+1}{2}]+1\leq i\leq n$ and $n+[\frac{n}{2}]+1\leq j\leq 2n$.
	\end{enumerate}	
Write the matrix $w_{\underline{k}^0}$ as a matrix of blocks $(w_{i,j})_{\scriptsize{\begin{matrix}i\leq n\\j\leq 2n\end{matrix}}}$, as explained right after \eqref{13.1}. Then $w_{i,j}=0$ iff $k^0_{i,j}=0$.
\end{prop}	
\begin{proof}
Let $1\leq i\leq n-1$. Let us show that, for $2\leq j\leq n-i+1$, or $n+2\leq j\leq 2n-i$, $k_{i,j}=0$. This can be directly verified when $n=2$. Assume that $n\geq 3$. We work by induction on $i$, where the case $i=1$ follows from Prop. \ref{prop 13.1}. Let $1\leq i<n-2$, and assume that our assertion holds for such $i$. Let $1\leq j\leq n-i$, or $n+2\leq j\leq 2n-i-1$. By the first inequality of Prop. \ref{prop 13.1},
$$
k_{2,j}+k_{3,j}+\cdots+k_{i+1,j}\leq k_{1,j+1}+k_{2,j+1}+\cdots+k_{i,j+1}.
$$
By induction, $k_{r,j}=0$, for $2\leq r\leq i$, so that the l.h.s. of the last inequality is $k_{i+1,j}$. Also, by induction, the r.h.s. is zero. This implies that $k_{i+1,j}=0$. Next, we show, in a similar fashion, that, for $0\leq i\leq n-2$, $k_{n-i,j}=0$, for $i+3\leq j\leq n+1$, or $n+i+3\leq j\leq 2n$. For this, we use the second inequality of the last proposition. The first case of the induction on $i$ is $i=0$, and here $k_{n,j}=0$ follows from the last proposition. We conclude that, for $1\leq j\leq 2n$,
\begin{multline}\label{13.4}
k_{1,j}=0,\ \textit{for}\ j\neq 1,n+1,2n;\ \ 
k_{n,j}=0,\ \textit{for}\ j\neq 1,2,n+2\\
\textit{Let}\ 2\leq i\leq n-1.\ \textit{Then}\ k_{i,j}=0, \ \textit{for}\ j\neq 1, n-i+2, 2n-i+1, 2n-i+2.
\end{multline}
It follows from \eqref{13.4}, \eqref{13.1}, that for $2\leq j\leq n+1$, 
\begin{equation}\label{13.5}
k_{n+2-j,j}=[\frac{m+1}{2}],
\end{equation}
and, for $n+2\leq j\leq 2n$,
\begin{equation}\label{13.6}
k_{2n+1-j,j}+k_{2n+2-j,j}=[\frac{m}{2}].
\end{equation}
From \eqref{13.4} and \eqref{13.3}, it follows that
\begin{equation}\label{13.7}
0\leq k_{n-1,n+2}\leq k_{n-2,n+3}\leq\cdots\leq k_{3,2n-2}\leq k_{2,2n-1}\leq k_{1,2n}\leq [\frac{m}{2}].
\end{equation}
Let $[\frac{n+1}{2}]\leq i\leq n-1$. We show that
\begin{equation}\label{13.8}
k_{i,2n+1-i}-k_{i+1,2n-i}\geq \mu_0.
\end{equation}
Indeed, for such $i$, it follows from \eqref{13.1}, \eqref{13.4}, \eqref{13.5}, that
\begin{equation}\label{13.9}
k_{i+1,1}+[\frac{m+1}{2}]+k_{i+1,2n-i}+k_{i+1,2n+1-i}=m-\mu_0.
\end{equation}
By \eqref{13.6}, it follows that
$$
k_{i+1,2n-i}+k_{i+1,2n+1-i}\leq k_{i,2n+1-i}+k_{i+1,2n+1-i}-\mu_0.
$$
This proves \eqref{13.8}. By \eqref{13.7}, we conclude that $[\frac{m}{2}]\geq \mu_0[\frac{n}{2}]$. By our assumptions, according to Cases 1-4, \eqref{12.7}-\eqref{12.16}, $[\frac{m}{2}]= \mu_0[\frac{n}{2}]$. This forces in \eqref{13.7}, that
\begin{multline}\label{13.10}
k_{n-1,n+2}=\mu_0, k_{n-2,n+3}=2\mu_0, k_{n-3,n+4}=3\mu_0,..., k_{[\frac{n+1}{2}],n+[\frac{n}{2}]+1}=[\frac{n}{2}]\mu_0=[\frac{m}{2}],\\
k_{[\frac{n+1}{2}]-1,n+[\frac{n}{2}]+2}=k_{[\frac{n+1}{2}]-2,n+[\frac{n}{2}]+3}=...=k_{3,2n-2}=k_{2,2n-1}=k_{1,2n}=[\frac{m}{2}],
\end{multline}
and hence
$$ k_{n,n+2}=[\frac{m}{2}]-\mu_0, k_{n-1,n+3}=[\frac{m}{2}]-2\mu_0,...,k_{[\frac{n+1}{2}]+2,n+[\frac{n}{2}]}=[\frac{m}{2}]-([\frac{n}{2}]-1)\mu_0=\mu_0.
$$
By \eqref{13.1}, \eqref{13.5}, \eqref{13.10}, we have, for $1\leq i< [\frac{n+1}{2}]$,
$$
k_{i,1}+[\frac{m+1}{2}]+[\frac{m}{2}]=m+\mu_0,
$$
and for $i=[\frac{n+1}{2}]$,
$$
k_{[\frac{n+1}{2}],1}+[\frac{m+1}{2}]+[\frac{m}{2}]=m+\nu_n \mu_0.
$$
We conclude that $k_{i,1}=\mu_0$, for $1\leq i< [\frac{n+1}{2}]$, and for $i=[\frac{n+1}{2}]$, $k_{[\frac{n+1}{2}],1}=\nu_n\mu_0$. This and \eqref{13.1} imply that $k_{i,1}=0$, for $[\frac{n+1}{2}]+1\leq i\leq n$. This completes the proof of the proposition.
	
\end{proof}	

Here are examples of $\underline{k}^0$ and $w_{\underline{k}^0}$ for $n=5$. Note that in this case $[\frac{m}{2}]=2\mu_0$,
$$
\underline{k}^0=\begin{pmatrix}\mu_0&0&0&0&0&[\frac{m+1}{2}]&0&0&0&[\frac{m}{2}]\\\mu_0&0&0&0&[\frac{m+1}{2}]&0&0&0&[\frac{m}{2}]&0\\0&0&0&[\frac{m+1}{2}]&0&0&0&[\frac{m}{2}]&0&0\\0&0&[\frac{m+1}{2}]&0&0&0&\mu_0&0&0&0\\0&[\frac{m+1}{2}]&0&0&0&0&\mu_0&0&0&0\end{pmatrix},
$$
$$
w_{\underline{k}^0}=\begin{pmatrix}w_{1,1}&0&0&0&0&w_{1,6}&0&0&0&w_{1,10}\\w_{2,1}&0&0&0&w_{2,5}&0&0&0&w_{2,9}&0\\0&0&0&w_{3,4}&0&0&0&w_{3,8}&0&0\\0&0&w_{4,3}&0&0&0&w_{4,7}&0&0&0\\0&w_{5,2}&0&0&0&0&w_{5,7}&0&0&0\end{pmatrix}.
$$
Recall that $\underline{k}^0$ determins each of the blocks $w_{i,j}$ of $w_{\underline{k}^0}$. Here are some examples: let $0\leq i< [\frac{n+1}{2}]$. Then
\begin{multline}\label{13.11}
\ \ \ \ \ w_{i,1}=\begin{pmatrix}0_{\mu_0\times (i-1)\mu_0}&I_{\mu_0}&0_{\mu_0\times  ([\frac{m}{2}]-i\mu_0)}\\0&0_{m\times \mu_0}&0\end{pmatrix},\\
\\
w_{i,n+2-i}=\begin{pmatrix}0_{\mu_0\times [\frac{m+1}{2}]}\\I_{[\frac{m+1}{2}]}\\0_{[\frac{m}{2}]\times [\frac{m+1}{2}]}\end{pmatrix},\ \ \ \ \ \ \ w_{i,2n+1-i}=\begin{pmatrix}0_{\mu_0\times [\frac{m}{2}]}\\0_{[\frac{m+1}{2}]\times [\frac{m}{2}]}\\I_{[\frac{m}{2}]}\end{pmatrix}\\
\\
w_{j,n+2-j}=\begin{pmatrix}I_{[\frac{m+1}{2}]}\\0_{([\frac{m}{2}]-\mu_0)\times [\frac{m+1}{2}]}\end{pmatrix},\ \ \ \ \ [\frac{n+1}{2}]+1\leq j\leq n .
\end{multline}
The last proposition shows that there is exactly one summand in \eqref{13.2}, that is
\begin{multline}\label{13.12}
J_{(\psi_v)_{V_{[\frac{m}{2}],[\frac{m+1}{2}]^n,[\frac{m}{2}]^{n-1}}}}(\tau')\\
\cong 
J_{(\psi_v)_{V_{[\frac{m}{2}],[\frac{m+1}{2}]^n,[\frac{m}{2}]^{n-1}}}}(ind^{P_{[\frac{m}{2}],[\frac{m+1}{2}]^n,[\frac{m}{2}]^{n-1}}(F_v)}_{C_{\underline{k}^0}}(\delta^{\frac{1}{2}}\chi_{\tau'})^{w_{\underline{k}^0}}),
\end{multline}
where
$$
C_{\underline{k}^0}=P_{[\frac{m}{2}],[\frac{m+1}{2}]^n,[\frac{m}{2}]^{n-1}}(F_v)\cap w_{\underline{k}^0}^{-1}P_{(m+\mu_0)^{[\frac{n}{2}]},m^{n-[\frac{n}{2}]}, (m-\mu_0)^{[\frac{n}{2}]}}(F_v)w_{\underline{k}^0}.
$$	  	
Let
$$
D_{\underline{k}^0}=L_{[\frac{m}{2}],[\frac{m+1}{2}]^n,[\frac{m}{2}]^{n-1}}(F_v)\cap w_{\underline{k}^0}^{-1}P_{(m+\mu_0)^{[\frac{n}{2}]},m^{n-[\frac{n}{2}]}, (m-\mu_0)^{[\frac{n}{2}]}}(F_v)w_{\underline{k}^0}.
$$	  		
Recall that $L_{[\frac{m}{2}],[\frac{m+1}{2}]^n,[\frac{m}{2}]^{n-1}}$ denotes the Levi part of $P_{[\frac{m}{2}],[\frac{m+1}{2}]^n,[\frac{m}{2}]^{n-1}}$.
\begin{lem}\label{lem 13.3}
The elements of $D_{\underline{k}^0}$ have the form $diag(g_1,...,g_{2n})$, where\\ 
$ g_1\in P_{\mu_0^{[\frac{n}{2}]}}(F_v)$, $g_2,...,g_{n+1}\in \GL_{[\frac{m+1}{2}]}(F_v)$, $g_{n+[\frac{n}{2}]+1},...,g_{2n}\in \GL_{[\frac{m}{2}]}(F_v)$,\\
and, for $[\frac{n+1}{2}]+1\leq j\leq n-1$, $g_{2n+1-j}\in P_{(n-j)\mu_0, [\frac{m}{2}]-(n-j)\mu_0}(F_v)$. Thus\\
\\
$D_{\underline{k}^0}\cong$
 $$
   P_{\mu_0^{[\frac{n}{2}]}}(F_v)\times \GL_{[\frac{m+1}{2}]}(F_v)^{n}\times (\prod_{j=[\frac{n+1}{2}]+1}^{n-1}P_{(n-j)\mu_0, [\frac{m}{2}]-(n-j)\mu_0}(F_v))\times \GL_{[\frac{m}{2}]}(F_v)^{[\frac{n+1}{2}]}.
 $$ 
\end{lem}

\begin{proof}
For $1\leq i\leq n$, denote the $i$-block row of $w_{\underline{k}^0}$ by $w^i_{\underline{k}^0}$  Let $diag(g_1,...,g_{2n})\in L_{[\frac{m}{2}],[\frac{m+1}{2}]^n,[\frac{m}{2}]^{n-1}}(F_v)$. This element lies in $D_{\underline{k}^0}$ if and only if, for all $1\leq j<i\leq n$
\begin{equation}\label{13.13}
w^i_{\underline{k}^0}\cdot diag(g_1,...,g_{2n})\cdot {}^tw^j_{\underline{k}^0}=0.
\end{equation}
We explicate the l.h.s. of \eqref{13.13} according to the structure of $w_{\underline{k}^0}$, determined by Prop. \ref{prop 13.2} (and the description of the blocks $w_{i,j}$, right after \eqref{13.1}). Assume that $1\leq j<i\leq [\frac{n+1}{2}]$. Then one can verify that \eqref{13.13} becomes
\begin{equation}\label{13.14}
w_{i,1}g_1{}^tw_{j,1}=0.
\end{equation}
Note, that when $n$ is odd, $w_{[\frac{n+1}{2}],1}=0$. Thus, for $n$ odd, we may take $i\leq [\frac{n+1}{2}]-1=[\frac{n}{2}]$. Since, for $n$ even, $[\frac{n+1}{2}]=[\frac{n}{2}]$, let us write $g_1$ as a an $[\frac{n}{2}]\times [\frac{n}{2}]$ matrix $((g_1)_{i,j})$ of $\mu_0\times \mu_0$ blocks 
$(g_1)_{i,j}$. By \eqref{13.11}, we see that \eqref{13.14} means that $(g_1)_{i,j}=0$, for all $1\leq j<i\leq [\frac{n}{2}]$, i.e. $g_1\in P_{\mu_0^{[\frac{n}{2}]}}(F_v)$. Next, assume that $[\frac{n+1}{2}]+1\leq i\leq n$, and $1\leq j\leq [\frac{n+1}{2}]$. Then \eqref{13.13} holds, as can be directly checked. Finally, assume that $[\frac{n+1}{2}]+1\leq j< i\leq n$. Then, by looking at the structure of $w_{\underline{k}^0}$, \eqref{13.13} holds automatically, for $[\frac{n+1}{2}]+2\leq j+1< i\leq n$. Thus, let $[\frac{n+1}{2}]+1\leq j\leq n-1$, and $i=j+1$. Then \eqref{13.13} becomes
\begin{equation}\label{13.15}
w_{j+1,2n+1-j}g_{2n+1-j}{}^tw_{j,2n+1-j}=0.
\end{equation}
Now, one checks that \eqref{13.15} means that $g_{2n+-j}\in P_{(n-j)\mu_0,[\frac{m}{2}]-(n-j)\mu_0}(F_v)$, for $[\frac{n+1}{2}]+1\leq j\leq n-1$ . This proves the lemma.

\end{proof}

Let $T$ be the stabilizer inside  $L_{[\frac{m}{2}],[\frac{m+1}{2}]^n,[\frac{m}{2}]^{n-1}}(F_v)$ of the character\\
 $(\psi_v)_{V_{[\frac{m}{2}],[\frac{m+1}{2}]^n,[\frac{m}{2}]^{n-1}}}$. Then
$$
T=\{\begin{pmatrix}h_1\\&h_2^{\Delta_n}\\&&h_3^{\Delta_{n-1}}\end{pmatrix}\ |h_1, h_3\in \GL_{[\frac{m}{2}]}(F_v), h_2\in \GL_{[\frac{m+1}{2}]}(F_v) \}.
$$
The following proposition is similar to Prop. 4.4 in \cite{GJS15}.
\begin{prop}\label{prop 13.4}
Let $\mathcal{A}$ be the subset of elements of $L_{[\frac{m}{2}],[\frac{m+1}{2}]^n,[\frac{m}{2}]^{n-1}}(F_v)$,	
which support the Jacquet module \eqref{13.12},\\ $J_{(\psi_v)_{V_{[\frac{m}{2}],[\frac{m+1}{2}]^n,[\frac{m}{2}]^{n-1}}}}(ind^{P_{[\frac{m}{2}],[\frac{m+1}{2}]^n,[\frac{m}{2}]^{n-1}}(F_v)}_{C_{\underline{k}^0}}(\delta^{\frac{1}{2}}\chi_{\tau'})^{w_{\underline{k}^0}})$. Then
$$
\mathcal{A}=D_{\underline{k}^0}T.
$$
We conclude that, as $T$-modules,
\begin{multline}\label{13.16}
J_{(\psi_v)_{V_{[\frac{m}{2}],[\frac{m+1}{2}]^n,[\frac{m}{2}]^{n-1}}}}(ind^{P_{[\frac{m}{2}],[\frac{m+1}{2}]^n,[\frac{m}{2}]^{n-1}}(F_v)}_{C_{\underline{k}^0}}(\delta^{\frac{1}{2}}\chi_{\tau'})^{w_{\underline{k}^0}})\cong\\
 ind_{T\cap D_{\underline{k}^0}}^T\delta_0(\delta^{\frac{1}{2}}\chi_{\tau'})^{w_{\underline{k}^0}},
\end{multline}
where $\delta_0$ is the character of $T\cap D_{\underline{k}^0}$ given by the Jacobian of the conjugation map $a\mapsto (v\mapsto ava^{-1})$ on the $F_v$ points of
\begin{multline}\nonumber
\bar{V}_{[\frac{m}{2}],[\frac{m+1}{2}]^n,[\frac{m}{2}]^{n-1}}^0=\\
(V_{[\frac{m}{2}],[\frac{m+1}{2}]^n,[\frac{m}{2}]^{n-1}}\cap w^{-1}_{\underline{k}^0}P_{(m+\mu_0)^{[\frac{n}{2}]},m^{n-[\frac{n}{2}]}, (m-\mu_0)^{[\frac{n}{2}]}}w_{\underline{k}^0})\backslash V_{[\frac{m}{2}],[\frac{m+1}{2}]^n,[\frac{m}{2}]^{n-1}}.
\end{multline} 
\end{prop}
\begin{proof}
Let $h=diag(g_1,...g_{2n})$ be an element in $\mathcal{A}$. We repeat the argument of the proof of Prop. \ref{prop 13.1} with $k_{i,j}=k^0_{i,j}$. The first part of the argument with the matrices $v(X)$ gives no condition on the matrices $g_{j+1}g_j^{-1}$, for $2\leq j \leq n$. The second part of the argument with the matrices $w(X)$ gives no condition on the matrices $g_{j+1}g_j^{-1}$, for $n+[\frac{n}{2}]\leq j\leq 2n-1$, but, for $n+2\leq j\leq n+[\frac{n}{2}]-1$, we get that $g_{j+1}g_j^{-1}$ has the form
$$	
g_{j+1}g_j^{-1}=\begin{pmatrix}A_1&A_2\\0_{\scriptsize{([\frac{m}{2}]-(j-n)\mu_0)\times (j-n-1)\mu_0}}&A_4\end{pmatrix}\in \GL_{[\frac{m}{2}]}(F_v).
$$	
It follows, by Lemma \ref{lem 13.3}, that we have the equality of left cosets 
$$
D_{\underline{k}^0}h=D_{\underline{k}^0}diag(g_1, I_{n[\frac{m+1}{2}]},g_{n-2}^{\Delta_{[\frac{n}{2}]-1}},I_{[\frac{m}{2}][\frac{n+1}{2}]})=D_{\underline{k}^0}diag(g_1, I_{n[\frac{m+1}{2}]},g_{n-2}^{\Delta_{n-1}}).
$$	
The representative of the last coset is in $T$. The isomorphism \eqref{13.16} is induced by the map which sends a function $f$ in the space of $ind^{P_{[\frac{m}{2}],[\frac{m+1}{2}]^n,[\frac{m}{2}]^{n-1}}(F_v)}_{C_{\underline{k}^0}}(\delta^{\frac{1}{2}}\chi_{\tau'})^{w_{\underline{k}^0}}$ to the function on $T$ given by
$$t\mapsto \int_{\bar{V}_{[\frac{m}{2}],[\frac{m+1}{2}]^n,[\frac{m}{2}]^{n-1}}^0(F_v)}f(ut)(\psi^{-1}_v)_{V_{[\frac{m}{2}],[\frac{m+1}{2}]^n,[\frac{m}{2}]^{n-1}}}(u)du.
$$
This proves the proposition.		
	
\end{proof}

Note that, by Lemma \ref{lem 13.3},
\begin{equation}\label{13.17}
T\cap D_{\underline{k}^0}=\{ \begin{pmatrix}h_1\\&h_2^{\Delta_n}\\&&h_3^{\Delta_{n-1}}\end{pmatrix}\ |h_1,h_3\in P_{\mu_0^{[\frac{n}{2}]}}(F_v), h_2\in \GL_{[\frac{m+1}{2}]}(F_v)\}.
\end{equation}

Now, we need to compute the character $\delta_0(\delta^{\frac{1}{2}}\chi_{\tau'})^{w_{\underline{k}^0}}$ in \eqref{13.16}. 

\begin{lem}\label{lem 13.5}
Let $h_2\in \GL_{[\frac{m+1}{2}]}(F_v)$, and
$$
h_1=\begin{pmatrix}a_1\\&\ddots\\&&a_{[\frac{n}{2}]}\end{pmatrix},\ 
h_3=\begin{pmatrix}c_1\\&\ddots\\&&c_{[\frac{n}{2}]}\end{pmatrix},\
a_i,c_i\in \GL_{\mu_0}(F_v).
$$
Then
\begin{multline}\label{13.18}
w_{\underline{k}^0}\begin{pmatrix}h_1\\&h_2^{\Delta_n}\\&&h_3^{\Delta_{n-1}}\end{pmatrix}w^{-1}_{\underline{k}^0}=\\
\scriptsize{\begin{pmatrix}\begin{pmatrix}a_1\\&h_2\\&&h_3\end{pmatrix}\\&\ddots\\&&\begin{pmatrix}a_{[\frac{n}{2}]}\\&h_2\\&&h_3\end{pmatrix}\\&&&\nu_n\begin{pmatrix}h_2\\&h_3\end{pmatrix}\\&&&&\begin{pmatrix}h_2\\&h^{(1)}_3\end{pmatrix}\\&&&&&\ddots\\&&&&&&\begin{pmatrix}h_2\\&h^{([\frac{n}{2}])}_3\end{pmatrix}\end{pmatrix}},
\end{multline}
where $h_3^{(i)}$ is obtained from $h_3$ by deleting $c_{[\frac{n}{2}]-i+1}$ from the block diagonal. When $\nu_n=0$, the middle block should be omitted. We conclude that
\begin{multline}\label{13.19}
(\delta^{\frac{1}{2}}\chi_{\tau'})^{w_{\underline{k}^0}}(\begin{pmatrix}h_1\\&h_2^{\Delta_n}\\&&h_3^{\Delta_{n-1}}\end{pmatrix})=\\
\prod_{i=1}^{[\frac{n}{2}]}\chi_i^2(\det(h_2)\det(h_3))\prod_{i=1}^{[\frac{n}{2}]}\chi_i(\det(a_i))\chi_{[\frac{n}{2}]-i+1}^{-1}(\det(c_i))|\det(h_1)\det(h_3)|^{\frac{mn}{2}}\cdot\\
|\det(h_2)\det(h_3)|^{\frac{\mu_0}{2}(n-2[\frac{n}{2}])-[\frac{m}{2}][\frac{n+1}{2}]}\prod_{i=1}^{[\frac{n}{2}]}(|\det(a_i)|^{m+\mu_0}|\det(c_i)|^{m-\mu_0})^{-\frac{2i-1}{2}}.
\end{multline}
\end{lem}
\begin{proof}
The proofs are by direct computations. For example, after having verified \eqref{13.18}, we see that 
\begin{multline}\label{13.20}
\chi_{\tau'}^{w_{\underline{k}^0}}(\begin{pmatrix}h_1\\&h_2^{\Delta_n}\\&&h_3^{\Delta_{n-1}}\end{pmatrix})=\\
\prod_{i=1}^{[\frac{n}{2}]}\chi_i(\det(a_i)\det(h_2)\det(h_3))|\det(h_2)\det(h_3)|^{\frac{\mu_0}{2}(n-2[\frac{n}{2}])}\prod_{i=1}^{[\frac{n}{2}]}\chi_i(\det(h_2)\det(h^{(i)}_3)).
\end{multline}
Since $\det(h^{(i)}_3=\det(h_3)(\det(c_{[\frac{n}{2}]-i+1}))^{-1}$, we get that
the r.h.s. of \eqref{13.20}, after a simple change of index, is equal to
\begin{equation}\label{13.21}	
\prod_{i=1}^{[\frac{n}{2}]}\chi_i^2(\det(h_2)\det(h_3))|\det(h_2)\det(h_3)|^{\frac{\mu_0}{2}(n-2[\frac{n}{2}])}\prod_{i=1}^{[\frac{n}{2}]}\chi_i(\det(a_i))\chi_{[\frac{n}{2}]-i+1}^{-1}(\det(c_i)).
\end{equation}
Recall that the character $\chi_{\tau'}$ is defined right after \eqref{12.9}, \eqref{12.9.1}. Recall also that $\delta=\delta_{P_{(m+\mu_0)^{[\frac{n}{2}]},m^{n-2[\frac{n}{2}]},(m-\mu_0)^{[\frac{n}{2}]}}}$. We have
\begin{multline}\nonumber
(\delta^{\frac{1}{2}})^{w_{\underline{k}^0}}(\begin{pmatrix}h_1\\&h_2^{\Delta_n}\\&&h_3^{\Delta_{n-1}}\end{pmatrix})=
\prod_{i=1}^{[\frac{n}{2}]}|\det(a_i)\det(h_2)\det(h_3)|^{\frac{mn-(2i-1)(m+\mu_0)}{2}}\cdot \\
|\det(h_2)\det(h_3)|^{-(n-2[\frac{n}{2}])\mu_0[\frac{n}{2}]}\prod_{i=1}^{[\frac{n}{2}]}|\det(h_2)\frac{\det(h_3)}{\det(c_{[\frac{n}{2}]-i+1})}|^{([\frac{n}{2}]-\frac{2i-1}{2})(m-\mu_0)-\frac{mn}{2}}.
\end{multline}
Easy simplifications, together with \eqref{13.21}, give \eqref{13.19}.	
	
\end{proof}

The main work now is to compute $\delta_0$ (in \eqref{13.16}), and for this, we need to describe the subgroup $\bar{V}_{[\frac{m}{2}],[\frac{m+1}{2}]^n,[\frac{m}{2}]^{n-1}}^0$.

\begin{prop}\label{prop 13.6}
Assume that $n=2n'$ is even. Then the subgroup\\ $V_{[\frac{m}{2}],[\frac{m+1}{2}]^n,[\frac{m}{2}]^{n-1}}(F_v)\cap w^{-1}_{\underline{k}^0}P_{(m+\mu_0)^{[\frac{n}{2}]},m^{n-[\frac{n}{2}]}, (m-\mu_0)^{[\frac{n}{2}]}}(F_v)w_{\underline{k}^0}$ consists of the matrices 
\begin{equation}\label{13.22}
v=\begin{pmatrix}I_{[\frac{m}{2}]}&X^1&X^2&X^3&X^4\\&I_{n'[\frac{m+1}{2}]}&0&X^5&0\\&&I_{n'[\frac{m+1}{2}]}&X^6&X^7\\&&&X^8&0\\&&&&I_{n'[\frac{m}{2}]}\end{pmatrix},
\end{equation}
where $X^1, X^3, X^6$ are arbitrary, and $X^2, X^4, X^5, X^7, X^8$ are as follows.\\
\begin{enumerate}
	\item Write $X^2$ as a $n'\times n'$ matrix $(X^2_{i,j})$ of blocks $X^2_{i,j}\in M_{\mu_0\times [\frac{m+1}{2}]}(F_v)$. Then, for $2\leq j\leq n'$, $X^2_{i,j}=0$, for $n'-j+2\leq i\leq n'$. Thus, in column $j$, the last $j-1$ blocks are zero.\\
	\item The matrix $X^4$ has a form similar to $X^2$. Write $X^4$ as a $n'\times n'$ matrix $(X^4_{i,j})$ of blocks $X^4_{i,j}\in M_{\mu_0\times [\frac{m}{2}]}(F_v)$. Then, for $2\leq j\leq n'$, $X^4_{i,j}=0$, for $n'-j+2\leq i\leq n'$.\\
	\item Write $X^5$ as a $n'\times (n'-1)$ matrix $(X^5_{i,j})$ of blocks $X^5_{i,j}\in M_{[\frac{m+1}{2}]\times [\frac{m}{2}]}(F_v)$. Then, for $1\leq i\leq n'-2$, and $i<j\leq n'-1$, $X^5_{i,j}=0$. Also, for $1\leq i\leq n'-1$, write $X^5_{i,i}=(x_i(1),...,x_i(n'))$, where $x_i(j')\in M_{\mu_0\times \mu_0}(F_v)$. Then, for $1\leq j'\leq i$, $x_i(j')=0$.\\
	\item Write $X^7$ as a $n'\times n'$ matrix $(X^7_{i,j})$ of blocks $X^7_{i,j}\in M_{[\frac{m+1}{2}]\times [\frac{m}{2}]}(F_v)$. Then, for $1\leq i<j\leq n'$, $X^7_{i,j}=0$.\\
	\item Write $X^8$ as a $(n'-1)\times (n'-1)$ matrix $(X^8_{i,j})$ of blocks $X^8_{i,j}\in M_{[\frac{m}{2}]\times [\frac{m}{2}]}(F_v)$. Then $X^8$ is upper unipotent, with $X^8_{i,i}=I_{[\frac{m}{2}]}$, $X^8_{i,i+i'}=0$, for $1\leq i\leq n'-2$, $2\leq i'\leq n-i$, and, for $1\leq i\leq n'-1$, let us write $X^8_{i,i+1}$ as a $n'\times n'$ matrix  $X^8_{i,i+1}=(x_i(i',j'))$ of blocks $x_i(i',j')\in M_{\mu_0\times \mu_0}(F_v)$. Then $x_i(i',j')=0$, unless $1\leq i'\leq i$ and $i+2\leq j'\leq n'$. 
\end{enumerate}

\end{prop} 

\begin{proof}
The proof is by a straightforward (patient) calculation.\\ 
Let $v\in V_{[\frac{m}{2}],[\frac{m+1}{2}]^n,[\frac{m}{2}]^{n-1}}(F_v)$. Write $v$ as a $2n\times 2n$ matrix $(v_{i,j})$ of blocks $v_{i,j}$, according to the partition $([\frac{m}{2}],[\frac{m+1}{2}]^n,[\frac{m}{2}]^{n-1})$ (of $mn$). Think of the matrix $w_{\underline{k}^0}vw^{-1}_{\underline{k}^0}$ as a $n\times n$ matrix of blocks according to the partition\\ $((m+\mu_0)^{n'},(m-\mu_0)^{n'})$. Then $w_{\underline{k}^0}vw^{-1}_{\underline{k}^0}$ lies in $P_{(m+\mu_0)^{n'},(m-\mu_0)^{n'}}(F_v)$ iff for all $1\leq j<i\leq n$,
\begin{equation}\label{13.23}
w^i_{\underline{k}^0}v{}^tw^j_{\underline{k}^0}=0.
\end{equation}
Recall that $w^i_{\underline{k}^0}$ denotes the $i-th$ block row of $w_{\underline{k}^0}$. We consider \eqref{13.23} in three cases. The first is for $1\leq j<i\leq n'$. The second is for $1\leq j\leq n'$ and $n'+1\leq i\leq n$, and the third is for $n'+1\leq j<i\leq n$. We will show the first case, and we leave the rest for the reader. Assume that $1\leq j<i\leq n'$. By the description of $w_{\underline{k}^0}$ in Prop. \ref{prop 13.2}, \eqref{13.23} is the same as
\begin{multline}\label{13.24}
w_{i,1}v_{1,n+2-j}{}^tw_{j,n+2-j}+w_{i,1}v_{1,2n+1-j}{}^tw_{j,2n+1-j}+
w_{i,n+2-i}v_{n+2-i,n+2-j}{}^tw_{j,n+2-j}\\
+w_{i,n+2-i}v_{n+2-i,2n+1-j}{}^tw_{j,2n+1-j}+w_{i,2n+1-i}v_{2n+1-i,2n+1-j}{}^tw_{j,2n+1-j}=0
\end{multline}	
Note that since $j<i$, $w_{i,1}{}^tw_{j,1}=0$. By examining the description of $w_{\underline{k}^0}$ in Prop. \ref{prop 13.2}, it follows that \eqref{13.24} is equivalent to each summand of the l.h.s. being zero, that is
\begin{eqnarray}
 w_{i,1}v_{1,n+2-j}{}^tw_{j,n+2-j}=0\label{13.25}\\
 w_{i,1}v_{1,2n+1-j}{}^tw_{j,2n+1-j}=0 \label{13.26}\\
 w_{i,n+2-i}v_{n+2-i,n+2-j}{}^tw_{j,n+2-j}=0 \label{13.27}\\
 w_{i,n+2-i}v_{n+2-i,2n+1-j}{}^tw_{j,2n+1-j}=0 \label{13.28} \\
 w_{i,2n+1-i}v_{2n+1-i,2n+1-j}{}^tw_{j,2n+1-j}=0 \label{13.29}.
\end{eqnarray}	
Write $v_{1,n+2-j}$ as a column of block matrices $v_{1,n+2-j}(i')\in M_{\mu_0\times [\frac{m+1}{2}]}(F_v)$, $1\leq i'\leq n'$. Then by \eqref{13.11}, \eqref{13.25} is the same as $v_{1,n+2-j}(i)=0$. We conclude that $v_{1,n+2-j}(i)=0$, for all $j<i\leq n'$, that is the last $n'-j$ blocks of $v_{1,n+2-j}$ are zero, and this holds for $1\leq j\leq n'-1$.	This is the description of $X^2$ in the first part of the proposition. The same argument shows that when we write $v_{1,2n+1-j}$ as a column of block matrices $v_{1,2n+1-j}(i')\in M_{\mu_0\times [\frac{m}{2}]}(F_v)$, then $v_{1,2n+1-j}(i)=0$, for all $j<i\leq n'$, and we get the description of $X^4$ in the second part of the proposition. Next, by \eqref{13.11}, \eqref{13.27} is equivalent to $v_{n+2-i,n+2-j}=0$. Similarly, \eqref{13.28} is equivalent to
$v_{n+2-i,2n+1-j}=0$, and this gives us the description of $X^7$ in the fourth part of the proposition. Finally, in the same way as in the last two cases, \eqref{13.29} is equivalent to $v_{2n+1-i,2n+1-j}=0$. This completes the case where $1\leq j<i\leq n'$. The remaining two cases of the proof are done similarly, and we leave them for the reader.	
	
\end{proof}

We have a very similar description of the subgroup $V_{[\frac{m}{2}],[\frac{m+1}{2}]^n,[\frac{m}{2}]^{n-1}}(F_v)\cap w^{-1}_{\underline{k}^0}P_{(m+\mu_0)^{[\frac{n}{2}]},m^{n-[\frac{n}{2}]}, (m-\mu_0)^{[\frac{n}{2}]}}(F_v)w_{\underline{k}^0}$ when $n$ is odd. The difference from the previous case is very subtle, and, of course, we need this for the computation of the character $\delta_0$. The proof of the proposition is carried out exactly as in Prop. \ref{prop 13.6}, and we leave it to the reader.

\begin{prop}\label{prop 13.7}
	Assume that $n=2n'+1$ is odd. Then the subgroup\\ $V_{[\frac{m}{2}],[\frac{m+1}{2}]^n,[\frac{m}{2}]^{n-1}}(F_v)\cap w^{-1}_{\underline{k}^0}P_{(m+\mu_0)^{n'},m, (m-\mu_0)^{n'}}(F_v)w_{\underline{k}^0}$ consists of the matrices 
	\begin{equation}\label{13.30}
	v=\begin{pmatrix}I_{[\frac{m}{2}]}&X^1&X^2&X^3&X^4\\&I_{n'[\frac{m+1}{2}]}&0&X^5&0\\&&I_{(n'+1)[\frac{m+1}{2}]}&X^6&X^7\\&&&X^8&0\\&&&&I_{(n'+1)[\frac{m}{2}]}\end{pmatrix},
	\end{equation}
	where $X^1, X^3, X^6$ are arbitrary, and $X^2, X^4, X^5, X^7, X^8$ are as follows.\\
	\begin{enumerate}
		\item Write $X^2$ as a $n'\times (n'+1)$ matrix $(X^2_{i,j})$ of blocks $X^2_{i,j}\in M_{\mu_0\times [\frac{m+1}{2}]}(F_v)$. Then, for $3\leq j\leq n'+1$, $X^2_{i,j}=0$, for $n'-j+3\leq i\leq n'$.\\
		\item The matrix $X^4$ has a form similar to $X^2$. Thus, write $X^4$ as a $n'\times (n'+1)$ matrix $(X^4_{i,j})$ of blocks $X^4_{i,j}\in M_{\mu_0\times [\frac{m}{2}]}(F_v)$. Then, for $3\leq j\leq n'+1$, $X^4_{i,j}=0$, for $n'-j+3\leq i\leq n'$.\\
		\item The matrix $X^5$ has the same description as in Prop. \ref{prop 13.6}. Thus, write $X^5$ as a $n'\times (n'-1)$ matrix $(X^5_{i,j})$ of blocks $X^5_{i,j}\in M_{[\frac{m+1}{2}]\times [\frac{m}{2}]}(F_v)$. Then, for $1\leq i\leq n'-2$, and $i<j\leq n'-1$, $X^5_{i,j}=0$. Also, for $1\leq i\leq n'-1$, write $X^5_{i,i}=(x_i(1),...,x_i(n'))$, where $x_i(j')\in M_{\mu_0\times \mu_0}(F_v)$. Then, for $1\leq j'\leq i$, $x_i(j')=0$.\\
		\item Write $X^7$ as a $(n'+1)\times (n'+1)$ matrix $(X^7_{i,j})$ of blocks $X^7_{i,j}\in M_{[\frac{m+1}{2}]\times [\frac{m}{2}]}(F_v)$. Then, for $1\leq i<j\leq n'+1$, $X^7_{i,j}=0$.\\
		\item The matrix $X^8$ has the same the description as in Prop. \ref{prop 13.6}. Thus, write $X^8$ as a $(n'-1)\times (n'-1)$ matrix $(X^8_{i,j})$ of blocks $X^8_{i,j}\in M_{[\frac{m}{2}]\times [\frac{m}{2}]}(F_v)$. Then $X^8$ is upper unipotent, with $X^8_{i,i}=I_{[\frac{m}{2}]}$, $X^8_{i,i+i'}=0$, for $1\leq i\leq n'-2$, $2\leq i'\leq n-i$, and, for $1\leq i\leq n'-1$, let us write $X^8_{i,i+1}$ as a $n'\times n'$ matrix  $X^8_{i,i+1}=(x_i(i',j'))$ of blocks $x_i(i',j')\in M_{\mu_0\times \mu_0}(F_v)$. Then $x_i(i',j')=0$, unless $1\leq i'\leq i$ and $i+2\leq j'\leq n'$. 
	\end{enumerate}
	
\end{prop}
 
Let $\delta'_0$ be the character of $T\cap D_{\underline{k}^0}$ (see \eqref{13.17}) given by the Jacobian of the conjugation map of its elements on the subgroup\\
$V_{[\frac{m}{2}],[\frac{m+1}{2}]^n,[\frac{m}{2}]^{n-1}}(F_v)\cap w^{-1}_{\underline{k}^0}P_{(m+\mu_0)^{[\frac{n}{2}]},m^{n-[\frac{n}{2}]}, (m-\mu_0)^{[\frac{n}{2}]}}(F_v)w_{\underline{k}^0}$. Then, for\\
 $g\in T\cap D_{\underline{k}^0}$,
\begin{equation}\label{13.31}
\delta_0(g)=\frac{\delta_{P_{[\frac{m}{2}],[\frac{m+1}{2}]^n,[\frac{m}{2}]^{n-1}}}(g)}{\delta'_0(g)}.
\end{equation}
Let $\begin{pmatrix}h_1\\&h_2^{\Delta_n}\\&&h_3^{\Delta_{n-1}}\end{pmatrix}$ be an element of $T\cap D_{\underline{k}^0}$, where $h_1,h_3\in P_{\mu_0^{[\frac{n}{2}]}}(F_v)$,\\
 $h_2\in \GL_{[\frac{m+1}{2}]}(F_v)$. Then a simple calculation shows that
\begin{multline}\label{13.32}
 \delta_{P_{[\frac{m}{2}],[\frac{m+1}{2}]^n,[\frac{m}{2}]^{n-1}}}(\begin{pmatrix}h_1\\&h_2^{\Delta_n}\\&&h_3^{\Delta_{n-1}}\end{pmatrix})=\\
 |\det(h_1)|^{mn-[\frac{m}{2}]}|\det(h_2)|^{n(n-2)[\frac{m}{2}]}|\det(h_3)|^{-(n-1)([\frac{m}{2}]+n[\frac{m+1}{2}])}.
 \end{multline}
 \begin{prop}\label{prop 13.8}
 	\begin{enumerate}
 	\item Let $h_1=\begin{pmatrix}a_1\\&\ddots\\&&a_{[\frac{n}{2}]}\end{pmatrix}$ where
 	$a_i\in \GL_{\mu_0}(F_v)$. Then
 		\begin{equation}\label{13.33}
 		\delta'_0(\begin{pmatrix}h_1\\&I_{[\frac{m+1}{2}]}^{\Delta_n}\\&&I_{[\frac{m}{2}]}^{\Delta_{n-1}}\end{pmatrix})=|\det(h_1)|^{mn+[\frac{m+1}{2}]}\prod_{i=1}^{[\frac{n}{2}]}|\det(a_i)|^{-mi}.
 		\end{equation}
 	\item Let $h_2\in \GL_{[\frac{m+1}{2}]}(F_v)$. Then
 	\begin{equation}\label{13.34}
 	\delta'_0(\begin{pmatrix}I_{[\frac{m}{2}]}\\&h_2^{\Delta_n}\\&&I_{[\frac{m}{2}]}^{\Delta_{n-1}}\end{pmatrix})=\{\begin{matrix}|\det(h_2)|^{[\frac{m}{2}]\frac{n^2-2n-2}{2}},&& n&\textit{even}\\
 	|\det(h_2)|^{[\frac{m}{2}]\frac{n^2-2n-3}{2}},&&n&\textit{odd}\end{matrix}
 	\end{equation}
 	\item Let $h_3=\begin{pmatrix}c_1\\&\ddots\\&&c_{[\frac{n}{2}]}\end{pmatrix}$, where
 	$c_i\in \GL_{\mu_0}(F_v)$. Then
 	\begin{multline}\label{13.35}
 	\delta'_0(\begin{pmatrix}I_{[\frac{m}{2}]}\\&I_{[\frac{m+1}{2}]}^{\Delta_n}\\&&h_3^{\Delta_{n-1}}\end{pmatrix})=\\
 	=|\det(h_3)|^{-[\frac{m}{2}][\frac{n+1}{2}]-[\frac{m+1}{2}]\frac{n^2-n-2}{2}-\mu_0}\prod_{i=1}^{[\frac{n}{2}]}|\det(c_i)|^{-(m-2\mu_0)i}.
 	\end{multline} 		
 	\end{enumerate}
 \end{prop}
 \begin{proof}
Assume that $n=2n'$ is even. We compute the Jacobian of the conjugation by the element $diag(h_1,I_{[\frac{m+1}{2}]}^{\Delta_n},I_{[\frac{m}{2}]}^{\Delta_{n-1}})$ on the subgroup of elements \eqref{13.22} of Prop. \ref{prop 13.6}. Only $X^1,...,X^4$ contribute to the Jacobian. The contributions of $X^1$, $X^3$ are $|\det(h_1)|^{n'[\frac{m+1}{2}]}$, $|\det(h_1)|^{(n'-1)[\frac{m}{2}]}$, respectively. The contributions of $X^2$ and $X^4$ are, respectively,
$$
\prod_{i=1}^{n'}|\det(a_i)|^{[\frac{m+1}{2}](n'-i+1)},\ \   
\prod_{i=1}^{n'}|\det(a_i)|^{[\frac{m}{2}](n'-i+1)}.
$$
Altogether, we get
\begin{equation}\label{13.36}
|\det(h_1)|^{[\frac{m+1}{2}](n+1)+[\frac{m}{2}]n}\prod_{i=1}^{[\frac{n}{2}]}|\det(a_i)|^{-mi}=
|\det(h_1)|^{mn+[\frac{m+1}{2}]}\prod_{i=1}^{[\frac{n}{2}]}|\det(a_i)|^{-mi}.
\end{equation}
When $n=2n'+1$ is odd, we consider \eqref{13.30} in Prop. \ref{prop 13.7}. The contributions of $X^1$, $X^3$ are $|\det(h_1)|^{n'[\frac{m+1}{2}]}$, $|\det(h_1)|^{(n'-1)[\frac{m}{2}]}$, respectively. The contributions of $X^2$ and $X^4$ are, respectively,
$$
\prod_{i=1}^{n'}|\det(a_i)|^{[\frac{m+1}{2}](n'-i+2)},\ \   
\prod_{i=1}^{n'}|\det(a_i)|^{[\frac{m}{2}](n'-i+2)}.
$$
Altogether, we get that the product is as in \eqref{13.36}. This proves (1).

Consider, next, for $n=2n'$ even, the Jacobian of the conjugation by the element $diag(I_{[\frac{m}{2}]},h_2^{\Delta_n},I_{[\frac{m}{2}]}^{\Delta_{n-1}})$ on the subgroup of elements \eqref{13.22} of Prop. \ref{prop 13.6}. Here, only $X^1, X^2, X^5, X^6, X^7$ contribute. The contributions of $X^1$ and $X^6$ are $|\det(h_2)|^{-n'[\frac{m}{2}]}$, $|\det(h_2)|^{n'(n'-1)[\frac{m}{2}]}$, respectively. The contribution of $X^2$ is
\begin{equation}\label{13.37}
|\det(h_2)|^{-\mu_0(1+2+\cdots+n')}=|\det(h_2)|^{-[\frac{m}{2}]\frac{n'+1}{2}}.
\end{equation}
We used the fact that $[\frac{m}{2}]=\mu_0[\frac{n}{2}]=\mu_0n'$. The contribution of $X^5$ is 
\begin{equation}\label{13.38}
|\det(h_2)|^{([\frac{m}{2}]+\mu_0)(1+2+\cdots+(n'-1))}=|\det(h_2)|^{[\frac{m}{2}]\frac{(n')^2-1}{2}}.
\end{equation}
The contribution of $X^7$ is 
\begin{equation}\label{13.39}
|\det(h_2)|^{[\frac{m}{2}](1+2+\cdots+n')}=|\det(h_2)|^{[\frac{m}{2}]\frac{n'(n'+1)}{2}}.
\end{equation}
The product of these contributions is easily seen to be \eqref{13.34}. When $n=2n'+1$ is odd, we do the same with respect to the elements \eqref{13.30}. Here, the contributions of $X^1, X^6$ are $|\det(h_2)|^{-n'[\frac{m}{2}]}$, $|\det(h_2)|^{(n'+1)(n'-1)[\frac{m}{2}]}$, respectively. The contribution of $X^2$, analogous to \eqref{13.37}, is $|\det(h_2)|^{-[\frac{m}{2}]\frac{n'+3}{2}}$. The contribution of $X^5$ is as in \eqref{13.38}, and the contribution of $X^7$, analogous to \eqref{13.39}, is   
$|\det(h_2)|^{[\frac{m}{2}]\frac{(n'+1)(n'+2)}{2}}$. The product of all these contributions is equal to \eqref{13.34}. This proves (2).

Finally, consider, for $n=2n'$ even, the Jacobian of the conjugation by the element $diag(I_{[\frac{m}{2}]},I_{[\frac{m+1}{2}]}^{\Delta_n},h_3^{\Delta_{n-1}})$ on the subgroup of elements \eqref{13.22} of Prop. \ref{prop 13.6}. Here, only $X^3$--$X^8$ contribute. The contributions of $X^3$, $X^6$ are $\det(h_3)|^{-(n'-1)[\frac{m}{2}]}$, $\det(h_3)|^{-n'(n'-1)[\frac{m+1}{2}]}$, respectively. The contribution of $X^4$ is
\begin{equation}\label{13.40}
|\det(h_3)|^{-\mu_0(1+2+\cdots+n')}=|\det(h_3)|^{-[\frac{m}{2}]\frac{n'+1}{2}}.
\end{equation}
The contribution of $X^5$ is
\begin{multline}\label{13.41}
|\det(h_3)|^{-[\frac{m+1}{2}](1+2+\cdots+(n'-1))}\prod_{i=2}^{n'}|\det(c_ic_{i+1}\cdots c_{n'})|^{-[\frac{m+1}{2}]}=\\
=|\det(h_3)|^{-[\frac{m+1}{2}]\frac{n'(n'-1)}{2}}\prod_{i=1}^{n'}|\det(c_i)|^{-(i-1)[\frac{m+1}{2}]}.
\end{multline}
The contribution of $X^7$ is
\begin{equation}\label{13.42}
|\det(h_3)|^{-[\frac{m+1}{2}](1+2+\cdots+n')}=|\det(h_3)|^{-[\frac{m+1}{2}]\frac{n'(n'+1)}{2}}.
\end{equation}
The contribution of $X^8$ is
\begin{multline}\label{13.43}
\prod_{i=1}^{n'-2}\frac{|\det(c_1c_2\cdots c_i)|^{\mu_0(n'-i-1)}}{|\det(c_{i+2}c_{i+3}\cdots c_{n'})|^{i\mu_0}}=\frac{\prod_{i=1}^{n'-2}|\det(c_i)|^{\mu_0(1+2+\cdots+(n'-i-1))}}{\prod_{i=3}^{n'}|\det(c_i)|^{\mu_0(1+2+\cdots+(i-2))}}=\\
=\prod_{i=1}^{n'}|\det(c_i)|^{\frac{\mu_0n'(n'-1)}{2}-\mu_0-(\mu_0n'-2\mu_0)i}=|\det(h_3)|^{[\frac{m}{2}]\frac{n'-1}{2}-\mu_0}\prod_{i=1}^{n'}|\det(c_i)|^{-([\frac{m}{2}]-2\mu_0)i}.
\end{multline}
The product of the contributions of $X^3--X^8$ is \eqref{13.35}. When $n=2n'+1$ is odd, we do the same with respect to the elements \eqref{13.30}. Here, the contributions of $X^3$, $X^6$ are $\det(h_3)|^{-(n'-1)[\frac{m}{2}]}$, $\det(h_3)|^{-(n'-1)(n'+1)[\frac{m+1}{2}]}$, respectively. The contribution of $X^4$, analogous to \eqref{13.40} is $|\det(h_3)|^{-[\frac{m}{2}]\frac{n'+3}{2}}$. The contribution of $X^5$ is as \eqref{13.41}. The contribution of $X^7$, analogous to \eqref{13.42} is
$|\det(h_3)|^{-[\frac{m+1}{2}]\frac{(n'+1)(n'+2)}{2}}$, and the contribution of $X^8$ is as \eqref{13.43}. The product of the contributions of $X^3--X^8$ is \eqref{13.35}. This prove (3).

\end{proof}

We are close to completing the proof of Theorem \ref{thm 12.3}. Embed $\GL_{[\frac{m}{2}]}(F_v)\times \GL_{[\frac{m}{2}]}(F_v)$  inside $\GL_{[\frac{mn}{2}]}(F_v)$ by 
$$
\nu(h,g)=diag(h,\nu'(g)^{\Delta_n},(g^*)^{\Delta_{n-1}}),
$$
where $g^*=w_{[\frac{m}{2}]}{}tg^{-1}w^{-1}_{[\frac{m}{2}]}$, and $\nu'(g)=g$, if $m$ is even, and if $m$ is odd, then
 $\nu'(g)=\begin{pmatrix}g\\&1\end{pmatrix}$. Note that, by \eqref{12.73}, \eqref{12.74}, we need to know the restriction of $J_{(\psi_v)_{V_{[\frac{m}{2}],[\frac{m+1}{2}]^n,[\frac{m}{2}]^{n-1}}}}(\tau')$ to $\nu(\GL_{[\frac{m}{2}]}(F_v)\times \GL_{[\frac{m}{2}]}(F_v))$. The following theorem is a corollary of all the work we did in this section.
 \begin{thm}\label{thm 13.9}
 	We have an isomorphism 
 	\begin{equation}\label{13.44}
 	Res_{\GL_{[\frac{m}{2}]}(F_v)\times \GL_{[\frac{m}{2}]}(F_v)}(J_{(\psi_v)_{V_{[\frac{m}{2}],[\frac{m+1}{2}]^n,[\frac{m}{2}]^{n-1}}}}(\tau')\circ \nu)\cong
 	\end{equation}
 	\begin{multline}\nonumber 
 	( |\det\cdot|^{\frac{m(n-2)+[\frac{m+1}{2}]}{2}}\otimes |\det\cdot|^{\frac{mn(n-2)+[\frac{m+1}{2}]}{2}})\cdot\\
 	\Ind_{P_{\mu_0^{[\frac{n}{2}]}}(F_v)}^{\GL_{[\frac{m}{2}]}(F_v)}(\otimes_{i=1}^{[\frac{n}{2}]}(\chi_i\circ det_{\GL_{\mu_0}}))\otimes \Ind_{P_{\mu_0^{[\frac{n}{2}]}(F_v)}}^{\GL_{[\frac{m}{2}]}(F_v)}(\otimes_{i=1}^{[\frac{n}{2}]}(\chi_i\circ det_{\GL_{\mu_0}})).
 	\end{multline}
 As a result, we conclude the isomorphism \eqref{12.17} of Theorem \ref{thm 12.3}. 			
\end{thm}
 \begin{proof}
 By \eqref{13.12}, \eqref{13.16}, the restriction of $J_{(\psi_v)_{V_{[\frac{m}{2}],[\frac{m+1}{2}]^n,[\frac{m}{2}]^{n-1}}}}(\tau')$ to\\
 $\nu(\GL_{[\frac{m}{2}]}(F_v)\times \GL_{[\frac{m}{2}]}(F_v))$ is isomorphic to the restriction of $ind_{T\cap D_{\underline{k}^0}}^T\delta_0(\delta^{\frac{1}{2}}\chi_{\tau'})^{w_{\underline{k}^0}}$ to $\nu(\GL_{[\frac{m}{2}]}(F_v)\times \GL_{[\frac{m}{2}]}(F_v))$. This restriction is the non-normalized induction from the character of $\nu(P_{\mu_0^{[\frac{n}{2}]}}(F_v)\times P_{\mu_0^{[\frac{n}{2}]}}(F_v))$ given by $\delta_0(\delta^{\frac{1}{2}}\chi_{\tau'})^{w_{\underline{k}^0}}$. Of course, it is trivial on the unipotent radical  $V_{\mu_0^{[\frac{n}{2}]}}(F_v)\times V_{\mu_0^{[\frac{n}{2}]}}(F_v)$, and we need to find this character on the Levi part. For this, we need to multiply the characters obtained from \eqref{13.18} and \eqref{13.32}, and then divide by the character explicated in Prop. \ref{prop 13.8}. The calculation is straightforward. Let us do this for the first coordinate. Let $h=diag(a_1,...,a_{[\frac{n}{2}]})$, $a_i\in \GL_{\mu_0}(F_v)$. Recall that $[\frac{m}{2}]=\mu_0[\frac{n}{2}]$. By \eqref{13.18}, \eqref{13.32} and \eqref{13.33},
 \begin{multline}\nonumber
  \delta_0(\delta^{\frac{1}{2}}\chi_{\tau'})^{w_{\underline{k}^0}}(\nu(h,I_{[\frac{m}{2}]}))=
  |\det(h)|^{\frac{mn}{2}+(mn-[\frac{m}{2}])+(-mn-[\frac{m+1}{2}])}\cdot\\
 \cdot  \prod_{i=1}^{[\frac{n}{2}]}|\det(a_i)|^{-(m+\mu_0)\frac{2i-1}{2}+mi}\chi_i(\det(a_i))=\\
  |\det(h)|^{\frac{mn}{2}-m}\prod_{i=1}^{[\frac{n}{2}]}|\det(a_i)|^{-\mu_0\frac{2i-1}{2}+\frac{m}{2}}\chi_i(\det(a_i))=\\
  |\det(h)|^{\frac{m(n-1)}{2}}\prod_{i=1}^{[\frac{n}{2}]}|\det(a_i)|^{-\mu_0\frac{2i-1}{2}}\chi_i(\det(a_i))=\\
  	|\det(h)|^{\frac{m(n-2)}{2}+\frac{1}{2}[\frac{m+1}{2}]}\prod_{i=1}^{[\frac{n}{2}]}|\det(a_i)|^{\mu_0([\frac{n}{2}]-\frac{2i-1}{2})}\chi_i(\det(a_i))=\\
  	|\det(h)|^{\frac{m(n-2)+[\frac{m+1}{2}]}{2}}\delta^{\frac{1}{2}}_{P_{\mu_0^{[\frac{n}{2}]}}}(h)\prod_{i=1}^{[\frac{n}{2}]}\chi_i(\det(a_i)).
  	\end{multline}
  We get a similar calculation for the second coordinate. Let $g=diag(b_1,...,b_{[\frac{n}{2}]})$, $b_i\in \GL_{\mu_0}(F_v)$,
  $$
  \delta_0(\delta^{\frac{1}{2}}\chi_{\tau'})^{w_{\underline{k}^0}}(\nu(I_{[\frac{m}{2}]},g))=
  |\det(g)|^{\frac{mn(n-2)+[\frac{m+1}{2}]}{2}}\delta^{\frac{1}{2}}_{P_{\mu_0^{[\frac{n}{2}]}}}(g)\prod_{i=1}^{[\frac{n}{2}]}\chi_i(\det(b_i)).
  $$	  
We leave this part for the reader. Note that if we write $g^*=\diag(c_1,...,c_{[\frac{n}{2}]})$, $c_i\in \GL_{\mu_0}(F_v)$, then $b_i=c^*_{[\frac{n}{2}]-i+1}$, and hence $\prod_{i=1}^{[\frac{n}{2}]}\chi^{-1}_{[\frac{n}{2}]-i+1}(\det(c_i))=\prod_{i=1}^{[\frac{n}{2}]}\chi_i(\det(b_i))$.

Finally, we obtain the isomorphism \eqref{12.17} of Theorem \ref{thm 12.3}. To see this, we review briefly the work done in this section and the previous one. We started with the isomorphism \eqref{12.19},
$$
J_{(\psi_v)_{U_{m^{n-1}}(F_v)}}(\rho_{\chi,\gamma_{\psi_v}^{(\epsilon)},\eta,k_0})\cong J_{(\psi_v)_{U_{m^{n-1}}(F_v)}}(\Ind^{H(F_v)}_{Q^{(\epsilon)}_{mn}(F_v)}\tau'\gamma_{\psi_v}^{(\epsilon)}).
$$
Then, by Theorem \ref{thm 12.4} and Prop. \ref{prop 12.5}, we got \eqref{12.52},
\begin{multline}\nonumber
J_{(\psi_v)_{U_{m^{n-1}}(F_v)}}(\Ind^{H(F_v)}_{Q^{(\epsilon)}_{mn}(F_v)}\tau'\gamma_{\psi_v}^{(\epsilon)})\cong\\ ind^{t^{(\epsilon)}(H^{(\epsilon)}_m(F_v)\times H^{(\epsilon)}_m(F_v))}_{t^{(\epsilon)}(H^{(\epsilon)}_m(F_v)\times H^{(\epsilon)}_m(F_v))\cap Q^{(0),(\epsilon)}(F_v)}\sigma_{[\tau',[\frac{m}{2}]]}.
\end{multline}
In the remaining part of the previus section, we showed that the last representation is isomorphic to the non-normalized parabolic induction from the representation $\pi_{\tau'}$ (\eqref{12.73}, \eqref{12.74}) of
$t^{(\epsilon)}(Q_{[\frac{m}{2}]}(F_v)\times Q_{[\frac{m}{2}]}(F_v))$. Note, that in the linear case, for $h,g\in \GL_{[\frac{m}{2}]}(F_v)$,
$$
\pi_{\tau'}(t(\hat{h},\hat{g}))=|\det(h)|^{\frac{m(2-n)-\delta_H}{2}}|\det(g)|^{\frac{mn(2-n)-\delta_H}{2}}J_{(\psi_v)_{V_{[\frac{m}{2}],[\frac{m+1}{2}]^n,[\frac{m}{2}]^{n-1}}}}(\tau')(\nu(h,g)).
$$
In the metaplectic case, we need to multiply by $\gamma_{\psi_v}(det(h))\gamma_{\psi_v}(det(g))$. Now, by \eqref{13.44}, we get the isomorphism \eqref{12.17}. This completes the proof of the theorem. This also completes the proof of Theorem \ref{thm 12.3}, a special case of which is Theorem \ref{thm 12.2}.

\end{proof}

\end{document}